\documentclass[letterpaper,11pt,oneside,reqno]{amsart}
\usepackage[utf8]{inputenc}%
\usepackage[english]{babel}%
\usepackage{amsmath,amssymb,amsthm,amsfonts}%
\usepackage{hyperref}%
\usepackage{enumerate}%
\usepackage{array}%
\usepackage{graphicx}
\usepackage[mathscr]{euscript}
\usepackage{color,tikz}
\usetikzlibrary{calc}
\usetikzlibrary{patterns}
\usetikzlibrary{arrows}
\usetikzlibrary{decorations.pathreplacing}
\usetikzlibrary{decorations.pathmorphing}
\usetikzlibrary{arrows}
\usepackage[DIV12]{typearea}
\usepackage{adjustbox}
\usepackage[width=.9\textwidth]{caption}
\allowdisplaybreaks%
\numberwithin{equation}{section}%

\newcommand{\Z}{\mathbb{Z}}
\newcommand{\C}{\mathbb{C}}
\newcommand{\R}{\mathbb{R}}
\DeclareMathOperator{\E}{\mathbb{E}}
\renewcommand{\i}{\mathbf{i}}

\newcommand{\al}{\alpha}

\newcommand{\la}{\lambda}
\newcommand{\La}{\Lambda}
\newcommand{\be}{\beta}
\newcommand{\ga}{\gamma}

\newcommand{\Om}{\boldsymbol\phi}

\DeclareMathOperator{\prob}{\mathrm{Prob}}

\let\oldphi\phi \let\phi\varphi \let\varphi\oldphi

\newcommand{\GT}{\mathbb{GT}^+}

\newcommand{\MP}{\mathscr{M}}
\newcommand{\MM}{\mathscr{MM}}

\newcommand{\xx}{\mathbf{x}}

\newcommand{\alb}{\boldsymbol\al}
\newcommand{\beb}{\boldsymbol\be}
\newcommand{\lab}{\boldsymbol\la}
\newcommand{\nub}{\boldsymbol\nu}

\newcommand{\Sym}{\mathsf{Sym}}

\newcommand{\prech}{\prec_{\mathsf{h}}}
\newcommand{\succh}{\succ_{\mathsf{h}}}
\newcommand{\precv}{\prec_{\mathsf{v}}}

\newcommand{\gap}{\mathsf{gap}}
\newcommand{\p}{\mathsf{p}}

\newcommand{\Qspec}[1]{\mathscr{Q}[{#1}]}
\newcommand{\Pspec}[1]{\mathscr{P}[{#1}]}
\newcommand{\Qapb}{\mathscr{Q}^{q=0}_{\textnormal{pb}}[\al]}
\newcommand{\Qbpb}{\mathscr{Q}^{q=0}_{\textnormal{pb}}[\hat\be]}
\newcommand{\Qarow}{\mathscr{Q}^{q=0}_{\textnormal{row}}[\al]}
\newcommand{\Qacol}{\mathscr{Q}^{q=0}_{\textnormal{col}}[\al]}
\newcommand{\Qbrow}{\mathscr{Q}^{q=0}_{\textnormal{row}}[\hat\be]}
\newcommand{\Qbcol}{\mathscr{Q}^{q=0}_{\textnormal{col}}[\hat\be]}
\newcommand{\Qqarow}{\mathscr{Q}^{q}_{\textnormal{row}}[\al]}
\newcommand{\Qqacol}{\mathscr{Q}^{q}_{\textnormal{col}}[\al]}
\newcommand{\Qqbrow}{\mathscr{Q}^{q}_{\textnormal{row}}[\hat\be]}
\newcommand{\Qqbcol}{\mathscr{Q}^{q}_{\textnormal{col}}[\hat\be]}

\newcommand{\muq}{\xi}

\newcommand{\GRV}[1]{\mathop{\mathrm{Gamma}}(#1)}
\newcommand{\GRVI}[1]{\mathop{\mathrm{Gamma}^{-1}}(#1)}

\newcommand{\LL}{\mathsf{L}}
\newcommand{\RR}{\mathsf{R}}

\newcommand{\xxi}{\chi}

\newcounter{la1}
\newcounter{la2}
\newcounter{la3}
\newcounter{la4}
\newcounter{la5}
\newcounter{la6}
\newcounter{lab1}
\newcounter{lab2}
\newcounter{lab3}
\newcounter{lab4}
\newcounter{lab5}

\makeatletter
\newcommand{\proofstep}[1]{%
  \par
  \addvspace{\medskipamount}
  \noindent\textbf{#1\@addpunct{.}}\enspace\ignorespaces
}

\newcommand{\proofa}[1]{%
  \par
  \addvspace{\medskipamount}
  \noindent\textbf{#1\@addpunct{)}}\enspace\ignorespaces
}
\makeatother

\newtheorem{proposition}{Proposition}[section]
\newtheorem{lemma}[proposition]{Lemma}
\newtheorem{corollary}[proposition]{Corollary}
\newtheorem{theorem}[proposition]{Theorem}

\theoremstyle{definition}
\newtheorem{definition}[proposition]{Definition}
\newtheorem{remark}[proposition]{Remark}

\usetikzlibrary{matrix}

\begin{document}

\title[$q$-randomized RSK correspondences and random polymers]
{$q$-randomized Robinson--Schensted--Knuth correspondences and random polymers}

\author[K. Matveev]{Konstantin Matveev}
\address{K. Matveev,
Harvard University, Department of Mathematics,
1 Oxford Street, Cambridge, MA 02138, USA}
\email{kosmatveev@gmail.com}

\author[L. Petrov]{Leonid Petrov}
\address{L. Petrov, University of Virginia, Department of Mathematics,
141 Cabell Drive, Kerchof Hall,
P.O. Box 400137,
Charlottesville, VA 22904, USA,
and Institute for Information Transmission Problems, Bolshoy Karetny per. 19, Moscow, 127994, Russia}
\email{lenia.petrov@gmail.com}

\date{}

\begin{abstract}
	We introduce and study $q$-randomized 
	Robinson--Schensted--Knuth (RSK) correspondences
	which interpolate between the classical
	($q=0$) and geometric ($q\nearrow1$)
	RSK correspondences (the latter ones 
	are sometimes also called tropical).

	For $0<q<1$
	our correspondences are randomized, i.e., 
	the result of an insertion is a certain
	probability distribution on semistandard Young tableaux.
	Because of this randomness, we use the language
	of discrete time Markov dynamics on two-dimensional
	interlacing particle arrays 
	(these arrays are in a natural bijection with semistandard tableaux).
	Our dynamics act nicely on a certain class of 
	probability measures on arrays, 
	namely, on $q$-Whittaker processes
	(which are $t=0$ versions of Macdonald processes
	of Borodin--Corwin \cite{BorodinCorwin2011Macdonald}).
	We present four Markov dynamics	which for $q=0$
	reduce to the classical 
	row or column RSK correspondences 
	applied to a random 
	input matrix with independent
	geometric or Bernoulli entries.
	
	Our new two-dimensional discrete time dynamics
	generalize and extend several known constructions:
	(1)~The discrete time $q$-TASEPs studied by Borodin--Corwin
	\cite{BorodinCorwin2013discrete}
	arise as one-dimen\-sional marginals of our ``column'' dynamics.
	In a similar way, our ``row'' dynamics lead to discrete time
	$q$-PushTASEPs --- new
	integrable particle systems in the 
	Kardar--Parisi--Zhang universality class.
	We employ these new one-dimen\-sional discrete time systems to 
	establish a Fredholm determinantal formula
	for the two-sided continuous time $q$-PushASEP
	conjectured by Corwin--Petrov \cite{CorwinPetrov2013}.
	(2)~In a certain Poisson-type limit (from discrete to continuous time),
	our two-dimensional 
	dynamics reduce to the $q$-randomized column and row Robinson--Schensted
	correspondences introduced by O'Connell--Pei \cite{OConnellPei2012}
	and Borodin--Petrov \cite{BorodinPetrov2013NN}, respectively.
	(3)~In a scaling limit as $q\nearrow1$, 
	two of our four dynamics on interlacing arrays
	turn into the geometric RSK correspondences	
	associated with log-Gamma
	(introduced by Sepp\"al\"ainen \cite{Seppalainen2012}) 
	or strict-weak (introduced
	independently
	by O'Connell--Ortmann
	\cite{OConnellOrtmann2014}
	and
	Corwin--Sepp\"al\"ainen--Shen
	\cite{CorwinSeppalainenShen2014}) 
	directed random lattice polymers.
\end{abstract}

\maketitle

\setcounter{tocdepth}{1}
\tableofcontents
\setcounter{tocdepth}{3}

\newpage
\section{Introduction} 
\label{sec:introduction}

\subsection{Overview} 
\label{sub:overview}

The classical Robinson--Schensted--Knuth (RSK) correspondence
associates to an integer matrix a pair of semistandard Young tableaux
of the same shape 
\cite{Knuth1970},
\cite{fulton1997young}, 
\cite{Stanley1999},
\cite{sagan2001symmetric}.
It is informative to view 
an integer matrix $M=(M_{ij})$ 
as a configuration
of points (``balls'') in cells of the lattice $\Z_{\ge0}^{2}$,
with $M_{ij}$ balls in the $(i,j)$-th cell (see Fig.~\ref{fig:RSK}, 
left).\begin{figure}[htbp]
	\begin{tabular}{ccc}
	\begin{adjustbox}{max height=.3\textwidth}
	\begin{tikzpicture}
		[scale=.7, thick]
		\draw[->, thick] (0,0) -- (4,0) node[right] {$j$};
		\draw[->, thick] (0,0) -- (0,4) node[above] {$i$};
		\foreach \h in {1,2,3}
		{
			\draw (0,\h)--++(3.5,0);
			\draw (\h,0)--++(0,3.5);
			\node[left] at (0,\h-1/2) {$\h$};
			\node[below] at (\h-1/2,0) {$\h$};
		}
		\draw[fill] (2.5,2.5) circle (4pt);
		\draw[fill] (2.25,1.5) circle (4pt);
		\draw[fill] (2.75,1.5) circle (4pt);
		\draw[fill] (.75,2.5) circle (4pt);
		\draw[fill] (.25,2.5) circle (4pt);
		\draw[fill] (.5,.5) circle (4pt);
		\draw[fill] (1.25,1.25) circle (4pt);
		\draw[fill] (1.75,1.75) circle (4pt);
		\draw[fill] (1.75,1.25) circle (4pt);
		\draw[fill] (1.25,1.75) circle (4pt);
		\node[anchor=east] at (5,4.5) 
		{\scriptsize$M=\begin{pmatrix}
			1&0&0\\
			0&4&2\\
			2&0&1
		\end{pmatrix}$};
	\end{tikzpicture}
	\end{adjustbox}
	&
	\begin{adjustbox}{max height=.3\textwidth}
	\begin{tikzpicture}
		[scale=.7, thick]
		\node[anchor=east] at (5,4.5) {$w=(12313)$};
		\draw[->, thick] (0,0) -- (4,0) node[right] {$j$ cont.};
		\draw[->, thick] (0,0) -- (0,4) node[above] {$i$};
		\foreach \h in {1,2,3}
		{
			\draw (0,\h)--++(3.5,0);
			\node[left] at (0,\h-1/2) {$\h$};
		}
		\draw[fill] (2.5,2.5) circle (4pt);
		\draw[fill] (.25,.5) circle (4pt);
		\draw[fill] (1.2,1.5) circle (4pt);
		\draw[fill] (1.4,2.5) circle (4pt);
		\draw[fill] (1.8,.5) circle (4pt);
		\draw[dotted] (2.5,0)--++(0,3);
		\draw[dotted] (.25,0)--++(0,3);
		\draw[dotted] (1.2,0)--++(0,3);
		\draw[dotted] (1.4,0)--++(0,3);
		\draw[dotted] (1.8,0)--++(0,3);
		\node[below] at (1-1/2,0) {\phantom{$1$}};
	\end{tikzpicture}
	\end{adjustbox}
	&
	\begin{adjustbox}{max height=.3\textwidth}
	\begin{tikzpicture}
		[scale=.7, thick]
		\node[below] at (1-1/2,0) {\phantom{$1$}};
		\node[anchor=east] at (5,4.5) {$\sigma=(514362)$};
		\draw[->, thick] (0,0) -- (4,0) node[right] {$j$ cont.};
		\draw[->, thick] (0,0) -- (0,4) node[above] {$i$ cont.};
		\draw[fill] (1.2,1.5) circle (4pt);
		\draw[fill] (1,.3) circle (4pt);
		\draw[fill] (2.2,2.5) circle (4pt);
		\draw[fill] (1.8,1.1) circle (4pt);
		\draw[fill] (.3,2.1) circle (4pt);
		\draw[fill] (2.9,.6) circle (4pt);
		\draw[dotted] (2.9,0)--++(0,3);
		\draw[dotted] (.3,0)--++(0,3);
		\draw[dotted] (1.8,0)--++(0,3);
		\draw[dotted] (2.2,0)--++(0,3);
		\draw[dotted] (1,0)--++(0,3);
		\draw[dotted] (1.2,0)--++(0,3);
		\draw[dotted] (0,1.5)--++(3,0);
		\draw[dotted] (0,.3)--++(3,0);
		\draw[dotted] (0,2.5)--++(3,0);
		\draw[dotted] (0,1.1)--++(3,0);
		\draw[dotted] (0,2.1)--++(3,0);
		\draw[dotted] (0,.6)--++(3,0);
	\end{tikzpicture}
	\end{adjustbox}
	\end{tabular}
	\caption{Left: an integer matrix 
	as an input to the RSK. 
	Center: an integer word as an input into the RS viewed as a matrix with the continuous $j$ coordinate 
	(at most one ball at a given horizontal position is allowed; the word encodes vertical positions of
	consecutive balls).
	Right: a permutation viewed as a matrix with both continuous coordinates
	(at most one ball at a given horizontal or vertical position is allowed;
	balls represent the graph of the permutation).}
	\label{fig:RSK}
\end{figure}
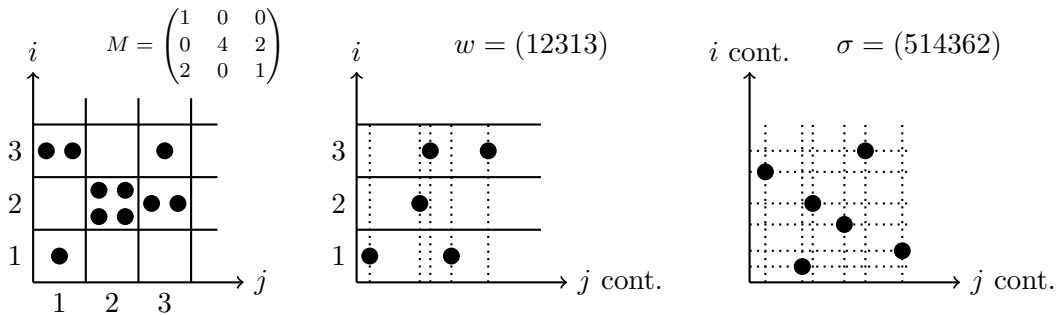
There are also simpler 
correspondences obtained from the RSK 
if one makes one or both dimensions
of the input \emph{continuous}, see
Fig.~\ref{fig:RSK}, center and right.
In particular, the Robinson--Schensted (RS)
correspondence maps integer words into pairs of Young tableaux
of the same shape, but now one 
of them is standard.

The idea of applying the RSK correspondence to a \emph{random} input can be traced back to
\cite{Vershik1986} where it was used in connection with the asymptotic theory 
of characters of the infinite symmetric group (see also \cite{BufetovCLT}). 
Together with combinatorial properties
of the RSK this idea has been extensively employed in studying various stochastic systems,
e.g., TASEP (=~totally asymmetric simple exclusion process), the
last-passage percolation \cite{johansson2000shape}, or
longest increasing subsequences of random permutations
\cite{baik1999distribution}, \cite{Baik1999second}.

Reading the random input matrix column by column adds a dynamical
perspective to random systems (with $j$ in all three cases on Fig.~\ref{fig:RSK} playing the role of time).
This direction has been substantially developed in, e.g., 
\cite{OConnell2003Trans}, \cite{OConnell2003}, 
\cite{BBO2004}.

The geometric version 
(also sometimes called ``\emph{tropical}'')
of the RS and the RSK 
correspondences\footnote{The geometric RSK maps 
arrays of positive real numbers into 
other such arrays in a birational way,
and is obtained from the classical RSK by a 
certain ``detropicalization'', see
\cite{Kirillov2000_Tropical},
\cite{NoumiYamada2004}.}
has also been
employed in the study of stochastic systems
\cite{Oconnell2009_Toda},
\cite{COSZ2011}, \cite{OSZ2012}, \cite{OConnell2013geomToda}. 
The systems one obtains at this level are related to directed
random polymers in random media, in particular, to the 
O'Connell--Yor,
log-Gamma,
and strict-weak 
random polymers
introduced in \cite{OConnellYor2001},
\cite{Seppalainen2012}, and 
\cite{OConnellOrtmann2014},
\cite{CorwinSeppalainenShen2014}, respectively.
Each such polymer model 
can be viewed as a 
positive temperature version 
of a certain last-passage percolation-like model. 

In the stochastic systems mentioned above, the RSK and related constructions 
provide a way to observe and understand their \emph{integrability}. 
The integrability property refers to the
presence of concise and exact formulas describing observables,
which allows to study the asymptotic behavior of such systems, 
and also gives access to exact descriptions of limiting universal
distributions, such as the Tracy-Widom distributions
which are features of the Kardar--Parisi--Zhang (KPZ) universality class
\cite{CorwinKPZ}, 
\cite{BorodinGorinSPB12}, \cite{BorodinPetrov2013Lect}
\cite{QuastelSpohnKPZ2015}.

\medskip

The classical RSK is deeply connected to Schur symmetric functions
\cite[Ch. I]{Macdonald1995}, while the geometric RSK
is relevant to the $\mathfrak{gl}_{n}$ Whittaker functions \cite{Kostant1977Whitt}, 
\cite{Etingof1999}.
Both families of functions are degenerations of more general
Macdonald symmetric functions depending on two parameters
$(q,t)$ \cite[Ch. VI]{Macdonald1995}:
the Schur functions correspond to $q=t$, and the Whittaker functions
arise in the limit as $t=0$ and $q\nearrow1$, \cite{GerasimovLebedevOblezin2011}.

In the recent years, there has been 
a progress in understanding analogues of the RS
correspondences at other levels of the Macdonald hierarchy: 
$q$-Whittaker ($t=0$ and $0<q<1$)
\cite{OConnellPei2012}, \cite{Pei2013Symmetry},
\cite{BorodinPetrov2013NN}
and Hall--Littlewood ($q=0$ and $0<t<1$)
\cite{BufetovPetrov2014}. 
At these levels, the correspondences become \emph{randomized},
that is, the image of a deterministic word (as on Fig.~\ref{fig:RSK}, center)
is no longer a fixed pair of Young tableaux, but rather a 
\emph{random} such pair. Because of this randomness, 
an appropriate language for describing the correspondences 
seems to be that of \emph{Markov dynamics} on
two-dimensional
interlacing integer arrays (these arrays are in a natural bijection with semistandard tableaux,
see Remark \ref{rmk:SSYT} below for 
more detail).
The dynamics which are analogues of the RS correspondences
evolve in continuous time according to the $j$ axis 
on Fig.~\ref{fig:RSK}, center.
These dynamics act nicely on certain families of probability distributions
on interlacing arrays, namely, the Macdonald processes \cite{BorodinCorwin2011Macdonald}.

The $q$-Whittaker level is relevant to
integrable one-dimensional
particle systems such as 
(continuous time) $q$-TASEP and the stochastic $q$-Boson
system \cite{SasamotoWadati1998}, \cite{BorodinCorwin2011Macdonald},
\cite{BorodinCorwinSasamoto2012},
\cite{BorodinCorwinPetrovSasamoto2013}, \cite{FerrariVeto2013},
and (continuous time) $q$-PushTASEP (=~$q$-deformed pushing TASEP) 
\cite{CorwinPetrov2013}.\footnote{These 
systems are in fact quantum integrable in the sense of the coordinate Bethe ansatz
\cite{Bethe1931}, \cite{Lieb67}, \cite{baxter2007exactly}, \cite{BorodinCorwinPetrovSasamoto2013}.}
In particular, continuous time Markov dynamics
on interlacing arrays constructed in 
\cite{OConnellPei2012} and
\cite{BorodinPetrov2013NN}
are two-dimensional extensions of, respectively, 
the $q$-TASEP and the $q$-PushTASEP. That is,
the latter one-dimensional processes
are Markovian marginals of the dynamics on two-dimensional interlacing 
arrays.\footnote{The two-dimensional dynamics at the Hall--Littlewood 
level \cite{BufetovPetrov2014}, however, 
do not seem to lead to any new one-dimensional
integrable particle systems.}

\medskip

In the present paper we advance further at the $q$-Whittaker level, and 
introduce four $q$-randomized RSK correspondences,
or, in other words,
four discrete time Markov dynamics on interlacing
arrays which act nicely on $q$-Whittaker processes (these are Macdonald
processes with $t=0$).
These dynamics unify, generalize and extend all of the above RSK-type constructions:

\smallskip

$\bullet$ When $q=0$, our four $q$-randomized correspondences become
	usual or dual, row or column classical RSKs (four classical correspondences in total). 
	The input matrix $M$ in the usual RSKs has $M_{ij}\in\{0,1,2,\ldots\}$,
	and in the dual RSKs one has $M_{ij}\in\{0,1\}$. When one takes 
	$M_{ij}$ to be independent geometric (for usual) or Bernoulli (for dual)
	random variables and applies a suitable classical RSK, the 
	shape of the resulting random Young diagram
	is distributed according to the Schur measure
	\cite{okounkov2001infinite}.\footnote{In the present paper, the word ``geometric''
	is attached to two separate concepts --- the geometric RSKs, and 
	the geometric and $q$-geometric random variables. 
	To avoid confusion where it can occur, we will call the correspondences
	the geometric (tropical) RSKs.
	See also Remark \ref{rmk:names_notation}.} Similarly, 
	our $q$-randomized correspondences applied to $q$-geometric or Bernoulli
	random inputs 
	(note that the Bernoulli input needs not to be $q$-deformed)
	give rise to $q$-Whittaker distributed random Young diagrams.
	The latter property is an instance of ``acting nicely'' on $q$-Whittaker processes (see also \eqref{adding_spec} in \S \ref{sec:macdonald_processes_and_markov_dynamics}
	for 
	more detail).

$\bullet$ 
	In a limit from discrete to continuous time, 
	our $q$-randomized RSKs turn into the (simpler)
	$q$-randomized RS correspondences 
	introduced and studied in \cite{OConnellPei2012},
	\cite{BorodinPetrov2013NN}.

$\bullet$ 
	The two discrete time $q$-TASEPs (associated with $q$-geometric
	or Bernoulli random variables) studied by Borodin--Corwin
	\cite{BorodinCorwin2013discrete}
	arise as one-dimensional marginals of our two
	``column'' dynamics on interlacing arrays.
	In a similar way, our two ``row'' dynamics lead to discrete time
	$q$-PushTASEPs --- new
	integrable particle systems in the 
	KPZ universality class.

$\bullet$
	In a scaling limit as $q\nearrow1$, the dynamics on interlacing
	arrays associated with the $q$-geometric random input
	(these are two out of our four $q$-randomized RSK correspondences)
	converge to geometric (tropical) RSK correspondences.
	The latter correspondences (which are deterministic birational maps 
	between arrays of positive reals)
	are relevant to the 
	log-Gamma \cite{Seppalainen2012}, \cite{COSZ2011}, \cite{OSZ2012}
	and strict-weak 
	\cite{CorwinSeppalainenShen2014}
	random lattice polymers.
	
\smallskip

In \S \ref{sub:_q_randomized_row_insertion_with_q_geometric_input} below
we describe one of our four dynamics in detail, and in \S \ref{sub:other_results}
we briefly discuss other dynamics and results.


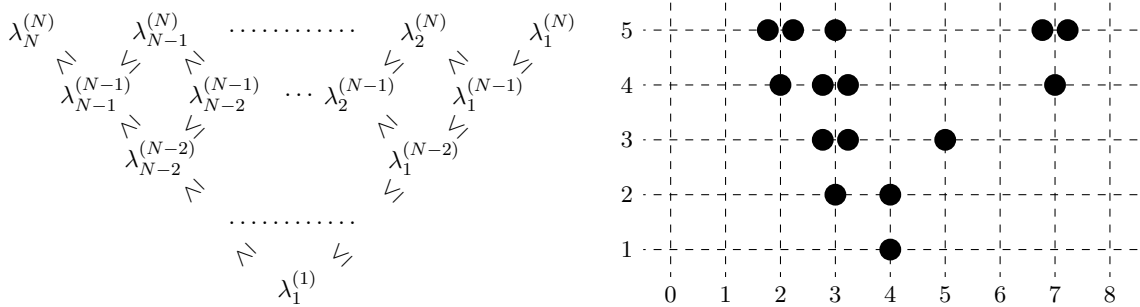
\begin{figure}[htbp]
	\begin{tabular}{cc}
	\begin{adjustbox}{max width=.48\textwidth}
		\begin{tikzpicture}[scale=1]
			\def\h{1}
			\def\x{1.9}
			\node at (-5,5) {$\la_N^{(N)}$};
			\node at (-3,5) {$\la_{N-1}^{(N)}$};
			\node at (0-\x/2,5) {$\ldots\ldots\ldots\ldots$};
			\node at (3-\x,5) {$\la_{2}^{(N)}$};
			\node at (5-\x,5) {$\la_{1}^{(N)}$};
			\node at (-4,5-\h) {$\la_{N-1}^{(N-1)}$};
			\node at (-2,5-\h) {$\la_{N-2}^{(N-1)}$};
			\node at (0-\x/2+.1,5-\h) {$\ldots$};
			\node at (2-\x,5-\h) {$\la_{2}^{(N-1)}$};
			\node at (4-\x,5-\h) {$\la_{1}^{(N-1)}$};
			\node at (-3,5-2*\h) {$\la_{N-2}^{(N-2)}$};
			\node at (3-\x,5-2*\h) {$\la_{1}^{(N-2)}$};
			\foreach \LePoint in {(4.5-\x,5-\h/2),(2.5-\x,5-\h/2),(-3.5,5-\h/2),
			(-2.5,5-3*\h/2),(3.5-\x,5-3*\h/2),(2.5-\x,5-5*\h/2),(1.7-\x,5-7*\h/2)} {
				\node [rotate=45] at \LePoint {$\le$};
			};
			\foreach \GePoint in {(3.5-\x,5-\h/2),(-4.5,5-\h/2),(-2.5,5-\h/2),
			(-3.5,5-3*\h/2),(-2.5,5-5*\h/2),(.5,5-3*\h/2),(-1.7,5-7*\h/2)} {
				\node [rotate=135] at \GePoint {$\ge$};
			};
			\node at (0-\x/2,5-3*\h) {$\ldots\ldots\ldots\ldots$};
			\node at (0-\x/2+.08,5-4*\h) {$\la_1^{(1)}$};
		\end{tikzpicture}
	\end{adjustbox}
	&
	\begin{tikzpicture}[scale = .73]
			\tikzset{every node/.style={
				font=\footnotesize
			}}
			\foreach \yy in {1,2,3,4,5}
			{
				\draw[dashed] (8.5, \yy) -- (-0.5,\yy) node[left] {\yy};
			};
			\foreach \xxy in {0,1,2,3,4,5,6,7,8}
			{
				\draw[dashed]  (\xxy,5.5) -- (\xxy,.5) node[below] {\xxy};
			};
			\def\x{0.16}
			\def\y{0.23}
			\def\ys{-1}
			\def\xs{7}
			\def\xss{-3}
			\draw[ultra thick, fill] (4,1) circle(\x);
			\draw[ultra thick, fill] (3,2) circle(\x);
			\draw[ultra thick, fill] (4,2) circle(\x);
			\draw[ultra thick, fill] (3-\y,3) circle(\x);
			\draw[ultra thick, fill] (3+\y,3) circle(\x);
			\draw[ultra thick, fill] (5,3) circle(\x);
			\draw[ultra thick, fill] (2,4) circle(\x);
			\draw[ultra thick, fill] (3-\y,4) circle(\x);
			\draw[ultra thick, fill] (3+\y,4) circle(\x);
			\draw[ultra thick, fill] (7,4) circle(\x);
			\draw[ultra thick, fill] (2-\y,5) circle(\x);
			\draw[ultra thick, fill] (2+\y,5) circle(\x);
			\draw[ultra thick, fill] (3,5) circle(\x);
			\draw[ultra thick, fill] (7-\y,5) circle(\x);
			\draw[ultra thick, fill] (7+\y,5) circle(\x);
		\end{tikzpicture}
	\end{tabular}
	\caption{Left: An interlacing array $\lab$; 
	we require that $\la^{(j)}_{i}\in\Z_{\ge0}$.
	Right:
	A configuration of particles corresponding to
	an interlacing array of depth $N=5$ (right).}
	\label{fig:array}
\end{figure}

\subsection{$q$-randomized row insertion with $q$-geometric input} 
\label{sub:_q_randomized_row_insertion_with_q_geometric_input}

Discrete time Markov dynamics 
(i.e., the $q$-randomized
RSK correspondences)
which we construct in the present paper
live on the space of integer arrays $\lab$ (see Fig.~\ref{fig:array}).
Neighboring levels of the array satisfy certain inequalities
which we call the \emph{interlacing property} (see \eqref{interlace} for the definition).
Each level $\la^{(j)}=(\la^{(j)}_{1}\ge \ldots\ge\la^{(j)}_{j})$
of an array can be viewed as a partition 
(equivalently, a Young diagram \cite[I.1]{Macdonald1995}),
so $\lab$ is a sequence of interlacing Young diagrams.
\begin{figure}[htbp]
	\begin{adjustbox}{max height=.2\textwidth}
	\begin{tikzpicture}[scale = .5, thick]
		\draw (0,0) --++ (7,0);
		\draw (0,-1) --++ (7,0);
		\draw (0,-2) --++ (7,0);
		\draw (0,-3) --++ (3,0);
		\draw (0,-4) --++ (2,0);
		\draw (0,-5) --++ (2,0);
		\draw (0,0) --++ (0,-5);
		\draw (1,0) --++ (0,-5);
		\draw (2,0) --++ (0,-5);
		\draw (3,0) --++ (0,-3);
		\draw (4,0) --++ (0,-2);
		\draw (5,0) --++ (0,-2);
		\draw (6,0) --++ (0,-2);
		\draw (7,0) --++ (0,-2);
		\def\xxxx{.05}
		\node[anchor=west] at (\xxxx,-.5) {1};
		\node[anchor=west] at (\xxxx+1,-.5) {1};
		\node[anchor=west] at (\xxxx+2,-.5) {1};
		\node[anchor=west] at (\xxxx+3,-.5) {1};
		\node[anchor=west] at (\xxxx+4,-.5) {3};
		\node[anchor=west] at (\xxxx+5,-.5) {4};
		\node[anchor=west] at (\xxxx+6,-.5) {4};
		\node[anchor=west] at (\xxxx,-1.5) {2};
		\node[anchor=west] at (\xxxx+1,-1.5) {2};
		\node[anchor=west] at (\xxxx+2,-1.5) {2};
		\node[anchor=west] at (\xxxx+3,-1.5) {5};
		\node[anchor=west] at (\xxxx+4,-1.5) {5};
		\node[anchor=west] at (\xxxx+5,-1.5) {5};
		\node[anchor=west] at (\xxxx+6,-1.5) {5};
		\node[anchor=west] at (\xxxx,-2.5) {3};
		\node[anchor=west] at (\xxxx+1,-2.5) {3};
		\node[anchor=west] at (\xxxx+2,-2.5) {3};
		\node[anchor=west] at (\xxxx,-3.5) {4};
		\node[anchor=west] at (\xxxx+1,-3.5) {4};
		\node[anchor=west] at (\xxxx,-4.5) {5};
		\node[anchor=west] at (\xxxx+1,-4.5) {5};
	\end{tikzpicture}
	\end{adjustbox}
	\caption{A semistandard Young tableau corresponding
	to the array on Fig.~\ref{fig:array}, right.}
	\label{fig:SSYT}
\end{figure}
\begin{remark}\label{rmk:SSYT}
	Each $\lab$ can be also viewed as a \emph{semistandard Young tableau}
	of shape $\la^{(N)}$ filled
	with numbers from $1$ to $N$. Here ``semistandard'' means that 
	in this filling, numbers increase weakly along rows, and strictly down columns.
	Then each $\la^{(j)}$ is the portion of the semistandard tableau filled
	with numbers from 1 to $j$, see Fig.~\ref{fig:SSYT}.
\end{remark}

Let us now define the ($q$-randomized) operation of 
inserting a word $w=(1^{m_1}2^{m_2}\ldots N^{m_N})$
(i.e., the word has $m_1$ ones, $m_2$ twos, etc.)
into an array $\lab$. The result is a new, random 
array $\nub$. At the first level we have
$\nu^{(1)}_{1}=\la^{(1)}_{1}+m_1$. 
Then, sequentially at all levels $j=2,\ldots,N$, 
given the existing change $\la^{(j-1)}\to\nu^{(j-1)}$
at the previous level and the old state $\la^{(j)}$
at the current level, construct the new state $\nu^{(j)}$
as follows. Each move $\nu^{(j-1)}_{i}-\la^{(j-1)}_{i}$,
$i=1,\ldots,j-1$, is \emph{randomly}
split into two pieces $r_i^{(j-1)}+\ell_i^{(j-1)}$, and the piece
$r_i^{(j-1)}$ is added to the new move of the upper right neighbor $\la^{(j)}_{i}$, 
while the piece $\ell_i^{(j-1)}$
is added to the new move of the upper left neighbor $\la^{(j)}_{i+1}$.
Moreover, $\la^{(j)}_{1}$ receives an additional move of size $m_j$.
All these splittings and moves at level $j$ happen in parallel.
That is (here and below $\mathbf{1}_{\cdots}$ stands for the indicator),
\begin{align*}
	\nu^{(j)}_{i}-\la^{(j)}_{i}=m_j\mathbf{1}_{i=1}
	+
	r_i^{(j-1)}\mathbf{1}_{i<j}+\ell_{i-1}^{(j-1)}\mathbf{1}_{i>1},\qquad
	i=1,\ldots,j
\end{align*}
(see Fig.~\ref{fig:moves_intro}).
To 
complete the definition, it now remains to describe the distribution of the splitting 
of the move
$\nu^{(j-1)}_{i}-\la^{(j-1)}_{i}=r_i^{(j-1)}+\ell_i^{(j-1)}$.
This is a certain $q$-deformed version of the Beta-binomial distribution, namely, 
$r_i^{(j-1)}$ is randomly chosen to be equal to
$r\in\{0,1,2,\ldots,\nu^{(j-1)}_{i}-\la^{(j-1)}_{i}\}$ with probability
\begin{align}\label{Om_intro}
	\Om_{q^{-1},q^{a},q^{b}}(r\mid c):=
	q^{ar}\frac{(q^{b-a};q^{-1})_{r}}{(q^{-1};q^{-1})_r}
	\frac{(q^{a};q^{-1})_{c-r}}{(q^{-1};q^{-1})_{c-r}}
	\frac{(q^{-1};q^{-1})_{c}}{(q^{b};q^{-1})_{c}},
\end{align}
where
\begin{align*}
	a=\la^{(j)}_{i}-\la^{(j-1)}_{i},\qquad
	b=\la^{(j-1)}_{i-1}-\la^{(j-1)}_{i},
	\qquad
	c=\nu^{(j-1)}_{i}-\la^{(j-1)}_{i},
\end{align*}
$(u;q)_{m}=(1-u)(1-uq)\ldots(1-uq^{m-1})$ are the $q$-Pochhammer symbols,
and we adopt the convention $\la^{j-1}_{0}=+\infty$.
The quantity $\ell^{(j-1)}_{i}$ is simply equal to $\nu^{(j-1)}_{i}-\la^{(j-1)}_{i}-r^{(j-1)}_{i}$.

The quantities \eqref{Om_intro} define a probability distribution in $r$
for $a\le b$, $c\le b$
(these conditions follow from the interlacing).
Moreover, this distribution is supported on $\{0,1,\ldots,c\}\cap\{c-a,c-a+1,\ldots,b-a-1,b-a\}$,
which in fact ensures that the new array $\nub$ is also interlacing (see lemma  \ref{lemma:interlacing_after} for details).

\begin{figure}[htbp]
	\begin{adjustbox}{max width=\textwidth}
	\begin{tikzpicture}[scale = 1, thick]
		\draw[fill] (0,0) circle (3pt) node[below] {$\la^{(j-1)}_{i}$};
		\draw[fill] (5,0) circle (3pt) node [below] {$\nu^{(j-1)}_{i}$};
		\draw[fill] (-3,2) circle (3pt) node [above] {$\la^{(j)}_{i+1}$};
		\draw[fill] (6,2) circle (3pt) node[above] {$\la^{(j)}_{i}$};
		\draw[fill] (2,2) circle (3pt) node[above] {$\nu^{(j)}_{i+1}$};
		\draw[fill] (11,2) circle (3pt) node[above] {$\nu^{(j)}_{i}$};
		\node (lab) at (0,0) {};
		\node (nub) at (5,0) {};
		\node (lal) at (-3,2) {};
		\node (lar) at (6,2) {};
		\node (nul) at (2,2) {};
		\node (nur) at (11,2) {};
		\draw[->, very thick] (lab)--(nub);
		\def\xxxx{.15}
		\draw[ultra thick] (3,-\xxxx)--(3,\xxxx);
		\draw[ultra thick] (-1,2-\xxxx)--(-1,2+\xxxx);
		\draw[ultra thick] (8,2-\xxxx)--(8,2+\xxxx);
		\node[above] at (1.5,0) {$\ell_i^{(j-1)}$};
		\node[above] at (4.5,0) {$r_i^{(j-1)}$};
		\node[below] at (-2.2,2) {$r_{i+1}^{(j-1)}$};
		\node[below] at (9.8,2) {$\ell_{i-1}^{(j-1)}$};
		\draw[dotted] (lab) -- (-1,2);
		\draw[dotted] (3,0) -- (lar);
		\draw[dotted] (3,0) -- (nul);
		\draw[dotted] (nub) -- (8,2);
		\draw[dotted] (lal) -- (-3.5,1.5);
		\draw[dotted] (-1,2) -- (-1.5,1.5);
		\draw[dotted] (11,2) -- (11.5,1.5);
		\draw[dotted] (8,2) -- (8.5,1.5);
		\draw[densely dashed, very thick, ->] (lal) -- (nul);
		\draw[densely dashed, very thick, ->] (lar) -- (nur);
	\end{tikzpicture}
	\end{adjustbox}
	\caption{Splitting of the move at level $j-1$ and 
	its propagation to the level $j$. Here we are using 
	the particle interpretation of interlacing arrays as on 
	Fig.~\ref{fig:array}, right.}
	\label{fig:moves_intro}
\end{figure}
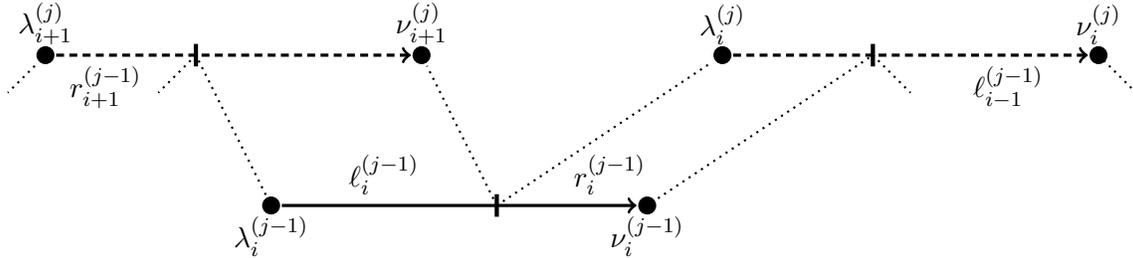

\begin{remark}
	The interlacing array $\lab$ plays the role of the \emph{insertion tableau}
	in our $q$-randomized RSK correspondence (cf. Remark \ref{rmk:SSYT}).
	One can readily define an accompanying \emph{recording tableau}
	in the same way as 
	it is done for the classical RSK correspondences.
	In the present paper we will not focus on recording tableaux.
\end{remark}

Now, let us take the insertion word 
$w=(1^{m_1}2^{m_2}\ldots N^{m_N})$ to be random itself. 
More precisely, let $m_j$, $j=1,\ldots,N$, 
be independent $q$-geometric random variables:
\begin{align}\label{qgeom_intro}
	\prob(m_j=k) = \frac{\al^{k}}{(q;q)_{k}}(\al;q)_{\infty},\qquad
	k=0,1,2,\ldots,\qquad 0<\al<1.
\end{align}
Inserting this random word $w$ into an array $\lab$ defines 
one step of a discrete time Markov chain on interlacing arrays.
We denote this Markov chain by $\Qqarow$.
\begin{theorem}\label{thm:sampling_intro}
	Start the Markov dynamics $\Qqarow$ from the interlacing array with all
	$\la^{(j)}_{i}(0)=0$. Then the distribution of the array $\lab(T)$ after $T$
	steps of this dynamics is given by the $q$-Whittaker process:
	\begin{align*}
		\prob(\lab(T)=\lab)=
		(\al;q)^{TN}_{\infty}{P_{\la^{(1)}}(1)
		P_{\la^{(2)}/\la^{(1)}}(1)\ldots P_{\la^{(N)}/\la^{(N-1)}}(1)
		Q_{\la^{(N)}}(\underbrace{\al,\al,\ldots,\al}_{T})}.
	\end{align*}
\end{theorem}
Here $P_{\la/\mu}$ and $Q_\la$ are the 
$q$-Whittaker polynomials, see \S \ref{sec:macdonald_processes_and_markov_dynamics}
for more detail. Theorem \ref{thm:sampling_intro} 
follows from Theorem \ref{thm:QqrRSK_alpha} which we prove in 
\S \ref{sub:row_insertion_dynamics_q__hatal_q-rrsk_}.
\begin{remark}
	In fact, we can (and will) consider a more general situation 
	when the parameters $\al$
	in \eqref{qgeom_intro} may depend on $j$
	and on the time step as $\al_t a_j$. Then the 
	$q$-Whittaker process above takes the form
	\begin{align*}
		\bigg(\prod_{j=1}^{N}\prod_{t=1}^{T}
		(\al_t a_j;q)_{\infty}\bigg)
		{P_{\la^{(1)}}(a_1)
		P_{\la^{(2)}/\la^{(1)}}(a_2)\ldots P_{\la^{(N)}/\la^{(N-1)}}(a_N)
		Q_{\la^{(N)}}(\al_1,\al_2,\ldots,\al_T)}.
	\end{align*}
	We omit the dependence on $j$ and $t$ in Introduction.	
\end{remark}

Let us now describe three degenerations of the dynamics $\Qqarow$:
\smallskip

$\bullet$ For $q=0$, the splitting distributions \eqref{Om_intro}
become supported at a single $r\in\{0,1,\ldots,c\}$, so the randomness in the 
insertion disappears, and the insertion itself
turns into the classical RSK row insertion (we recall its
definition in \S \ref{sub:rsk_type_dynamics_schur}). The $q$-geometric random variables $m_j$
\eqref{qgeom_intro}
become geometric, and the $q$-Whittaker polynomials in Theorem \ref{thm:sampling_intro}
turn into the Schur polynomials. 
This justifies our treatment of the dynamics 
$\Qqarow$ as the $q$-randomized row RSK correspondences.

$\bullet$ Fix $0<q<1$. When
$\al\searrow0$ in \eqref{qgeom_intro} and one rescales time from discrete to continuous time 
(see \S\ref{sub:small_al_continuous_time_limit} for details on this scaling), 
the random input matrix turns into $N$ independent Poisson processes
running in parallel (i.e., we are passing from the left to the center
situation on Fig.~\ref{fig:RSK}). Then in the splitting distributions one has
$c=0$ or $1$, and the dynamics $\Qqarow$ turns into a continuous
time dynamics on $q$-Whittaker processes 
which was introduced in \cite{BorodinPetrov2013NN}. The latter continuous time dynamics
should be viewed as a $q$-randomized row RS correspondence. 

$\bullet$ Let $q=e^{-\epsilon}$ and $\al=e^{-\theta\epsilon}$ with $\epsilon\searrow0$ and $\theta>0$.
Define the positive random variables $\hat R^{j}_{k}(t,\epsilon)$ via the following scaling:
\begin{align*}
	\la^{(j)}_{k}(t)=(t+j-2k+1)\epsilon^{-1}\log\epsilon^{-1}
	+\epsilon^{-1}\log\big(\hat R^{j}_{k}(t,\epsilon)\big).
\end{align*}
If the quantities $\la^{(j)}_{k}$ evolve under the dynamics $\Qqarow$
started from all $\la^{(j)}_{k}(0)=0$, then 
the rescaled quantities $\hat R^{j}_{k}(t,\epsilon)$ converge
to certain ratios of partition functions in the log-Gamma lattice polymer model
(see \S \ref{sub:polymer_partition_functions} and Theorem \ref{thm:TR} in particular for details).
Moreover, under this scaling the randomness 
in the splitting \eqref{Om_intro}
disappears, and the $q$-randomized insertion described above turns into the 
geometric RSK insertion.

\begin{remark}
	It is worth noting that there is also a strong
	connection between the geometric 
	(tropical) RSK and representation theory,
	cf. \cite{BBO2004}, 
	\cite{Chhaibi2013}.
	At the $q$-randomized level this connection does not (yet)
	seem to be present.
\end{remark}

When restricted to the rightmost particles $\la^{(j)}_{1}$, $j=1,\ldots,N$,
of the interlacing array, the dynamics $\Qqarow$
induces a marginally Markovian evolution which we call the 
(discrete time) geometric $q$-PushTASEP. 
This is a new 
integrable particle system in the 
KPZ universality class.
In the shifted coordinates
$x_i(t):=-\la_1^{(i)}(t)-i$ (so $x_N<\ldots<x_1$), 
the evolution of this system during time step $t\to t+1$ looks as follows.
Sequentially for $i=1,2,\ldots,N$, each particle $x_i$
jumps to the left by $m_i+r^{(i-1)}_{1}$, where $m_i$ is an independent
$q$-geometric random variable \eqref{qgeom_intro}, and 
$r^{(i-1)}_{1}$ is a random variable with distribution 
\begin{align*}
	\Om_{q^{-1},q^{a},0}(r\mid c)=
	q^{ar}(q^{a};q^{-1})_{c-r}
	\frac{(q^{-1};q^{-1})_{c}}{(q^{-1};q^{-1})_{r}(q^{-1};q^{-1})_{c-r}},
	\qquad
	\begin{array}{l}
		a=x_{i-1}(t)-x_i(t)-1,
		\\
		c=x_{i-1}(t+1)-x_{i-1}(t)	
	\end{array}
\end{align*}
(this is simply the splitting distribution \eqref{Om_intro} with $b=+\infty$).
Note that if $c>a$, then $r^{(i-1)}_{1}$ chosen according to the above distribution
will be at least $c-a$. See Fig.~\ref{fig:qpush_alpha}.
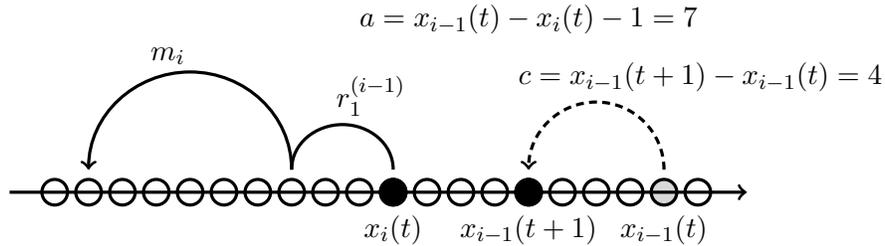
\begin{figure}[htbp]
	\begin{tikzpicture}
		[scale=1,very thick]
		\def\pt{.17}
		\def\ee{.1}
		\def\h{.45}
		\draw[fill, color=gray!30!white] (12*\h,0) circle(\pt);
		\draw[->] (-3.3,0) -- (6.5,0);
		\foreach \ii in {
		(-6*\h,0),(-5*\h,0),(-4*\h,0),(-3*\h,0),(-2*\h,0),(-1*\h,0),
		(0*\h,0), (1*\h,0), (2*\h,0), (3*\h,0), (4*\h,0), (5*\h,0), 
		(6*\h,0), (7*\h,0), (8*\h,0), (9*\h,0), (10*\h,0), (11*\h,0),(12*\h,0), (13*\h,0)}
		{
			\draw \ii circle(\pt);
		}
		\foreach \ii in {(8*\h,0),
		(4*\h,0)}
		{
			\draw[fill] \ii circle(\pt);
		}
		\node at (8*\h,-5*\ee) {$x_{i-1}(t+1)$};
		\node at (12*\h,-5*\ee) {$x_{i-1}(t)$};
		\node at (4*\h,-5*\ee) {$x_{i}(t)$};
		\node at (8*\h,23*\ee) {$a=x_{i-1}(t)-x_i(t)-1=7$};
	    \draw[->, very thick, densely dashed] (12*\h,.3) to [in=0, out=90] (10*\h,1.2)
	    to [in=90, out=180] (8*\h,.3)
	    node [xshift=65,yshift=35] {$c=x_{i-1}(t+1)-x_{i-1}(t)=4$};
	    \draw[very thick] (4*\h,.3) to [in=0, out=90] (2.5*\h,0.9)
	    to [in=90, out=180] (1*\h,.3) node [xshift=30,yshift=28] {$r^{(i-1)}_{1}$};
	    \draw[->, very thick] (1*\h,.3) to [in=0, out=90] (-2*\h,1.6)
	    to [in=90, out=180] (-5*\h,.3) node [xshift=30,yshift=43] {$m_i$};
	\end{tikzpicture}
	\caption{The discrete time geometric $q$-PushTASEP.}
	\label{fig:qpush_alpha}
\end{figure}

In a continuous time limit as $\al\searrow0$, the geometric $q$-PushTASEP 
turns into the continuous time $q$-PushTASEP
of \cite{BorodinPetrov2013NN}, \cite{CorwinPetrov2013}.
The $q$-moments of the form 
$\E q^{k(x_n(t)+n)}$ (and more general such moments)
of both $q$-PushTASEPs are given in terms of nested 
contour integrals.
For the geometric $q$-PushTASEP only finitely many such moments exist (i.e.,
the expectation is infinite for sufficiently large $k$), and for the 
continuous time $q$-PushTASEP the moments grow too fast and
also do not determine the distribution of $x_n(t)$. 
However, it is still possible to write down a 
Fredholm determinantal formula for the distribution of $x_n(t)$
for both processes
(started from the step initial configuration $x_i(0)=-i$)
using the theory of Macdonald processes 
\cite{BorodinCorwin2011Macdonald}, see \cite[Theorem 3.3]{BorodinCorwinFerrariVeto2013}.
We refer to \S \ref{sec:moments_and_fredholm_determinants_for_bernoulli_} for
further details.


\subsection{Other dynamics and results} 
\label{sub:other_results}

Besides the dynamics $\Qqarow$ discussed in \S \ref{sub:_q_randomized_row_insertion_with_q_geometric_input}
above, we introduce three other dynamics on $q$-Whittaker processes:

\smallskip

$\bullet$ $\Qqacol$ (\S \ref{sub:col_insertion_dynamics_q__hatal_q-rrsk_}
and Theorem \ref{thm:QqcRSK_alpha}). At $q=0$ this dynamics 
becomes the classical RSK column insertion applied to a geometric random input 
$\Qacol$ (\S \ref{sub:rsk_type_dynamics_schur}). In a scaling limit as 
$q\nearrow1$, $\Qqacol$
turns into a geometric (tropical) RSK associated with the strict-weak lattice polymer introduced
in 
\cite{CorwinSeppalainenShen2014} (Theorem \ref{thm:TL}). 
In a continuous time limit, $\Qqacol$ turns into the $q$-randomized column RS correspondence
introduced in \cite{OConnellPei2012}.
Under $\Qqacol$,
the leftmost particles $\la^{(j)}_{j}$ of the interlacing array
evolve according to the discrete time geometric $q$-TASEP of \cite{BorodinCorwin2013discrete}.

$\bullet$ $\Qqbrow$ (\S \ref{sub:row_insertion_dynamics_q__hatbe_q-rrsk_}
and Theorem \ref{thm:QqrRSK_beta}). At $q=0$ this dynamics 
becomes the dual RSK row insertion applied to a Bernoulli random input
$\Qbrow$ (\S \ref{sub:rsk_type_dynamics_schur}). 
In a continuous time limit, $\Qqbrow$ turns into the $q$-randomized row RS correspondence
\cite{BorodinPetrov2013NN}.
Under $\Qqbrow$,
the rightmost particles $\la^{(j)}_{1}$ of the interlacing array
evolve according to a new particle system,
the discrete time Bernoulli $q$-PushTASEP (Definition \ref{def:Bernoulli_qpush}).

$\bullet$
$\Qqbcol$ (\S \ref{sub:column_insertion_dynamics_q__hatbe_q-rrsk_}
and Theorem \ref{thm:QqcRSK_beta}). At $q=0$ this dynamics 
becomes the dual RSK column insertion applied to a Bernoulli random input
$\Qbcol$ (\S \ref{sub:rsk_type_dynamics_schur}). 
In a continuous time limit, $\Qqbrow$ turns into the $q$-randomized column RS correspondence
of \cite{OConnellPei2012}.
Under $\Qqbrow$,
the leftmost particles $\la^{(j)}_{j}$ of the array
evolve according to
the discrete time Bernoulli $q$-TASEP 
\cite{BorodinCorwin2013discrete}.\footnote{In contrast with 
$\Qqarow$ and $\Qqacol$,
there is 
(yet) no known polymer-like 
limits of $\Qqbrow$ or $\Qqbcol$.}

\begin{remark}
	We do not attempt a full classification of $q$-randomized RSK correspondences
	as it was done for the RS correspondences 
	by solving certain linear equations for transition
	probabilities in 
	\cite{BorodinPetrov2013NN}.
	Similar equations 
	for the discrete time situation
	seem to be much more involved,
	and in this paper we demonstrate particular solutions
	to these equations which lead to 
	discrete time Markov dynamics
	(see also \S \ref{ssub:classification} for further discussion).

	We however believe that the four dynamics 
	we construct are the most ``natural''
	discrete time dynamics on $q$-Whittaker processes 
	having all the desired properties and prescribed degenerations:
	\begin{itemize}
		\item The update in the dynamics is sequential,
		from lower to upper levels of the interlacing array.
		\item The dynamics act nicely on $q$-Whittaker
		measures and
		processes.
		\item The continuous time 
		limits ($\al$ or $\be\to0$) of the dynamics
		coincide with continuous time RS 
		dynamics of \cite{OConnellPei2012} or 
		\cite{BorodinPetrov2013NN}.
		\item For $q=0$, the dynamics degenerate to the ones 
		related to the classical RSK correspondences.
		\item In the $q\nearrow1$
		limit, the $(\al)$ dynamics  
		converge to the ones related to the geometric (tropical) RSKs.
	\end{itemize}
\end{remark}

The dynamics
$\Qqbrow$ and $\Qqbcol$
are related to each other via a straightforward procedure
we call complementation (\S \ref{sub:Complementation})
which shortens the proofs for $\Qqbcol$.
Moreover, one can say that this procedure provides a direct link
between the column and the row $q$-randomized RS correspondences
of \cite{OConnellPei2012} and \cite{BorodinPetrov2013NN}
(which are continuous time limits of $\Qqbcol$ and $\Qqbrow$, respectively).
This also provides a direct coupling between
the Bernoulli $q$-TASEP and $q$-PushTASEP (Proposition
\ref{prop:TASEP_coupling}).

We employ the discrete time Bernoulli processes
to obtain a Fredholm determinantal formula for the continuous time $q$-PushTASEP
and for its two-sided extension, 
the continuous time $q$-PushASEP (the latter formula was 
conjectured in \cite{CorwinPetrov2013}), see Theorem \ref{claim:Fredholm}. See also a related discussion 
in the end of 
\S \ref{sub:_q_randomized_row_insertion_with_q_geometric_input}.



\subsection{Outline} 
\label{sub:outline}

In \S \ref{sec:macdonald_processes_and_markov_dynamics}
we recall the necessary background
on Macdonald and $q$-Whittaker symmetric functions and
$q$-Whittaker processes,
and also write down and discuss main linear equations
which must be satisfied by our Markov dynamics on interlacing arrays.
In \S \ref{sec:push_block_and_rsk_type_dynamics} we
discuss two particular types of 
Markov dynamics, namely, push-block and RSK-type dynamics, and 
explain the differences between them.
In \S \ref{sec:schur_degeneration} we 
illustrate our main definitions
and concepts in the $q=0$ situation,
when the $q$-Whittaker polynomials
reduce to the simpler Schur polynomials, 
and the dynamics on interlacing arrays are relevant to the classical RSK correspondences.
In \S \ref{sec:bernoulli_} and \S \ref{sec:geometric_q_rsks}
we explain in detail the constructions of 
four discrete time RSK-type dynamics on interlacing arrays,
and prove that these dynamics act on the $q$-Whittaker processes
in desired ways. In \S \ref{sec:moments_and_fredholm_determinants_for_bernoulli_}
we discuss moment and Fredholm determinantal formulas
for our one-dimensional interacting particle systems.
In \S \ref{sec:polymer_limit}
we consider
scaling limits as $q\nearrow1$
of our two $(\al)$ dynamics on interlacing arrays,
and show that they
turn into the geometric RSK correspondences	
associated with log-Gamma or strict-weak directed random lattice polymers.


\subsection{Acknowledgments} 
\label{sub:acknowledgments}

We are grateful to Alexei Borodin, Vadim Gorin, and Ivan Corwin
for numerous discussions which were extremely helpful.
LP would like to thank
Sergey Fomin, Greta Panova, and
Guillaume Barraquand for useful remarks.
We are also grateful to Columbia University and to Ivan Corwin
for the warm hospitality during our short visit at a final stage of this project.
We are indebted to Christian Krattenthaler
for providing us with proofs of certain $q$-binomial identities
(Propositions \ref{prop:Kr1} and \ref{prop:Kr2})
which are crucial ingredients for
the construction of one of our four dynamics.
We would also like to thank the anonymous referee for valuable comments which 
has lead to many improvements.

KM was partially supported by the Natural Sciences and Engineering Research Council of Canada through the PGS D Scholarship. LP was partially supported by the University of Virginia through the EDF Fellowship.




\section{Macdonald processes and Markov dynamics} 
\label{sec:macdonald_processes_and_markov_dynamics}

In this section
we collect main notation and definitions related to Macdonald
processes
used throughout the paper,
and also
write down and discuss linear equations 
satisfied by Markov dynamics
on $q$-Whittaker processes which we aim to construct.

\subsection{Preliminaries} 
\label{sub:preliminaries}

A \emph{signature}\footnote{These objects 
are also sometimes called \emph{highest weights}, 
cf. \cite{Weyl1946}, as they are the highest weights
of irreducible representations of the unitary group
$U(N)$.} of length $N\ge1$
is a nonincreasing collection of integers 
$\la=(\la_1\ge \ldots\ge\la_N)\in\Z^{N}$.
We will work with signatures which have only nonnegative 
parts, i.e., $\la_N\ge0$ (in which case they are also called \emph{partitions}).
Denote the set of all such objects by $\GT_N$.
Let also $\GT:=\bigcup_{N=1}^{\infty}\GT_N$,
with the understanding that we identify
$\la\cup0=(\la_1,\ldots,\la_N,0,0,\ldots,0)\in\GT_{N+M}$
($M$ zeros)
with $\la\in\GT_N$
for any $M\ge1$.

We will use two ways to depict signatures (see Fig.~\ref{fig:YD}):
\begin{enumerate}
	\item Any signature $\la\in\GT_N$ 
	can be identified with a \emph{Young diagram}
	(having at most $N$ rows)
	as in \cite[I.1]{Macdonald1995}.
	\item A signature $\la\in\GT_N$ can also be represented
	as a configuration of $N$ particles on $\Z_{\ge0}$
	(with the understanding that there can be more than one particle
	at a given location).
\end{enumerate}
We denote by $|\la|:=\sum_{i=1}^{N}\la_i$ the number of 
boxes in the corresponding Young diagram,
and by $\ell(\la)$ the number of nonzero parts in $\la$ 
(which is finite for all $\la\in\GT$).
For $\mu,\la\in\GT$ we
will write $\mu\subseteq\la$
if (after possibly appending $\mu$ and $\la$ by zeros)
we have $\mu_i\le\la_i$ for all $i \in \Z_{>0}$.
In this case, the set difference 
of Young diagrams $\la$ and $\mu$
is denoted by $\la/\mu$ and is called a \emph{skew Young diagram}.

Two signatures $\mu,\la\in\GT$ are said to \emph{interlace}
if one can append them by zeros such that
$\mu\in\GT_{N-1}$ and $\la\in\GT_N$
for some $N$, and
\begin{align}\label{interlace}
	\la_1\ge\mu_1\ge\la_2\ge\mu_2 \ge\ldots\ge
	\la_{N-1}\ge\mu_{N-1}\ge
	\la_{N}.
\end{align}
In terms of Young diagrams, this means that
$\la$ is obtained from $\mu$
by adding a \emph{horizontal strip} (or, equivalently, that
\emph{the skew diagram $\la/\mu$ is a horizontal strip} which is,
by definition, a skew Young diagram having at most one box in each 
vertical column),
and we denote this by $\mu\prech\la$.

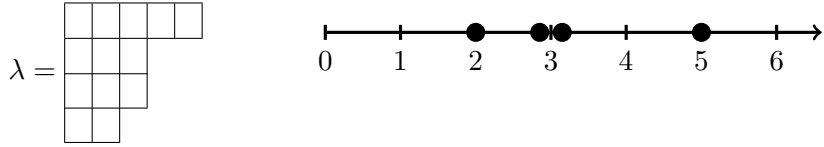
\begin{figure}[htbp]
\begin{equation*}
	\la=\begin{array}{|c|c|c|c|c|}
	    \hline
	    \ &\ &\ &\ &\ \\
	    \hline 
	    \ &\ &\ \\
	    \cline{1-3}
	    \ &\ &\ \\
		\cline{1-3}
		\ &\ \\
		\cline{1-2}
	\end{array}
	\hspace{40pt}
	\begin{tikzpicture}
		[scale=1, very thick]
		\draw[->] (0,0) -- (6.6,0);
		\foreach \h in {0,1,2,3,4,5,6}
		{
			\draw (\h,.1)--++(0,-.2) node [below] {$\h$};
		}
		\draw[fill] (2,0) circle (3pt);
		\draw[fill] (3-.15,0) circle (3pt);
		\draw[fill] (3+.15,0) circle (3pt);
		\draw[fill] (5,0) circle (3pt);
	\end{tikzpicture}
\end{equation*}
\caption{Young diagram $\la=(5,3,3,2)\in\GT_4$,
and the corresponding particle configuration.
Note that there are two particles at location 3.}
\label{fig:YD}
\end{figure}

Let $\la'$ denote the transposition of the 
Young diagram $\la$. For the diagram on
Fig.~\ref{fig:YD}, we have $\la'=(4,4,3,1,1)$.
If $\la/\mu$ is a horizontal strip, then $\la'/\mu'$
is called a \emph{vertical strip}.
We will denote the corresponding relation by $\mu' \precv\la'$.


\subsection{Macdonald polynomials} 
\label{sub:macdonald_polynomials}

Probability measures and Markov dynamics studied in the 
present paper are 
based on Macdonald polynomials.
Here let us briefly recall their definition and properties which
are essential for us. An excellent exposition 
and much more details
may be found in \cite[Ch. VI]{Macdonald1995}. 
We also refer to \cite[\S2]{BorodinCorwin2011Macdonald}
for a discussion of specializations of Macdonald polynomials.

\begin{definition}
	Let $q,t\in[0,1)$ be two parameters.
	Consider the first order $q$-difference operator
	acting on functions in $N$ variables:
	\begin{align*}
		(\mathscr{D}^{(1)}f)(x_1,\ldots,x_N)
		:=\sum_{i=1}^{N}\prod_{j\ne i}
		\frac{tx_i-x_j}{x_i-x_j}f(x_1,\ldots,x_{i-1},qx_i,
		x_{i+1},\ldots,x_N).
	\end{align*}
	This operator preserves the space 
	$\mathbb{Q}(q,t)[x_1,\ldots,x_N]^{S(N)}$
	of symmetric
	polynomials with coefficients which are rational
	functions in $q$ and $t$.

	Eigenfunctions of $\mathscr{D}^{(1)}$ are given by the 
	\emph{Macdonald symmetric polynomials} 
	$P_\la(x_1,\ldots,x_N\mid q,t)$ indexed by $\la\in\GT_N$,
	with eigenvalues
	\begin{align*}
		\mathscr{D}^{(1)}P_\la=(q^{\la_1}t^{N-1}+
		q^{\la_2}t^{N-2}+\ldots+q^{\la_{N-1}}t+q^{\la_N})P_\la
	\end{align*}
	(which are pairwise distinct for generic $q,t$).
	The polynomials $P_\la$ are homogeneous, and 
	form a linear basis for $\mathbb{Q}(q,t)[x_1,\ldots,x_N]^{S(N)}$.
\end{definition}
The Macdonald polynomials are \emph{stable} in the sense that
for any $\la\in\GT_N$,
\begin{align*}
	P_{\la\cup0}(x_1,\ldots,x_N,0\mid q,t)=
	P_\la(x_1,\ldots,x_N\mid q,t).
\end{align*}
Therefore, one may speak about \emph{Macdonald symmetric functions}
$P_\la(x_1,x_2,\ldots\mid q,t)$ 
in infinitely many variables,
indexed by arbitrary $\la\in\GT$.
These are elements of the algebra of symmetric
functions, which may be
viewed as a free unital algebra 
$\Sym=\mathbb{Q}(q,t)[p_1,p_2,\ldots]$
generated by the Newton power sums
$p_k(x_1,x_2,\ldots)=\sum_{j=1}^{\infty}x_j^{k}$.
In other words, symmetric functions can be viewed as 
usual polynomials in $p_1,p_2,\ldots$. 
Note that
$P_\la(x_1,\ldots,x_N)\equiv 0$ if $\ell(\la)>N$.
The Macdonald symmetric functions admit an equivalent alternative definition:
\begin{definition}
	Let $(\cdot,\cdot)_{q,t}$ be the scalar product on 
	$\Sym$ defined on products of 
	power sums $p_\la=p_{\la_1}p_{\la_2}\ldots$
	as
	\begin{align*}
		(p_\la,p_\mu)_{q,t}=\mathbf{1}_{\la=\mu}z_\la(q,t),
		\qquad z_\la(q,t):=\bigg(\prod_{i\ge1}i^{m_i}(m_i)!\bigg)
		\cdot
		\bigg(\prod_{i=1}^{\ell(\la)}
		\frac{1-q^{\la_i}}{1-t^{\la_i}}\bigg),
	\end{align*}
	where $\la=(1^{m_1}2^{m_2}\ldots)$ means that $\la$ has $m_1$ parts equal to 1, $m_2$ parts equal to 2, etc.

	The $P_\la$'s
	form a unique family of 
	homogeneous symmetric functions such that:
	\begin{enumerate}
		\item They are pairwise 
		orthogonal with respect to 
		the scalar product $(\cdot,\cdot)_{q,t}$.
		\item For every $\la$, we have
		\begin{align*}
			P_\la(x_1,x_2,\ldots\mid q,t)=
			x_1^{\la_1}x_2^{\la_2}\ldots 
			x_{{\ell(\la)}}^{\la_{\ell(\la)}}{}+
			{}\mbox{lower monomials in lexicographic order}.
		\end{align*}
		The dependence on the parameters $(q,t)$ 
		is in coefficients of the 
		lexicographically 
		lower monomials.\footnote{Lexicographic order 
		means that, for example, $x_1^{2}$ is higher 
		than $\mathrm{const}\cdot x_1x_2$ 
		which is in turn higher 
		than $\mathrm{const}\cdot x_2^{2}$.}
	\end{enumerate}
\end{definition}
Set $b_\la(q,t):=1/(P_\la,P_\la)_{q,t}$; 
this is an explicit quantity determined via the
shape of the 
Young diagram $\la$. Then the  
symmetric functions
$Q_\la(\cdot\mid q,t):=b_\la(q,t)P_\la(\cdot\mid q,t)$
are biorthonormal with the $P_\la$'s: 
$(P_\la,Q_\mu)_{q,t}=\mathbf{1}_{\la=\mu}$.

\begin{definition}
	The skew Macdonald symmetric functions $Q_{\la/\mu}$,
	$\mu,\la,\in\GT$, are defined as the only symmetric
	functions for which 
	$(Q_{\la/\mu},P_\nu)_{q,t}=(Q_\la,P_\mu P_\nu)_{q,t}$
	for any $\nu\in\GT$. The ``$P$'' versions are 
	given by 
	$P_{\la/\mu}=(b_\mu(q,t)/b_\la(q,t)) Q_{\la/\mu}$.
	These skew functions 
	are identically zero
	unless $\mu\subseteq\la$.
\end{definition}
The 
skew Macdonald symmetric functions enter the following
recurrence relations:
\begin{align}
	P_\la(x_1,\ldots,x_N)=\sum_{\mu\in\GT_{N-K}}
	P_{\la/\mu}(x_1,\ldots,x_K)
	P_\mu(x_{K+1},\ldots,x_N),\qquad
	\la\in\GT_N,\qquad K\ge1
	\label{Macdonald_recurrence}
\end{align}
(and similarly for the $Q_{\la}$'s). 
This may be viewed as an alternative definition of the 
skew Macdonald polynomials $P_{\la/\mu}$
in finitely many variables.
If $K=1$ in \eqref{Macdonald_recurrence}, 
then the summation is over 
the interlacing signatures $\mu\prech\la$.
In this case $P_{\la/\mu}(x_1)$ is proportional to $x_1^{|\la|-|\mu|}$
by homogeneity
(cf. \eqref{P_one_variable} below),
and \eqref{Macdonald_recurrence} is also sometimes referred to 
as the branching rule for the Macdonald polynomials. 

\medskip

From now on let us set the second Macdonald parameter
$t$ to zero. Then $P_\la(\cdot\mid q,0)$
are known as the \emph{$q$-Whittaker functions},
i.e., the $q$-deformed $\mathfrak{gl}_{n}$ 
Whittaker functions, cf. \cite{GerasimovLebedevOblezin2011}
and \cite[\S3]{BorodinCorwin2011Macdonald}. 
\begin{remark}
	Other notable degenerations of the Macdonald
	polynomials include the Hall--Littlewood
	polynomials (for $q=0$, $t>0$), 
	and the Schur polynomials (for $q=t$).
	We refer to \cite{Macdonald1995}
	and \cite{Kerov-book}
	for details.
\end{remark}
We will use $q$-binomial coefficients and 
$q$-Pochhammer symbols
\begin{align}\label{q_notation}
	\binom nk_{q}:=\frac{(q;q)_{n}}{(q;q)_{k}(q;q)_{n-k}},
	\qquad
	(a;q)_{m}:=
	\begin{cases}
		(1-a)(1-aq)\ldots(1-aq^{m-1}),&m>0;\\
		1,&m=0;\\
		(1-aq^{-1})^{-1}(1-aq^{-2})^{-1}\ldots(1-aq^{m})^{-1},&m<0
	\end{cases}
\end{align}
to record certain explicit $q$-dependent quantities 
related to $q$-Whittaker functions.\footnote{In the $q$-Pochhammer symbol, 
$m$ may be $+\infty$ since 
$0\le q<1$. Note also that in all cases,
$(a;q)_{m}={(a;q)_{\infty}}/{(aq^{m};q)_{\infty}}$}
We have
\begin{align}
	P_{\la/\mu}(x_1\mid q,0)&=
	\psi_{\la/\mu}x_1^{|\la|-|\mu|},
	\qquad 
	\psi_{\la/\mu}=\psi_{\la/\mu}(q):=\mathbf{1}_{\mu\prech\la}
	\prod_{i=1}^{\ell(\mu)}\binom{\la_i-\la_{i+1}}{\la_i-\mu_i}_{q};
	\label{P_one_variable}
	\\
	\label{Q_one_variable}
	Q_{\la/\mu}(x_1\mid q,0)&=
	\varphi_{\la/\mu}x_1^{|\la|-|\mu|},
	\qquad 
	\varphi_{\la/\mu}=\varphi_{\la/\mu}(q):=\mathbf{1}_{\mu\prech\la}
	\frac{1}{(q;q)_{\la_1-\mu_1}}
	\prod_{i=1}^{\ell(\la)}
	\binom{\mu_i-\mu_{i+1}}{\mu_i-\la_{i+1}}_{q}.
\end{align}

\begin{definition}\label{def:spec}
	A \emph{specialization} 
	of the algebra of symmetric functions
	$\Sym$
	is an algebra morphism $\Sym\to\C$.
	This is a generalization of the operation
	of taking the value of a symmetric function
	at a point.
	We will deal with specializations
	$\mathbf{A}=(\alb;\beb;\ga)$,
	where $\alb=(\al_1\ge\al_2\ge \ldots\ge0)$
	$\beb=(\be_1\ge\be_2\ge \ldots\ge0)$,
	$\ga\ge0$, and $\sum_{i}(\al_i+\be_i)<\infty$,
	which may be defined via the generating function
	corresponding to signatures $(n)\in\GT_1$:
	\begin{align}\label{Pi_definition}
		\sum_{n=0}^{\infty}Q_{(n)}(\mathbf{A})\cdot u^n=
		e^{\ga u}\prod_{i=1}^{\infty}
		\frac{1+\be_i u}{(\al_i u;q)_{\infty}}
		:=\Pi(u;\mathbf{A}).
	\end{align}
	Under these specializations, 
	we have $P_{\la/\mu}(\mathbf{A})\ge0$ for
	any $\mu,\la\in\GT$ (\emph{nonnegativity}). 
	The \emph{Kerov's conjecture} (see \cite[\S2.9.3]{Kerov-book})
	states that 
	the specializations 
	of the form $\mathbf{A}=(\alb;\beb;\ga)$
	exhaust all nonnegative specializations.
\end{definition}

\begin{remark}
The specialization with all $\be_{i}=0$ and $\gamma=0$ is the same as 
assigning values to the formal variables,
$x_j=\al_j$. We will refer to this as the pure $(\al)$ specialization, and to the parameters $\al_j$ as the \emph{usual parameters}. 

If we go back to the case of the nonzero $t$ parameter, then
the corresponding 
specialization with all $\al_{j}=0$ 
and $\gamma=0$ would send $P_{\la/\mu}(\cdot\mid q,t)$ to $Q_{\la'/\mu'}(\be_1,\be_2,\ldots\mid t,q)$, the value of the usual specialization into $(\be_1,\be_2,\ldots)$ 
with $q$ and $t$ swapped. (Formula \eqref{Pi_definition} for nonzero $t$
contains an additional factor
$\prod_{i=1}^{\infty}(t\al_iu;q)_{\infty}$.)
Hence we will refer to such specializations as
pure $(\hat \be)$ specializations, 
and to the parameters $\be_i$ as the \emph{dual parameters} (though
setting $t=0$ as we do in the rest of the paper 
eliminates the ``full'' dual nature of these parameters). 

Finally, 
$\ga$ will be called the \emph{Plancherel parameter}, and the corresponding
specialization can be defined as a limit of the specializations with, 
e.g., $\be_i=\ga/L$, $i=1,\ldots,L$ (and all other parameters zero), as $L\to\infty$.

The present paper mostly deals with 
specializations with 
$\ga=0$.
A treatment of the case $\alb=\beb=0$, $\ga>0$
may be found in \cite{BorodinPetrov2013NN},
\cite{BufetovPetrov2014}. 
\end{remark}

Let $\mathbf{A}\cup\mathbf{B}$ denote the union of 
specializations (a generalization of concatenating the sets 
of variables). Formally it is defined as
$p_k(\mathbf{A}\cup\mathbf{B})=
p_k(\mathbf{A})+
p_k(\mathbf{B})$, $k\ge1$. 
An obvious generalization of the
recurrence
relation \eqref{Macdonald_recurrence}
allows to express $P_\la(\mathbf{A}\cup\mathbf{B})$
through $P_{\la/\mu}(\mathbf{A})$ and $P_\mu(\mathbf{B})$.
Thus, we can equivalently say that the 
specialization into usual parameters
is \emph{completely determined} by  
\eqref{P_one_variable} (or \eqref{Q_one_variable}) 
and \eqref{Macdonald_recurrence}.
Similarly, the specialization into dual parameters
is determined by 
the same recurrence \eqref{Macdonald_recurrence}, 
but with a different one-parameter formula:
\begin{align}
		Q_{\la/\mu}(\be_1\mid q,0)=
		\psi'_{\la/\mu}
		\be_1^{|\la|-|\mu|},
		\qquad
		\psi'_{\la/\mu}
		=\psi'_{\la/\mu}(q):= \mathbf{1}_{\mu\precv\la}\prod_{i\ge 1\colon \la_i=\mu_i,\;
		\la_{i+1}=\mu_{i+1}+1}(1-q^{\mu_i-\mu_{i+1}}).
	\label{Q_one_dual_spec}
\end{align}

We will also need Cauchy identities for
$q$-Whittaker symmetric functions recorded below.
Similar identities
(involving $t$) also exist for the general Macdonald
symmetric functions.
\begin{align}
	\label{Cauchy}
	\sum_{\la\in\GT}P_\la(a_1,\ldots,a_N)Q_\la(\mathbf{A})
	&=\Pi(a_1;\mathbf{A})\ldots
	\Pi(a_N;\mathbf{A});
	\\
	\label{skew_Cauchy}
	\sum_{\varkappa\in\GT}
	P_{\varkappa/\la}(\mathbf{A})
	Q_{\varkappa/\nu}(\mathbf{B})
	&=\Pi(\mathbf{A};\mathbf{B})
	\sum_{\mu\in\GT}
	Q_{\la/\mu}(\mathbf{B})
	P_{\nu/\mu}(\mathbf{A}).
\end{align}
In \eqref{skew_Cauchy}, 
$\Pi(\mathbf{A};\mathbf{B})$ is given by
\begin{align}\label{general_Pi}
	\Pi(\mathbf{A};\mathbf{B})=
	\exp\bigg(
	\sum_{n=1}^{\infty}
	\frac{1}{n}\frac{1}{1-q^{n}}\,
	p_n(\mathbf{A})p_n(\mathbf{B})
	\bigg).
\end{align}
For the proofs see \cite[VI.(2.6) and VI.7, Example 6]{Macdonald1995}.
This definition agrees with \eqref{Pi_definition}
when one of the specializations is into a 
single usual parameter. Note also that
$\Pi(\mathbf{A}\cup\mathbf{B};\mathbf{C})
=\Pi(\mathbf{A};\mathbf{C})
\Pi(\mathbf{B};\mathbf{C})$.

Finally, we will need the Pieri rules: For any $r\ge1$,
\begin{align}\label{Pieri}
	P_{(1^{r})} P_\mu=\sum_{\la\colon\text{$\la/\mu$ is a vertical $r$-strip}}\psi'_{\la/\mu}P_\la,
	\qquad \qquad
	Q_{(r)} P_\mu=
	\sum_{\la\colon\text{$\la/\mu$ is a horizontal $r$-strip}}\varphi_{\la/\mu}P_\la
\end{align}
(an $r$-strip means a strip consisting of $r$ boxes).
Here $P_{(1^{r})}=e_r$ is in fact equal to the $r$-th elementary symmetric
function $e_r(x_1,x_2,\ldots)=
\sum_{i_1<\ldots<i_r}x_{i_1}\ldots x_{i_r}$
(note that $e_1=p_1$),
and the $Q_{(r)}$'s are the quantities entering the generating
function \eqref{Pi_definition}.


\subsection{$q$-Whittaker processes} 
\label{sub:qwhit_proc}

The (depth $N$) \emph{$q$-Whittaker processes}
are probability measures on 
sequences of interlacing  
signatures
$\lab=(\la^{(1)}\prech\la^{(2)}\prech \ldots\prech \la^{(N)})$,
where $\la^{(j)}\in\GT_j$. 
Such sequences are sometimes referred to as
\emph{Gelfand--Tsetlin schemes}, or \emph{patterns},
they first appeared in connection with 
representation theory of unitary groups
\cite{gelfand1950finite}.\footnote{This justifies the notation
``$\mathbb{GT}$'' we are using.}
We will depict sequences $\lab$ as interlacing integer arrays,
and also associate to them configurations of 
particles 
$\{(\la^{(k)}_j,k)\colon k=1,\ldots,N,\; j=1,\ldots,k\}$
on $N$ horizontal copies of $\Z$. See Fig.~\ref{fig:array}.
Let us denote the set of all interlacing
arrays $\lab$ of depth $N$ by $\mathbb{GT}^{(N)}$.

The $q$-Whittaker process
$\MP_{\mathbf{A}}^{\vec a}$ 
depends on a nonnegative specialization\footnote{In 
the rest of the paper, we will speak only about nonnegative
specializations, and omit the word ``nonnegative''.}
$\mathbf{A}=(\alb;\beb;\ga)$ (Definition~\ref{def:spec})
and on additional parameters
$\vec a=(a_1,\ldots,a_N)$
with $a_j>0$,
satisfying 
$\al_ia_j<1$ for all possible $i$ and $j$
(this ensures the finiteness of the normalizing
constant $\Pi(\vec a;\mathbf{A})$ in \eqref{MP_weights} below). 
The probability weights 
$\MP_{\mathbf{A}}^{\vec a}(\lab)$ of interlacing arrays $\lab$
may be defined via the generating function\footnote{In \eqref{MP_genfunc}, $\Pi(\vec u;\mathbf{A})=
\Pi(u_1;\mathbf{A})\ldots \Pi(u_N;\mathbf{A})$, and similarly
for the denominator (cf. \eqref{Pi_definition}, \eqref{general_Pi}).
Here the $a_j$'s are regarded as constants, and 
the $u_j$'s as variables.}
\begin{align}\label{MP_genfunc}
	\sum_{\lab=(\la^{(1)}\prech \ldots
	\prech\la^{(N)})}
	\MP_{\mathbf{A}}^{\vec a}(\lab)
	\left(\frac{u_1}{a_1}\right)^{|\la^{(1)}|}
	\left(\frac{u_2}{a_2}\right)^{|\la^{(2)}|-|\la^{(1)}|}
	\ldots
	\left(\frac{u_N}{a_N}\right)^{|\la^{(N)}|-|\la^{(N-1)}|}
	=
	\frac{\Pi(\vec u;\mathbf{A})}
	{\Pi(\vec a;\mathbf{A})},
\end{align}
plus a certain \emph{$q$-Gibbs property} requiring that 
the quantities
\begin{align}\label{qGibbs}
	\frac{\MP_{\mathbf{A}}^{\vec a}(\lab)}{P_{\la^{(1)}}(a_1)
	P_{\la^{(2)}/\la^{(1)}}(a_2)\ldots P_{\la^{(N)}/\la^{(N-1)}}(a_N)}
\end{align}
depend only on the top row $\la^{(N)}$,
and not on $\la^{(1)},\ldots,\la^{(N-1)}$.
Note that setting $\vec u=\vec a$
turns \eqref{MP_genfunc} into 
an identity stating that the sum of all probability
weights is $1$.

\begin{remark}\label{rmk:q_Gibbs}
	It is natural to call the property 
	involving quantities \eqref{qGibbs}
	``$q$-Gibbs''
	because for $q=0$ and $a_1=\ldots=a_N=1$ 
	it reduces to the following 
	\emph{Gibbs property}: 
	The conditional distribution of the interlacing
	array $\lab$ under $\MP_{\mathbf{A}}^{\vec a}(\lab)\vert_{q=0,\,a_j\equiv 1}$
	obtained by fixing the top row $\la^{(N)}\in\GT_N$
	is the \emph{uniform distribution} 
	on the set of all interlacing arrays
	$\lab\in\mathbb{GT}^{(N)}$ with fixed top row $\la^{(N)}$
	(note that the latter set is finite).
	For general $q$ and $\vec a$, 
	the conditional distribution will not be uniform, 
	but instead each interlacing array will
	have the conditional weight proportional to 
	$P_{\la^{(1)}}(a_1)
	P_{\la^{(2)}/\la^{(1)}}(a_2)\ldots P_{\la^{(N)}/\la^{(N-1)}}(a_N)$.
\end{remark}

By the Cauchy identity
\eqref{Cauchy} and the fact that the $q$-Whittaker polynomials
form a linear basis, definition
\eqref{MP_genfunc}--\eqref{qGibbs} is equivalent to 
\begin{align}\label{MP_weights}
	\MP_{\mathbf{A}}^{\vec a}(\lab)=
	\frac{1}{\Pi(\vec a;\mathbf{A})}{P_{\la^{(1)}}(a_1)
	P_{\la^{(2)}/\la^{(1)}}(a_2)\ldots P_{\la^{(N)}/\la^{(N-1)}}(a_N)
	Q_{\la^{(N)}}(\mathbf{A})},
\end{align}
which is a more traditional definition of the measure
(first given in \cite{BorodinCorwin2011Macdonald}, and earlier 
in \cite{okounkov2003correlation} in the Schur case).
To see this, one also has to note that
$\frac
{P_{\la^{(1)}}(u_1)\ldots P_{\la^{(N)}/\la^{(N-1)}}(u_N)}
{P_{\la^{(1)}}(a_1)\ldots P_{\la^{(N)}/\la^{(N-1)}}(a_N)}$
is equal to the product of $(u_j/a_j)^{|\la^{(j)}|-|\la^{(j-1)}|}$
in the left-hand side of \eqref{MP_genfunc} (provided that the
$\la^{(j)}$'s satisfy the interlacing constraints).

The marginal distribution of the top row $\la^{(N)}$
under $\MP_{\mathbf{A}}^{\vec a}$
is the \emph{$q$-Whittaker measure} $\MM_{\mathbf{A}}^{\vec a}$
which is defined by either of the following equivalent ways:
\begin{align}\label{MP_measure}
	\sum_{\la\in\GT_N}\MM_{\mathbf{A}}^{\vec a}(\la)
	\frac{P_{\la}(\vec u)}
	{P_{\la}(\vec a)}&=\frac{\Pi(\vec u;\mathbf{A})}
	{\Pi(\vec a;\mathbf{A})},
	\\
	\label{MM_measure}
	\MM_{\mathbf{A}}^{\vec a}(\la)&=\frac{P_\la(\vec a)Q_\la(\mathbf{A})}
	{\Pi(\vec a;\mathbf{A})}.
\end{align}


\subsection{Markov dynamics} 
\label{sub:markov_dynamics}

One of the main goals of the present paper is the construction of
\emph{Markov dynamics}
preserving the family of $q$-Whittaker processes.
More precisely, we will deal with infinite matrices
$\Qspec{\mathbf{B}}$ (with rows and columns indexed by 
interlacing arrays) such that 
\begin{align}\label{adding_spec}
	\MP_{\mathbf{A}}^{\vec a}\Qspec{\mathbf{B}}
	=\MP_{\mathbf{A}\cup \mathbf{B}}^{\vec a},
	\qquad
	\sum_{\lab}
	\MP_{\mathbf{A}}^{\vec a}(\lab)\Qspec{\mathbf{B}}
	(\lab\to\nub)
	=\MP_{\mathbf{A}\cup \mathbf{B}}^{\vec a}(\nub),
	\qquad \nub\in\mathbb{GT}^{(N)}
\end{align}
(the second formula is simply an expansion of the matrix
notation in the first formula).
It suffices to consider three elementary cases
for the specialization $\mathbf{B}$ which is added by the dynamics:
\begin{align}\label{3_cases}
\parbox{.9\textwidth}{\begin{enumerate}
	\item $\mathbf{B}=(\al)$ is a specialization
	into one usual parameter $\al$.
	\item $\mathbf{B}=(\hat\be)$ is a specialization
	into one dual parameter $\be$.
	\item $\mathbf{B}$ is a specialization
	with $\alb=\beb\equiv0$ and $\ga>0$. 
\end{enumerate}}
\end{align}
Indeed, applying a sequence of the above elementary steps
one can get a general specialization $\mathbf{B}$
(if the number of parameters $\al_i$ or $\be_j$
is infinite, this also requires a relatively straightforward
limit transition).

\begin{remark}
	Note that setting all parameters
	in a specialization to zero
	leads to an \emph{empty specialization} $\varnothing$.
	The corresponding $q$-Whittaker process
	$\MP_{\varnothing}^{\vec a}$
	is simply a delta measure on the zero configuration
	$\la^{(k)}_{j}= 0$ for all $k,j$.
	Note also that
	$\Qspec{\varnothing}$
	is the identity matrix.
\end{remark}

The third case in \eqref{3_cases} leads to continuous time Markov
dynamics, in which the parameter $\ga$
plays the role of time. These continuous time
dynamics were studied in detail in \cite{BorodinPetrov2013NN}
(see also 
\cite{OConnellPei2012}).
They are simpler than the discrete time processes
(corresponding to the fist two cases in \eqref{3_cases})
considered in the present paper. 

We will thus not focus on continuous time dynamics,
and will deal with construction of 
matrices $\Qspec{\al}$
and $\Qspec{\hat \be}$
whose elements $\Qspec{\al}(\lab\to\nub)$
and $\Qspec{\hat\be}(\lab\to\nub)$
are transition probabilities from $\lab$ to $\nub$
(where $\lab,\nub\in\mathbb{GT}^{(N)}$)
in one step of the discrete time. These matrix elements are nonnegative, and
$\sum_{\nub}\Qspec{\al}(\lab\to\nub)=1$
for all $\lab$ (and similarly for the second matrix). 
It is also helpful to view $\Qspec{\al}$
and $\Qspec{\hat \be}$
as (\emph{Markov}) \emph{operators} acting on 
functions in the spatial variables $\lab$
(e.g., these operators act in the 
space of bounded functions).

Adding a specialization $\mathbf{B}=(\al)$ 
or $(\hat \be)$ to $\mathbf{A}$
as in \eqref{adding_spec}
corresponds to multiplying the right-hand side of 
\eqref{MP_genfunc} by 
\begin{align}\label{MP_multiplication_operators}
	\prod_{j=1}^{N}\frac{(\al a_j;q)_{\infty}}
	{(\al u_j;q)_{\infty}}
	\qquad\text{or}\qquad
	\prod_{j=1}^{N}\frac{1+\be u_j}
	{1+\be a_j},
\end{align}
respectively, since $\Pi(u; \mathbf{A})\Pi(u; \al) = \Pi(u, \mathbf{A} \cup (\al))$ and $\Pi(u; \mathbf{A})\Pi(u; \hat{\be}) = \Pi(u, \mathbf{A} \cup (\hat{\be}))$. (Factors containing $a_j$
correspond to normalization, and it is the dependence on $u_j$
in these expressions
which is crucial.)
The problem of finding Markov operators 
$\Qspec{\al}$
and $\Qspec{\hat \be}$
can thus be informally restated as the problem of \emph{turning} 
(by virtue of \eqref{MP_genfunc}) the
\emph{multiplication operators} in the variables $\vec u$
\eqref{MP_multiplication_operators}
into operators \emph{acting in the spatial variables} $\lab$.

\medskip

A similar problem of turning multiplication operators
\eqref{MP_multiplication_operators}
into operators acting in the spatial variables
$\la\in\GT_N$
may be posed for the generating function
for the $q$-Whittaker measures
\eqref{MP_measure}, \eqref{MM_measure}. 
In this case, the problem 
of finding the corresponding
matrices
$\Pspec{\al}$
and $\Pspec{\hat\be}$
(with rows and columns indexed by signatures $\la\in\GT_N$)
has a \emph{unique} solution:
\begin{proposition}\label{prop:univariate}
	There exist unique transition matrices 
	$\Pspec{\al}$ and $\Pspec{\hat\be}$
	which add specializations
	$(\al)$ or $(\hat\be)$, respectively,
	to the $q$-Whittaker measure $\MM_{\mathbf{A}}^{\vec a}$ for 
	an arbitrary nonnegative specialization $\mathbf{A}$,
	in the sense similar to \eqref{adding_spec}:
	\begin{align*}
		\MM_{\mathbf{A}}^{\vec a}\Pspec{\al}
		=
		\MM_{\mathbf{A}\cup(\al)}^{\vec a},\qquad
		\MM_{\mathbf{A}}^{\vec a}\Pspec{\hat\be}
		=
		\MM_{\mathbf{A}\cup(\hat\be)}^{\vec a}.
	\end{align*}
	Their matrix elements are given by
	\begin{align}
		\Pspec{\al}(\la\to\nu)&=
		\prod_{j=1}^{N}(\al a_j;q)_{\infty}\frac{P_\nu(\vec a)}{P_\la(\vec a)}\,
		\varphi_{\nu/\la}\al^{|\nu|-|\la|}
		\mathbf{1}_{\la\prech\nu};
		\label{univariate_operators1}
		\\
		\Pspec{\hat\be}(\la\to\nu)&=
		\prod_{j=1}^{N}
		\frac{1}{1+\be a_j}\,
		\frac{P_\nu(\vec a)}{P_\la(\vec a)}
		\psi'_{\nu/\la}\be^{|\nu|-|\la|}
		\mathbf{1}_{\la\precv\nu},
		\label{univariate_operators2}
	\end{align}
	where $\varphi_{\nu/\la}$
	and $\psi'_{\nu/\la}$ are explicit quantities
	given in 
	\eqref{Q_one_variable} and 
	\eqref{Q_one_dual_spec},
	respectively. 
\end{proposition}
Transition operators
$\Pspec{\al}$ and $\Pspec{\hat\be}$ 
were introduced in 
\cite{BorodinCorwin2011Macdonald}, 
see also \cite{Borodin2010Schur}
for a similar construction for the 
Schur measures (cf. \S \ref{sub:univariate_dynamics} below). 
\begin{proof}
	Let us consider only the case of $(\hat \be)$,
	the case of $(\al)$ is analogous.

	Multiply both sides of \eqref{MP_measure}
	by 
	$\prod_{j=1}^{N}\frac{1+\be u_j}
	{1+\be a_j}$.
	By the very definition of the $q$-Whittaker measures, 
	the right-hand side can be rewritten as 
	\begin{align*}
		\frac{\Pi(\vec u;\mathbf{A} \cup (\hat \be))}
	{\Pi(\vec a;\mathbf{A} \cup (\hat \be))} = \sum_{\nu\in\GT_N}\MM_{\mathbf{A}\cup(\hat\be)}^{\vec a}(\nu)
		\frac{P_{\nu}(\vec u)}
		{P_{\nu}(\vec a)}.
	\end{align*}
	In the left-hand side, 
	use the well-known property 
	$	
	\prod_{j=1}^{N}({1+\be u_j})
	=\sum_{r=0}^{N}e_r(u_1,\ldots,u_N)\be^{r}
	$ 
	of the elementary
	symmetric functions \cite[I.(2.2)]{Macdonald1995}
	together with the first Pieri rule \eqref{Pieri} to write 
	\begin{align*}
		P_\la(\vec u)\prod_{j=1}^{N}({1+\be u_j})
		=\sum_{\nu\colon\la\precv\nu}
		P_\nu(\vec u)\psi'_{\nu/\la}\be^{|\nu|-|\la|}.
	\end{align*}
	(In the
	$(\al)$ case, one needs to use the generating
	function \eqref{Pi_definition} and the second Pieri rule.)
	Then the left hand side of \eqref{MP_measure}
	multiplied by 
	$\prod_{j=1}^{N}\frac{1+\be u_j}
	{1+\be a_j}$
	becomes 
	\begin{align*}
		\prod_{j=1}^{N}
		\frac{1}{1+\be a_j} \sum_{\la\in\GT_N} \sum_{\nu\colon\la\precv\nu}\MM_{\mathbf{A}}^{\vec a}(\la)
		\frac{P_{\nu}(\vec u)\psi'_{\nu/\la}\be^{|\nu|-|\la|}}
		{P_{\la}(\vec a)}.
	\end{align*}
	Collecting the coefficients
	by $P_\nu(\vec u)/P_\nu(\vec a)$,
	one can rewrite this as
	\begin{align*}
		\sum_{\nu\in\GT_N}\frac{P_{\nu}(\vec u)}
		{P_{\nu}(\vec a)}
		\sum_{\la\colon\la\precv\nu}
		\MM_{\mathbf{A}}^{\vec a}(\la)
		\Pspec{\hat\be}(\la\to\nu),
	\end{align*}
	where
	the operator $\Pspec{\hat\be}$
	is
	given by 
	\eqref{univariate_operators2}. 

	Since $P_\nu(\vec u)/P_\nu(\vec a)$ are linearly independent as polynomials in $\vec{u}$,
	\begin{align*}
	\MM_{\mathbf{A}\cup(\hat\be)}^{\vec a}(\nu) = \sum_{\la\colon\la\precv\nu}\MM_{\mathbf{A}}^{\vec a}(\la)
		\Pspec{\hat\be}(\la\to\nu) = \sum_{\la}\MM_{\mathbf{A}}^{\vec a}(\la)
		\Pspec{\hat\be}(\la\to\nu)
	\end{align*}
	for all $\nu \in \GT_N$. To show uniqueness suppose there is another operator
	$\Pspec{\hat\be}'$ that satisfies 
	\begin{align*}
	\MM_{\mathbf{A}}^{\vec a}\Pspec{\hat\be}'
		=
		\MM_{\mathbf{A}\cup(\hat\be)}^{\vec a}.
	\end{align*}
	Pick $\la_{0}$ and $\nu_{0}$, such that $\Pspec{\hat\be}(\la_{0}\to\nu_{0}) \neq \Pspec{\hat\be}'(\la_{0}\to\nu_{0})$. For any specialization $\mathbf{A}$,
	\begin{align*}
	\sum_{\la \in \GT_N}\MM_{\mathbf{A}}^{\vec a}(\la)\left(\Pspec{\hat\be}(\la \to\nu_{0}) - \Pspec{\hat\be}'(\la\to\nu_{0})\right) = 0.
	\end{align*}
	Take $\mathbf{A}$ to be a pure specialization into usual parameters $(\al_{1}, \ldots, \al_{N})$ and multiply both sides by $\Pi(\vec a;\mathbf{A})$ to get 
	\begin{align*}
	\sum_{\la \in \GT_N}P_\la(\vec a)Q_\la(\al_{1}, \ldots, \al_{N})\left(\Pspec{\hat\be}(\la \to\nu_{0}) - \Pspec{\hat\be}'(\la\to\nu_{0})\right) = 0
	\end{align*}
	for any $\al_{1}, \ldots, \al_{N} \ge 0$
	(in fact, this sum is only over $\la\precv\nu_0$ 
	and thus is finite), which contradicts the fact that $Q_\la(\al_{1}, \ldots, \al_{N})$ are linearly independent as polynomials in $\al_{1}, \ldots, \al_{N}$.
\end{proof}

It follows from \eqref{skew_Cauchy} that both operators $\Pspec{\hat\be}$ and $\Pspec{\al}$ are stochastic, i.e. for any $\la \in \GT_N$
\begin{align}
\label{stochasticity}
\sum_{\nu \in \GT_N}\Pspec{\hat\be}(\la \to\nu) = \sum_{\nu \in \GT_N}\Pspec{\al}(\la \to\nu) = 1.
\end{align}
\begin{remark}\label{rmk:one_level_dynamics}
	If $N=1$ in Proposition \ref{prop:univariate}, 
	then both dynamics $\Pspec{\al}$
	and $\Pspec{\hat\be}$ 
	(living on $\Z_{\ge0}=\GT_{1}$)
	are rather simple.
	Namely, under both dynamics, at each discrete time step 
	the only particle $\la^{(1)}_{1}\in\Z_{\ge0}=\GT_1$
	jumps to the right according to 
	\begin{enumerate}
		\item the \emph{$q$-geometric
		distribution} with parameter $\al a_1$,
		i.e.,
		$\p_{\al a_1}(n):=(\al a_1;q)_{\infty}\frac{(\al a_1)^{n}}{(q;q)_{n}}$,
		$n=0,1,2,\ldots$,\footnote{The fact that 
		this is indeed a probability distribution
		follows from the $q$-binomial theorem.}
		in the case of dynamics $\Pspec{\al}$,
		or
		\item
		the \emph{Bernoulli distribution}
		with parameter $\be a_1$
		in the case of dynamics $\Pspec{\hat\be}$:
		the particle jumps to the 
		right by one with probability $\be a_1/(1+\be a_1)$,
		and stays put with the complementary probability
		$1/(1+\be a_1)$.\footnote{This parametrization of Bernoulli
		random variables will be used throughout the paper.}
	\end{enumerate}
	
	More generally, one can show that 
	under the dynamics on $\GT_N$,
	the quantities $|\la^{(N)}|$
	evolve as follows. For $\Pspec{\al}$,
	at each discrete time step $|\la^{(N)}|$
	is increased by the sum of $N$ independent 
	$q$-geometric random variables with parameters
	$\al a_1,\ldots,\al a_N$.
	For $\Pspec{\hat\be}$,
	at each discrete time step $|\la^{(N)}|$
	is increased by the sum of $N$ independent 
	Bernoulli random variables with parameters
	$\be a_1,\ldots,\be a_N$. To see this, use \eqref{stochasticity} to write 
	\begin{align*}
		\sum_{\nu \in \GT_N}\frac{P_\nu(\vec a)}{P_\la(\vec a)}\,
		\varphi_{\nu/\la}\al^{|\nu|-|\la|}
		\mathbf{1}_{\la\prech\nu} = \prod_{j=1}^{N}\frac{1}{(\al a_j;q)_{\infty}},
		\\
		\sum_{\nu \in \GT_N}
		\frac{P_\nu(\vec a)}{P_\la(\vec a)}
		\psi'_{\nu/\la}\be^{|\nu|-|\la|}
		\mathbf{1}_{\la\precv\nu} = \prod_{j=1}^{N}
		(1+\be a_j)
	\end{align*}
	for any $\la \in \GT_N$. Substituting $\al u$ instead of $\al$ (or $\be u$ instead of $\be$) in these equalities leads to
	\begin{align*}
		\sum_{\nu \in \GT_N}\Pspec{\al}(\la \to \nu)u^{|\nu|-|\la|}
		 = \prod_{j=1}^{N}\frac{(\al  a_{j} ;q)_{\infty}}{(\al  a_{j} u;q)_{\infty}},
		\\
		\sum_{\nu \in \GT_N}\Pspec{\hat \be}(\la \to \nu)u^{|\nu|-|\la|} = \prod_{j=1}^{N}
		\frac{1+\be a_{j}u}{1+\be a_{j}}.
	\end{align*}
	The observation follows, since both left hand sides are probability generating functions of $|\nu| - |\la|$ in the formal variable $u$, and the right-hand sides expand as probability generating functions 
	of sums of independent $q$-geometric or Bernoulli random variables.
\end{remark}

We will call the dynamics
$\Pspec{\al}$ and $\Pspec{\hat\be}$
the \emph{univariate dynamics},
and the corresponding 
dynamics on interlacing arrays
$\Qspec{\al}$ and $\Qspec{\hat\be}$
(which we aim to construct)
the \emph{multivariate dynamics}.
In a way, multivariate dynamics
on arrays $\lab=(\la^{(1)}\prech \ldots\prech\la^{(N)})$
\emph{stitch together} univariate
dynamics on all levels $\la^{(j)}$,
$j=1,\ldots,N$: Namely, started from a $q$-Gibbs distribution,
the multivariate evolution
of the array $\lab$
reduces to the corresponding univariate dynamics
on each of the levels $\la^{(j)}$, $j=1,\ldots,N$.
This fact follows from the proof of Theorem
\ref{thm:main_eq} below,
see also \cite[\S2.2]{BorodinPetrov2013NN}
for a related discussion.

Instead of the case of univariate dynamics
(driven by identity
\eqref{MP_measure}),
the problem of constructing
multivariate dynamics
(involving identity \eqref{MP_genfunc})
has a \emph{whole family of solutions}.
This phenomenon was known in the Schur ($q=0$)
case for some time, with the presence of 
the RSK-type (e.g., see \cite{OConnell2003Trans},
\cite{OConnell2003}) and the push-block 
\cite{BorFerr2008DF}
dynamics
(see \S \ref{sec:schur_degeneration}
below for more detail).
A similar phenomenon was investigated in 
\cite{BorodinPetrov2013NN}
for continuous time dynamics increasing the parameter $\ga$
in the $q$-Whittaker processes.
In that simpler continuous time 
setting, a classification result
was established in the latter paper.

\begin{remark}\label{rmk:univariate_hard}
	Since the $q$-Whittaker polynomials
	$P_\la(\vec a)$ entering \eqref{univariate_operators1}
	and \eqref{univariate_operators2} are not given by an especially nice formula, 
	transition probabilities of the univariate
	dynamics 
	are harder to analyze. On the other hand, 
	RSK-type multivariate dynamics
	which we construct in the present paper
	turn out to have simpler transition probabilities.
	Note also that multivariate dynamics 
	on $q$-Gibbs distributions
	can be 
	used to ``simulate'' the univariate ones, 
	cf. the above discussion about ``stitching''.
\end{remark}


\subsection{Main equations} 
\label{sub:main_equations}

Here we write down linear equations whose solutions correspond
to multivariate discrete time Markov dynamics 
on $q$-Whittaker processes.
Let us first narrow down the 
class of dynamics $\mathscr{Q}$ on interlacing arrays which we 
consider.

\begin{definition}\label{def:sequential_update}
	A dynamics $\mathscr{Q}$ on interlacing arrays
	will be called a
	\emph{sequential update dynamics}
	if its one-step transition probabilities from $\lab$ to $\nub$,
	$\lab,\nub\in\mathbb{GT}^{(N)}$,
	have a product form
	\begin{align}
		\mathscr{Q}(\lab\to\nub)=
		\mathscr{U}_1(\la^{(1)}\to\nu^{(1)})
		\mathscr{U}_2(\la^{(2)}\to\nu^{(2)}\mid \la^{(1)}\to\nu^{(1)})
		\ldots
		\mathscr{U}_N(\la^{(N)}\to\nu^{(N)}
		\mid \la^{(N-1)}\to\nu^{(N-1)}),
		\label{Q_sequential_update}
	\end{align}
	where $\mathscr{U}_j$'s are conditional probabilities
	of transitions at levels $j=1,\ldots,N$
	satisfying\footnote{By agreement,
	for $j=1$ 
	we mean
	$\mathscr{U}_j(\la^{(j)}\to\nu^{(j)}\mid 
	\la^{(j-1)}\to\nu^{(j-1)})\equiv
	\mathscr{U}_1(\la^{(1)}\to\nu^{(1)})$.}
	\begin{align}
		\mathscr{U}_j(\la^{(j)}\to\nu^{(j)}\mid 
		\la^{(j-1)}\to\nu^{(j-1)})\ge0,
		\qquad
		\sum_{\nu^{(j)}\in\GT_j}
		\mathscr{U}_j(\la^{(j)}\to\nu^{(j)}\mid 
		\la^{(j-1)}\to\nu^{(j-1)})=1.
		\label{U_properties}
	\end{align}
	In words, the transition $\lab\to\nub$ looks as follows.
	First, update $\la^{(1)}\to\nu^{(1)}$ at the
	bottom level $\GT_1$ according to the distribution
	$\mathscr{U}_1$.
	Then for each $j=2,\ldots,N$,
	given the transition $\la^{(j-1)}\to\nu^{(j-1)}$ at the previous level,
	update $\la^{(j)}\to\nu^{(j)}$
	at level $\GT_j$
	according to the conditional distribution~$\mathscr{U}_j$.
	We see that the evolution of several first levels
	$\la^{(1)},\ldots,\la^{(k)}$ of the interlacing array
	does not depend on what is happening at the upper levels
	$\la^{(k+1)},\ldots,\la^{(N)}$.
\end{definition}
This setting of sequential update dynamics
is not too restrictive as 
it covers all previously known examples of dynamics on 
Macdonald (in particular, on $q$-Whittaker and Schur)
processes,~cf.~\cite{BorodinPetrov2013NN}.
For a sequential update dynamics
it suffices to
describe the evolution
at any two consecutive levels $j-1$ and $j$.

\begin{theorem}\label{thm:main_eq}
	A sequential update dynamics
	$\mathscr{Q}$ defined via \eqref{Q_sequential_update}--\eqref{U_properties}
	preserves the class of $q$-Whittaker processes
	$\MP_{\mathbf{A}}^{\vec a}$
	and adds a new usual parameter $\al$
	to the specialization $\mathbf{A}$ if and only if
	\begin{align}\label{alpha_main_equation}
		\sum_{\bar\la\in\GT_{j-1}}
		\mathscr{U}_j
		(\la\to\nu\mid \bar\la\to\bar\nu)
		(\al a_j)^{|\la|-|\nu|-(|\bar\la|-|\bar\nu|)}
		\psi_{\la/\bar\la}\varphi_{\bar\nu/\bar\la}
		=
		(\al a_j;q)_{\infty}
		\psi_{\nu/\bar\nu}\varphi_{\nu/\la}
	\end{align}
	for any $j=1, 2,\ldots,N$ and
	any $\la,\nu\in\GT_j$, $\bar\nu\in\GT_{j-1}$,
	such that the four signatures
	$\bar\la,\bar\nu,\la,\nu$ are related 
	as on Fig.~\ref{fig:square}, left (in particular, 
	the above summation is taken only over $\bar\la$ satisfying
	$\bar\la\prech\bar\nu$, $\bar\la\prech\la$). 
	For $j=1$ we take $\bar \la = \bar \nu = \varnothing$ in this equation and it becomes equivalent to $\mathscr{U}_1=\Pspec{\al}$ at level $\GT_1$ (as in Remark \ref{rmk:one_level_dynamics}).

	Similarly, a dynamics $\mathscr{Q}$ 
	preserves the class of $q$-Whittaker processes
	and adds a new dual parameter $\be$
	to the specialization $\mathbf{A}$ if and only if
	\begin{align}\label{beta_main_equation}
		\sum_{\bar\la\in\GT_{j-1}}
		\mathscr{U}_j
		(\la\to\nu\mid \bar\la\to\bar\nu)
		(\be a_j)^{|\la|-|\nu|-(|\bar\la|-|\bar\nu|)}
		\psi_{\la/\bar\la}\psi'_{\bar\nu/\bar\la}
		=
		\frac{1}{1+\be a_j}
		\psi_{\nu/\bar\nu}\psi'_{\nu/\la}
	\end{align}
	for any $j=1, 2,\ldots,N$ and
	any $\la,\nu\in\GT_j$, $\bar\nu\in\GT_{j-1}$,
	such that the four signatures
	$\bar\la,\bar\nu,\la,\nu$ are related 
	as on Fig.~\ref{fig:square}, right (in particular, 
	the above summation is taken only over $\bar\la$ satisfying
	$\bar\la\precv\bar\nu$, $\bar\la\prech\la$).
	For $j=1$ we take $\bar \la = \bar \nu = \varnothing$ in this equation and it becomes equivalent to $\mathscr{U}_1=\Pspec{\hat \be}$ at level $\GT_1$ (as in Remark \ref{rmk:one_level_dynamics}).
\end{theorem}
\begin{figure}[htbp]
	\begin{tabular}{cccc}	
	\raisebox{70pt}{$(\al)$}\hspace*{10pt}&\begin{tikzpicture}
		[scale=1.5, very thick]
		\node at (1,0) {$j-1$};
		\node at (1,1) {$j$};
		\node at (2,1) {\framebox{$\la$\rule{0pt}{10pt}}};
		\node at (3,1) {\framebox{$\nu${}\rule{0pt}{10pt}}};
		\node at (2,0) {\framebox{$\bar{\la}${}\rule{0pt}{10pt}}};
		\node at (3,0) {\framebox{$\bar{\nu}${}\rule{0pt}{10pt}}};
		\node[rotate=90] at (3,.5) {$\prech$};
		\node[rotate=90] at (2,.5) {$\prech$};
		\node at (2.5,0) {$\prech$};
		\node at (2.5,1) {$\prech$};
		\node at (2.5,-.5) 
		{${\xrightarrow[\text{time}]{\hspace*{70pt}}}$};
	\end{tikzpicture}
	&\hspace{70pt}
	\raisebox{70pt}{$(\hat \be)$}\hspace*{10pt}&
	\begin{tikzpicture}
		[scale=1.5, very thick]
		\node at (4,0) {$j-1$};
		\node at (4,1) {$j$};
		\node at (2,1) {\framebox{$\la$\rule{0pt}{10pt}}};
		\node at (3,1) {\framebox{$\nu${}\rule{0pt}{10pt}}};
		\node at (2,0) {\framebox{$\bar{\la}${}\rule{0pt}{10pt}}};
		\node at (3,0) {\framebox{$\bar{\nu}${}\rule{0pt}{10pt}}};
		\node[rotate=90] at (3,.5) {$\prech$};
		\node[rotate=90] at (2,.5) {$\prech$};
		\node at (2.5,0) {$\precv$};
		\node at (2.5,1) {$\precv$};
		\node at (2.5,-.5) 
		{${\xrightarrow[\text{time}]{\hspace*{70pt}}}$};
	\end{tikzpicture}
	\end{tabular}
\caption{Squares of four signatures 
on two consecutive levels 
relevant to conditional transition $\la\to\nu$
on the upper level
given the transition $\bar\la\to\bar\nu$
on the lower level,
under dynamics
$\Qspec{\al}$ (left) and $\Qspec{\hat\be}$ (right).
Note the similarity to blocks in Fomin's 
growth diagrams (about the latter,
see
\cite{fomin1979thesis}, 
\cite{Fomin1986},
\cite{fomin1994duality},
\cite{fomin1995schensted}).}
\label{fig:square}
\end{figure}
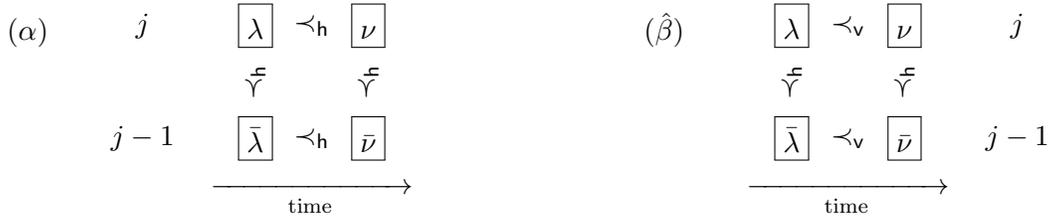
The proof of
	these equations was already established
	in \cite[\S2.2]{BorodinPetrov2013NN} 
	using a more general framework of Gibbs-like measures. However, for the sake of completeness, we reproduce it here in our particular setting of the
	$q$-Whittaker processes.
\begin{proof}
    Let us consider only the case of adding $(\al)$, as the case of $(\hat \be)$ is analogous.
    
    The fact that a sequential update dynamics
	$\mathscr{Q}$ defined via \eqref{Q_sequential_update}--\eqref{U_properties}
	preserves the class of $q$-Whittaker processes
	$\MP_{\mathbf{A}}^{\vec a}$
	and adds a new usual parameter $\al$
	to the specialization $\mathbf{A}$ means that  
	\begin{multline}
	\label{qWhittaker_preservation1}
	\sum_{\lab}\frac{1}{\Pi(\vec a;\mathbf{A})}{P_{\la^{(1)}}(a_1)
	P_{\la^{(2)}/\la^{(1)}}(a_2)\ldots P_{\la^{(N)}/\la^{(N-1)}}(a_N)
	Q_{\la^{(N)}}(\mathbf{A})}\\ \mathscr{U}_1(\la^{(1)}\to\nu^{(1)})
		\mathscr{U}_2(\la^{(2)}\to\nu^{(2)}\mid \la^{(1)}\to\nu^{(1)})
		\ldots
		\mathscr{U}_N(\la^{(N)}\to\nu^{(N)}
		\mid \la^{(N-1)}\to\nu^{(N-1)}) = \\ = \frac{1}{\Pi(\vec a;\mathbf{A} \cup (\al))}{P_{\nu^{(1)}}(a_1)
	P_{\nu^{(2)}/\nu^{(1)}}(a_2)\ldots P_{\nu^{(N)}/\nu^{(N-1)}}(a_N)
	Q_{\nu^{(N)}}(\mathbf{A} \cup (\al))} \quad \text{for every} \ \nub \ \text{and} \ \mathbf{A}. 
	\end{multline}
	Using \eqref{Macdonald_recurrence}, we can rewrite  \eqref{qWhittaker_preservation1} as 
	\begin{align*}
	\sum_{\lab} \left(\prod_{j=1}^{N}\mathscr{U}_j(\la^{(j)}\to\nu^{(j)}
		\mid \la^{(j-1)}\to\nu^{(j-1)})a_{j}^{(|\la^{(j)}|-|\nu^{(j)}|) - (|\la^{(j-1)}|-|\nu^{(j-1)}|)} \psi_{\la^{(j)}/\la^{(j-1)}}\right) 
	Q_{\la^{(N)}}(\mathbf{A})  = \\ =
		\left(\prod_{j=1}^{N}(\al a_{j};q)_{\infty} \psi_{\nu^{(j)}/\nu^{(j-1)}}\right) \sum_{\nu \in \GT_N}Q_{\nu}(\mathbf{A})\al^{|\nu^{(N)}|-|\nu|}\varphi_{\nu^{(N)}/\nu} \quad \text{for every} \ \nub \ \text{and} \ \mathbf{A}. 
	\end{align*}
    Since $Q_{\la}(\mathbf{A})$ are linearly independent as polynomials in $u_{1}, \ldots, u_{N}$ for a specialization $\mathbf{A}$ into usual variables $(u_{1}, \ldots, u_{N})$, this is equivalent to saying that     
    \begin{multline}
    \label{qWhittaker_preservation}
	\sum_{\lab: \la^{(N)} = \la} \prod_{j=1}^{N}\mathscr{U}_j(\la^{(j)}\to\nu^{(j)}
		\mid \la^{(j-1)}\to\nu^{(j-1)})(\al a_{j})^{(|\la^{(j)}|-|\nu^{(j)}|) - (|\la^{(j-1)}|-|\nu^{(j-1)}|)} \psi_{\la^{(j)}/\la^{(j-1)}} 
	  = \\ =
		\varphi_{\nu^{(N)}/\la^{(N)}} \prod_{j=1}^{N}(\al a_{j};q)_{\infty} \psi_{\nu^{(j)}/\nu^{(j-1)}} \quad \text{for all} \ \nub \ \text{and} \ \lab. 
	\end{multline}
	
	For the proof in one direction, suppose that $\mathscr{U}_1=\Pspec{\al}$ at level $\GT_1$, and $\mathscr{U}_j(\la^{(j)}\to\nu^{(j)}
		\mid \la^{(j-1)}\to\nu^{(j-1)})$ satisfy \eqref{alpha_main_equation} for $2 \le j \le N$. Then we can show by induction on $k$, that 
	 \begin{multline}
	 \label{qWhittaker_preservation2}
	\sum_{\lab: \la^{(k)} = \la} \prod_{j=1}^{k}\mathscr{U}_j(\la^{(j)}\to\nu^{(j)}
		\mid \la^{(j-1)}\to\nu^{(j-1)})(\al a_{j})^{(|\la^{(j)}|-|\nu^{(j)}|) - (|\la^{(j-1)}|-|\nu^{(j-1)}|)} \psi_{\la^{(j)}/\la^{(j-1)}} 
	  = \\ =
		\varphi_{\nu^{(k)}/\la^{(k)}} \prod_{j=1}^{k}(\al a_{j};q)_{\infty} \psi_{\nu^{(j)}/\nu^{(j-1)}}
		\ \text{for all} \ 1 \le k \le N, \nub = (\nu^{(1)}\prech\nu^{(2)}\prech \ldots\prech \nu^{(k)}), \la \in \GT_k.
	\end{multline}
	Base for $k=1$ follows from the fact that $\mathscr{U}_1=\Pspec{\al}$ at level $\GT_1$, while the inductive step follows from \eqref{alpha_main_equation}. So \eqref{qWhittaker_preservation} holds. 

	For the other direction, suppose that \eqref{qWhittaker_preservation} (and hence \eqref{qWhittaker_preservation1}) holds. For $1 \le k \le N$ by summing \eqref{qWhittaker_preservation1} over $\nu^{(k+1)}, \ldots, \nu^{(N)}$ and applying \eqref{skew_Cauchy} we get
	\begin{multline*}
	\sum_{\lab}\frac{1}{\Pi(a_{1}, \ldots, a_{k};\mathbf{A})}{P_{\la^{(1)}}(a_1)
	P_{\la^{(2)}/\la^{(1)}}(a_2)\ldots P_{\la^{(k)}/\la^{(k-1)}}(a_k)
	Q_{\la^{(k)}}(\mathbf{A})}\\ \mathscr{U}_1(\la^{(1)}\to\nu^{(1)})
		\mathscr{U}_2(\la^{(2)}\to\nu^{(2)}\mid \la^{(1)}\to\nu^{(1)})
		\ldots
		\mathscr{U}_k(\la^{(k)}\to\nu^{(k)}
		\mid \la^{(k-1)}\to\nu^{(k-1)}) = \\ = \frac{1}{\Pi(a_{1}, \ldots, a_{k};\mathbf{A} \cup (\al))}{P_{\nu^{(1)}}(a_1)
	P_{\nu^{(2)}/\nu^{(1)}}(a_2)\ldots P_{\nu^{(k)}/\nu^{(k-1)}}(a_k)
	Q_{\nu^{(k)}}(\mathbf{A} \cup (\al))} \ \text{for every} \ \nub \ \text{and} \ \mathbf{A}, 
	\end{multline*}
	which implies \eqref{qWhittaker_preservation2}. For $k=1$ it means that $\mathscr{U}_1=\Pspec{\al}$ at level $\GT_1$, while for $k \ge 2$ using  \eqref{qWhittaker_preservation2} for both $k$ and $k-1$ implies \eqref{alpha_main_equation}.
\end{proof}

In a continuous time setting, there also exist linear equations
governing multivariate dynamics, cf. \cite[\S2.4]{BorodinPetrov2013NN}. In fact, the latter equations arise as small $\al$
or small $\be$ limits of \eqref{alpha_main_equation}
or \eqref{beta_main_equation}, respectively.
Markov dynamics
on $q$-Whittaker processes 
corresponding to solutions to these continuous time
equations
were constructed in 
\cite{OConnellPei2012}, \cite{BorodinPetrov2013NN},
\cite{BufetovPetrov2014}.


\subsection{Discussion of main equations} 
\label{sub:discussion_of_main_equations}

Let us make a number of general remarks about the main equations 
of Theorem \ref{thm:main_eq}.

\subsubsection{}\label{ssub:classification}

The paper \cite{BorodinPetrov2013NN} contains a 
classification result in continuous time setting, 
which was achieved by further restricting the 
class of dynamics by imposing certain 
\emph{nearest neighbor interaction} constraints.
Under these constraints,
putting together continuous time linear equations
(which look similarly to \eqref{alpha_main_equation} and \eqref{beta_main_equation})
with fixed $\la$ and $\bar\nu$ in a generic position, 
at level $j$
one arrives at a system of
$j$ linear equations
with $3j-2$ variables.
Solutions of such a system admit a reasonable classification.

It remains unclear how to impose (preferably, natural) constraints
on solutions of discrete time equations
\eqref{alpha_main_equation} or \eqref{beta_main_equation}
so that the family of all solutions would admit 
a reasonable description. 
Indeed, for example, in the case of a usual parameter \eqref{alpha_main_equation},
the number of variables is infinite while the number
of available equations is finite.
Therefore, in \S \ref{sec:bernoulli_} and 
\S \ref{sec:geometric_q_rsks} below we devote our attention to constructing
certain particular
multivariate discrete time dynamics
satisfying
equations \eqref{beta_main_equation} 
and \eqref{alpha_main_equation}, respectively.

\subsubsection{}\label{ssub:refined_Cauchy}

Note that summing \eqref{alpha_main_equation}
or \eqref{beta_main_equation} over $\nu\in\GT_j$
leads to the skew Cauchy identity
with both specializations being into one parameter
(cf. \eqref{skew_Cauchy}):
\begin{align}
	\sum_{\bar\la\in\GT}P_{\la/\bar\la}(a_j)
	Q_{\bar\nu/\bar\la}(\mathbf{B})=
	\frac1{\Pi(a_j;\mathbf{B})}
	\sum_{\nu\in\GT}
	P_{\nu/\bar\nu}(a_j)Q_{\nu/\la}(\mathbf{B}),
	\qquad
	\text{$\mathbf{B}=(\al)$ or $(\hat\be)$}.
	\label{skew_Cauchy_concrete}
\end{align}

Identity \eqref{skew_Cauchy_concrete} 
may also be interpreted as 
a certain commutation relation between the univariate
Markov operators $\Pspec{\al}$
or $\Pspec{\hat\be}$ (of Proposition \ref{prop:univariate})
and Markov projection operators 
(or \emph{links})\footnote{These links in fact
determine the $q$-Gibbs property \eqref{qGibbs}; e.g.,
see \cite[\S2]{BorodinPetrov2013NN} for more detail.}
\begin{align*}
	\La^{j}_{j-1}(\la,\bar\la):=
	\frac{P_{\bar\la}(a_1,\ldots,a_{j-1})}
	{P_\la(a_1,\ldots,a_j)}P_{\la/\bar\la}(a_j),
	\qquad
	\la\in\GT_j,\quad 
	\bar\la\in\GT_{j-1},
\end{align*}
in the sense that
\begin{align}
	\Pspec{\al}^{(j)}\La^{j}_{j-1}
	=
	\La^{j}_{j-1}
	\Pspec{\al}^{(j-1)},
	\label{comm_rel}
\end{align}
and similarly for $\Pspec{\hat\be}$. 
Indices $j$ and $j-1$
in $\Pspec{\al}$ above mean the 
level of the interlacing array at which 
the transition operator of the univariate dynamics acts.

One can thus say that each
solution to the main equations \eqref{alpha_main_equation}
or \eqref{beta_main_equation}
(and, therefore, each discrete time Markov dynamics
on $q$-Whittaker processes)
corresponds to a \emph{refinement} of the skew
Cauchy identity \eqref{skew_Cauchy_concrete}
(or of the commutation relation \eqref{comm_rel}).

\begin{remark}\label{rmk:Fomin}
	When $\mathbf{B}=(\al)$ is a usual specialization,
	one may check that all quantities entering both sides of
	\eqref{skew_Cauchy_concrete}
	can be viewed as generating series in $q$, $\al$, and $a_j$
	with nonnegative integer coefficients. It would be very interesting 
	to understand whether there is a bijective mechanism behind identity
	\eqref{skew_Cauchy_concrete}
	similar to the one existing in the classical ($q=0$) case
	(see also the discussion after Theorem \ref{thm:schur_RSK}).
	We are very grateful to Sergey Fomin for this comment.
\end{remark}

\subsubsection{}

The parameters $a_1,\ldots,a_{j-1}$ (but not $a_j$)
essentially do not contribute to the main equations
\eqref{alpha_main_equation}, \eqref{beta_main_equation}:
they enter
the equations only as a requirement
that $\bar\nu\in\GT_{j-1}$
and $\la,\nu\in\GT_{j}$.
Thus, equations 
\eqref{alpha_main_equation}, \eqref{beta_main_equation} 
essentially depend on two specializations:
a specialization into one usual parameter
$\boldsymbol{\Lambda}=(a_j)$
which corresponds to increasing the level number, 
and a specialization $\mathbf{B}=(\al)$
or $(\hat\be)$ which corresponds 
to time evolution.
This allows to think of diagrams as on
Fig.~\ref{fig:square},
as well as of main equations,
for \emph{any} specializations
$\boldsymbol{\Lambda}$ and $\mathbf{B}$
(see 
Fig.~\ref{fig:square_general}).\begin{figure}[htbp]
	\begin{tabular}{cc}
	\begin{tikzpicture}
		[scale=1.7, very thick]
		\node at (2,1) {\framebox{$\la$\rule{0pt}{10pt}}};
		\node at (3,1) {\framebox{$\nu${}\rule{0pt}{10pt}}};
		\node at (2,0) {\framebox{$\bar{\la}${}\rule{0pt}{10pt}}};
		\node at (3,0) {\framebox{$\bar{\nu}${}\rule{0pt}{10pt}}};
		\node[rotate=90] at (3,.5) {$\prec_{\boldsymbol{\Lambda}}$};
		\node[rotate=90] at (2,.5) {$\prec_{\boldsymbol{\Lambda}}$};
		\node at (2.5,0) {$\prec_{\mathbf{B}}$};
		\node at (2.5,1) {$\prec_{\mathbf{B}}$};
		\node at (2.5,-.5) 
		{$\xrightarrow[\text{time}]{\hspace*{70pt}}$};
		\node[rotate=90] at (1.6,.5) 
		{$\xrightarrow[\hspace*{65pt}]{\text{level}}$};
	\end{tikzpicture}
	\end{tabular}
\caption{A square of four signatures corresponding
to arbitrary specializations $\boldsymbol{\Lambda}$ 
and $\mathbf{B}$. Notation $\bar\la\prec_{\boldsymbol{\Lambda}}\la$
means that $P_{\la/\bar\la}(\boldsymbol{\Lambda})>0$,
and similarly for $\prec_{\mathbf{B}}$.
When the specialization $\boldsymbol{\Lambda}$
is into a single usual or dual parameter,
$\prec_{\boldsymbol{\Lambda}}$ reduces to 
$\prech$ or $\precv$, respectively.}
\label{fig:square_general}
\end{figure}
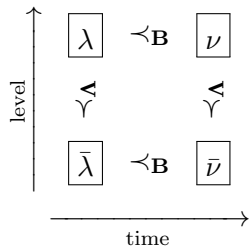
It suffices
to consider three elementary cases 
for $\boldsymbol{\Lambda}$ and $\mathbf{B}$
as in \eqref{3_cases}.
This yields 9 possible systems of equations for 
dynamics.
If one of the specializations is pure Plancherel
(case (3) in \eqref{3_cases}),
then the corresponding Markov dynamics
on $q$-Whittaker processes
were essentially
constructed in \cite{BorodinPetrov2013NN},
\cite{BufetovPetrov2014}.
This leaves four systems of equations in which
both $\boldsymbol{\Lambda}$ and $\mathbf{B}$ 
are specializations into a single usual or dual parameter.
In this paper we address two of these four cases
corresponding to $\boldsymbol{\Lambda}=(a_j)$,
which in particular give rise to two new discrete 
time $q$-PushTASEPs (as marginally Markovian
projections of dynamics on interlacing arrays, 
see \S \ref{sub:bernoulli_q_pushtasep} and \S \ref{sub:geometric_q_pushtasep}).

\subsubsection{}

In fact, one can 
define the quantities $\psi_{\la/\mu}(q,t)$,
$\varphi_{\la/\mu}(q,t)$,
$\psi'_{\la/\mu}(q,t)$
for the general Macdonald parameters $(q,t)$
(see \cite[Ch. VI]{Macdonald1995}),
and thus write down the corresponding 
main linear equations for any specializations 
$\boldsymbol{\Lambda}$ and $\mathbf{B}$. (In particular, for $t\ne0$
the right-hand side of the identity \eqref{Pi_definition} 
defining a specialization
should be multiplied by $\prod_{i=1}^{\infty}(t\al_i u;q)_{\infty}$.)
It is not known whether there exist other solutions
to the main equations for general $(q,t)$
yielding honest 
Markov dynamics (i.e., having \emph{nonnegative}
transition probabilities)
except the \emph{push-block solution} 
(see \S \ref{sec:push_block_and_rsk_type_dynamics}
below for the definition). We do not address this question in the present paper.

There is a rather simple transformation 
of the main equations for general $(q,t)$
(related to transposition
of Young diagrams)
which interchanges $q\leftrightarrow t$
and swaps usual and dual parameters
in both specializations $\boldsymbol{\Lambda}$
and $\mathbf{B}$
\cite{BufetovPetrov2014}.
This transformation relates the 
$q$-Whittaker ($t=0$)
and the Hall--Littlewood ($q=0$) settings. 

The remaining two cases of the ($q$-Whittaker) 
main equations mentioned above
(corresponding to 
$\boldsymbol{\Lambda}=(\hat b)$, a specialization into a dual parameter)
should thus be thought of as 
discrete time versions
of the
continuous time equations of \cite{BufetovPetrov2014}
(relevant to the Hall--Littlewood setting). As such, 
(conjectural) solutions to the former equations
leading to discrete time
dynamics on interlacing arrays 
are unlikely to produce new 
marginally Markovian
TASEP-like particle systems in one space dimension
(see also discussion in \cite[\S8.3]{BorodinPetrov2013NN}).
In the present paper, we do not address these two remaining cases
corresponding
to the Hall--Littlewood setting.



\section{Push-block and RSK-type dynamics} 
\label{sec:push_block_and_rsk_type_dynamics}

\subsection{Push-block dynamics} 
\label{sub:push_block_dynamics}

There is a rather straightforward general construction 
(dating back to an idea of \cite{DiaconisFill1990})
leading to certain particular multivariate dynamics.
Namely, 
assume that the conditional probabilities
$\mathscr{U}_j(\la\to\nu\mid \bar\la\to\bar\nu)$
entering the main equations (Theorem \ref{thm:main_eq})
\emph{do not depend} on $\bar\la$.
Then each equation (corresponding to 
fixed $\la,\nu\in\GT_j$, and $\bar\nu\in\GT_{j-1}$) contains only one 
unknown $\mathscr{U}_j(\la\to\nu\mid \bar\nu)$.
With this restriction the main equations admit a unique solution.
Let us consider the case of a usual parameter $\al$ \eqref{alpha_main_equation}.
Observe that
the left-hand side
of \eqref{alpha_main_equation} takes the following form
(where signatures satisfy conditions
on Fig.~\ref{fig:square}, left):
\begin{align*}
	&
	\mathscr{U}_j
	(\la\to\nu\mid \bar\nu)
	\sum_{\bar\la\in\GT_{j-1}}
	(\al a_j)^{|\la|-|\nu|-(|\bar\la|-|\bar\nu|)}
	\psi_{\la/\bar\la}\varphi_{\bar\nu/\bar\la}
	\\&\hspace{20pt}=
	\mathscr{U}_j
	(\la\to\nu\mid \bar\nu)\,
	\al^{|\la|-|\nu|}
	a_j^{-|\nu|+|\bar\nu|}
	\sum_{\bar\la\in\GT_{j-1}}
	P_{\la/\bar\la}(a_j)Q_{\bar\nu/\bar\la}(\al)
	\\&\hspace{20pt}=
	\mathscr{U}_j
	(\la\to\nu\mid \bar\nu)\,
	(\al a_j;q)_{\infty}
	\sum_{\varkappa\in\GT_j}
	(\al a_j)^{|\varkappa|-|\nu|}
	\psi_{\varkappa/\bar\nu}\varphi_{\varkappa/\la},
\end{align*}
where we have used the skew Cauchy identity 
\eqref{skew_Cauchy_concrete}.
Then \eqref{alpha_main_equation} yields the solution
\begin{align}\label{alpha_push_block_solution}
	\mathscr{U}_j
	(\la\to\nu\mid \bar\nu)=
	\frac{(\al a_j)^{|\nu|}
	\psi_{\nu/\bar\nu}\varphi_{\nu/\la}}
	{\sum_{\varkappa\in\GT_j}
	(\al a_j)^{|\varkappa|}
	\psi_{\varkappa/\bar\nu}\varphi_{\varkappa/\la}}.
\end{align}
In \eqref{alpha_push_block_solution}
as well as in the above computation,
it should be 
$\la\prech\nu$,
$\bar\nu\prech\nu$
and
$\la\prech\varkappa$,
$\bar\nu\prech\varkappa$, see Fig.~\ref{fig:square}, left.

Similarly, the solution of
\eqref{beta_main_equation} 
not depending on $\bar\la$
looks as
\begin{align}\label{beta_push_block_solution}
	\mathscr{U}_j
	(\la\to\nu\mid \bar\nu)=
	\frac{(\be a_j)^{|\nu|}
	\psi_{\nu/\bar\nu}\psi'_{\nu/\la}}
	{\sum_{\varkappa\in\GT_j}
	(\be a_j)^{|\varkappa|}
	\psi_{\varkappa/\bar\nu}\psi'_{\varkappa/\la}}.
\end{align}
The signatures
have to be related as on Fig.~\ref{fig:square}, right,
i.e., $\la\precv\nu$, $\bar\nu\prech\nu$,
and
$\la\precv\varkappa$, $\bar\nu\prech\varkappa$.

\begin{definition}
	\label{def:push_block}
	We will call
	the multivariate dynamics
	defined by \eqref{alpha_push_block_solution}
	or \eqref{beta_push_block_solution} 
	the (\emph{discrete time})
	\emph{push-block dynamics}
	on $q$-Whittaker processes
	adding a specialization $(\al)$
	or $(\hat\be)$, respectively. 
	About the 
	name see \S \ref{sub:push_block_dynamics_schur}
	below.
	We denote these dynamics by 
	$\Qapb$ and $\Qbpb$.
\end{definition}
The construction of push-block
dynamics can be equivalently described 
as follows. Recall the commutation relation
between the univariate dynamics $\mathscr{P}$
and the stochastic links $\La^{j}_{j-1}$
\eqref{comm_rel}. 
Then one can say that
the multivariate 
dynamics chooses $\nu$ at random according
to the distribution
of the middle signature
in a chain of Markov operators
\begin{align*}
	\la\xrightarrow[\hspace{30pt}]{\mathscr{P}^{(j)}_{\phantom{j-1}}}
	\nu\xrightarrow[\hspace{30pt}]{\La^{j}_{j-1}}
	\bar\nu,
\end{align*}
conditioned on 
the first signature $\la$
and the last signature $\bar\nu$.
Denominators in formulas \eqref{alpha_push_block_solution}
and \eqref{beta_push_block_solution} reflect this 
conditioning.

The push-block dynamics
(in the Schur case)
first appeared in \cite{BorFerr2008DF},
see also \S \ref{sub:push_block_dynamics_schur} below. 
For analogous dynamics in continuous time and space
in which the univariate dynamics is the Dyson's Brownian motion
see \cite{warren2005dyson}. 
As shown in \cite{GorinShkolnikov2012},
the latter dynamics is a diffusion limit of the discrete-space dynamics from \cite{BorFerr2008DF}.


\subsection{RSK-type dynamics} 
\label{sub:rsk_type_dynamics}

Let us now define an important subclass
of multivariate dynamics which is 
central to the present paper.

\begin{definition}\label{def:rsk}
	A multivariate sequential update dynamics $\mathscr{Q}$
	(which corresponds to conditional
	probabilities $\mathscr{U}_j(\la\to\nu\mid\bar\la\to\bar\nu)$
	satisfying \eqref{U_properties}
	and the main equations \eqref{alpha_main_equation} or
	\eqref{beta_main_equation})
	is called \emph{RSK-type} if
	\begin{align*}
		\mathscr{U}_j(\la\to\nu\mid \bar\la\to\bar\nu)
		=0
		\quad
		\text{unless $|\nu|-|\la|\ge |\bar\nu|-|\bar\la|$},
		\qquad \text{for all}\ \la,\nu\in\GT_j,\
		\bar\la,\bar\nu\in\GT_{j-1}.
	\end{align*}
\end{definition}
In the above definition, $|\bar\nu|-|\bar\la|$
is the total distance traveled by particles at level
$j-1$, and similarly $|\nu|-|\la|$
is the total distance traveled by particles at level $j$.
Informally, under an RSK-type dynamics all movement
at level $j-1$ \emph{must propagate further} to level $j$
(and, consequently, to all upper levels of the array).

By Remark \ref{rmk:one_level_dynamics},
under an RSK-type dynamics
the quantity $|\la^{(j)}|-|\la^{(j-1)}|$
(for any $j=1,\ldots,N$)
at each step of the discrete time
is increased by adding a $q$-geometric
random variable with parameter $\al a_j$
(in the case of $\Qspec{\al}$),
or a Bernoulli random variable
with parameter $\be a_j$
(in the case of $\Qspec{\hat\be}$).

\begin{remark}\label{rmk:push_block_RSK_different}
	This feature of
	RSK-type dynamics 
	separates them from the push-block dynamics of \S \ref{sub:push_block_dynamics}.
	Indeed, under a push-block dynamics
	movements at level $j-1$ generically
	\emph{do not propagate upwards}
	because
	the 
	quantities
	$\mathscr{U}_j(\la\to\nu\mid \bar\la\to\bar\nu)$
	do not depend on $\bar\la$. More precisely, the 
	only steps at level $j-1$ that can propagate to level $j$
	correspond to the situation $\bar\nu\not\prech\la$.
	Then a part of the movement $\la\to\nu$
	is \emph{mandatory}, as it is dictated by the 
	need to immediately 
	(i.e., during the same time step of the 
	multivariate dynamics)
	restore the interlacing between the levels $j-1$ and~$j$.
\end{remark}

RSK-type dynamics 
on $q$-Whittaker processes that we construct
in 
\S \ref{sec:bernoulli_}
and
\S \ref{sec:geometric_q_rsks}
give rise to discrete time $q$-TASEPs
and $q$-PushTASEPs as their Markovian marginals. 
On the other hand,
discrete time push-block dynamics
do not seem to produce any
TASEP-like processes.\footnote{The continuous 
time push-block dynamics
on $q$-Whittaker processes
has lead to the discovery of the 
continuous time 
$q$-TASEP in
\cite{BorodinCorwin2011Macdonald}.
A continuous time RSK-type dynamics on $q$-Whittaker processes
was later employed in \cite{BorodinPetrov2013NN}
to discover the continuous time $q$-PushTASEP,
a close relative of the $q$-TASEP 
(see also \S \ref{sub:small_be_continuous_time_limit} below). In
fact, $q$-PushTASEP and $q$-TASEP can be unified
to produce another nice particle system 
on $\Z$, namely, the \emph{$q$-PushASEP} (see \S \ref{sub:fredholm_determinants} below),
which also extends to a certain dynamics
on interlacing arrays \cite{CorwinPetrov2013}.}
Note also that in general the denominator in \eqref{alpha_push_block_solution}
or \eqref{beta_push_block_solution}
does not seem to be given by an explicit formula,
so the discrete
time push-block dynamics are not easy to work with
(cf.~Remark~\ref{rmk:univariate_hard}).
This provides an additional motivation for
constructing and studying RSK-type dynamics.



\section{Schur case} 
\label{sec:schur_degeneration}

In this section we discuss
the Schur ($q=0$) case,
and explain how in this case the RSK-type
multivariate dynamics are related to
the classical Robinson--Schensted--Knuth correspondences.

\subsection{Univariate dynamics in the Schur case} 
\label{sub:univariate_dynamics}

When $q=0$, 
the $q$-Whittaker polynomials $P_\la$ and $Q_\la$
reduce to the simpler Schur polynomials $s_\la$.
In particular,
we have $\psi_{\la/\mu}=\varphi_{\la/\mu}= \mathbf{1}_{\mu \prech \la}$ 
and $\psi'_{\la/\mu}=\mathbf{1}_{\mu \precv \la}$.
Univariate discrete time 
dynamics on the first level $\GT_{1}=\Z_{\ge0}$
look as in Remark \ref{rmk:one_level_dynamics}
with the understanding that 
the $q$-geometric
distribution in the case of $\Pspec{\al}$
has to be replaced 
by the usual geometric distribution
$\p_{\al a_1}(n)\vert_{q=0}=(1-\al a_1)(\al a_1)^{n}$,
$n=0,1,2,\ldots$. 
\begin{remark}
	The continuous time dynamics on $\GT_1$ 
	increasing the parameter $\ga$
	of the specialization is the usual Poisson process
	which can be obtained from either of
	the discrete time dynamics
	$\Pspec{\al}$
	or 
	$\Pspec{\hat\be}$
	in a small $\al$ or small $\be$ limit,
	respectively. In fact, this observation is also
	true in the general $q>0$ case.
\end{remark}

The univariate dynamics
$\Pspec{\al}$
and
$\Pspec{\hat\be}$
at any higher level $\GT_N$, $N=2,3,\ldots$
(described in a $q=0$ version of Proposition \ref{prop:univariate}),
can be obtained from the $N=1$ dynamics
via the \emph{Doob's $h$-transform} procedure. 
Informally, to get the dynamics
of $N$ distinct particles $(x_1>\ldots>x_N)$ on $\Z_{\ge0}$
(this state space is the same as $\GT_N$
up to a shift $x_i=\la_i+N-i$),
one should consider the 
dynamics of $N$ \emph{independent} particles $x_j$
each of which evolves according to the corresponding
$N=1$ dynamics, and then impose the \emph{condition} that 
the particles \emph{never collide}
and have relative asymptotic speeds $a_1,\ldots,a_N$,
respectively.
This conditioning gives rise to the 
presence of the factors 
$s_\nu(\vec a)/s_\la(\vec a)$
in transition probabilities (cf. Proposition \ref{prop:univariate}).
We refer to, e.g., 
\cite{konig2002non},
\cite{Konig2005}, \cite{OConnell2003}, \cite{OConnellYor2002}
for details on noncolliding dynamics.

It is worth noting that the Dyson's Brownian motion
coming from $N\times N$ GUE random matrices
\cite{dyson1962brownian}
arises via a similar procedure
by considering noncolliding Brownian particles.
One may thus think that the
univariate dynamics $\Pspec{\al}$
and
$\Pspec{\hat\be}$
on $\GT_N$
are certain discrete analogues of the Dyson's Brownian
motion.


\subsection{Push-block dynamics in the Schur case} 
\label{sub:push_block_dynamics_schur}

Setting $q=0$ greatly simplifies 
formulas \eqref{alpha_push_block_solution} and
\eqref{beta_push_block_solution} thus
leading to nice push-block multivariate dynamics on 
interlacing arrays. They were introduced and 
studied in \cite{BorFerr2008DF}. 

Due to the sequential nature of 
multivariate dynamics \eqref{Q_sequential_update},
we will consider evolution at 
consecutive levels $j-1$ and $j$.
Assuming that the movement $\bar\la\to\bar\nu$
at level $j-1$ and the old configuration $\la$
at level $j$
are given, we will describe the probability
distribution of $\nu\in\GT_j$ corresponding to 
$\mathscr{U}_j(\la\to\nu\mid \bar\la\to\bar\nu)$.

\begin{figure}[htbp]
	\begin{adjustbox}{max height=.17\textwidth}
	\begin{tikzpicture}
		[scale=.7, thick]
		\def\cir{.2}
		\def\ysh{6}
		\def\x{1.8}
		\def\y{2.2}
		\draw[fill] (0,0) circle(\cir) node [below,yshift=-\ysh] {$1$};
		\draw[fill] (2*\x,0) circle(\cir) node [below,yshift=-\ysh] {$1$};
		\draw[fill] (4*\x,0) circle(\cir) node [below,yshift=-\ysh] {$2+{\color{blue}1}$};
		\draw[fill] (6*\x,0) circle(\cir) node [below,yshift=-\ysh] {$4+{\color{blue}1}$};
		\draw[fill] (-1*\x,1*\y) circle(\cir) node [above,yshift=\ysh] {$0+{\color{blue}Y_3}$};
		\draw[fill] (1*\x,1*\y) circle(\cir) node [above,yshift=\ysh] {$1+{\color{blue}0}$};
		\draw[fill] (3*\x,1*\y) circle(\cir) node [above,yshift=\ysh] {$2+{\color{blue}Y_2}$};
		\draw[fill] (5*\x,1*\y) circle(\cir) node [above,yshift=\ysh] {$2+{\color{blue}1}$};
		\draw[fill] (7*\x,1*\y) circle(\cir) node [above,yshift=\ysh] {$7+{\color{blue}Y_1}$};
		\node at (9.5*\x,0) {$\bar\la+({\color{blue}\bar\nu-\bar\la})$};
		\node at (9.5*\x,\y) {$\la+({\color{blue}\nu-\la})$};
		\node (jb) at (4*\x,0) {};
		\node (pa) at (5*\x,\y) {};
		\node (sb) at (2*\x,0) {};
		\node (ba) at (1*\x,\y) {};
		\draw[->,very thick] (jb) -- (pa)
		node[rectangle,draw=black,fill=gray!20!white] [midway] {$+1$};
		\draw[very thick, dotted] (sb) -- (ba)
		node[rectangle,draw=black,fill=gray!20!white] [midway] {block};
	\end{tikzpicture}
	\end{adjustbox}
	\caption{An example of a step of 
	$\Qbpb$
	at levels $4$ and $5$. Here
	$\la=(7,2,2,1,0)$,
	$\bar\la=(4,2,1,1)$, and
	$\bar\nu=(5,3,1,1)$.
	The move $\la_2=2\to\nu_2=2+1$ on the 
	upper level is dictated by the corresponding
	move $\bar\la_2=2\to\bar\nu_2=2+1$
	on the lower level (due to the short-range pushing mechanism),
	so no further move of $\nu_2$ is possible.
	The particle $\la_4=1$ cannot move 
	because it is blocked by $\bar\nu_3=\la_4$.
	All other particles are free to move (including $\la_3$ which was
	blocked before the movement at the lower level), and their
	jumps 
	$Y_1,Y_2,Y_3$ are independent identically distributed 
	Bernoulli random variables 
	with $P(Y_1=0)=1/(1+\be a_j)$.}
	\label{fig:beta_schur}
\end{figure}
Let us first focus on the case of
$\Qbpb$ 
(see Fig.~\ref{fig:beta_schur}).\footnote{To simplify
pictures, here and below we will display 
interlacing arrays of integers (cf. Fig.~\ref{fig:array}), 
but will still 
speak about particles jumping to the right.} In this case \eqref{beta_push_block_solution} simplifies to 
\begin{align*}
\mathscr{U}_j
	(\la\to\nu\mid \bar\nu)=
	\frac{(\be a_j)^{|\nu|}
	\mathbf{1}_{\bar\nu \prech \nu}\mathbf{1}_{\la \precv \nu}}
	{\sum_{\varkappa\in\GT_j}
	(\be a_j)^{|\varkappa|}
	\mathbf{1}_{\bar\nu \prech \varkappa}\mathbf{1}_{\la \precv \varkappa}}, 
\end{align*}
i.e. for any $\nu', \nu'' \in \GT_j$, such that $\bar\nu \prech \nu', \la \precv \nu'$ and $\bar\nu \prech \nu'', \la \precv \nu''$, 
\begin{align*}
\frac{\mathscr{U}_j
	(\la\to\nu'\mid \bar\nu)}{\mathscr{U}_j
	(\la\to\nu''\mid \bar\nu)}= (\be a_j)^{|\nu'|-|\nu''|}.
\end{align*}
It is clear that the only dynamics with such property fits the following description.
During one step of the dynamics, each particle
$\la_i$, $1\le i\le j$,
can either stay, or jump to the right by one, according to the rules:
\begin{enumerate}
	\item (\emph{short-range pushing}) If $\bar\nu_i=\la_i+1$, then the move $\la_i\to\nu_i=\la_i+1$
	is mandatory to restore the interlacing 
	(which was broken by the move $\bar\la\to\bar\nu$)
	during the 
	same step of the discrete time.
	\item (\emph{blocking}) If $\la_i=\bar\nu_{i-1}$, then the 
	particle $\la_i$ is blocked and must stay, i.e., 
	$\nu_i$ is forced to be equal to $\la_i$.
	\item (\emph{independent jumps})
	All other particles
	$\la_i$ which are neither pushed nor blocked,
	jump to the right by $0$ or $1$ according
	to an independent Bernoulli random variable
	with probability of staying $1/(1+\be a_j)$.
\end{enumerate}

\begin{figure}[htbp]
	\begin{adjustbox}{max height=.17\textwidth}
	\begin{tikzpicture}
		[scale=.7, thick]
		\def\cir{.2}
		\def\ysh{6}
		\def\x{1.8}
		\def\y{2.2}
		\draw[fill] (0,0) circle(\cir) node [below,yshift=-\ysh] {$1$};
		\draw[fill] (2*\x,0) circle(\cir) node [below,yshift=-\ysh] {$1$};
		\draw[fill] (4*\x,0) circle(\cir) node [below,yshift=-\ysh] {$2+{\color{blue}2}$};
		\draw[fill] (6*\x,0) circle(\cir) node [below,yshift=-\ysh] {$4+{\color{blue}4}$};
		\draw[fill] (-1*\x,1*\y) circle(\cir) node [above,yshift=\ysh] {$0+{\color{blue}Y_4}$};
		\draw[fill] (1*\x,1*\y) circle(\cir) node [above,yshift=\ysh] {$1+{\color{blue}0}$};
		\draw[fill] (3*\x,1*\y) circle(\cir) node [above,yshift=\ysh] {$2+{\color{blue}Y_3}$};
		\draw[fill] (5*\x,1*\y) circle(\cir) node [above,yshift=\ysh-1.9] {$2+{\color{blue}2+Y_2}$};
		\draw[fill] (7*\x,1*\y) circle(\cir) node [above,yshift=\ysh-1.9] {$7+{\color{blue}1+Y_1}$};
		\node at (9.5*\x,0) {$\bar\la+({\color{blue}\bar\nu-\bar\la})$};
		\node at (9.5*\x,\y) {$\la+({\color{blue}\nu-\la})$};
		\node (jb) at (4*\x,0) {};
		\node (pa) at (5*\x,\y) {};
		\node (jb1) at (6*\x,0) {};
		\node (pa1) at (7*\x,\y) {};
		\node (sb) at (2*\x,0) {};
		\node (ba) at (1*\x,\y) {};
		\draw[->,very thick] (jb) -- (pa)
		node[rectangle,draw=black,fill=gray!20!white] [midway] {$+2$};
		\draw[->,very thick] (jb1) -- (pa1)
		node[rectangle,draw=black,fill=gray!20!white] [midway] {$+1$};
		\draw[very thick, dotted] (sb) -- (ba)
		node[rectangle,draw=black,fill=gray!20!white] [midway] {block};
	\end{tikzpicture}
	\end{adjustbox}
	\caption{An example of a step of 
	$\Qapb$ at levels 4 and 5.
	The move $\bar\la_1=4\to\bar\nu_1=4+4$ forces $\la_1$ to move 
	to the right by 1, and similarly the move
	$\bar\la_2=2\to\bar\nu_2=2+2$ forces $\la_2$
	to move to the right by 2 (short-range pushing); note that these
	forced moves do not exhaust all possible
	distance traveled
	by $\la_1$ or $\la_2$.
	The particle $\la_4=1$ is blocked by $\bar\nu_3=\la_4$
	and thus cannot move. All other parts of the movement $\la\to\nu$
	are determined by independent identically distributed
	geometric random variables $Y_i$, $1\le i\le 4$
	with parameter $\al a_j$, where each variable 
	is conditioned
	to stay in the maximal interval  
	not breaking the interlacing: $Y_1\ge0$ (i.e., no conditioning), 
	$0\le Y_2\le 4$, $0\le Y_3\le 2$, $0\le Y_4\le 1$.}
	\label{fig:alpha_schur}
\end{figure}

By the same explanation the dynamics
$\Qapb$ 
at two consecutive levels
looks as follows (see Fig.~\ref{fig:alpha_schur}). Each 
particle $\la_i$, $1\le i\le j$, independently
jumps to the 
right by a random distance which 
has the geometric distribution with parameter $\al a_j$
conditioned to stay in the interval from $(\bar\nu_i-\la_i)_{+}:=
\max\{0,\bar\nu_i-\la_i\}$ to $\bar\nu_{i-1}-\la_i$
(with the agreement that $\la_0=+\infty$).\footnote{Due to the 
memorylessness of the geometric distribution, 
this description is equivalent
to what is illustrated on 
Fig.~\ref{fig:alpha_schur}.}
This conditioning corresponds to
the denominator in \eqref{alpha_push_block_solution}.



\subsection{RSK-type dynamics in the Schur case} 
\label{sub:rsk_type_dynamics_schur}

Let us now discuss four discrete time multivariate
RSK-type dynamics
$\Qarow$, $\Qbrow$,
$\Qacol$, $\Qbcol$
on Schur processes. 
The former two dynamics arise from the 
\emph{row RSK algorithm}\footnote{The row RSK 
is the most classical version of the Robinson--Schensted--Knuth algorithm.}
applied to geometric or Bernoulli random input,
respectively (cf. Remark \ref{rmk:random_input} below).
Similarly, the latter two dynamics
correspond to the \emph{column RSK algorithm}
applied to the same random inputs.
We refer to 
\cite{Knuth1970},
\cite{fulton1997young}, 
\cite{Stanley1999}
for relevant background and details on RSK correspondences
(though
descriptions of dynamics below in this subsection
may serve as equivalent
definitions of RSK algorithms).
See also \cite[\S7]{BorodinPetrov2013NN}
for a ``dictionary'' between interlacing arrays
and semistandard Young tableaux viewpoints.

Let us first recall two elementary operations
of deterministic long-range pulling and pushing
from \cite{BorodinPetrov2013NN}
(in the language 
of semistandard Young tableaux
they
correspond to row and column bumping, respectively).

\begin{definition}(Deterministic long-range pulling, Fig.~\ref{fig:pulling})
	\label{def:pull}
	Let $j=2,\ldots,N$, 
	and signatures 
	$\bar\la,\bar\nu\in\GT_{j-1}$, 
	$\la\in\GT_{j}$ satisfy $\bar\la\prech \la$, 
	$\bar\nu=\bar\la+\bar{\mathrm{e}}_{i}$,
	where $\bar{\mathrm{e}}_{i}=(0,0,\ldots,0,1,0,\ldots,0)$ (for some $i=1,\ldots,j-1$) 
	is the $i$th basis vector of length $j-1$.
	Define $\nu\in\GT_j$ to be
	\begin{align*}
		\nu=\mathsf{pull}(\la\mid \bar\la\to\bar\nu):=
		\begin{cases}
			\la+\mathrm{e}_{i},&\text{if $\bar\la_{i}=\la_i$};\\
			\la+\mathrm{e}_{i+1},&\text{otherwise}.
		\end{cases}
	\end{align*}
	Here $\mathrm{e}_{i}$ and $\mathrm{e}_{i+1}$
	are basis vectors of length $j$.

	In words, the particle $\bar\la_i$
	at level $j-1$  which moved to the right by one
	generically pulls its upper left neighbor $\la_{i+1}$, or 
	pushes it upper right neighbor $\la_i$
	if the latter operation is needed to preserve the interlacing.
	Note that the long-range pulling mechanism
	does not encounter any blocking issues.
\end{definition}
\begin{figure}[htbp]
	\begin{tabular}{cc}
	\framebox{\begin{adjustbox}{max height=.16\textwidth}\begin{tikzpicture}
		[scale=.7, thick]
		\def\cir{.2}
		\def\ysh{6}
		\def\x{1.8}
		\def\y{2.2}
		\draw[fill] (4*\x,0) circle(\cir) node [below,yshift=-\ysh] {$2+{\color{blue}1}$};
		\draw[fill] (6*\x,0) circle(\cir) node [below,yshift=-\ysh] {$4$};
		\draw[fill] (3*\x,1*\y) circle(\cir) node [above,yshift=\ysh] {$1$};
		\draw[fill] (5*\x,1*\y) circle(\cir) node [above,yshift=\ysh] {$2+{\color{blue}1}$};
		\draw[fill] (7*\x,1*\y) circle(\cir) node [above,yshift=\ysh] {$7$};
		\node (jb) at (4*\x,0) {};
		\node (pa) at (5*\x,\y) {};
		\draw[->,very thick] (jb) -- (pa)
		node[rectangle,draw=black,fill=gray!20!white] [midway] {$+1$};
	\end{tikzpicture}
	\end{adjustbox}}
	\hspace{70pt}
	&
	\framebox{\begin{adjustbox}{max height=.16\textwidth}\begin{tikzpicture}
		[scale=.7, thick]
		\def\cir{.2}
		\def\ysh{6}
		\def\x{1.8}
		\def\y{2.2}
		\draw[fill] (4*\x,0) circle(\cir) node [below,yshift=-\ysh] {$2+{\color{blue}1}$};
		\draw[fill] (6*\x,0) circle(\cir) node [below,yshift=-\ysh] {$4$};
		\draw[fill] (3*\x,1*\y) circle(\cir) node [above,yshift=\ysh] {$1+{\color{blue}1}$};
		\draw[fill] (5*\x,1*\y) circle(\cir) node [above,yshift=\ysh] 
		{$3$};
		\draw[fill] (7*\x,1*\y) circle(\cir) node [above,yshift=\ysh] {$7$};
		\node (jb) at (4*\x,0) {};
		\node (pa) at (3*\x,\y) {};
		\draw[->,very thick] (jb) -- (pa)
		node[rectangle,draw=black,fill=gray!20!white] [midway] {$+1$};
	\end{tikzpicture}\end{adjustbox}}
	\end{tabular}
	\caption{An example of pulling mechanism
	for $i=2$
	at levels 2 and 3 (i.e., $j=3$).
	Left: $\bar\la_2=\la_2$,
	which forces the pushing of the upper right neighbor.
	Right: in the generic
	situation $\bar\la_2<\la_2$ the upper left neighbor is pulled.}
	\label{fig:pulling}
\end{figure}
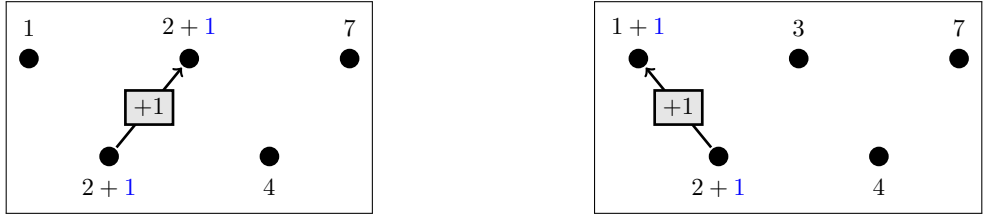

\begin{definition}(Deterministic long-range pushing, Fig.~\ref{fig:pushing})
	\label{def:push}
	As in the previous definition, 
	let
	$j=2,\ldots,N$, 
	$\bar\la,\bar\nu\in\GT_{j-1}$, 
	$\la\in\GT_{j}$ be such that $\bar\la\prech \la$
	and
	$\bar\nu=\bar\la+\bar{\mathrm{e}}_{i}$.
	Define $\nu\in\GT_j$ to be
	\begin{align*}
		\nu=\mathsf{push}(\la\mid \bar\la\to\bar\nu):=
		\la+\mathrm{e}_{m},
		\qquad
		\text{where
		$m=\max\{p\colon 
		1\le p\le i\text{ and }\la_{p}<\bar\la_{p-1}\}$}.
	\end{align*}

	In words, the particle $\bar\la_i$ 
	at level $j-1$
	which moved to the right by one,
	pushes its first upper right neighbor $\la_m$
	which is not blocked (and therefore is free to move
	without violating the interlacing).
	Generically (when all particles are sufficiently
	far apart) $\la_m=\la_i$, so the immediate upper right
	neighbor is pushed.
\end{definition}
\begin{remark}[Move donation]\label{rmk:move_donation}
	It is useful to equivalently 
	interpret the mechanism of Definition
	\ref{def:push} in a slightly different way.
	Namely, let us say that 
	when the particle $\bar\la_i$
	at level $j-1$ moves, it gives the 
	particle $\la_i$ at level $j$ a \emph{moving impulse}.
	If $\la_i$ is blocked (i.e., if $\la_i=\bar\la_{i-1}$),
	this moving impulse is \emph{donated} to the next
	particle $\la_{i-1}$ to the right of $\la_i$. If $\la_{i-1}$
	is blocked, too, then the impulse is donated further, and so on.
	Note that the particle $\la_{1}$ cannot be blocked, so
	this moving impulse will always result in an actual move.
\end{remark}

\begin{figure}[htbp]
	\begin{adjustbox}{max height=.17\textwidth}
	\begin{tikzpicture}
		[scale=.7, thick]
		\def\cir{.2}
		\def\ysh{6}
		\def\x{1.8}
		\def\y{2.2}
		\draw[fill] (0,0) circle(\cir) node [below,yshift=-\ysh] {$1$};
		\draw[fill] (2*\x,0) circle(\cir) node [below,yshift=-\ysh] {$2+
		{\color{blue}1}$};
		\draw[fill] (4*\x,0) circle(\cir) node [below,yshift=-\ysh] {$4$};
		\draw[fill] (6*\x,0) circle(\cir) node [below,yshift=-\ysh] {$6$};
		\draw[fill] (-1*\x,1*\y) circle(\cir) node [above,yshift=\ysh] 
		{$0$};
		\draw[fill] (1*\x,1*\y) circle(\cir) node [above,yshift=\ysh] 
		{$1$};
		\draw[fill] (3*\x,1*\y) circle(\cir) node [above,yshift=\ysh] 
		{$4$};
		\draw[fill] (5*\x,1*\y) circle(\cir) node [above,yshift=\ysh-1.9] {$6$};
		\draw[fill] (7*\x,1*\y) circle(\cir) node [above,yshift=\ysh-1.9] {$7+{\color{blue}1}$};
		\node at (10*\x,0) {$\bar\la+({\color{blue}\bar\nu-\bar\la})$};
		\node at (10*\x,\y) {$\la+({\color{blue}\nu-\la})$};
		\node (sb) at (4*\x,0) {};
		\node (ba) at (3*\x,\y) {};
		\node (sb1) at (6*\x,0) {};
		\node (ba1) at (5*\x,\y) {};
		\node (la1) at (7*\x,\y) {};
		\node (lab3) at (2*\x,0) {};
		\draw[->,very thick] (lab3) .. controls 
		(4*\x,-\y/3)
		and
		(5*\x,4/3*\y)
		.. (la1) 
		node[rectangle,draw=black,fill=gray!20!white] [midway] {$+1$};
		\draw[very thick,dotted] (sb) -- (ba)
		node[rectangle,draw=black,fill=gray!20!white] [midway] {block};
		\draw[very thick,dotted] (sb1) -- (ba1)
		node[rectangle,draw=black,fill=gray!20!white] [midway] {block};
	\end{tikzpicture}
	\end{adjustbox}
	\caption{An example of pushing mechanism
	for $i=3$ at levels 4 and 5 
	(i.e., $j=5$). Since the particles
	$\la_3=\bar\la_2$ and $\la_2=\bar\la_1$
	are blocked, the first particle 
	which can be pushed is $\la_1$.}
	\label{fig:pushing}
\end{figure}
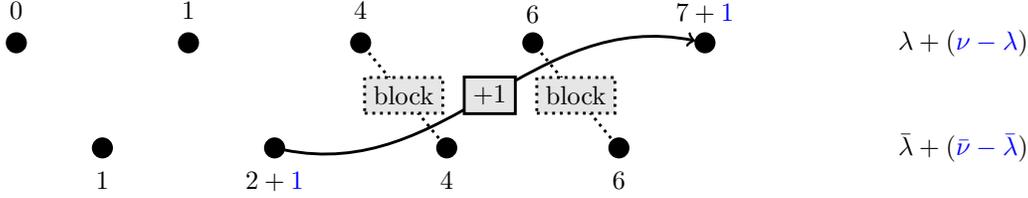

Let us now describe four RSK-type dynamics
on Schur processes. Under each of the dynamics
the interlacing array $\lab$
is updated sequentially (cf. \eqref{Q_sequential_update})
at each level $j=1,\ldots,N$. 
At each step of the discrete time corresponding to 
an update $\lab\to\nub$,
new randomness is introduced
via $N$ independent random variables $V_1,\ldots,V_N$,
which are either
geometric random variables (belonging to $\Z_{\ge0}$)
with parameters $\al a_1,\ldots, \al a_N$
in the case of $\Qarow$ and
$\Qacol$,
or
Bernoulli random variables~$\in\{0,1\}$
with parameters $\be a_1,\ldots,\be a_N$
in the case of 
$\Qbrow$ and
$\Qbcol$.
These random variables are resampled during each time step. 

\begin{remark}\label{rmk:random_input}
	We see that 
	all randomness in each of the four RSK-type dynamics
	can be organized into a matrix $(V^{(t)}_{j})_{1\le j\le N,\; t=1,2,\ldots}$
	(with appropriate distribution of the $V^{(t)}_{j}$'s).
	Such matrices containing nonnegative integers 
	are usually thought of as inputs 
	for classical Robinson--Schensted--Knuth correspondences.
\end{remark}

Under each of the four dynamics, the particle
at the first level of the array
is updated as $\nu^{(1)}_{1}=\la^{(1)}_{1}+V_1$.
Then, for each 
$j=2,\ldots,N$,
assume that we are given signatures
$\bar\la,\bar\nu\in\GT_{j-1}$, 
$\la\in\GT_{j}$
satisfying relations as on Fig.~\ref{fig:square}
(note that these relations depend on the type ($\al$) or $(\hat\be)$ of the dynamics). 
Let us represent the movement $\bar\la\to\bar\nu$
at level $j-1$ as
\begin{align*}
	\bar\nu-\bar\la=\sum_{i=1}^{j-1}c_i \bar{\mathrm{e}}_{i},
	\qquad
	\begin{cases}
		c_i\in\Z_{\ge0}&\text{in the case of $\Qarow$ and
		$\Qacol$};\\
		c_i\in\{0,1\}
		&
		\text{in the case of 
		$\Qbrow$ and
		$\Qbcol$}
	\end{cases}
\end{align*}
(recall that $\bar{\mathrm{e}}_{i}$ is the $i$th basis vector of length $j-1$).
Also denote $|c|:=\sum_{i=1}^{j-1}c_i$.

Depending on the dynamics, we will construct the signature $\nu\in\GT_j$  
(which also fits into relations on Fig.~\ref{fig:square})
as follows:
\medskip

$\bullet$ ($\Qarow$, Fig.~\ref{fig:QrRSK_alpha})
	First, do $|c|$ operations $\mathsf{pull}$ (Definition \ref{def:pull})
	in order \emph{from left to right},
	starting from position $j-1$ all the way up to position $1$.
	In more detail, let $\mu(j-1,0):=\la$ and
	for $p=1,\ldots,c_{j-1}$ let
	\begin{align*}
		\mu(j-1,p):=\mathsf{pull}(\mu(j-1,p-1)\mid
		\bar\la+(p-1)\,\bar{\mathrm{e}}_{j-1}
		\to\bar\la+p\,\bar{\mathrm{e}}_{j-1}),
	\end{align*}
	then let $\mu(j-2,0):=\mu(j-1,c_{j-1})$ and 
	for $p=1,\ldots,c_{j-2}$ let
	\begin{align*}
		\mu(j-2,p):=\mathsf{pull}(\mu(j-2,p-1)\mid
		\bar\la+c_{j-1}\bar{\mathrm{e}}_{j-1}+
		(p-1)\,\bar{\mathrm{e}}_{j-2}
		\to\bar\la+c_{j-1}\bar{\mathrm{e}}_{j-1}
		+p\,\bar{\mathrm{e}}_{j-2}),
	\end{align*}
	etc., 
	all the way up to 
	$\mu(1,c_1):=
	\mathsf{pull}(\mu(1,c_1-1)\mid \bar\nu-\bar{\mathrm{e}}_{1}\to\bar\nu)$.
	(Clearly, if some $c_i=0$, then 
	the steps corresponding to $\mu(i,\cdot)$ should be omitted.)
	
	After these $|c|$ operations, 
	define $\nu:=\mu(1,c_1)+V_j\mathrm{e}_{1}$.
	That is, let the rightmost particle at level $j$
	jump to the right by $V_j$ (which is a geometric random
	variable with parameter~$\al a_j$).

\begin{figure}[htbp]
	\begin{adjustbox}{max height=.16\textwidth}
	\begin{tikzpicture}
		[scale=.7, thick]
		\def\cir{.2}
		\def\ysh{6}
		\def\x{1.8}
		\def\y{2.2}
		\draw[fill] (0,0) circle(\cir) node [below,yshift=-\ysh] {$3+{\color{blue}2}$};
		\draw[fill] (2*\x,0) circle(\cir) node [below,yshift=-\ysh] {$6$};
		\draw[fill] (4*\x,0) circle(\cir) node [below,yshift=-\ysh] {$6+{\color{blue}1}$};
		\draw[fill] (6*\x,0) circle(\cir) node [below,yshift=-\ysh] {$7+{\color{blue}3}$};
		\draw[fill] (-1*\x,1*\y) circle(\cir) node [above,yshift=\ysh] {$1+{\color{blue}1}$};
		\draw[fill] (1*\x,1*\y) circle(\cir) node [above,yshift=\ysh] {$4+{\color{blue}1}$};
		\draw[fill] (3*\x,1*\y) circle(\cir) node [above,yshift=\ysh] {$6$};
		\draw[fill] (5*\x,1*\y) circle(\cir) node [above,yshift=\ysh] {$6+{\color{blue}3}$};
		\draw[fill] (7*\x,1*\y) circle(\cir) node [above,yshift=-2.2+\ysh] {$9+{\color{blue}1+V_j}$};
		\node at (9.5*\x,0) {$\bar\la+({\color{blue}\bar\nu-\bar\la})$};
		\node at (9.5*\x,\y) {$\la+({\color{blue}\nu-\la})$};
		\node (lab4) at (0*\x,0) {};
		\node (lab3) at (2*\x,0) {};
		\node (lab2) at (4*\x,0) {};
		\node (lab1) at (6*\x,0) {};
		\node (la5) at (-1*\x,\y) {};
		\node (la4) at (1*\x,\y) {};
		\node (la3) at (3*\x,\y) {};
		\node (la2) at (5*\x,\y) {};
		\node (la1) at (7*\x,\y) {};
		\draw[->,very thick] (lab4) -- (la5) node[rectangle,draw=black,fill=gray!20!white] [midway] {$+1$};
		\draw[->,very thick] (lab4) -- (la4) node[rectangle,draw=black,fill=gray!20!white] [midway] {$+1$};
		\draw[->,very thick] (lab2) -- (la2) node[rectangle,draw=black,fill=gray!20!white] [midway] {$+1$};
		\draw[->,very thick] (lab1) -- (la2) node[rectangle,draw=black,fill=gray!20!white] [midway] {$+2$};
		\draw[->,very thick] (lab1) -- (la1) node[rectangle,draw=black,fill=gray!20!white] [midway] {$+1$};
	\end{tikzpicture}
	\end{adjustbox}
	\caption{An example of a step of $\Qarow$
	at levels 4 and 5.
	Propagation steps (represented by numbers on arrows)
	are performed from left to right, according
	to $\mathsf{pull}$ operation. After that,
	the rightmost particle at level $j$
	jumps to the right by $V_j$.}
	\label{fig:QrRSK_alpha}
\end{figure}
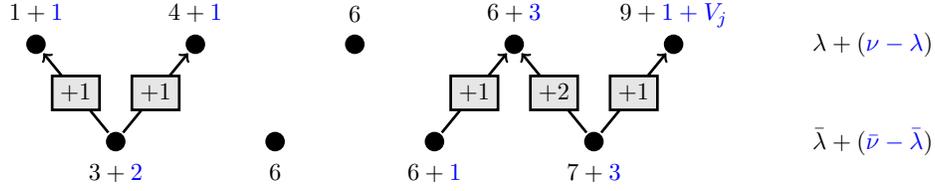

$\bullet$ ($\Qbrow$, Fig.~\ref{fig:QrRSK_beta})
	First, define $\mu(1,0):=\la+ V_j \mathrm{e}_1$.
	That is, let the rightmost particle 
	at level $j$ jump to the right by $V_j$
	(which is a Bernoulli random
	variable with parameter~$\be a_j$).

	After that,
	perform $|c|$ operations $\mathsf{pull}$
	(Definition \ref{def:pull})
	in order \emph{from right to left},
	starting from position $1$ all the way up to position $j-1$
	(details are analogous to the above 
	dynamics $\Qarow$).
	Then set $\nu:=\mu(j-1,c_{j-1})$.

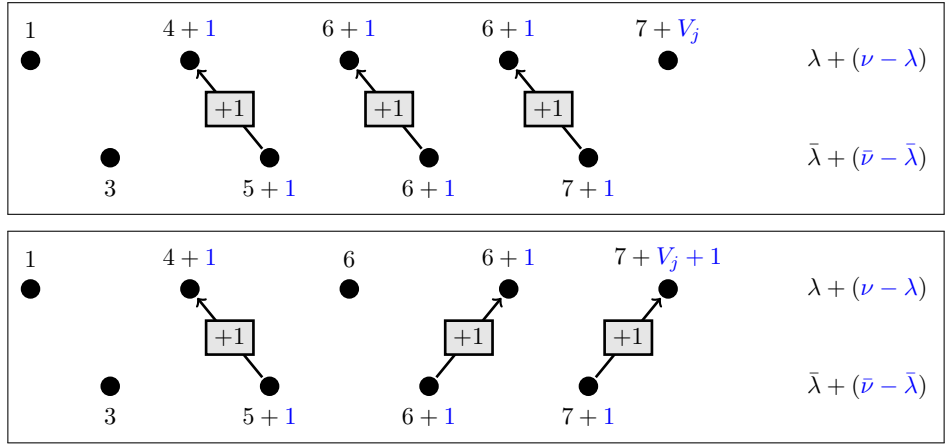
\begin{figure}[htbp]
	\framebox{\begin{adjustbox}{max height=.16\textwidth}
	\begin{tikzpicture}
		[scale=.7, thick]
		\def\cir{.2}
		\def\ysh{6}
		\def\x{1.8}
		\def\y{2.2}
		\draw[fill] (0,0) circle(\cir) node [below,yshift=-\ysh] {$3$};
		\draw[fill] (2*\x,0) circle(\cir) node [below,yshift=-\ysh] {$5+{\color{blue}1}$};
		\draw[fill] (4*\x,0) circle(\cir) node [below,yshift=-\ysh] {$6+{\color{blue}1}$};
		\draw[fill] (6*\x,0) circle(\cir) node [below,yshift=-\ysh] {$7+{\color{blue}1}$};
		\draw[fill] (-1*\x,1*\y) circle(\cir) node [above,yshift=\ysh] {$1$};
		\draw[fill] (1*\x,1*\y) circle(\cir) node [above,yshift=\ysh] {$4+{\color{blue}1}$};
		\draw[fill] (3*\x,1*\y) circle(\cir) node [above,yshift=\ysh] {$6+{\color{blue}1}$};
		\draw[fill] (5*\x,1*\y) circle(\cir) node [above,yshift=\ysh] {$6+{\color{blue}1}$};
		\draw[fill] (7*\x,1*\y) circle(\cir) node [above,yshift=-2.2+\ysh] {$7+{\color{blue}V_j}$};
		\node at (9.5*\x,0) {$\bar\la+({\color{blue}\bar\nu-\bar\la})$};
		\node at (9.5*\x,\y) {$\la+({\color{blue}\nu-\la})$};
		\node (lab4) at (0*\x,0) {};
		\node (lab3) at (2*\x,0) {};
		\node (lab2) at (4*\x,0) {};
		\node (lab1) at (6*\x,0) {};
		\node (la5) at (-1*\x,\y) {};
		\node (la4) at (1*\x,\y) {};
		\node (la3) at (3*\x,\y) {};
		\node (la2) at (5*\x,\y) {};
		\node (la1) at (7*\x,\y) {};
		\draw[->,very thick] (lab1) -- (la2) node[rectangle,draw=black,fill=gray!20!white] [midway] {$+1$};
		\draw[->,very thick] (lab2) -- (la3) node[rectangle,draw=black,fill=gray!20!white] [midway] {$+1$};
		\draw[->,very thick] (lab3) -- (la4) node[rectangle,draw=black,fill=gray!20!white] [midway] {$+1$};
	\end{tikzpicture}
	\end{adjustbox}}
	\\
	\vspace*{5pt}
	\framebox{\begin{adjustbox}{max height=.16\textwidth}
	\begin{tikzpicture}
		[scale=.7, thick]
		\def\cir{.2}
		\def\ysh{6}
		\def\x{1.8}
		\def\y{2.2}
		\draw[fill] (0,0) circle(\cir) node [below,yshift=-\ysh] {$3$};
		\draw[fill] (2*\x,0) circle(\cir) node [below,yshift=-\ysh] {$5+{\color{blue}1}$};
		\draw[fill] (4*\x,0) circle(\cir) node [below,yshift=-\ysh] {$6+{\color{blue}1}$};
		\draw[fill] (6*\x,0) circle(\cir) node [below,yshift=-\ysh] {$7+{\color{blue}1}$};
		\draw[fill] (-1*\x,1*\y) circle(\cir) node [above,yshift=\ysh] {$1$};
		\draw[fill] (1*\x,1*\y) circle(\cir) node [above,yshift=\ysh] {$4+{\color{blue}1}$};
		\draw[fill] (3*\x,1*\y) circle(\cir) node [above,yshift=\ysh] {$6$};
		\draw[fill] (5*\x,1*\y) circle(\cir) node [above,yshift=\ysh] {$6+{\color{blue}1}$};
		\draw[fill] (7*\x,1*\y) circle(\cir) node [above,yshift=-2.2+\ysh] {$7+{\color{blue}V_j+1}$};
		\node at (9.5*\x,0) {$\bar\la+({\color{blue}\bar\nu-\bar\la})$};
		\node at (9.5*\x,\y) {$\la+({\color{blue}\nu-\la})$};
		\node (lab4) at (0*\x,0) {};
		\node (lab3) at (2*\x,0) {};
		\node (lab2) at (4*\x,0) {};
		\node (lab1) at (6*\x,0) {};
		\node (la5) at (-1*\x,\y) {};
		\node (la4) at (1*\x,\y) {};
		\node (la3) at (3*\x,\y) {};
		\node (la2) at (5*\x,\y) {};
		\node (la1) at (7*\x,\y) {};
		\draw[->,very thick] (lab1) -- (la1) node[rectangle,draw=black,fill=gray!20!white] [midway] {$+1$};
		\draw[->,very thick] (lab2) -- (la2) node[rectangle,draw=black,fill=gray!20!white] [midway] {$+1$};
		\draw[->,very thick] (lab3) -- (la4) node[rectangle,draw=black,fill=gray!20!white] [midway] {$+1$};
	\end{tikzpicture}
	\end{adjustbox}}
	\caption{An example of a step of 
	$\Qbrow$ at levels 4 and 5.
	Propagation steps
	are performed from right to left, according
	to $\mathsf{pull}$ operation. 
	Above: $V_j=1$, below: $V_j=0$.}
	\label{fig:QrRSK_beta}
\end{figure}

$\bullet$ ($\Qacol$, Fig.~\ref{fig:QcRSK_alpha})
	First, the leftmost particle $\la_j$ at level $j$
	receives $V_j$ moving impulses (here $V_j$
	is a geometric random variable with parameter $\al a_j$).
	Each moving impulse means that
	$\la_j$ tries to jump to the right by one, and
	if it is blocked (i.e., if $\la_{j}=\bar\la_{j-1}$),
	then the moving impulse is donated to $\la_{j-1}$, etc.
	(see Remark
	\ref{rmk:move_donation}).
	Denote 
	the signature at level $j$ 
	arising after these $V_j$ moving impulses
	by $\mu(j-1,0)$.

	After that, perform $|c|$ operations $\mathsf{push}$
	(Definition \ref{def:push}),
	in order \emph{from left to right},
	starting from position $j-1$
	all the way up to position $1$
	(details are analogous to the above).
	Then we set $\nu:=\mu(1,c_1)$.

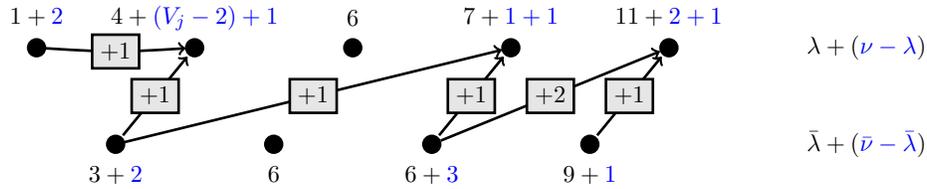
\begin{figure}[htbp]
	\begin{adjustbox}{max height=.16\textwidth}
	\begin{tikzpicture}
		[scale=.7, thick]
		\def\cir{.2}
		\def\ysh{6}
		\def\x{1.8}
		\def\y{2.2}
		\draw[fill] (0,0) circle(\cir) node [below,yshift=-\ysh] {$3+{\color{blue}2}$};
		\draw[fill] (2*\x,0) circle(\cir) node [below,yshift=-\ysh] {$6$};
		\draw[fill] (4*\x,0) circle(\cir) node [below,yshift=-\ysh] {$6+{\color{blue}3}$};
		\draw[fill] (6*\x,0) circle(\cir) node [below,yshift=-\ysh] {$9+{\color{blue}1}$};
		\draw[fill] (-1*\x,1*\y) circle(\cir) node [above,yshift=\ysh] {$1+{\color{blue}2}$};
		\draw[fill] (1*\x,1*\y) circle(\cir) node [above,yshift=-2.2+\ysh] {$4+{\color{blue}(V_j-2)+1}$};
		\draw[fill] (3*\x,1*\y) circle(\cir) node [above,yshift=\ysh] {$6$};
		\draw[fill] (5*\x,1*\y) circle(\cir) node [above,yshift=\ysh] {$7+{\color{blue}1+1}$};
		\draw[fill] (7*\x,1*\y) circle(\cir) node [above,yshift=\ysh] {$11+{\color{blue}2+1}$};
		\node at (9.5*\x,0) {$\bar\la+({\color{blue}\bar\nu-\bar\la})$};
		\node at (9.5*\x,\y) {$\la+({\color{blue}\nu-\la})$};
		\node (lab4) at (0*\x,0) {};
		\node (lab3) at (2*\x,0) {};
		\node (lab2) at (4*\x,0) {};
		\node (lab1) at (6*\x,0) {};
		\node (la5) at (-1*\x,\y) {};
		\node (la4) at (1*\x,\y) {};
		\node (la3) at (3*\x,\y) {};
		\node (la2) at (5*\x,\y) {};
		\node (la1) at (7*\x,\y) {};
		\draw[->,very thick] (la5) .. controls 
		(0*\x,.95*\y)
		.. (la4) node[rectangle,draw=black,fill=gray!20!white] [midway] {$+1$};
		\draw[->,very thick] (lab4) -- (la4) node[rectangle,draw=black,fill=gray!20!white] [midway] {$+1$};
		\draw[->,very thick] (lab4) -- (la2) node[rectangle,draw=black,fill=gray!20!white] [midway] {$+1$};
		\draw[->,very thick] (lab2) -- (la2) node[rectangle,draw=black,fill=gray!20!white] [midway] {$+1$};
		\draw[->,very thick] (lab2) -- (la1) node[rectangle,draw=black,fill=gray!20!white] [midway] {$+2$};
		\draw[->,very thick] (lab1) -- (la1) node[rectangle,draw=black,fill=gray!20!white] [midway] {$+1$};
	\end{tikzpicture}
	\end{adjustbox}
	\caption{An example of a step of $\Qacol$ at levels 4 and 5.
	We have $V_j=3$, which means that initially
	the particle $\la_5$ jumps to the right by 2 and the 
	particle $\la_4$ jumps by 1 (because of
	move donation).
	After that, propagation steps
	are performed from left to right, according
	to $\mathsf{push}$ operation.}
	\label{fig:QcRSK_alpha}
\end{figure}

$\bullet$ ($\Qbcol$, Fig.~\ref{fig:QcRSK_beta})
	First, perform $|c|$ operations
	$\mathsf{push}$
	(Definition \ref{def:push}),
	in order \emph{from right to left},
	starting from position $1$
	all the way up to position $j-1$
	(details are analogous to what is done above).
	Let $\mu(j-1,c_{j-1})$ be the signature at
	level $j$ arising after these $|c|$ operations.

	After that, let the leftmost particle
	at level $j$ receives $V_j$ moving impulses
	(here $V_j$ is a Bernoulli random variable
	with parameter $\be a_j$). That is, if $V_j=0$,
	then set $\nu:=\mu(j-1,c_{j-1})$.
	Otherwise, if $V_j=1$, 
	the $j$th particle at level $j$ tries to 
	jump to the right by one. If it is blocked,
	the impulse is donated to the $(j-1)$th
	particle at level $j$, etc. 
	In this case, denote by $\nu$
	the signature at level $j$
	arising after this moving impulse.

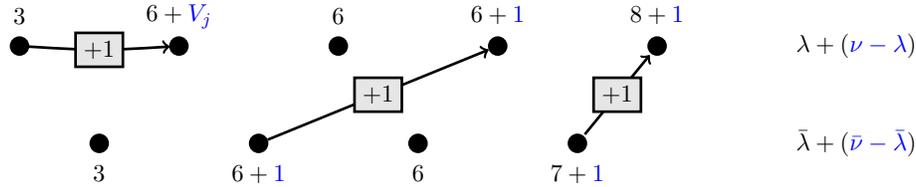
\begin{figure}[htbp]
	\begin{adjustbox}{max height=.16\textwidth}
	\begin{tikzpicture}
		[scale=.7, thick]
		\def\cir{.2}
		\def\ysh{6}
		\def\x{1.8}
		\def\y{2.2}
		\draw[fill] (0,0) circle(\cir) node [below,yshift=-\ysh] {$3$};
		\draw[fill] (2*\x,0) circle(\cir) node [below,yshift=-\ysh] {$6+{\color{blue}1}$};
		\draw[fill] (4*\x,0) circle(\cir) node [below,yshift=-\ysh] {$6$};
		\draw[fill] (6*\x,0) circle(\cir) node [below,yshift=-\ysh] {$7+{\color{blue}1}$};
		\draw[fill] (-1*\x,1*\y) circle(\cir) node [above,yshift=\ysh] {$3$};
		\draw[fill] (1*\x,1*\y) circle(\cir) node [above,yshift=-2.2+\ysh] {$6+{\color{blue}V_j}$};
		\draw[fill] (3*\x,1*\y) circle(\cir) node [above,yshift=\ysh] {$6$};
		\draw[fill] (5*\x,1*\y) circle(\cir) node [above,yshift=\ysh] {$6+{\color{blue}1}$};
		\draw[fill] (7*\x,1*\y) circle(\cir) node [above,yshift=\ysh] {$8+{\color{blue}1}$};
		\node at (9.5*\x,0) {$\bar\la+({\color{blue}\bar\nu-\bar\la})$};
		\node at (9.5*\x,\y) {$\la+({\color{blue}\nu-\la})$};
		\node (lab4) at (0*\x,0) {};
		\node (lab3) at (2*\x,0) {};
		\node (lab2) at (4*\x,0) {};
		\node (lab1) at (6*\x,0) {};
		\node (la5) at (-1*\x,\y) {};
		\node (la4) at (1*\x,\y) {};
		\node (la3) at (3*\x,\y) {};
		\node (la2) at (5*\x,\y) {};
		\node (la1) at (7*\x,\y) {};
		\draw[->,very thick] (lab1) -- (la1) node[rectangle,draw=black,fill=gray!20!white] [midway] {$+1$};
		\draw[->,very thick] (lab3) -- (la2) node[rectangle,draw=black,fill=gray!20!white] [midway] {$+1$};
		\draw[->,very thick] (la5) .. controls 
		(0*\x,.95*\y)
		.. (la4) node[rectangle,draw=black,fill=gray!20!white] [midway] {$+1$};
	\end{tikzpicture}
	\end{adjustbox}
	\caption{An example of a step of 
	$\Qbcol$ at levels 4 and 5.
	Propagation steps
	are performed from right to left, according
	to $\mathsf{push}$ operation. We have $V_j=1$, and the 
	jump of the rightmost particle at level $j$
	is donated to the right.}
	\label{fig:QcRSK_beta}
\end{figure}

The above four rules of constructing the signature $\nu\in\GT_j$
complete the description of the RSK-type dynamics
$\Qarow$, $\Qbrow$,
$\Qacol$, 
and $\Qbcol$, respectively.

\begin{remark}
	By the very construction, 
	at each step of any of the four above RSK-type dynamics 
	the quantity $|\la^{(N)}|$
	is increased by $V_1+\ldots+V_N$,
	as it should be
	(cf. the discussion
	before Remark~\ref{rmk:push_block_RSK_different}).
\end{remark}

\begin{theorem}\label{thm:schur_RSK}
	The RSK-type dynamics 
	$\Qarow$, 
	$\Qacol$, 
	$\Qbrow$,
	and $\Qbcol$
	described above satisfy $q=0$ versions of the main equations
	of Theorem \ref{thm:main_eq}
	and hence act on Schur processes by adding 
	a new usual parameter $\al$ or a new dual parameter
	$\be$, respectively (as in \eqref{adding_spec}).
	\label{thm:schurRSK}
\end{theorem}
\begin{proof}
	This statement follows from bijective properties of RSK correspondences
	(briefly discussed below in this section),
	or, equivalently, it may be regarded as $q=0$
	degeneration of our main results
	about RSK-type dynamics on $q$-Whittaker processes
	(Theorems \ref{thm:QqrRSK_beta}, \ref{thm:QqcRSK_beta}, \ref{thm:QqrRSK_alpha}, and \ref{thm:QqcRSK_alpha}).
\end{proof}

Each of four RSK-type dynamics described above
gives rise to a certain bijection
between sets $\{\la,\bar\la,\bar\nu,V_j\}$
and $\{\la,\nu,\bar\nu\}$ (at each time step
and at each level $j$ of the interlacing array). 
In more detail, each of the dynamics $\Qarow$
and $\Qacol$
(see Fig.~\ref{fig:QrRSK_alpha}
and \ref{fig:QcRSK_alpha}) produces a bijection
between the following sets:
\begin{align}\label{alpha_bijection}
	\big\{\la,\bar\la,\bar\nu
	\colon
	\la\succh\bar\la\prech\bar\nu\big\}
	\cup\big\{V_j\in\Z_{\ge0}\big\}
	\longleftrightarrow
	\big\{\la,\nu,\bar\nu\colon\la\prech\nu\succh\bar\nu\big\}.
\end{align}
Similarly, each of the
dynamics $\Qbrow$
and $\Qbcol$
(Fig.~\ref{fig:QrRSK_beta}
and \ref{fig:QcRSK_beta})
establishes a bijection between the sets 
\begin{align}\label{beta_bijection}
	\big\{\la,\bar\la,\bar\nu
	\colon
	\la\succh\bar\la\precv\bar\nu\big\}
	\cup\big\{V_j\in\{0,1\}\big\}
	\longleftrightarrow
	\big\{\la,\nu,\bar\nu\colon\la\precv\nu\succh\bar\nu\big\}.
\end{align}
In \eqref{alpha_bijection} and \eqref{beta_bijection}
we have $\bar\la,\bar\nu\in\GT_{j-1}$
and $\la,\nu\in\GT_j$, as usual.

The understanding of RSK correspondences
via bijections as in 
\eqref{alpha_bijection} and \eqref{beta_bijection}
was presented in 
\cite{fomin1995schur}.\footnote{Starting multivariate
dynamics 
from initial condition
$\la^{(j)}_{i}\equiv 0$ for all $1\le i\le j\le N$
and considering all levels of an interlacing array,
bijections \eqref{alpha_bijection}
and \eqref{beta_bijection} extend to bijective 
correspondences between certain
integer matrices and pairs of semistandard Young tableaux
(in agreement with the well-known 
understanding of RSK correspondences).}
It also implies that 
fixing $\la$ and $\bar\nu$
and
taking generating functions
of both sets in \eqref{alpha_bijection}, by
weighting elements
of the left set by $(a_j\al)^{V_j+|\bar\nu|-|\bar\la|}$,
and of the right set by 
$(a_j\al)^{|\nu|-|\la|}$
(under the bijections, these powers are equal to each other),
one recovers the skew Cauchy identity
\eqref{skew_Cauchy_concrete} for $\mathbf{B}=(\al)$.
Similarly, \eqref{beta_bijection} leads to 
\eqref{skew_Cauchy_concrete} with $\mathbf{B}=(\hat\be)$.
This observation agrees with the
understanding of multivariate dynamics as refinements of 
the skew Cauchy identity
(\S \ref{ssub:refined_Cauchy}).

\begin{remark}
	In RSK-type dynamics
	on $q$-Whittaker processes
	considered in 
	\S \ref{sec:bernoulli_}
	and \S \ref{sec:geometric_q_rsks}
	below, a part of new randomness at each step
	also comes 
	from independent random
	variables $V_1,\ldots,V_N$
	(having $q$-geometric or Bernoulli distribution, 
	cf. Remark \ref{rmk:one_level_dynamics}).
	Moreover, for $q>0$ the bijective
	mechanisms \eqref{alpha_bijection}, \eqref{beta_bijection}
	will be \emph{$q$-randomized} (i.e. will no longer be deterministic bijections).
	This would lead to four \emph{$q$-randomized
	RSK correspondences}: the row and column $(\al)$,
	and the row and column $(\hat\be)$. 
	In fact, for $q>0$
	the step-by-step nature 
	of the $q=0$ case (when $\mathsf{push}$ or $\mathsf{pull}$
	operations are performed one at a time) will be broken,
	and certain series of 
	$\mathsf{push}$ or $\mathsf{pull}$ 
	operations will be clumped together and $q$-randomized as a whole.
	This will make the dynamics at the $q$-Whittaker level more complicated.
\end{remark}

Each of the four RSK-type dynamics
possesses a 
marginally Markovian projection
(onto the leftmost or the rightmost particles of the interlacing
array) leading to a certain discrete time
particle system on $\Z$. 
Namely, 
$\Qarow$ and $\Qbrow$
give rise to
the \emph{geometric and Bernoulli PushTASEPs},
respectively,
on the rightmost particles 
$\la^{(j)}_{1}$, $j=1,\ldots,N$.
Similarly, 
$\Qacol$ and $\Qbcol$
lead to the
\emph{geometric and Bernoulli TASEPs},
respectively, 
on the leftmost particles 
$\la^{(j)}_{j}$.
The $q$-deformed dynamics 
of \S \ref{sec:bernoulli_}
and \S \ref{sec:geometric_q_rsks} below would
lead to $q$-deformations 
of these four particle systems. 

\begin{remark}
	By imposing some reasonable 
	nearest neighbor constraints on discrete time multivariate dynamics,
	one may seek a full classification of solutions
	of the main equations of Theorem~\ref{thm:main_eq}
	in the Schur ($q=0$) case.
	Such classification in continuous time setting was obtained in 
	\cite{BorodinPetrov2013NN}.
	We do not pursue this direction here.
\end{remark}



\section{RSK-type dynamics 
$\Qqbrow$ 
and $\Qqbcol$ 
adding a dual parameter} 
\label{sec:bernoulli_}

In this section we explain the construction of 
two RSK-type dynamics on $q$-Whittaker processes
adding a dual parameter $\be$ to the specialization 
(in the sense of \eqref{adding_spec}).
For $q=0$, these dynamics degenerate to 
$(\hat\be)$ dynamics on Schur processes 
arising from row and column RSK insertion. 
We also discover that for $0<q<1$,
the row and column dynamics 
$\Qqbrow$ 
and $\Qqbcol$
are related by a certain transformation (we call it \emph{complementation}).
Moreover, in a small $\be$ limit the complementation provides a direct connection
between continuous time RSK-type dynamics on $q$-Whittaker processes
introduced in \cite{OConnellPei2012} (column version) and 
\cite{BorodinPetrov2013NN} (row version).

\subsection{Row insertion dynamics $\Qqbrow$} 
\label{sub:row_insertion_dynamics_q__hatbe_q-rrsk_}

Let us now describe one time step $\lab\to\nub$ 
of the multivariate Markov dynamics $\Qqbrow$
on $q$-Whittaker processes of depth $N$. 
A part of randomness during this step comes from
independent Bernoulli random variables
$V_1,\ldots,V_N\in\{0,1\}$
with parameters $\be a_1,\ldots,\be a_N$, respectively
(these random variables are resampled during each time step).

The bottommost particle of the interlacing array is updated as $\nu^{(1)}_{1}=\la^{(1)}_{1}+V_1$
(as it should be, cf. Remark \ref{rmk:one_level_dynamics}).
Next, sequentially for each $j=2,\ldots,N$, given the movement $\bar\la\to \bar\nu$ at level $j-1$,
we will randomly update $\la\to\nu$ at level $j$.
To describe this update, write 
\begin{align*}
	\bar\nu-\bar\la=\sum_{i=1}^{j-1}c_i \bar{\mathrm{e}}_{i},
	\qquad c_i\in\{0,1\},\qquad
	\text{$\bar{\mathrm{e}}_{i}$ are basis vectors of length $j-1$},
\end{align*}
and say that numbers $(k,m)$, where $1\le k\le m\le j-1$, form \emph{island}$(k,m)$
if 
\begin{align*}
	\text{$c_{k-1}=0$\quad(or $k=1$),\quad$c_k=c_{k+1}=\ldots=c_m=1$,\quad{}and\quad$c_{m+1}=0$\quad(or $m=j-1$)}.
\end{align*}
That is, all particles that have moved at level $j-1$
split into several disjoint islands.
Also denote for any $i=1,\ldots,j-1$:
\begin{align}\label{f_g_definition}
	\mathsf{f}_{i}=\mathsf{f}_{i}(\bar\nu,\la):=\frac{1-q^{\la_i-\bar\nu_i+1}}{1-q^{\bar\nu_{i-1}-\bar\nu_i+1}}
	,\qquad
	\mathsf{g}_{i}=\mathsf{g}_{i}(\bar\nu,\la):=1-q^{\la_{i}-\bar\nu_{i}+1}
\end{align}
(by agreement, let $\bar\nu_0:=+\infty$).
Note that all these quantities are between 0 and 1.

\setcounter{la1}{7}
\setcounter{la2}{5}
\setcounter{la3}{5}
\setcounter{la4}{3}
\setcounter{la5}{3}
\setcounter{la6}{2}
\setcounter{lab1}{6}
\setcounter{lab2}{5}
\setcounter{lab3}{5}
\setcounter{lab4}{3}
\setcounter{lab5}{2}

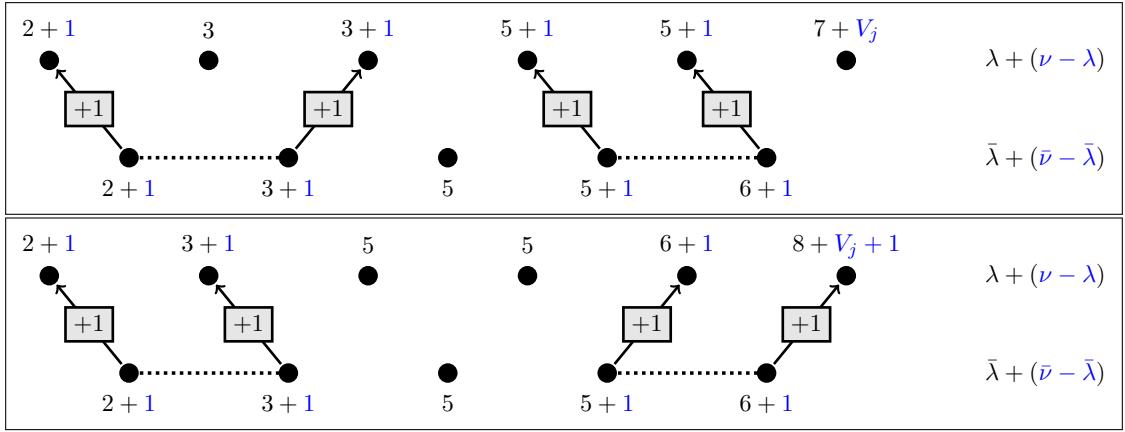
\begin{figure}[htbp]
	\framebox{\begin{adjustbox}{max height=.16\textwidth}
	\begin{tikzpicture}
		[scale=.7, thick]
		\def\cir{.2}
		\def\ysh{6}
		\def\x{1.8}
		\def\y{2.2}
		\draw[fill] (-2*\x,0) circle(\cir) node [below,yshift=-\ysh] {$\arabic{lab5}+{\color{blue}1}$};
		\draw[fill] (0,0) circle(\cir) node [below,yshift=-\ysh] {$\arabic{lab4}+{\color{blue}1}$};
		\draw[fill] (2*\x,0) circle(\cir) node [below,yshift=-\ysh] {$\arabic{lab3}$};
		\draw[fill] (4*\x,0) circle(\cir) node [below,yshift=-\ysh] {$\arabic{lab2}+{\color{blue}1}$};
		\draw[fill] (6*\x,0) circle(\cir) node [below,yshift=-\ysh] {$\arabic{lab1}+{\color{blue}1}$};
		\draw[fill] (-3*\x,1*\y) circle(\cir) node [above,yshift=\ysh] {$\arabic{la6}+{\color{blue}1}$};
		\draw[fill] (-1*\x,1*\y) circle(\cir) node [above,yshift=\ysh] {$\arabic{la5}$};
		\draw[fill] (1*\x,1*\y) circle(\cir) node [above,yshift=\ysh] {$\arabic{la4}+{\color{blue}1}$};
		\draw[fill] (3*\x,1*\y) circle(\cir) node [above,yshift=\ysh] {$\arabic{la3}+{\color{blue}1}$};
		\draw[fill] (5*\x,1*\y) circle(\cir) node [above,yshift=\ysh] {$\arabic{la2}+{\color{blue}1}$};
		\draw[fill] (7*\x,1*\y) circle(\cir) node [above,yshift=-2.2+\ysh] {$\arabic{la1}+{\color{blue}V_j}$};
		\node at (9.5*\x,0) {$\bar\la+({\color{blue}\bar\nu-\bar\la})$};
		\node at (9.5*\x,\y) {$\la+({\color{blue}\nu-\la})$};
		\node (lab5) at (-2*\x,0) {};
		\node (lab4) at (0*\x,0) {};
		\node (lab3) at (2*\x,0) {};
		\node (lab2) at (4*\x,0) {};
		\node (lab1) at (6*\x,0) {};
		\node (la6) at (-3*\x,\y) {};
		\node (la5) at (-1*\x,\y) {};
		\node (la4) at (1*\x,\y) {};
		\node (la3) at (3*\x,\y) {};
		\node (la2) at (5*\x,\y) {};
		\node (la1) at (7*\x,\y) {};
		\draw[ultra thick, dotted] (lab1) -- (lab2);
		\draw[ultra thick, dotted] (lab4) -- (lab5);
		\draw[->,very thick] (lab1) -- (la2) node[rectangle,draw=black,fill=gray!20!white] [midway] {$+1$};
		\draw[->,very thick] (lab2) -- (la3) node[rectangle,draw=black,fill=gray!20!white] [midway] {$+1$};
		\draw[->,very thick] (lab4) -- (la4) node[rectangle,draw=black,fill=gray!20!white] [midway] {$+1$};
		\draw[->,very thick] (lab5) -- (la6) node[rectangle,draw=black,fill=gray!20!white] [midway] {$+1$};
	\end{tikzpicture}
	\end{adjustbox}}
	\\
\setcounter{la1}{8}
\setcounter{la2}{6}
\setcounter{la3}{5}
\setcounter{la4}{5}
\setcounter{la5}{3}
\setcounter{la6}{2}
\setcounter{lab1}{6}
\setcounter{lab2}{5}
\setcounter{lab3}{5}
\setcounter{lab4}{3}
\setcounter{lab5}{2}
	\framebox{\begin{adjustbox}{max height=.16\textwidth}
	\begin{tikzpicture}
		[scale=.7, thick]
		\def\cir{.2}
		\def\ysh{6}
		\def\x{1.8}
		\def\y{2.2}
		\draw[fill] (-2*\x,0) circle(\cir) node [below,yshift=-\ysh] {$\arabic{lab5}+{\color{blue}1}$};
		\draw[fill] (0,0) circle(\cir) node [below,yshift=-\ysh] {$\arabic{lab4}+{\color{blue}1}$};
		\draw[fill] (2*\x,0) circle(\cir) node [below,yshift=-\ysh] {$\arabic{lab3}$};
		\draw[fill] (4*\x,0) circle(\cir) node [below,yshift=-\ysh] {$\arabic{lab2}+{\color{blue}1}$};
		\draw[fill] (6*\x,0) circle(\cir) node [below,yshift=-\ysh] {$\arabic{lab1}+{\color{blue}1}$};
		\draw[fill] (-3*\x,1*\y) circle(\cir) node [above,yshift=\ysh] {$\arabic{la6}+{\color{blue}1}$};
		\draw[fill] (-1*\x,1*\y) circle(\cir) node [above,yshift=\ysh] {$\arabic{la5}+{\color{blue}1}$};
		\draw[fill] (1*\x,1*\y) circle(\cir) node [above,yshift=\ysh] {$\arabic{la4}$};
		\draw[fill] (3*\x,1*\y) circle(\cir) node [above,yshift=\ysh] {$\arabic{la3}$};
		\draw[fill] (5*\x,1*\y) circle(\cir) node [above,yshift=\ysh] {$\arabic{la2}+{\color{blue}1}$};
		\draw[fill] (7*\x,1*\y) circle(\cir) node [above,yshift=-2.2+\ysh] {$\arabic{la1}+{\color{blue}V_j+1}$};
		\node at (9.5*\x,0) {$\bar\la+({\color{blue}\bar\nu-\bar\la})$};
		\node at (9.5*\x,\y) {$\la+({\color{blue}\nu-\la})$};
		\node (lab5) at (-2*\x,0) {};
		\node (lab4) at (0*\x,0) {};
		\node (lab3) at (2*\x,0) {};
		\node (lab2) at (4*\x,0) {};
		\node (lab1) at (6*\x,0) {};
		\node (la6) at (-3*\x,\y) {};
		\node (la5) at (-1*\x,\y) {};
		\node (la4) at (1*\x,\y) {};
		\node (la3) at (3*\x,\y) {};
		\node (la2) at (5*\x,\y) {};
		\node (la1) at (7*\x,\y) {};
		\draw[ultra thick, dotted] (lab1) -- (lab2);
		\draw[ultra thick, dotted] (lab4) -- (lab5);
		\draw[->,very thick] (lab1) -- (la1) node[rectangle,draw=black,fill=gray!20!white] [midway] {$+1$};
		\draw[->,very thick] (lab2) -- (la2) node[rectangle,draw=black,fill=gray!20!white] [midway] {$+1$};
		\draw[->,very thick] (lab4) -- (la5) node[rectangle,draw=black,fill=gray!20!white] [midway] {$+1$};
		\draw[->,very thick] (lab5) -- (la6) node[rectangle,draw=black,fill=gray!20!white] [midway] {$+1$};
	\end{tikzpicture}
	\end{adjustbox}}
	\caption{An example of a step of 
	$\Qqbrow$ at levels 5 and 6.
	There are two islands, $(1,2)$ and $(4,5)$,
	moving at level $j-1$.
	Above: $V_j=1$, 
	and the probability of the displayed transition is
	$1\cdot (1-\mathsf{f}_4)\mathsf{g}_5=1-q$
	(note that here the particle $\la_4=3$ cannot be chosen not to move
	because $\mathsf{f}_4=0$).
	Below: $V_j=0$,
	and the probability of the displayed transition is
	$(1-\mathsf{f}_1)(1-\mathsf{g}_2)\cdot\mathsf{f}_4=q^{3}$
	(note that here the particle $\la_4=5$ must be chosen not to move
	because $\mathsf{f}_4=1$).}
	\label{fig:QqrRSK_beta}
\end{figure}

The update $\la\to\nu$ at level $j$ goes as follows (see Fig.~\ref{fig:QqrRSK_beta}).
First, the rightmost particle jumps to the right by $V_j$, i.e.,
$\nu_1=\la_1+V_j{\mathrm{e}}_{1}$.
Then, independently for every island$(k,m)$ 
of particles that have moved at level $j-1$, perform the following updates:
\begin{enumerate}
	\item If $V_j=1$ and $k=1$ (i.e., the particle $\la_1$
	has already moved, and the island contains the first particle at level $j-1$), 
	then move the
	particles $\la_2,\ldots,\la_{m+1}$
	at level $j$ to the right by one with probability 1.

	\item If $V_j=1$ and $k>1$, or $V_j=0$ (i.e., island$(k,m)$ does not 
	interfere with the movement of $\la_1$ coming from $V_j$, or there is no independent
	movement of $\la_1$), then
	island$(k,m)$ triggers the movement (to the right by one)
	of
	all particles $\la_k,\ldots,\la_{m+1}$
	except one. The particle which does not move 
	is chosen at random:
	\begin{itemize}
		\item $\la_k$ is chosen not to move with probability
		\begin{align}\label{island_prob_1}
			\mathsf{f}_k=\frac{1-q^{\la_k-\bar\nu_k+1}}{1-q^{\bar\nu_{k-1}-\bar\nu_k+1}};
		\end{align}
		\item each $\la_s$, $k+1\le s \le m$,
		is chosen not to move with probability
		\begin{align}\label{island_prob_2}
			(1-\mathsf{f}_k)(1-\mathsf{g}_{k+1})\ldots(1-\mathsf{g}_{s-1})\mathsf{g}_{s}
			=\frac{q^{\la_k-\bar\nu_k+1}-q^{\bar\nu_{k-1}-\bar\nu_k+1}}{1-q^{\bar\nu_{k-1}-\bar\nu_k+1}}
			q^{\sum_{i=k+1}^{s-1}(\la_i-\bar\nu_i+1)}
			(1-q^{\la_s-\bar\nu_s+1});
		\end{align}
		\item $\la_{m+1}$ is chosen not to move with probability
		\begin{align}
			\label{island_prob_3}
			(1-\mathsf{f}_k)(1-\mathsf{g}_{k+1})\ldots(1-\mathsf{g}_{m-1})(1-\mathsf{g}_{m})
			=
			\frac{q^{\la_k-\bar\nu_k+1}-q^{\bar\nu_{k-1}-\bar\nu_k+1}}{1-q^{\bar\nu_{k-1}-\bar\nu_k+1}}
			q^{\sum_{i=k+1}^{m}(\la_i-\bar\nu_i+1)}.
		\end{align}
	\end{itemize}
	Probabilities \eqref{island_prob_1}, \eqref{island_prob_2}, 
	and \eqref{island_prob_3}
	are nonnegative, and
	their sum telescopes to 1.
\end{enumerate}
This completes the description of the $(\hat \be)$ row insertion RSK-type dynamics $\Qqbrow$.
Clearly, thus defined conditional probabilities $\mathscr{U}_j$, $j=1,\ldots,N$,
for this dynamics
satisfy \eqref{U_properties}.

\begin{remark}\label{rmk:automatically_push_pull}
	The $q$-deformed probabilities \eqref{island_prob_1}, \eqref{island_prob_2}, 
	and \eqref{island_prob_3}
	ensure that mandatory pushing and blocking mechanisms
	(built into Definitions \ref{def:pull} and \ref{def:push})
	work automatically:
	\begin{itemize}
		\item 
		If 
		$\la_s=\bar\nu_s-1$ for any $k\le s \le m$, then 
		the particle
		$\la_s$ cannot be chosen not to move.
		This agrees with the mandatory pushing of $\la_s$
		by the move of $\bar\la_s=\la_s$
		which is necessary to restore the interlacing.
		\item If $\la_{k}=\bar\nu_{k-1}$ (i.e., $\la_k$ is blocked), 
		then $\mathsf{f}_k=1$,
		so $\la_k$ must be chosen not to move.
		This means that in this dynamics no move donations ever arise
		(cf Remark \ref{rmk:move_donation}).
	\end{itemize}
\end{remark}

\begin{theorem}\label{thm:QqrRSK_beta}
	The dynamics $\Qqbrow$
	defined above satisfies the main equations
	\eqref{beta_main_equation},
	and hence
	preserves the class of $q$-Whittaker processes
	and adds a new dual parameter $\be$
	to the specialization $\mathbf{A}$
	as in \eqref{adding_spec}.
\end{theorem}
\begin{proof}
	We need to prove \eqref{beta_main_equation}
	for any fixed $j=2,\ldots,N$
	and 
	$\la,\nu\in\GT_j$, $\bar\nu\in\GT_{j-1}$,
	where $\la\precv\nu\succh\bar\nu$ (cf. Fig.~\ref{fig:square}, right).
	For a subset $I\subseteq\{1,2,\ldots,j-1\}$,
	set
	\begin{align*}
		U_I:=(1+\be a_j)\mathscr{U}_j(\la\to\nu\mid\bar\la\to\bar\nu)
		\frac{\psi_{\la/\bar\la}\psi'_{\bar\nu/\bar\la}}
		{\psi_{\nu/\bar\nu}\psi'_{\nu/\la}},
	\end{align*}
	where $\bar\la=\bar\nu-\sum_{i\in I}\bar{\mathrm{e}}_{i}$,
	i.e., $\bar\la\in\GT_{j-1}$ is obtained
	from $\bar\nu$ by shifting back (by one)
	all particles with indices belonging to $I$.
	By agreement, if $I$ is such that
	$\bar\la$
	does not satisfy
	$\la\succh\bar\la\precv\bar\nu$ 
	(cf.~Fig.~\ref{fig:square}, right),
	then $U_I=0$.
	With this notation, the desired identity 
	\eqref{beta_main_equation} turns into 
	\begin{align}\label{beta_main_equation_in_proof}
		\sum_{I\subseteq\{1,2,\ldots,j-1\}}
		U_I(\be a_j)^{|\la|-|\nu|-(|\bar\la|-|\bar\nu|)}=1.
	\end{align}
	Note that the denominator $(1+\be a_j)$ coming from 
	the Bernoulli distribution of $V_j$
	will always cancel the corresponding 
	factor in all $U_I$'s.

	\smallskip

	First, let us consider a particular case when $\nu=\la+\sum_{i=k}^{m}\mathrm{e}_i$,
	i.e., the movement $\la\to\nu$ involves a consecutive group of particles
	from $k$ to $m$, where $1\le k\le m\le j-1$.
	There are four subcases:

	\smallskip

	\textbf{1.} If $k>1$ and $m<j$, then necessarily $V_j=0$, and 
	\eqref{beta_main_equation_in_proof} becomes
	\begin{align}
		U_{[k-1,m-1]}+\sum_{s=k}^{m-1}U_{[k-1,s-1]\cup[s+1,m]}
		+
		U_{[k,m]}=1
		\label{beta_main_equation_in_proof1}
	\end{align}
	(here and below
	by $[k-1,m-1]$, etc., we mean the corresponding interval of indices).
	See Fig.~\ref{fig:beta_main_equation_in_proof}.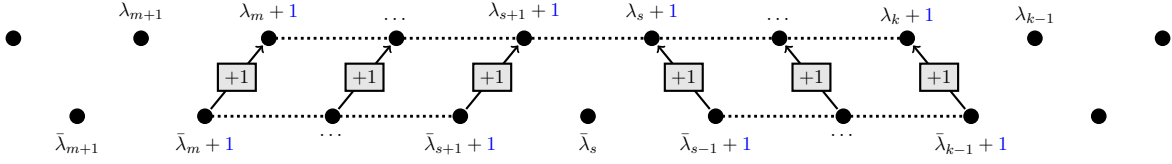
\begin{figure}[htbp]
		\begin{adjustbox}{max height=.136\textwidth}
		\begin{tikzpicture}
			[scale=.7, thick]
			\def\cir{.2}
			\def\ysh{6}
			\def\x{1.8}
			\def\y{2.2}
			\draw[fill] (-6*\x,0) circle(\cir) node [below,yshift=-\ysh] {$\bar\la_{m+1}$};
			\draw[fill] (-4*\x,0) circle(\cir) node [below,yshift=-\ysh] {$\bar\la_m+{\color{blue}1}$};
			\draw[fill] (-2*\x,0) circle(\cir) node [below,yshift=-\ysh] {\ldots};
			\draw[fill] (0,0) circle(\cir) node [below,yshift=-\ysh] {$\bar\la_{s+1}+{\color{blue}1}$};
			\draw[fill] (2*\x,0) circle(\cir) node [below,yshift=-\ysh] {$\bar\la_s$};
			\draw[fill] (4*\x,0) circle(\cir) node [below,yshift=-\ysh] {$\bar\la_{s-1}+{\color{blue}1}$};
			\draw[fill] (6*\x,0) circle(\cir) node [below,yshift=-\ysh] {\ldots};
			\draw[fill] (8*\x,0) circle(\cir) node [below,yshift=-\ysh] {$\bar\la_{k-1}+{\color{blue}1}$};
			\draw[fill] (10*\x,0) circle(\cir) node [below,yshift=-\ysh] {};
			\draw[fill] (-7*\x,1*\y) circle(\cir) node [above,yshift=\ysh] {};
			\draw[fill] (-5*\x,1*\y) circle(\cir) node [above,yshift=\ysh] {$\la_{m+1}$};
			\draw[fill] (-3*\x,1*\y) circle(\cir) node [above,yshift=\ysh] {$\la_{m}+{\color{blue}1}$};
			\draw[fill] (-1*\x,1*\y) circle(\cir) node [above,yshift=\ysh] {\ldots};
			\draw[fill] (1*\x,1*\y) circle(\cir) node [above,yshift=\ysh] {$\la_{s+1}+{\color{blue}1}$};
			\draw[fill] (3*\x,1*\y) circle(\cir) node [above,yshift=\ysh] {$\la_s+{\color{blue}1}$};
			\draw[fill] (5*\x,1*\y) circle(\cir) node [above,yshift=\ysh] {\ldots};
			\draw[fill] (7*\x,1*\y) circle(\cir) node [above,yshift=-2.2+\ysh] {$\la_{k}+{\color{blue}1}$};
			\draw[fill] (9*\x,1*\y) circle(\cir) node [above,yshift=-2.2+\ysh] {$\la_{k-1}$};
			\draw[fill] (11*\x,1*\y) circle(\cir) node [above,yshift=-2.2+\ysh] {};
			\draw[ultra thick, dotted] (-4*\x,0) -- (0,0);
			\draw[ultra thick, dotted] (4*\x,0) -- (8*\x,0);
			\draw[ultra thick, dotted] (-3*\x,\y) -- (7*\x,\y);
			\node (labk1) at (8*\x,0) {}; 
			\node (labk2) at (6*\x,0) {};
			\node (labk3) at (4*\x,0) {};
			\node (labk4) at (0*\x,0) {};
			\node (labk5) at (-2*\x,0) {};
			\node (labk6) at (-4*\x,0) {};
			\node (lak1) at (7*\x,\y) {};
			\node (lak2) at (5*\x,\y) {};
			\node (lak3) at (3*\x,\y) {};
			\node (lak4) at (1*\x,\y) {};
			\node (lak5) at (-1*\x,\y) {};
			\node (lak6) at (-3*\x,\y) {};
			\draw[->,very thick] (labk1) -- (lak1) node[rectangle,draw=black,fill=gray!20!white] [midway] {$+1$};
			\draw[->,very thick] (labk2) -- (lak2) node[rectangle,draw=black,fill=gray!20!white] [midway] {$+1$};
			\draw[->,very thick] (labk3) -- (lak3) node[rectangle,draw=black,fill=gray!20!white] [midway] {$+1$};
			\draw[->,very thick] (labk4) -- (lak4) node[rectangle,draw=black,fill=gray!20!white] [midway] {$+1$};
			\draw[->,very thick] (labk5) -- (lak5) node[rectangle,draw=black,fill=gray!20!white] [midway] {$+1$};
			\draw[->,very thick] (labk6) -- (lak6) node[rectangle,draw=black,fill=gray!20!white] [midway] {$+1$};
		\end{tikzpicture}
		\end{adjustbox}
		\caption{Situation corresponding to the $s$-th term in 
		\eqref{beta_main_equation_in_proof1}.}
		\label{fig:beta_main_equation_in_proof}
	\end{figure}
	Using \eqref{P_one_variable}, \eqref{Q_one_dual_spec},
	we have (as before, here and below in the proof we agree that $\bar\nu_0=+\infty$)
	\begin{align*}
		U_{[k-1,m-1]}&=\underbrace{\mathsf{f}_{k-1}(\bar\nu,\la)}_{\mathscr{U}_j}\cdot
		\underbrace{
		\frac{\binom{\la_{m}-\la_{m+1}}{\la_m-\bar\nu_m}_{q}
		}
		{\binom{\la_{m}+1-\la_{m+1}}{\la_m+1-\bar\nu_m}_{q}
		}
		\frac{\binom{\la_{k-1}-\la_k}{\la_{k-1}-\bar\nu_{k-1}+1}_q}
		{\binom{\la_{k-1}-\la_k-1}{\la_{k-1}-\bar\nu_{k-1}}_q}
		}_{\psi_{\la/\bar\la}/\psi_{\nu/\bar\nu}}
		\cdot
		\underbrace{\frac{1-q^{\bar\nu_{k-2}-\bar\nu_{k-1}+1}}
		{1-q^{\la_{k-1}-\la_{k}}}}_{\psi'_{\bar\nu/\bar\la}/\psi'_{\nu/\la}}
		\\&=
		\frac{1-q^{\la_{k-1}-\bar\nu_{k-1}+1}}{1-q^{\bar\nu_{k-2}-\bar\nu_{k-1}+1}}
		\frac{1-q^{\la_{m}-\bar\nu_m+1}}{1-q^{\la_m-\la_{m+1}+1}}
		\frac{1-q^{\la_{k-1}-\la_k}}{1-q^{\la_{k-1}-\bar\nu_{k-1}+1}}
		\frac{1-q^{\bar\nu_{k-2}-\bar\nu_{k-1}+1}}
		{1-q^{\la_{k-1}-\la_{k}}}
		\\&=
		\frac{1-q^{\la_{m}-\bar\nu_m+1}}{1-q^{\la_m-\la_{m+1}+1}}.
	\end{align*}
	Also for any $k\le s\le m-1$,
	\begin{align*}
		&
		U_{[k-1,s-1]\cup[s+1,m]}\\&=
		\underbrace{\mathsf{f}_{k-1}(\bar\nu,\la)
		\cdot(1-\mathsf{f}_{s+1}(\bar\nu,\la))
		(1-\mathsf{g}_{s+2}(\bar\nu,\la))\ldots
		(1-\mathsf{g}_{m}(\bar\nu,\la))}_{\mathscr{U}_j}\\&\phantom{={}}\times
		\underbrace{
		\frac{\binom{\la_{m}-\la_{m+1}}{\la_m-\bar\nu_m+1}_{q}
		}
		{\binom{\la_{m}+1-\la_{m+1}}{\la_m+1-\bar\nu_m}_{q}
		}
		\frac{\binom{\la_{s}-\la_{s+1}}{\la_s-\bar\nu_{s}}_{q}}
		{\binom{\la_{s}-\la_{s+1}}{\la_s+1-\bar\nu_{s}}_{q}
		}
		\frac{\binom{\la_{k-1}-\la_k}{\la_{k-1}-\bar\nu_{k-1}+1}_q}
		{\binom{\la_{k-1}-\la_k-1}{\la_{k-1}-\bar\nu_{k-1}}_q}
		}_{\psi_{\la/\bar\la}/\psi_{\nu/\bar\nu}}
		\cdot
		\underbrace{\frac{(1-q^{\bar\nu_{k-2}-\bar\nu_{k-1}+1})
		(1-q^{\bar\nu_{s}-\bar\nu_{s+1}+1})}
		{1-q^{\la_{k-1}-\la_{k}}}}_{\psi'_{\bar\nu/\bar\la}/\psi'_{\nu/\la}}
		\\&=
		\frac{1-q^{\la_{k-1}-\bar\nu_{k-1}+1}}{1-q^{\bar\nu_{k-2}-\bar\nu_{k-1}+1}}
		\frac{q^{\la_{s+1}-\bar\nu_{s+1}+1}-q^{\bar\nu_{s}-\bar\nu_{s+1}+1}}
		{1-q^{\bar\nu_{s}-\bar\nu_{s+1}+1}}
		q^{\sum_{i=s+2}^{m}(\la_i-\bar\nu_i+1)}
		\\&\phantom{={}}\times
		\frac{1-q^{\bar\nu_m-\la_{m+1}}}{1-q^{\la_m+1-\la_{m+1}}}
		\frac{1-q^{\la_s+1-\bar\nu_s}}{1-q^{\bar\nu_s-\la_{s+1}}}
		\frac{1-q^{\la_{k-1}-\la_k}}{1-q^{\la_{k-1}-\bar\nu_{k-1}+1}}
		\frac{(1-q^{\bar\nu_{k-2}-\bar\nu_{k-1}+1})
		(1-q^{\bar\nu_{s}-\bar\nu_{s+1}+1})}
		{1-q^{\la_{k-1}-\la_{k}}}
		\\&=
		\frac{(1-q^{\la_s+1-\bar\nu_s})(1-q^{\bar\nu_m-\la_{m+1}})}{1-q^{\la_m+1-\la_{m+1}}}
		q^{\sum_{i=s+1}^{m}(\la_i-\bar\nu_i+1)},
	\end{align*}
	and
	\begin{align*}
		U_{[k,m]}&=
		\underbrace{(1-\mathsf{f}_{k}(\bar\nu,\la))
		(1-\mathsf{g}_{k+1}(\bar\nu,\la))\ldots
		(1-\mathsf{g}_{m}(\bar\nu,\la))}_{\mathscr{U}_j}\\&\phantom{={}}\times
		\underbrace{
		\frac{\binom{\la_{m}-\la_{m+1}}{\la_m-\bar\nu_m+1}_{q}
		}
		{\binom{\la_{m}+1-\la_{m+1}}{\la_m+1-\bar\nu_m}_{q}
		}
		\frac{\binom{\la_{k-1}-\la_k}{\la_{k-1}-\bar\nu_{k-1}}_q}
		{\binom{\la_{k-1}-\la_k-1}{\la_{k-1}-\bar\nu_{k-1}}_q}
		}_{\psi_{\la/\bar\la}/\psi_{\nu/\bar\nu}}
		\cdot
		\underbrace{\frac{1-q^{\bar\nu_{k-1}-\bar\nu_{k}+1}}
		{1-q^{\la_{k-1}-\la_{k}}}}_{\psi'_{\bar\nu/\bar\la}/\psi'_{\nu/\la}}
		\\&=
		\frac{q^{\la_k-\bar\nu_k+1}-q^{\bar\nu_{k-1}-\bar\nu_k+1}}{1-q^{\bar\nu_{k-1}-\bar\nu_k+1}}
		q^{\sum_{i=k+1}^{m}(\la_i-\bar\nu_i+1)}
		\frac{1-q^{\bar\nu_m-\la_{m+1}}}{1-q^{\la_m+1-\la_{m+1}}}
		\frac{1-q^{\la_{k-1}-\la_k}}{1-q^{\bar\nu_{k-1}-\la_k}}
		\frac{1-q^{\bar\nu_{k-1}-\bar\nu_{k}+1}}
		{1-q^{\la_{k-1}-\la_{k}}}
		\\&=
		\frac{1-q^{\bar\nu_m-\la_{m+1}}}{1-q^{\la_m+1-\la_{m+1}}}
		q^{\sum_{i=k}^{m}(\la_i-\bar\nu_i+1)}.
	\end{align*}
	The summation
	in \eqref{beta_main_equation_in_proof1} 
	thus telescopes and gives 1 as desired (similarly
	to the sum of expressions \eqref{island_prob_1}, 
	\eqref{island_prob_2}, and \eqref{island_prob_3}).

	\smallskip

	\textbf{2.} If $k>1$ and $m=j$,
	then also necessarily $V_j=0$,
	and there is only one $I$, namely, $[k-1,j-1]$, contributing to \eqref{beta_main_equation_in_proof}.
	We have
	\begin{align*}
		U_{[k-1,j-1]}&=\mathsf{f}_{k-1}(\bar\nu,\la)\cdot 
		\frac{\binom{\la_{k-1}-\la_{k}}{\la_{k-1}-\bar\nu_{k-1}+1}_{q}}
		{\binom{\la_{k-1}-\la_{k}-1}{\la_{k-1}-\bar\nu_{k-1}} _{q}}\cdot
		\frac{1-q^{\bar\nu_{k-2}-\bar\nu_{k-1}+1}}{1-q^{\la_{k-1}-\la_{k}}}
		\\&=\frac{1-q^{\la_{k-1}-\bar\nu_{k-1}+1}}{1-q^{\bar\nu_{k-2}-\bar\nu_{k-1}+1}}
		\frac{1-q^{\la_{k-1}-\la_{k}}}{1-q^{\la_{k-1}-\bar\nu_{k-1}+1}}
		\frac{1-q^{\bar\nu_{k-2}-\bar\nu_{k-1}+1}}{1-q^{\la_{k-1}-\la_{k}}}
		=1,
	\end{align*}
	so we see that \eqref{beta_main_equation_in_proof} holds.

	\smallskip

	\textbf{3.}
	If $k=1$ and $m<j$,
	then $V_j$ can be either 0 or 1, and 
	\eqref{beta_main_equation_in_proof} now looks as
	\begin{align*}
		U_{[1,m]}+(a_j\be)^{-1}\sum_{s=1}^{m-1}U_{[1,s-1]\cup[s+1,m]}
		+(a_j\be)^{-1} U_{[1,m-1]}=1.
	\end{align*}
	This identity
	is established 
	similarly to the subcase 1. Namely, one readily sees that
	\begin{align*}
		U_{[1,m]}&=
		\frac{1 - q^{\bar\nu_{m}-\la_{m + 1}}}{1-q^{\la_{m} - \la_{m + 1} +1}}
		q^{\sum_{i = 1}^{m}(\la_{i} - \bar\nu_{i} + 1)};
		\\
		U_{[1,s-1]\cup[s+1,m]}&=(a_j\be)
		\frac{(1 - q^{\bar\nu_{m} - \la_{m +1}})(1 - q^{\la_{s} - \bar\nu_{s} + 1})}
		{1 - q^{\la_{m} - \la_{m+1} + 1}}
		q^{\sum_{i = s+1}^{m} (\la_{i} - \bar\nu_{i} + 1)};
		\\
		U_{[1,m-1]}&=(a_j\be)\frac{1-q^{\la_m-\bar\nu_{m}+1}}{1 - q^{\la_{m} - \la_{m+1} + 1}},
	\end{align*}
	and the sum of these quantities telescopes and gives 1.

	\smallskip

	\textbf{4.} If $k=1$ and $m=j$, this means that necessarily
	$V_j=1$, and 
	the only term that enters \eqref{beta_main_equation_in_proof} is
	$U_{[1,j-1]}=\be a_j$,
	so the desired identity also holds.

	\smallskip 
		
	We have now established
	the desired identity
	in the particular case $\nu=\la+\sum_{i=k}^{m}\mathrm{e}_i$. 
	In the general case there could be several consecutive groups of
	particles forming the move $\la\to\nu$ at level $j$. 
	Let there be gaps of at least two not moving particles
	between neighboring moving groups.
	Then, by the product nature of the quantities
	$\psi$ and $\psi'$ \eqref{P_one_variable}, \eqref{Q_one_dual_spec},
	as well as by the independence of propagation for different islands at level $j-1$, 
	cf. Fig.~\ref{fig:QqrRSK_beta},
	the sum in the left-hand side of 
	\eqref{beta_main_equation_in_proof}
	can clearly be represented as a product 
	of sums corresponding to 
	individual groups of moving particles.
	Each such individual sum is the same as in 
	one of the 
	subcases 1--4 above, and therefore is equal to 1.
	This implies \eqref{beta_main_equation_in_proof}
	in the case when moving groups at level $j$
	are sufficiently far apart.\begin{figure}[htbp]
		\begin{adjustbox}{max height=.136\textwidth}
		\begin{tikzpicture}
			[scale=.7, thick]
			\def\cir{.2}
			\def\ysh{6}
			\def\x{1.8}
			\def\y{2.2}
			\draw[fill] (-6*\x,0) circle(\cir) node [below,yshift=-\ysh] {};
			\draw[fill] (-4*\x,0) circle(\cir) node [below,yshift=-\ysh] {};
			\draw[fill] (-2*\x,0) circle(\cir) node [below,yshift=-\ysh] {$\bar\la_{m-1}+{\color{blue}1}$};
			\draw[fill] (0,0) circle(\cir) node [below,yshift=-\ysh] {$\bar\la_{s}+{\color{blue}1}$};
			\draw[fill] (2*\x,0) circle(\cir) node [below,yshift=-\ysh] {$\bar\la_{s-1}+{\color{blue}1}$};
			\draw[fill] (4*\x,0) circle(\cir) node [below,yshift=-\ysh] {\ldots};
			\draw[fill] (6*\x,0) circle(\cir) node [below,yshift=-\ysh] {$\bar\la_{k}+{\color{blue}1}$};
			\draw[fill] (8*\x,0) circle(\cir) node [below,yshift=-\ysh] {};
			\draw[fill] (10*\x,0) circle(\cir) node [below,yshift=-\ysh] {};
			\draw[fill] (-7*\x,1*\y) circle(\cir) node [above,yshift=\ysh] {};
			\draw[fill] (-5*\x,1*\y) circle(\cir) node [above,yshift=\ysh] {};
			\draw[fill] (-3*\x,1*\y) circle(\cir) node [above,yshift=\ysh] {$\la_{m}+{\color{blue}1}$};
			\draw[fill] (-1*\x,1*\y) circle(\cir) node [above,yshift=\ysh] {\ldots};
			\draw[fill] (1*\x,1*\y) circle(\cir) node [above,yshift=\ysh] {$\la_{s}$};
			\draw[fill] (3*\x,1*\y) circle(\cir) node [above,yshift=\ysh] {\ldots};
			\draw[fill] (5*\x,1*\y) circle(\cir) node [above,yshift=\ysh] {\ldots};
			\draw[fill] (7*\x,1*\y) circle(\cir) node [above,yshift=-2.2+\ysh] {$\la_{k}+{\color{blue}1}$};
			\draw[fill] (9*\x,1*\y) circle(\cir) node [above,yshift=-2.2+\ysh] {};
			\draw[fill] (11*\x,1*\y) circle(\cir) node [above,yshift=-2.2+\ysh] {};
			\draw[ultra thick, dotted] (-2*\x,0) -- (6*\x,0);
			\draw[ultra thick, dotted] (-3*\x,\y) -- (-1*\x,\y);
			\draw[ultra thick, dotted] (3*\x,\y) -- (7*\x,\y);
			\node (labk1) at (6*\x,0) {}; 
			\node (labk2) at (4*\x,0) {};
			\node (labk3) at (2*\x,0) {};
			\node (labk4) at (0*\x,0) {};
			\node (labk5) at (-2*\x,0) {};
			\node (lak1) at (7*\x,\y) {};
			\node (lak2) at (5*\x,\y) {};
			\node (lak3) at (3*\x,\y) {};
			\node (lak4) at (-1*\x,\y) {};
			\node (lak5) at (-3*\x,\y) {};
			\draw[->,very thick] (labk1) -- (lak1) node[rectangle,draw=black,fill=gray!20!white] [midway] {$+1$};
			\draw[->,very thick] (labk2) -- (lak2) node[rectangle,draw=black,fill=gray!20!white] [midway] {$+1$};
			\draw[->,very thick] (labk3) -- (lak3) node[rectangle,draw=black,fill=gray!20!white] [midway] {$+1$};
			\draw[->,very thick] (labk4) -- (lak4) node[rectangle,draw=black,fill=gray!20!white] [midway] {$+1$};
			\draw[->,very thick] (labk5) -- (lak5) node[rectangle,draw=black,fill=gray!20!white] [midway] {$+1$};
		\end{tikzpicture}
		\end{adjustbox}
		\caption{Two islands at level $j$ corresponding to a single
		island at level $j-1$.}
		\label{fig:beta_main_equation_in_proof1}
	\end{figure}
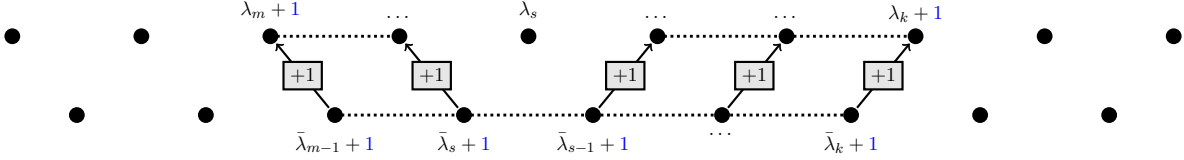
	\begin{figure}[htbp]
		\begin{tabular}{cc}
		\framebox{\begin{adjustbox}{max height=.126\textwidth}
		\begin{tikzpicture}
			[scale=.7, thick]
			\def\cir{.2}
			\def\ysh{6}
			\def\x{1.8}
			\def\y{2.2}
			\draw[fill] (-4*\x,0) circle(\cir) node [below,yshift=-\ysh] {};
			\draw[fill] (-2*\x,0) circle(\cir) node [below,yshift=-\ysh] {$\bar\la_{m-1}+{\color{blue}1}$};
			\draw[fill] (0,0) circle(\cir) node [below,yshift=-\ysh] {$\bar\la_{s}+{\color{blue}1}$};
			\draw[fill] (2*\x,0) circle(\cir) node [below,yshift=-\ysh] {$\bar\la_{s-1}$};
			\draw[fill] (-5*\x,1*\y) circle(\cir) node [above,yshift=\ysh] {};
			\draw[fill] (-3*\x,1*\y) circle(\cir) node [above,yshift=\ysh] {$\la_{m}+{\color{blue}1}$};
			\draw[fill] (-1*\x,1*\y) circle(\cir) node [above,yshift=\ysh] {$\la_{s+1}+{\color{blue}1}$};
			\draw[fill] (1*\x,1*\y) circle(\cir) node [above,yshift=\ysh] {$\la_{s}$};
			\draw[fill] (3*\x,1*\y) circle(\cir) node [above,yshift=\ysh] {\ldots};
			\draw[ultra thick, dotted] (-2*\x,0) -- (0*\x,0);
			\draw[ultra thick, dotted] (-3*\x,\y) -- (-1*\x,\y);
			\node (labk4) at (0*\x,0) {};
			\node (labk5) at (-2*\x,0) {};
			\node (lak4) at (-1*\x,\y) {};
			\node (lak5) at (-3*\x,\y) {};
			\draw[->,very thick] (labk4) -- (lak4) node[rectangle,draw=black,fill=gray!20!white] [midway] {$+1$};
			\draw[->,very thick] (labk5) -- (lak5) node[rectangle,draw=black,fill=gray!20!white] [midway] {$+1$};
		\end{tikzpicture}
		\end{adjustbox}}
		&
		\framebox{\begin{adjustbox}{max height=.126\textwidth}
		\begin{tikzpicture}
			[scale=.7, thick]
			\def\cir{.2}
			\def\ysh{6}
			\def\x{1.8}
			\def\y{2.2}
			\draw[fill] (0,0) circle(\cir) node [below,yshift=-\ysh] {$\bar\la_{s}$};
			\draw[fill] (2*\x,0) circle(\cir) node [below,yshift=-\ysh] {$\bar\la_{s-1}+{\color{blue}1}$};
			\draw[fill] (4*\x,0) circle(\cir) node [below,yshift=-\ysh] {\ldots};
			\draw[fill] (6*\x,0) circle(\cir) node [below,yshift=-\ysh] {$\bar\la_{k}+{\color{blue}1}$};
			\draw[fill] (8*\x,0) circle(\cir) node [below,yshift=-\ysh] {};
			\draw[fill] (-1*\x,1*\y) circle(\cir) node [above,yshift=\ysh] {\ldots};
			\draw[fill] (1*\x,1*\y) circle(\cir) node [above,yshift=\ysh] {$\la_{s}$};
			\draw[fill] (3*\x,1*\y) circle(\cir) node [above,yshift=\ysh] {$\la_{s-1}+{\color{blue}1}$};
			\draw[fill] (5*\x,1*\y) circle(\cir) node [above,yshift=\ysh] {\ldots};
			\draw[fill] (7*\x,1*\y) circle(\cir) node [above,yshift=-2.2+\ysh] {$\la_{k}+{\color{blue}1}$};
			\draw[fill] (9*\x,1*\y) circle(\cir) node [above,yshift=-2.2+\ysh] {};
			\draw[ultra thick, dotted] (2*\x,0) -- (6*\x,0);
			\draw[ultra thick, dotted] (3*\x,\y) -- (7*\x,\y);
			\node (labk1) at (6*\x,0) {}; 
			\node (labk2) at (4*\x,0) {};
			\node (labk3) at (2*\x,0) {};
			\node (lak1) at (7*\x,\y) {};
			\node (lak2) at (5*\x,\y) {};
			\node (lak3) at (3*\x,\y) {};
			\draw[->,very thick] (labk1) -- (lak1) node[rectangle,draw=black,fill=gray!20!white] [midway] {$+1$};
			\draw[->,very thick] (labk2) -- (lak2) node[rectangle,draw=black,fill=gray!20!white] [midway] {$+1$};
			\draw[->,very thick] (labk3) -- (lak3) node[rectangle,draw=black,fill=gray!20!white] [midway] {$+1$};
		\end{tikzpicture}
		\end{adjustbox}}
		\end{tabular}
		\caption{Two configurations giving the same contribution
		as the one on Fig.~\ref{fig:beta_main_equation_in_proof1}.}
		\label{fig:beta_main_equation_in_proof2}
	\end{figure}
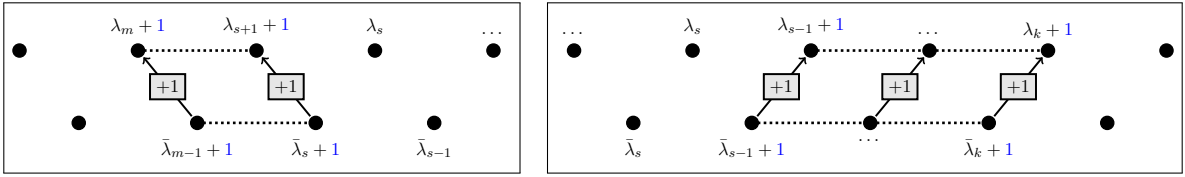

	Finally, it remains to check 
	\eqref{beta_main_equation_in_proof} 
	in the case
	when there could be moving
	groups at level $j$
	separated by one not moving particle. 
	Consider two such neighboring groups.
	The only configuration
	of moves at level $j-1$
	(corresponding to these two groups at level $j$)
	that could prevent 
	the sum in \eqref{beta_main_equation_in_proof}
	to be of product form is given
	on Fig.~\ref{fig:beta_main_equation_in_proof1}.
	However, one readily sees that the contribution of this configuration 
	is the same as the product of contributions of two configurations
	on Fig.~\ref{fig:beta_main_equation_in_proof2}.
	Indeed, factors involving the
	quantities $\psi$
	are already in a product form, and the 
	remaining factors (coming from $\mathscr{U}_j$ and the quantities $\psi'$)
	are
	\begin{align*}
		&
		\underbrace{
		\frac{q^{\la_k-\bar\la_k}-q^{\bar\la_{k-1}-\bar\la_k}}{1-q^{\bar\la_{k-1}-\bar\la_k}}
		q^{\sum_{i=k+1}^{s-1}(\la_i-\bar\la_i)}
		(1-q^{\la_s-\bar\la_s})
		}_{\text{$(1-\mathsf{f}_{k})(1-\mathsf{g}_{k+1})\ldots(1-\mathsf{g}_{s-1})
		\mathsf{g}_{s}$ on Fig.~\ref{fig:beta_main_equation_in_proof1}}}\cdot
		\frac{1-q^{\bar\la_{k-1}-\bar\la_{k}}}
		{(1-q^{\la_{k-1}-\la_k})(1-q^{\la_{s}-\la_{s+1}})}
		\\&\hspace{20pt}=
		\underbrace{\frac{1-q^{\la_s-\bar\la_s}}{1-q^{\bar\la_{s-1}-\bar\la_s}}}_{\text{$\mathsf{f}_{s}$
		on Fig.~\ref{fig:beta_main_equation_in_proof2}, left}}
		\cdot \frac{1-q^{\bar\la_{s-1}-\bar\la_{s}}}{1-q^{\la_{s}-\la_{s+1}}}
		\times
		\underbrace{
		\frac{q^{\la_k-\bar\la_k}-q^{\bar\la_{k-1}-\bar\la_k}}{1-q^{\bar\la_{k-1}-\bar\la_k}}
		q^{\sum_{i=k+1}^{s-1}(\la_i-\bar\la_i)}
		}_{\text{$(1-\mathsf{f}_{k})(1-\mathsf{g}_{k+1})\ldots(1-\mathsf{g}_{s-1})$
		on Fig.~\ref{fig:beta_main_equation_in_proof2}, right}}
		\cdot 
		\frac{1-q^{\bar\la_{k-1}-\bar\la_{k}}}{1-q^{\la_{k-1}-\la_k}}.
	\end{align*}
	Note that we have expressed everything in terms of 
	signatures $\la$ and $\bar\la$ because the signatures
	$\bar\nu$ differ on
	Fig.~\ref{fig:beta_main_equation_in_proof1}
	and
	Fig.~\ref{fig:beta_main_equation_in_proof2}.
	
	Therefore, in the last remaining case we can still rewrite 
	\eqref{beta_main_equation_in_proof} in a product form.
	This completes the proof of the theorem.
\end{proof}

\begin{remark}[Schur degeneration]\label{rmk:schur_degeneration}
	If one sets $q=0$, then 
	in a generic situation 
	(when
	particles at levels $j-1$ and $j$
	are sufficiently far apart from each other)
	all quantities $\mathsf{f}_i$
	and
	$\mathsf{g}_i$ become equal to one,
	see \eqref{f_g_definition}.
	One readily sees that the dynamics 
	$\Qqbrow$
	reduces to 
	the dynamics 
	$\Qbrow$
	on Schur processes. The latter dynamics is 
	based on the classical Robinson--Schensted--Knuth
	row insertion (\S \ref{sub:rsk_type_dynamics_schur}).
\end{remark}


\subsection{Bernoulli $q$-PushTASEP} 
\label{sub:bernoulli_q_pushtasep}

One can readily check that under the dynamics 
$\Qqbrow$
we have just constructed, 
the rightmost $N$ particles $\la^{(j)}_{1}$
of the interlacing array
evolve in a \emph{marginally Markovian manner}
(i.e., their evolution does not depend
on the dynamics of the rest of the interlacing
array). 
Namely, at each discrete time step $t\to t+1$
the bottommost particle
is updated as
$\la^{(1)}_{1}(t+1)=\la^{(1)}_{1}(t)+V_1$, and for any $j=2,\ldots,N$:
\begin{itemize}
	\item If $\la^{(j-1)}_{1}$ has not moved, then the rightmost
	particle at level $j$ is updated as
	\begin{align*}
		\la^{(j)}_{1}(t+1)=\la^{(j)}_{1}(t)+V_j;
	\end{align*}
	\item
	If $\la^{(j-1)}_{1}$ has moved to the right by one,
	then the same particle is updated as
	\begin{align*}
		\la^{(j)}_{1}(t+1)=\la^{(j)}_{1}(t)+V_j+(1-V_j)\cdot\mathbf{1}_{\text{pushing
		by $\la^{(j-1)}_{1}$}},
	\end{align*}
	where pushing by $\la^{(j-1)}_{1}$
	happens with probability
	$1-\mathsf{f}_1=q^{\la_1^{(j)}(t)-\la_{1}^{(j-1)}(t)}$
	which depends only on the rightmost particles of the array.
\end{itemize}
(Recall that the $V_i$'s are independent Bernoulli random variables which are 
independently resampled each step of the discrete time.)
This evolution of the 
rightmost particles 
$\la^{(j)}_{1}$,
$1\le j\le N$,
leads to a new interacting particle system
on $\Z$ which we call the (\emph{discrete time}) \emph{Bernoulli \mbox{$q$-PushTASEP}}.
We discuss this process in detail in \S \ref{sec:moments_and_fredholm_determinants_for_bernoulli_}
below.


\begin{figure}[htbp]
	\begin{adjustbox}{max height=.136\textwidth}
	\begin{tikzpicture}
		[scale=.7, thick]
		\def\cir{.17}
		\def\ysh{6}
		\def\x{.8}
		\def\y{2.2}
		\draw[fill, red!90!black] (-22*\x,0) circle(1.5*\cir) node [below,yshift=-\ysh] {};
		\draw[fill, red!90!black] (-20*\x,0) circle(1.5*\cir) node [below,yshift=-\ysh] {};
		\draw[fill] (-18*\x,0) circle(\cir) node [below,yshift=-\ysh] {$\color{blue}+1$};
		\draw[fill] (-16*\x,0) circle(\cir) node [below,yshift=-\ysh] {$\color{blue}+1$};
		\draw[fill] (-14*\x,0) circle(\cir) node [below,yshift=-\ysh] {$\color{blue}+1$};
		\draw[fill] (-12*\x,0) circle(\cir) node [below,yshift=-\ysh] {$\color{blue}+1$};
		\draw[fill, red!90!black] (-10*\x,0) circle(1.5*\cir) node [below,yshift=-\ysh] {};
		\draw[fill, red!90!black] (-8*\x,0) circle(1.5*\cir) node [below,yshift=-\ysh] {};
		\draw[fill] (-6*\x,0) circle(\cir) node [below,yshift=-\ysh] {$\color{blue}+1$};
		\draw[fill] (-4*\x,0) circle(\cir) node [below,yshift=-\ysh] {$\color{blue}+1$};
		\draw[fill, red!90!black] (-2*\x,0) circle(1.5*\cir) node [below,yshift=-\ysh] {};
		\draw[fill, red!90!black] (0,0) circle(1.5*\cir) node [below,yshift=-\ysh] {};
		\draw[fill, red!90!black] (2*\x,0) circle(1.5*\cir) node [below,yshift=-\ysh] {};
		\draw[fill] (4*\x,0) circle(\cir) node [below,yshift=-\ysh] {$\color{blue}+1$};
		\draw[fill] (6*\x,0) circle(\cir) node [below,yshift=-\ysh] {$\color{blue}+1$};
		\draw[fill] (8*\x,0) circle(\cir) node [below,yshift=-\ysh] {$\color{blue}+1$};
		\draw[fill, red!90!black] (10*\x,0) circle(1.5*\cir) node [below,yshift=-\ysh] {};
		\draw[fill, red!90!black] (-23*\x,1*\y) circle(1.5*\cir) node [above,yshift=\ysh] {};
		\draw[fill, red!90!black] (-21*\x,1*\y) circle(1.5*\cir) node [above,yshift=\ysh] {};
		\draw[fill] (-19*\x,1*\y) circle(\cir) node [above,yshift=\ysh] {$\color{blue}+1$};
		\draw[fill] (-17*\x,1*\y) circle(\cir) node [above,yshift=\ysh] {$\color{blue}+1$};
		\draw[fill, red!90!black] (-15*\x,1*\y) circle(1.5*\cir) node [above,yshift=\ysh] {};
		\draw[fill] (-13*\x,1*\y) circle(\cir) node [above,yshift=\ysh] {$\color{blue}+1$};
		\draw[fill] (-11*\x,1*\y) circle(\cir) node [above,yshift=\ysh] {$\color{blue}+1$};
		\draw[fill, red!90!black] (-9*\x,1*\y) circle(1.5*\cir) node [above,yshift=\ysh] {};
		\draw[fill] (-7*\x,1*\y) circle(\cir) node [above,yshift=\ysh] {$\color{blue}+1$};
		\draw[fill] (-5*\x,1*\y) circle(\cir) node [above,yshift=\ysh] {$\color{blue}+1$};
		\draw[fill, red!90!black] (-3*\x,1*\y) circle(1.5*\cir) node [above,yshift=\ysh] {};
		\draw[fill, red!90!black] (-1*\x,1*\y) circle(1.5*\cir) node [above,yshift=\ysh] {};
		\draw[fill, red!90!black] (1*\x,1*\y) circle(1.5*\cir) node [above,yshift=\ysh] {};
		\draw[fill] (3*\x,1*\y) circle(\cir) node [above,yshift=\ysh] {$\color{blue}+1$};
		\draw[fill, red!90!black] (5*\x,1*\y) circle(1.5*\cir) node [above,yshift=\ysh] {};
		\draw[fill] (7*\x,1*\y) circle(\cir) node [above,yshift=-2.2+\ysh] {$\color{blue}+1$};
		\draw[fill] (9*\x,1*\y) circle(\cir) node [above,yshift=-2.2+\ysh] {$\color{blue}+1$};
		\draw[fill, red!90!black] (11*\x,1*\y) circle(1.5*\cir) node [above,yshift=-2.2+\ysh] {};
		\draw[ultra thick, dotted] (-18*\x,0) -- (-12*\x,0);
		\draw[ultra thick, dotted] (-6*\x,0) -- (-4*\x,0);
		\draw[ultra thick, dotted] (4*\x,0) -- (8*\x,0);
		\draw[ultra thick, dotted] (-19*\x,\y) -- (-17*\x,\y);
		\draw[ultra thick, dotted] (-13*\x,\y) -- (-11*\x,\y);
		\draw[ultra thick, dotted] (-7*\x,\y) -- (-5*\x,\y);
		\draw[ultra thick, dotted] (7*\x,\y) -- (9*\x,\y);
		\node (a1) at (-18*\x,0) {};
		\node (b1) at (-19*\x,\y) {};
		\draw[->,very thick] (a1) -- (b1);
		\node (a1) at (-16*\x,0) {};
		\node (b1) at (-17*\x,\y) {};
		\draw[->,very thick] (a1) -- (b1);
		\node (a1) at (-14*\x,0) {};
		\node (b1) at (-13*\x,\y) {};
		\draw[->,very thick] (a1) -- (b1);
		\node (a1) at (-12*\x,0) {};
		\node (b1) at (-11*\x,\y) {};
		\draw[->,very thick] (a1) -- (b1);
		\node (a1) at (-6*\x,0) {};
		\node (b1) at (-7*\x,\y) {};
		\draw[->,very thick] (a1) -- (b1);
		\node (a1) at (-4*\x,0) {};
		\node (b1) at (-5*\x,\y) {};
		\draw[->,very thick] (a1) -- (b1);
		\node (a1) at (-4*\x,0) {};
		\node (b1) at (-5*\x,\y) {};
		\draw[->,very thick] (a1) -- (b1);
		\node (a1) at (4*\x,0) {};
		\node (b1) at (3*\x,\y) {};
		\draw[->,very thick] (a1) -- (b1);
		\node (a1) at (6*\x,0) {};
		\node (b1) at (7*\x,\y) {};
		\draw[->,very thick] (a1) -- (b1);
		\node (a1) at (8*\x,0) {};
		\node (b1) at (9*\x,\y) {};
		\draw[->,very thick] (a1) -- (b1);
		\node (a1) at (-20*\x,0) {};
		\node (b1) at (-15*\x,\y) {};
		\draw[->, line width=3 ,red!90!black, densely dashed] (a1) -- (b1);
		\node (a1) at (-22*\x,0) {};
		\node (b1) at (-21*\x,\y) {};
		\draw[->, line width=3 ,red!90!black, densely dashed] (a1) -- (b1);
		\node (a1) at (-10*\x,0) {};
		\node (b1) at (-9*\x,\y) {};
		\draw[->, line width=3 ,red!90!black, densely dashed] (a1) -- (b1);
		\node (a1) at (-8*\x,0) {};
		\node (b1) at (-3*\x,\y) {};
		\draw[->, line width=3 ,red!90!black, densely dashed] (a1) -- (b1);
		\node (a1) at (-2*\x,0) {};
		\node (b1) at (-1*\x,\y) {};
		\draw[->, line width=3 ,red!90!black, densely dashed] (a1) -- (b1);
		\node (a1) at (0*\x,0) {};
		\node (b1) at (1*\x,\y) {};
		\draw[->, line width=3 ,red!90!black, densely dashed] (a1) -- (b1);
		\node (a1) at (2*\x,0) {};
		\node (b1) at (5*\x,\y) {};
		\draw[->, line width=3 ,red!90!black, densely dashed] (a1) -- (b1);
		\node (a1) at (10*\x,0) {};
		\node (b1) at (11*\x,\y) {};
		\draw[->, line width=3 ,red!90!black, densely dashed] (a1) -- (b1);
	\end{tikzpicture}
	\end{adjustbox}
	\caption{Complementation of propagation rules 
	turning the dynamics $\Qqbrow$
	(with move propagation given by thin solid arrows)
	into
	$\Qqbcol$
	(corresponding to thick dashed arrows).}
	\label{fig:involutive}
\end{figure}
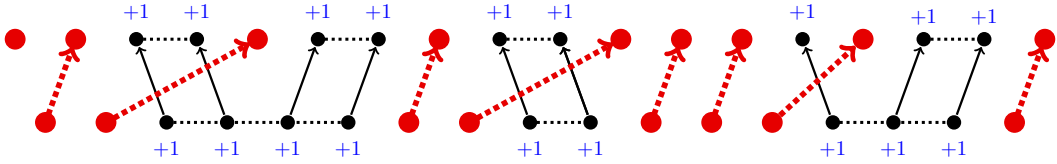

\subsection{Complementation} 
\label{sub:Complementation}

Let us take another look at propagation rules 
employed in the definition of the row insertion
dynamics
$\Qqbrow$
on $q$-Whittaker processes 
(see the beginning of \S \ref{sub:row_insertion_dynamics_q__hatbe_q-rrsk_}).
These rules state that, generically, an island
of moving particles at level $j-1$
splits (at random) into two moving
islands at level $j$
separated by exactly one staying particle
(either of two moving islands at level $j$ is allowed to be empty).
Now consider the
\emph{pattern of staying particles} at levels $j-1$ and $j$.
We see that an island$(k,m)$ (where $k\le m$) of staying particles
at level $j-1$ always gives rise to an 
island$(k+1,m)$ of staying particles at level $j$,
plus one more staying particle somewhere to the right of 
$k$ (but to the left of the next staying particle at level $j$). 
The latter staying particle (whose index is chosen at random) is precisely the one separating the 
two moving islands at level $j$. See Fig.~\ref{fig:involutive}.

The transformation of propagation rules of $\Qqbrow$
that we just described informally in fact leads to a new RSK-type
multivariate dynamics on $q$-Whittaker processes. Let us work in a more general setting:
\begin{definition}[Complementation of a dynamics]\label{def:complementation}
	Assume that $\mathscr{Q}$
	is a multivariate sequential update dynamics
	on $q$-Whittaker processes
	adding a specialization $(\hat\be)$.
	For $j=2,\ldots,N$
	and signatures $\la,\nu\in\GT_j$, $\bar\la,\bar\nu\in\GT_{j-1}$
	satisfying conditions on Fig.~\ref{fig:square}, right,
	let $\mathscr{U}_j(\la\to\nu\mid\bar\la\to\bar\nu)$
	be the corresponding conditional probabilities.
	Assume that the dynamics is \emph{translation invariant}, i.e., that the
	values $\mathscr{U}_j(\la\to\nu\mid\bar\la\to\bar\nu)$
	do not change if one adds the same number to all coordinates of all four signatures.

	For $S$ a sufficiently large positive integer, 
	define the \emph{complement conditional probabilities} as
	\begin{align*}
		\mathscr{U}_j'(\la\to\nu
		\mid\bar\la\to\bar\nu)
		:=
		(a_j\be)^{-2(|\la|-|\nu|-|\bar\la|+|\bar\nu|)-1}
		\mathscr{U}_j\Big([S-\la]\to [S+1-\nu]
		\;\Big\vert\; [S-\bar\la]\to [S+1-\bar\nu]\Big),
	\end{align*}
	where
	\begin{align*}
		[S-\la]:=\big(S-\la_j\ge S-\la_{j-1}\ge \ldots\ge S-\la_1\big)
	\end{align*}
	is the complement of the Young diagram $\la$ in the $j\times S$ rectangle,
	and similarly for
	$[S+1-\nu]$, $[S-\bar\la]$, and $[S+1-\bar\nu]$
	(hence the name ``complementation'').
	Note that these four new signatures also satisfy 
	conditions on Fig.~\ref{fig:square}, right.

	Let us denote by $\mathscr{Q}'$
	the dynamics on interlacing arrays corresponding to $\mathscr{U}_j'$, $j=2,\ldots,N$.
	Note that due to translation invariance, the complement dynamics 
	$\mathscr{Q}'$ does not depend on $S$ provided that $S$ is large enough.
\end{definition}

\begin{lemma}\label{lemma:psi_complementation}
	Let $S$ be sufficiently large.
	For $\bar\mu\in\GT_{k-1}$, $\mu\in\GT_k$
	such that
	$\bar\mu\prech\mu$, we have
	\begin{align*}
		\psi_{[S-\mu]/[S-\bar\mu]}=\psi_{\mu/\bar\mu}.
	\end{align*}
	For $\mu,\varkappa\in\GT_k$ such that 
	$\mu\precv\varkappa$, we have
	\begin{align*}
		\psi'_{[S+1-\varkappa]/[S-\mu]}=\psi'_{\varkappa/\mu}.
	\end{align*}
\end{lemma}
\begin{proof}
	A straightforward verification using definitions 
	\eqref{P_one_variable}, \eqref{Q_one_dual_spec}.
\end{proof}

\begin{proposition}\label{prop:complementation_dynamics_is_good}
	If $\mathscr{Q}$ is a multivariate
	sequential update 
	dynamics
	adding a specialization $(\hat\be)$,
	then so is the complement dynamics $\mathscr{Q}'$.
\end{proposition}
\begin{proof}
	One can show that the 
	complement dynamics $\mathscr{Q}'$
	satisfies the same main equations
	\eqref{beta_main_equation}
	as the original dynamics 
	$\mathscr{Q}$. 
	Indeed, Lemma \ref{lemma:psi_complementation} ensures that 
	the coefficients
	$\psi_{\la/\bar\la}\psi'_{\bar\nu/\bar\la}$
	and
	$\psi_{\nu/\bar\nu}\psi'_{\nu/\la}$
	do not change under complementation,
	and powers of $(a_j\be)$ 
	also transform as they should:
	\begin{align*}
		&\mathscr{U}_j'(\la\to\nu
		\mid\bar\la\to\bar\nu)
		(a_j\be)^{|\la|-|\nu|-|\bar\la|+|\bar\nu|}
		\\&\hspace{15pt}=
		(a_j\be)^{-|\la|+|\nu|+|\bar\la|-|\bar\nu|-1}
		\mathscr{U}_j\Big([S-\la]\to [S+1-\nu]
		\;\Big\vert\; [S-\bar\la]\to [S+1-\bar\nu]\Big)
		\\&\hspace{15pt}=
		(a_j\be)^{\big|[S-\la]\big|-\big|[S+1-\nu]\big|-\big|[S-\bar\la]\big|+\big|[S+1-\bar\nu]\big|}
		\mathscr{U}_j\Big([S-\la]\to [S+1-\nu]
		\;\Big\vert\; [S-\bar\la]\to [S+1-\bar\nu]\Big).
	\end{align*}
	This establishes the main equations for the
	complement dynamics.
\end{proof}


\subsection{Column insertion dynamics $\Qqbcol$} 
\label{sub:column_insertion_dynamics_q__hatbe_q-rrsk_}

Clearly, the row insertion dynamics 
$\Qqbrow$
on $q$-Whittaker processes is translation invariant
(in the sense of Definition \ref{def:complementation}),
so one can define the complement dynamics.
Denote it by $\Qqbcol$.
Let us describe (in an explicit way) 
the evolution of the interlacing array
under this new dynamics
during one step of the discrete time. See Fig.~\ref{fig:col_beta} for an
example.

As before, a part of randomness comes from
independent Bernoulli random variables
$V_j\in\{0,1\}$ with $P(V_j=0)=1/(1+\be a_j)$, $j=1,\ldots,N$.
The bottommost particle of the interlacing array is updated as $\nu^{(1)}_{1}=\la^{(1)}_{1}+V_1$.
Sequentially for each $j=2,\ldots,N$, given the movement $\bar\la\to \bar\nu$ at level $j-1$,
we will randomly update $\la\to\nu$ at level $j$. Let us denote
(as usual, $\bar\nu_0=+\infty$)
\begin{align}\label{f_g_prime_definition}
	\mathsf{f}_k'=\mathsf{f}_k'(\bar\nu,\la):=\frac{1-q^{\bar\nu_{k-1}-\la_{k}}}{1-q^{\bar\nu_{k-1}-\bar\nu_{k}+1}},
	\qquad \qquad
	\mathsf{g}_s'=\mathsf{g}_s'(\bar\nu,\la):=1-q^{\bar\nu_{s-1}-\la_{s}}.
\end{align}

\setcounter{la1}{7}
\setcounter{la2}{7}
\setcounter{la3}{5}
\setcounter{la4}{3}
\setcounter{la5}{3}
\setcounter{la6}{0}
\setcounter{lab1}{7}
\setcounter{lab2}{6}
\setcounter{lab3}{3}
\setcounter{lab4}{3}
\setcounter{lab5}{2}
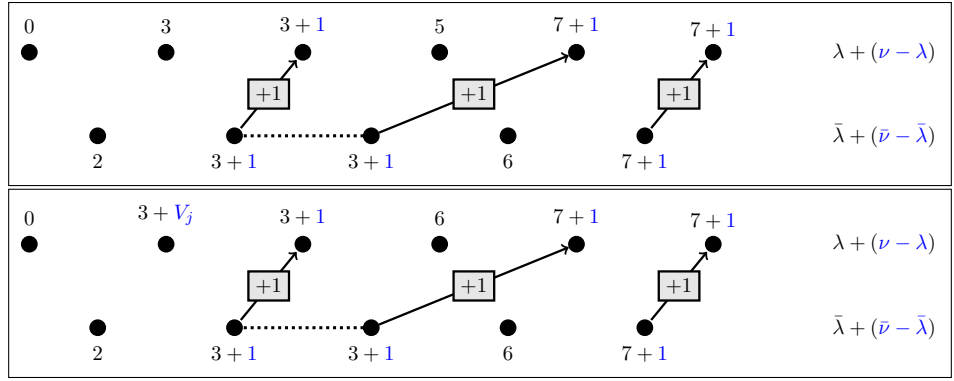
\begin{figure}[htbp]
	\framebox{\begin{adjustbox}{max width=.76\textwidth}
	\begin{tikzpicture}
	[scale=.7, thick]
	\def\cir{.2}
	\def\ysh{6}
	\def\x{1.8}
	\def\y{2.2}
	\draw[fill] (-2*\x,0) circle(\cir) node [below,yshift=-\ysh] {$\arabic{lab5}$};
	\draw[fill] (0,0) circle(\cir) node [below,yshift=-\ysh] {$\arabic{lab4}+{\color{blue}1}$};
	\draw[fill] (2*\x,0) circle(\cir) node [below,yshift=-\ysh] {$\arabic{lab3}+{\color{blue}1}$};
	\draw[fill] (4*\x,0) circle(\cir) node [below,yshift=-\ysh] {$\arabic{lab2}$};
	\draw[fill] (6*\x,0) circle(\cir) node [below,yshift=-\ysh] {$\arabic{lab1}+{\color{blue}1}$};
	\draw[fill] (-3*\x,1*\y) circle(\cir) node [above,yshift=\ysh] {$\arabic{la6}$};
	\draw[fill] (-1*\x,1*\y) circle(\cir) node [above,yshift=\ysh] {$\arabic{la5}$};
	\draw[fill] (1*\x,1*\y) circle(\cir) node [above,yshift=\ysh] {$\arabic{la4}+{\color{blue}1}$};
	\draw[fill] (3*\x,1*\y) circle(\cir) node [above,yshift=\ysh] {$\arabic{la3}$};
	\draw[fill] (5*\x,1*\y) circle(\cir) node [above,yshift=\ysh] {$\arabic{la2}+{\color{blue}1}$};
	\draw[fill] (7*\x,1*\y) circle(\cir) node [above,yshift=-2.2+\ysh] {$\arabic{la1}+{\color{blue}1}$};
	\node at (9.5*\x,0) {$\bar\la+({\color{blue}\bar\nu-\bar\la})$};
	\node at (9.5*\x,\y) {$\la+({\color{blue}\nu-\la})$};
	\node (lab5) at (-2*\x,0) {};
	\node (lab4) at (0*\x,0) {};
	\node (lab3) at (2*\x,0) {};
	\node (lab2) at (4*\x,0) {};
	\node (lab1) at (6*\x,0) {};
	\node (la6) at (-3*\x,\y) {};
	\node (la5) at (-1*\x,\y) {};
	\node (la4) at (1*\x,\y) {};
	\node (la3) at (3*\x,\y) {};
	\node (la2) at (5*\x,\y) {};
	\node (la1) at (7*\x,\y) {};
	\draw[ultra thick, dotted] (lab3) -- (lab4);
	\draw[->,very thick] (lab4) -- (la4) node[rectangle,draw=black,fill=gray!20!white] [midway] {$+1$};
	\draw[->,very thick] (lab3) -- (la2) node[rectangle,draw=black,fill=gray!20!white] [midway] {$+1$};
	\draw[->,very thick] (lab1) -- (la1) node[rectangle,draw=black,fill=gray!20!white] [midway] {$+1$};
	\end{tikzpicture}
	\end{adjustbox}}
	\\\setcounter{la1}{7}
\setcounter{la2}{7}
\setcounter{la3}{6}
\setcounter{la4}{3}
\setcounter{la5}{3}
\setcounter{la6}{0}
\setcounter{lab1}{7}
\setcounter{lab2}{6}
\setcounter{lab3}{3}
\setcounter{lab4}{3}
\setcounter{lab5}{2}
	\framebox{\begin{adjustbox}{max width=.76\textwidth}
	\begin{tikzpicture}
	[scale=.7, thick]
	\def\cir{.2}
	\def\ysh{6}
	\def\x{1.8}
	\def\y{2.2}
	\draw[fill] (-2*\x,0) circle(\cir) node [below,yshift=-\ysh] {$\arabic{lab5}$};
	\draw[fill] (0,0) circle(\cir) node [below,yshift=-\ysh] {$\arabic{lab4}+{\color{blue}1}$};
	\draw[fill] (2*\x,0) circle(\cir) node [below,yshift=-\ysh] {$\arabic{lab3}+{\color{blue}1}$};
	\draw[fill] (4*\x,0) circle(\cir) node [below,yshift=-\ysh] {$\arabic{lab2}$};
	\draw[fill] (6*\x,0) circle(\cir) node [below,yshift=-\ysh] {$\arabic{lab1}+{\color{blue}1}$};
	\draw[fill] (-3*\x,1*\y) circle(\cir) node [above,yshift=\ysh] {$\arabic{la6}$};
	\draw[fill] (-1*\x,1*\y) circle(\cir) node [above,yshift=\ysh] {$\arabic{la5}+{\color{blue}V_j}$};
	\draw[fill] (1*\x,1*\y) circle(\cir) node [above,yshift=\ysh] {$\arabic{la4}+{\color{blue}1}$};
	\draw[fill] (3*\x,1*\y) circle(\cir) node [above,yshift=\ysh] {$\arabic{la3}$};
	\draw[fill] (5*\x,1*\y) circle(\cir) node [above,yshift=\ysh] {$\arabic{la2}+{\color{blue}1}$};
	\draw[fill] (7*\x,1*\y) circle(\cir) node [above,yshift=-2.2+\ysh] {$\arabic{la1}+{\color{blue}1}$};
	\node at (9.5*\x,0) {$\bar\la+({\color{blue}\bar\nu-\bar\la})$};
	\node at (9.5*\x,\y) {$\la+({\color{blue}\nu-\la})$};
	\node (lab5) at (-2*\x,0) {};
	\node (lab4) at (0*\x,0) {};
	\node (lab3) at (2*\x,0) {};
	\node (lab2) at (4*\x,0) {};
	\node (lab1) at (6*\x,0) {};
	\node (la6) at (-3*\x,\y) {};
	\node (la5) at (-1*\x,\y) {};
	\node (la4) at (1*\x,\y) {};
	\node (la3) at (3*\x,\y) {};
	\node (la2) at (5*\x,\y) {};
	\node (la1) at (7*\x,\y) {};
	\draw[ultra thick, dotted] (lab3) -- (lab4);
	\draw[->,very thick] (lab4) -- (la4) node[rectangle,draw=black,fill=gray!20!white] [midway] {$+1$};
	\draw[->,very thick] (lab3) -- (la2) node[rectangle,draw=black,fill=gray!20!white] [midway] {$+1$};
	\draw[->,very thick] (lab1) -- (la1) node[rectangle,draw=black,fill=gray!20!white] [midway] {$+1$};
	\end{tikzpicture}
	\end{adjustbox}}
	\caption{An example of a step of 
	$\Qqbcol$ at levels 5 and 6.
	Above: $V_j=0$, 
	and the probability of the displayed transition is
	$1-\mathsf{f}'_3=(q+q^{2})/(1+q+q^{2})$.
	Below: $V_j=1$,
	and the probability of the displayed transition is
	$(1-\mathsf{g}'_6)(1-\mathsf{f}'_3)=q^{2}$. Note that in the latter case the particle $\la_3=6$
	cannot be chosen to move because it is blocked by $\bar\la_2=\la_3$
	which is not moving; this agrees with $\mathsf{f}'_3=0$.}
	\label{fig:col_beta}
\end{figure}

The update $\la\to\nu$ looks as follows:
\begin{enumerate}
	\item 
	Consider a pair of moved particles
	$(\bar\la_{r}, \bar\la_k)$ at level $j-1$, where $0\le r < k\le j-1$,
	such that the particles 
	$\bar\la_{r+1},\ldots,\bar\la_{k-1}$
	in between
	did not move
	(by agreement, $r=0$
	corresponds to $\bar\la_{k}$ being the rightmost moved particle at level $j-1$).
	Regardless of the value of $V_j$,
	each such pair of moved particles at level $j-1$ triggers the move (to the right by one)
	of exactly one particle $\la_s$, 
	$r+1\le s\le k$,
	between them at level $j$. 
	If $r+1=k$, then there is only one choice $s=k$, so $\la_k$
	must move. Otherwise, the moving particle
	$\la_s$
	is chosen at random
	(independently of everything else)
	with the following probabilities:
	\begin{itemize}
		\item If $s=k$, then $\la_s$ is chosen to move with probability
		\begin{align}\label{beta_col_prob1}
			\mathsf{f}'_{k}=\frac{1-q^{\bar\nu_{k-1}-\la_{k}}}{1-q^{\bar\nu_{k-1}-\bar\nu_{k}+1}};
		\end{align}
		\item If $r+1<s<k$, then $\la_s$ is chosen to move with probability
		\begin{align}\label{beta_col_prob2}
			(1-\mathsf{f}'_{k})(1-\mathsf{g}'_{k-1})\ldots(1-\mathsf{g}'_{s+1})\mathsf{g}'_{s}
			=\frac{q^{\bar\nu_{k-1}-\la_{k}}-q^{\bar\nu_{k-1}-\bar\nu_{k}+1}}{1-q^{\bar\nu_{k-1}-\bar\nu_{k}+1}}
			q^{\sum_{i=s}^{k-2}(\bar\nu_{i}-\la_{i+1})}(1-q^{\bar\nu_{s-1}-\la_{s}});
		\end{align}
			
		\item If $s=r+1$, then $\la_s$ is chosen to move with probability
		\begin{align}\label{beta_col_prob3}
			(1-\mathsf{f}'_{k})(1-\mathsf{g}'_{k-1})\ldots(1-\mathsf{g}'_{r+3})(1-\mathsf{g}'_{r+2})
			=\frac{q^{\bar\nu_{k-1}-\la_{k}}-q^{\bar\nu_{k-1}-\bar\nu_{k}+1}}{1-q^{\bar\nu_{k-1}-\bar\nu_{k}+1}}
			q^{\sum_{i=r+1}^{k-2}(\bar\nu_{i}-\la_{i+1})}.
		\end{align}
	\end{itemize}
	Clearly, these probabilities are nonnegative, and 
	their sum telescopes to 1.

	\item If $V_j=1$, then in addition to the moves described above,
	exactly one more particle at level $j$
	is chosen to move (to the right by one).
	Namely, let $\bar\la_{m}$
	be the leftmost moved particle at level $j-1$.
	If $m=j-1$, then 
	the additional moving particle at level $j$ is $\la_j$, the leftmost particle.
	If $m<j-1$, then one of the particles $\la_s$
	with $m+1\le s\le j$ is randomly chosen to move (independently of everything else)
	with the following probabilities:
	\begin{itemize}
		\item If $s=j$, then $\la_s$ is chosen to move with probability
		\begin{align}\label{beta_col_prob4}
			\mathsf{g}'_{j}=1-q^{\bar\nu_{j-1}-\la_j};
		\end{align}
		\item If $m+1<s<j$, then $\la_s$ is chosen to move with probability
		\begin{align}\label{beta_col_prob5}
			(1-\mathsf{g}'_{j})(1-\mathsf{g}'_{j-1})\ldots(1-\mathsf{g}'_{s+1})\mathsf{g}'_{s}=
			(1-q^{\bar\nu_{s-1}-\la_s})q^{\sum_{i=s}^{j-1}(\bar\nu_{i}-\la_{i+1})};
		\end{align}
		\item If $s=m+1$, then $\la_s$ is chosen to move with probability
		\begin{align}\label{beta_col_prob6}
			(1-\mathsf{g}'_{j})(1-\mathsf{g}'_{j-1})\ldots(1-\mathsf{g}'_{m+3})(1-\mathsf{g}'_{m+2})=
			q^{\sum_{i=m+1}^{j-1}(\bar\nu_{i}-\la_{i+1})}.
		\end{align}
	\end{itemize}
	The sum of these probabilities also telescopes to 1.
\end{enumerate}
This completes the description of the $(\hat\be)$ column insertion RSK-type dynamics 
$\Qqbcol$.

\begin{theorem}\label{thm:QqcRSK_beta}
	The dynamics $\Qqbcol$
	defined above satisfies the main equations
	\eqref{beta_main_equation},
	and hence
	preserves the class of $q$-Whittaker processes
	and adds a new dual parameter $\be$
	to the specialization $\mathbf{A}$
	as in \eqref{adding_spec}.
\end{theorem}
\begin{proof}
	One can readily check that
	$\Qqbcol$ is the complement of 
	$\Qqbrow$. Then the desired statement 
	follows from 
	Theorem \ref{thm:QqrRSK_beta}
	and
	Proposition 
	\ref{prop:complementation_dynamics_is_good}.
\end{proof}

\begin{remark}\label{rmk:col_automatically_push_pull}
	Similarly to $\Qqbrow$ (cf. Remark \ref{rmk:automatically_push_pull}), 
	probabilities \eqref{beta_col_prob1}--\eqref{beta_col_prob6}
	employed in the definition of 
	$\Qqbcol$
	ensure the
	mandatory pushing, blocking, and move donation mechanisms
	(described in Definitions \ref{def:pull}
	and \ref{def:push} and Remark \ref{rmk:move_donation}). Namely, observe that
	\begin{itemize}
		\item 
		If $\bar\la_{k}=\la_k$ for some $k$
		and $\bar\la_{k}$ has moved at level $j-1$,
		then $\mathsf{f}'_{k}=1$,
		which means that $\la_{k}$ 
		is chosen to move with probability 1.

		\item If $\la_{s}=\bar\la_{s-1}$, and $\bar\la_{s-1}$
		has not moved, then $\mathsf{g}'_{s}=\mathsf{f}'_{s}=0$,
		so according to \eqref{beta_col_prob1}, \eqref{beta_col_prob2} the particle $\la_s$ at level $j$
		cannot be chosen to move.
		If, moreover, $\bar\la_{s}$ has moved at level $j-1$,
		then this move will trigger some other particle to the 
		right of $\la_s$ at level $j$ to move.
		In other words, the moving impulse 
		coming from $\bar\la_s\to\bar\nu_s=\bar\la_s+1$
		will be donated further to the right of $\la_s$.
	\end{itemize}
\end{remark}

\begin{remark}[Schur degeneration]\label{rmk:col_schur_degeneration}
	When $q=0$, one readily sees 
	from \eqref{f_g_prime_definition}
	that
	generically (i.e., when particles
	at levels $j-1$ and $j$ are sufficiently far apart) 
	we have
	$\mathsf{f}'_{k}=\mathsf{g}'_{s}=1$.
	This implies that
	the dynamics 
	$\Qqbcol$
	degenerates to the multivariate dynamics
	$\Qbcol$
	on Schur processes.
	The latter is
	based on the classical Robinson--Schensted--Knuth
	column insertion (\S \ref{sub:rsk_type_dynamics_schur}).
\end{remark}

\subsection{Bernoulli $q$-TASEP} 
\label{sub:bernoulli_tasep}

Under the dynamics 
$\Qqbcol$, 
the leftmost $N$ particles
$\la^{(j)}_{j}$
of the interlacing array
evolve in a \emph{marginally Markovian manner}.
Indeed, one can readily check that
at each discrete time step $t\to t+1$
the bottommost particle
is updated as
$\la^{(1)}_{1}(t+1)=\la^{(1)}_{1}(t)+V_1$, and for any $j=2,\ldots,N$:
\begin{itemize}
	\item If $\la^{(j-1)}_{j-1}$ has moved, then 
	the leftmost particle at level $j$ is updated as
	\begin{align*}
		\la^{(j)}_{j}(t+1)=\la^{(j)}_{j}(t)+V_j;
	\end{align*}
	\item
	If $\la^{(j-1)}_{j-1}$ has not moved, then the same
	particle is updated as
	\begin{align*}
		\la^{(j)}_{j}(t+1)=\la^{(j)}_{j}(t)+V_j\cdot\mathbf{1}_{\text{$\la^{(j)}_{j}$ is chosen to move}},
	\end{align*}
	where $\la^{(j)}_{j}$ is chosen to move
	with probability 
	$\mathsf{g}_j'=1-q^{\la^{(j-1)}_{j-1}(t)-\la^{(j)}_{j}(t)}$
	which depends only on the leftmost particles
	of the array.
\end{itemize}
This evolution of the 
leftmost particles 
$\la^{(j)}_{j}$,
$1\le j\le N$,
is the 
(\emph{discrete time}) 
\emph{Bernoulli \mbox{$q$-TASEP}}
which was introduced and studied in 
\cite{BorodinCorwin2013discrete}.



\subsection{Small $\be$ continuous time limit} 
\label{sub:small_be_continuous_time_limit}

If one sends the parameter $\be$
to zero and simultaneously 
rescales time from discrete to continuous,
then both dynamics
$\Qqbrow$ and $\Qqbcol$
turn into certain
continuous time Markov dynamics
on $q$-Whittaker processes.
At the level $j=1$ (cf. Remark \ref{rmk:one_level_dynamics}), 
this limit transition coincides with the 
one bringing the (one-sided) discrete time random walk to the 
continuous time Poisson process.
In continuous time setting, 
at most one particle can move at each level $j=1,\ldots,N$
during an instance of continuous time. 

\medskip

The continuous time limit of $\Qqbrow$ looks as follows.
Each rightmost particle $\la^{(j)}_{1}$
of the interlacing array has an independent 
exponential clock with rate $a_j$.
When the clock rings, the particle jumps to the right by one.

There is also a jump propagation mechanism present:
If at level $j-1$ some particle 
$\la^{(j-1)}_{m}$
has moved (to the right by one), then this move 
instantaneously triggers the move 
of the upper left neighbor $\la^{(j)}_{m+1}$
with probability $\mathsf{f}_{m}=
\frac{1-q^{\la^{(j)}_{m}-\la^{(j-1)}_m}}
{1-q^{\la^{(j-1)}_{m-1}-\la^{(j-1)}_m}}$,\footnote{Note 
that here we are rewriting $\mathsf{f}_m$ through the 
particle coordinates before the 
move at level $j-1$ (cf. \eqref{f_g_definition})
so that the probabilistic meaning is clearer. 
One could also similarly rewrite 
all other formulas for $\mathsf{f}_m,
\mathsf{g}_m,\mathsf{f}_m',\mathsf{g}_m'$
above in this section, but for the purposes of checking
the main equations of Theorem \ref{thm:main_eq}
it is more convenient to use the notation 
involving the coordinates $\bar\nu$ after the jump.} 
or the move of the upper right neighbor $\la^{(j)}_{m}$
with the complementary probability $1-\mathsf{f}_{m}$.
This dynamics was introduced
in \cite{BorodinPetrov2013NN} (Dynamics 8). 
Under it, the rightmost particles of the array 
also evolve in a marginally Markovian manner.
This leads to the \emph{continuous time $q$-PushTASEP} on $\Z$
\cite[\S 8.3]{BorodinPetrov2013NN},
\cite{CorwinPetrov2013}.

\medskip

The continuous time limit of $\Qqbcol$ looks as follows.
Each particle $\la_k$, $1\le k \le j$,
at level $j$ has an independent exponential clock with
rate
\begin{align*}
	\begin{cases}
		a_j\mathsf{g}'_j,&\hspace{15pt}k=j;\\
		a_j(1-\mathsf{g}'_j)(1-\mathsf{g}'_{j-1})\ldots(1-\mathsf{g}'_{k+1})\mathsf{g}'_{k}
		,&\hspace{15pt}1<k<j;\\
		a_j(1-\mathsf{g}'_j)(1-\mathsf{g}'_{j-1})\ldots(1-\mathsf{g}'_{3})(1-\mathsf{g}'_{2}),&\hspace{15pt}k=1.
	\end{cases}
\end{align*}
These quantities correspond to
\eqref{beta_col_prob4}--\eqref{beta_col_prob6} 
with $\bar\nu=\bar\la$ (because if an independent jump occurs at level $j$
then at level $j-1$ there could be no movement).
When the clock of $\la_k$ rings, this particle jumps to the right by one.
Note that the move donation mechanism described in Remark \ref{rmk:move_donation}
follows from the above probabilities.

There is also a jump propagation mechanism: If a particle
$\bar\la_{k}$ has moved at level $j-1$,
then it triggers the move (to the right by one) of exactly one
particle $\la_s$, $1\le s\le k$, at level $j$,
where $s$ is chosen at random with probabilities
\begin{align*}
	\begin{cases}
		\mathsf{f}'_{k},
		&\hspace{15pt}s=k;\\
		(1-\mathsf{f}'_{k})(1-\mathsf{g}'_{k-1})\ldots(1-\mathsf{g}'_{s+1})\mathsf{g}'_{s},
		&\hspace{15pt}1<s<k;\\
		(1-\mathsf{f}'_{k})(1-\mathsf{g}'_{k-1})\ldots(1-\mathsf{g}'_{3})(1-\mathsf{g}'_{2}),
		&\hspace{15pt}s=1.
	\end{cases}
\end{align*}
The above probabilities are given by 
\eqref{beta_col_prob1}--\eqref{beta_col_prob3} where
$\bar\nu$ differs from $\bar\la$
as $\bar\nu=\bar\la+\bar{\mathrm{e}}_{k}$.
This dynamics on $q$-Whittaker processes was introduced in 
\cite{OConnellPei2012}. Under it, the 
leftmost particles of the interlacing array
evolve in a marginally Markovian manner
as a $q$-TASEP. 
This continuous time particle system
was introduced in
\cite{BorodinCorwin2011Macdonald}.
See also, e.g.,
\cite{BorodinCorwinSasamoto2012}, 
\cite{BorodinCorwinPetrovSasamoto2013},
\cite{FerrariVeto2013} for further results on the $q$-TASEP.
 
\medskip

Thus, the two continuous time dynamics on $q$-Whittaker processes 
(or, in other words, $q$-randomized Robinson--Schensted correspondences)
introduced in
\cite{OConnellPei2012} and \cite{BorodinPetrov2013NN}
are the $\be\to0$ degenerations 
of $\Qqbcol$ and $\Qqbrow$, respectively.
On the other hand, 
complementation (\S \ref{sub:Complementation}) provides a 
straightforward link between the 
two latter discrete time dynamics.



\section{RSK-type dynamics 
$\Qqarow$ 
and $\Qqacol$ 
adding a usual parameter} 
\label{sec:geometric_q_rsks}

In this section we explain the construction of 
two RSK-type dynamics $\Qqarow$ and $\Qqacol$ on $q$-Whittaker processes
adding a usual parameter $\al$ to the specialization 
(as in \eqref{adding_spec}).
For $q=0$, these dynamics degenerate to 
$(\al)$ dynamics on Schur processes 
arising from row and column RSK insertion. 
As in the case of $\Qqbrow$ and $\Qqbcol$ dynamics, 
in a small $\al$ limit 
the dynamics 
$\Qqarow$ and $\Qqacol$
degenerate to continuous time 
RSK-type dynamics from \cite{OConnellPei2012} (column version) and \cite{BorodinPetrov2013NN} (row version).

\subsection{The $q$-deformed Beta-binomial distribution} 
\label{sub:the_q_deformed_binomial_distribution}

We will use the following quantities:

\begin{definition}
	Let $y\in\{0,1,2,\ldots\}\cup\{+\infty\}$, and $s\in\{0,1,\ldots,y\}$.
	Recall the $q$-notation from \eqref{q_notation}. Let
	\begin{align}\label{Om_qmunu_definition}
		\Om_{q,\muq,\eta}(s\mid y):=\muq^{s}\frac{(\eta/\muq;q)_{s}(\muq;q)_{y-s}}{(\eta;q)_{y}}\frac{(q;q)_{y}}{(q;q)_{s}(q;q)_{y-s}}.
	\end{align}
	If $y=+\infty$, the limits of the above quantities are
	\begin{align}\label{Om_qmunu_limit}
		\Om_{q,\muq,\eta}(s\mid +\infty)=
		\muq^{s}\frac{(\eta/\muq;q)_{s}}{(q;q)_{s}}
		\frac{(\muq;q)_{\infty}}{(\eta;q)_{\infty}}.
	\end{align}
\end{definition}
An important property of the quantities \eqref{Om_qmunu_definition} and \eqref{Om_qmunu_limit}
is that
for all $y\in\{0,1,2,\ldots\}\cup\{+\infty\}$, we have
\begin{align} \label{Om_qmunu_sum}
\sum_{s=0}^{y} \Om_{q,\muq,\eta}(s\mid y)=1.
\end{align}
This statement may be rewritten as the $q$-Chu-Vandermonde identity for the basic 
hypergeometric series $_{2}\varphi_{1}$. For the proof and more 
details see \cite{GasperRahman}, 
\cite{Corwin2014qmunu}. Recall that in general the unilateral basic hypergeometric series $_{j}\varphi_{k}$ is defined via
\begin{align}
_{j}\varphi_{k} \left[\begin{array}{ccc} a_{1} & \ldots & a_{j} \\ b_{1} & \ldots & b_{k} \end{array}; q, z \right] := \sum_{n=0}^{\infty} 
\frac{(a_{1}, \ldots, a_{j}; q)_{n}}{(b_{1}, \ldots, b_{k}, q; q)_{n}} \left((-1)^{n}q^{\binom {n}{2}}\right)^{1+k-j} z^{n}, 
\end{align}
where $(c_{1}, \ldots, c_{m}; q)_{n} = \prod_{i=1}^{m} (c_{i}; q)_{n}$.
Later on in this section to prove some identities we will need to apply transformation formulas for certain hypergeometric series. 

Therefore, for all values of the parameters $(q,\muq,\eta)$
for which $\Om_{q,\muq,\eta}(s\mid y)$
is well-defined and nonnegative for every $0 \leq y \leq s$,
\eqref{Om_qmunu_definition} 
defines a probability distribution on $\{0, 1, \ldots, y\}$. 
One such family of parameters is $0\le q<1$, $0\le \eta\le\muq<1$, cf. \cite{Povolotsky2013},
\cite{Corwin2014qmunu}.
Another choice of parameters leading to a probability distribution 
which we will use
is $\Om_{q^{-1},q^{a},q^{b}}(\cdot\mid c)$, where $a\le b$, $c\le b$ are nonnegative integers.

The distribution $\Om_{q,\muq,\eta}$ 
appears (under a simple change of parameters) 
as the orthogonality weight of the classical
$q$-Hahn orthogonal polynomials \cite[\S3.6]{Koekoek1996},
and is also related to a very natural $q$-deformation
of the Polya urn scheme \cite{Gnedin2009}.
As such, $\Om_{q,\muq,\eta}$ 
may be regarded as a
\emph{$q$-deformed Beta-binomial distribution},
since the latter is the orthogonality weight for the Hahn 
orthogonal polynomials, and also arises from the ordinary 
Polya urn scheme. We can also directly see by taking $q=e^{-\epsilon}, \muq=e^{-\al\epsilon}, \eta=e^{-(\al+\be)\epsilon}$ and letting $\epsilon \to 0+$, that $\Om_{q,\muq,\eta}(s\mid y)$ converges to 
\begin{align*}
\frac{\Gamma(\al+\be)}{\Gamma{(\al)}\Gamma(\be)}
\frac{\Gamma(\al+y-s)\Gamma(\be+s)}
{\Gamma(\al+\be+y)}\binom ys,
\end{align*}
which is the probability of $s$ under the beta-binomial distribution with parameters $y, \al, \be$.

\smallskip

Let us now record two straightforward observations which we will be using below. 
First,
\begin{align}\label{qqinverse}
	\binom nk_{q^{-1}} = q^{-k(n-k)} \binom nk_{q}.
\end{align}
Second, if $a\le b$, $c\le b$ are nonnegative integers ($b$ might also be $+\infty$),
then for any
$s\in\{0,1,\ldots,c\}$ one has
\begin{align}\label{q=0_Omega}
	\lim_{q\searrow0}\Om_{q^{-1},q^{a},q^{b}}(s\mid c)=\mathbf{1}_{s=\max\{c-a,0\}}.
\end{align}
Indeed,  in this case 
\begin{align*}
\Om_{q^{-1},q^{a},q^{b}}(s\mid c) = q^{s(a-c+s)}\frac{(q^{a}; q^{-1})_{c-s}(q^{b-a};q^{-1})_{s}}{(q^{b};q^{-1})_{c}}\binom sc_{q}.
\end{align*}
If $a > c$, as $q \to 0$ this converges to $1$ for $s=0$ and to $0$ for $s > 0$. If $a \le c$, as $q \to 0$ this converges to $0$ for $0 \le s < c-a$, since $(q^{a}; q^{-1})_{c-s}$ vanishes,  to $0$ for $s > c-a$, since a positive power of $q$ tends to $0$, and to $1$ for $s=c-a$.
	

\subsection{Row insertion dynamics $\Qqarow$} 
\label{sub:row_insertion_dynamics_q__hatal_q-rrsk_}
Let us now describe one time step $\boldsymbol\la \to \boldsymbol\nu$ of the multivariate Markov dynamics $\Qqarow$ on $q$-Whittaker processes of depth $N$. A part of randomness a time step comes from independent $q$-geometric random variables $V_1,\ldots,V_N\in \Z_{\ge0}$ with parameters $\al a_1,\ldots,\al a_N$, respectively (these random variables are resampled during each time step).

The bottommost particle of the interlacing array is updated as $\nu_{1}^{(1)} = \la_{1}^{(1)} + V_{1}$.  Next, sequentially for each $j=2,\ldots,N$, given the movement $\bar\la = \la^{(j-1)}\to \bar\nu = \nu^{(j-1)}$ at level $j-1$,
we will randomly update $\la = \la^{(j)} \to\nu = \nu^{(j)}$ at level $j$.
To describe this update, write 
\begin{align*}
	\bar\nu-\bar\la=\sum_{i=1}^{j-1}c_i \bar{\mathrm{e}}_{i},
	\qquad c_i\in\Z_{\ge0},\qquad
	\text{$\bar{\mathrm{e}}_{i}$ are basis vectors of length $j-1$}.
\end{align*}
Note that by interlacing, it must be that $c_i\le\bar \la_{i-1} - \bar \la_{i}$.

Sample independent random variables $W_{1}, \ldots, W_{j-1}$, 
such that each $W_{i}\in\{0,1,\ldots,c_i\}$ is distributed according to 
\begin{align}
	\Om_{q^{-1},\muq_{i},\eta_{i}}(\cdot\mid c_{i}),
	\quad \text{where $\muq_{i}: = q^{\la_{i} - \bar \la_{i}}$ 
	and  $\eta_{i}: = q^{\bar \la_{i-1} - \bar \la_{i}}$}
	\label{splitting_distribution}
\end{align}
(this is a probability distribution because $\la_{i} - \bar \la_{i}\le 
\bar \la_{i-1} - \bar \la_{i}$ and $c_i\le \bar \la_{i-1} - \bar \la_{i}$, 
cf. \S \ref{sub:the_q_deformed_binomial_distribution}).
We will use the conventions $\bar \la_{0} = +\infty$ and  $\eta_{1} = 0$. Define a sequence of signatures 
\begin{align*}
	\la=\mu({0}), \mu({1}), \ldots, \mu({j-1})
\end{align*}
via
\begin{align*} 
 & \mu({i}) := 
 \mu({i-1}) + W_{j-i}\mathrm{e}_{j-i} + (c_{j-i} - W_{j-i})\mathrm{e}_{j-i+1} \quad 
 \text{for $1 \leq i \leq j-1$}
\end{align*}
(where $\mathrm{e}_{i}$ are basis vectors of length $j$).
Finally, define $\nu := \mu({j-1}) + V_{j}e_{1}$, this is our new 
signature at level $j$.

In words, each $i$th particle on the $(j-1)$-st level 
which has moved by $c_i$, must trigger a total of $c_i$
moves (to the right by one) at level $j$ (this the RSK-type property,
see Definition \ref{def:rsk}). Each such particle at level $j-1$
independently from the others, \emph{in parallel}, splits contribution from its jump between its nearest neighbors on the level $j$, according to 
the distribution \eqref{splitting_distribution}. After this pushing, the rightmost particle on the $j$-th level additionally performs an independent jump according to the $q$-geometric distribution with parameter $\al a_{j}$. Clearly, thus defined conditional probabilities $\mathscr{U}_j$, $j=1,\ldots,N$, for this dynamics are nonnegative and satisfy \eqref{U_properties}. See Fig.~\ref{fig:row_alpha}.

One must verify that the interlacing properties (as on
Fig.~\ref{fig:square}, left) are preserved by this dynamics:
\begin{lemma} \label{lemma:interlacing_after}
	If $\bar\la \prech \bar \nu$, $\bar\la\prech \la$ 
	and $\mathscr{U}_j(\la \to \nu \mid \bar \la \to \bar \nu) > 0$, 
	then $\bar \nu\prech\nu$ and $\la \prech \nu$.
\end{lemma}
\begin{proof}
	Observe that for $a \le b$ and $c \le b$
	\begin{align} \label{phinequalities} 
	\Om_{q^{-1},q^{a}, q^{b}}(s \mid c) = 0, \quad \text{if $s > b-a$ or $c-s > a$}.
	\end{align}  
	Apply this for $a=\la_{i}-\bar \la_{i}$, $b = \bar \la_{i-1} - \bar \la_{i}$, $c = c_{i}$ to get $c_i-\la_i+\bar\la_i\le W_i\le \bar\la_{i-1}-\la_i$.

	Since $\nu_i=\la_i+W_i+c_{i-1}-W_{i-1}$, we have
	\begin{align*}
		\nu_i\le \la_i+\bar\la_{i-1}-\la_i+c_{i-1}-W_{i-1}=\bar\la_{i-1}+c_{i-1}-W_{i-1}
		=\bar\nu_{i-1}-W_{i-1}\le \bar\nu_{i-1}
	\end{align*}
	and $\bar\nu_{i} = \bar\la_{i}+c_{i} \le \la_{i}+W_{i} \le \nu_{i}$,
	so $\nu\succh\bar\nu$. Moreover,
	we can also write 
	\begin{align*}
		\nu_i\le \la_i+\bar\la_{i-1}-\la_i
		+\la_{i-1}-\bar\la_{i-1}=\la_{i-1}, \qquad \la_{i} \le \nu_{i}
	\end{align*}
	which implies that 
	$\nu\succh\la$.
\end{proof}
This verification completes the description of 
the $(\al)$ row insertion RSK-type dynamics 
$\Qqarow$.

\begin{remark} (Schur degeneration).
If one sets $q=0$, then the dynamics $\Qqarow$ reduces to the dynamics $\Qarow$ on Schur processes based on the classical 
Robinson--Schensted--Knuth row insertion 
(\S \ref{sub:rsk_type_dynamics_schur}). To see this,
observe that \eqref{q=0_Omega} implies 
\begin{align*}
\Om_{q^{-1},\muq_{i}, \eta_{i}}(s \mid c_{i}) \to \mathbf{1}_{s = \max\{c_{i} -\la_{i} + \bar \la_{i}, 0\}} \quad \text{as $q \to 0$},
\end{align*}
that is, each $W_i$ becomes equal to 
$\max\{c_{i} -\la_{i} + \bar \la_{i}, 0\}$ in the $q\searrow0$ limit.
Therefore, the update $\la\to\nu$
is reduced to applying $c_i$ operations $\mathsf{pull}$ at positions $i$ from $j-1$ to $1$, 
plus 
an
additional independent jump 
of the 
rightmost particle according to the geometric distribution with parameter $\al a_{j}$.
\end{remark}

\setcounter{la1}{35}
\setcounter{la2}{25}
\setcounter{la3}{15}
\setcounter{la4}{9}
\setcounter{la5}{6}
\setcounter{la6}{0}
\setcounter{lab1}{29}
\setcounter{lab2}{19}
\setcounter{lab3}{12}
\setcounter{lab4}{8}
\setcounter{lab5}{2}
\begin{figure}[htbp]
	{\begin{adjustbox}{max width=.86\textwidth}
	\begin{tikzpicture}
	[scale=.7, thick]
	\def\cir{.2}
	\def\ysh{6}
	\def\x{1.8}
	\def\y{2.2}
	\draw[fill] (-2*\x,0) circle(\cir) node [below,yshift=-\ysh] {$\arabic{lab5}+{\color{blue}3}$};
	\draw[fill] (0,0) circle(\cir) node [below,yshift=-\ysh] {$\arabic{lab4}$};
	\draw[fill] (2*\x,0) circle(\cir) node [below,yshift=-\ysh] {$\arabic{lab3}+{\color{blue}4}$};
	\draw[fill] (4*\x,0) circle(\cir) node [below,yshift=-\ysh] {$\arabic{lab2} +{\color{blue}5}$};
	\draw[fill] (6*\x,0) circle(\cir) node [below,yshift=-\ysh] {$\arabic{lab1}+{\color{blue}4}$};
	\draw[fill] (-3*\x,1*\y) circle(\cir) node [above,yshift=\ysh] {$\arabic{la6}+{\color{blue}2}$};
	\draw[fill] (-1*\x,1*\y) circle(\cir) node [above,yshift=\ysh] {$\arabic{la5}+{\color{blue}1}$};
	\draw[fill] (1*\x,1*\y) circle(\cir) node [above,yshift=\ysh] {$\arabic{la4}+{\color{blue}2}$};
	\draw[fill] (3*\x,1*\y) circle(\cir) node [above,yshift=\ysh] {$\arabic{la3}+{\color{blue}5}$};
	\draw[fill] (5*\x,1*\y) circle(\cir) node [above,yshift=\ysh] {$\arabic{la2}+{\color{blue}5}$};
	\draw[fill] (7*\x,1*\y) circle(\cir) node [above,yshift=-2.2+\ysh] {$\arabic{la1}+{\color{blue}1}+{V_j}$};
	\node at (9.5*\x,0) {$\bar\la+({\color{blue}\bar\nu-\bar\la})$};
	\node at (9.5*\x,\y) {$\la+({\color{blue}\nu-\la})$};
	\node (lab5) at (-2*\x,0) {};
	\node (lab4) at (0*\x,0) {};
	\node (lab3) at (2*\x,0) {};
	\node (lab2) at (4*\x,0) {};
	\node (lab1) at (6*\x,0) {};
	\node (la6) at (-3*\x,\y) {};
	\node (la5) at (-1*\x,\y) {};
	\node (la4) at (1*\x,\y) {};
	\node (la3) at (3*\x,\y) {};
	\node (la2) at (5*\x,\y) {};
	\node (la1) at (7*\x,\y) {};
	\draw[->,very thick] (lab5) -- (la6) node[rectangle,draw=black,fill=gray!20!white] [midway] {$+2$};
           \draw[->,very thick] (lab5) -- (la5) node[rectangle,draw=black,fill=gray!20!white] [midway] {$+1$};
	\draw[->,very thick] (lab3) -- (la4) node[rectangle,draw=black,fill=gray!20!white] [midway] {$+2$};
           \draw[->,very thick] (lab3) -- (la3) node[rectangle,draw=black,fill=gray!20!white] [midway] {$+2$};
	\draw[->,very thick] (lab2) -- (la3) node[rectangle,draw=black,fill=gray!20!white] [midway] {$+3$};
           \draw[->,very thick] (lab2) -- (la2) node[rectangle,draw=black,fill=gray!20!white] [midway] {$+2$};
           \draw[->,very thick] (lab1) -- (la2) node[rectangle,draw=black,fill=gray!20!white] [midway] {$+3$};
           \draw[->,very thick] (lab1) -- (la1) node[rectangle,draw=black,fill=gray!20!white] [midway] {$+1$};
	\end{tikzpicture}
	\end{adjustbox}}
	\caption{An example of a step of 
	$\Qqarow$ at levels 5 and 6, with
	$V_{j}=3$. The probability of this update is equal to
    ${\Om_{q^{-1},q^{4}, q^{6}}(1 \mid 3)}{\Om_{q^{-1},q^{3}, q^{7}}(2 \mid 4)}{\Om_{q^{-1},q^{6}, q^{10}}(2 \mid  5)}\Om_{q^{-1},q^{6}, 0}(1 \mid 4)(\al a_{6}; q)_{\infty}\frac{(\al a_{6})^{3}}{(q;q)_{3}}.$
    Note that, e.g., ${\Om_{q^{-1},q^{3}, q^{7}}(0 \mid 4)}=0$,
    which ensures the mandatory pushing (by at least $1$) of $\la_3$ by the move
    of $\bar\la_3$.}
	\label{fig:row_alpha}
\end{figure}
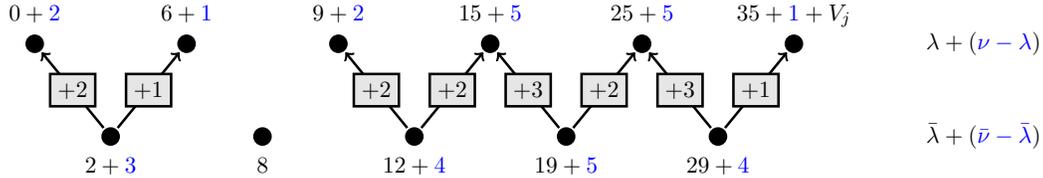

\begin{theorem}\label{thm:QqrRSK_alpha}
The dynamics $\Qqarow$ defined above satisfies the main equations \eqref{alpha_main_equation}, and hence preserves the class of $q$-Whittaker processes and adds a new usual parameter $\al$ to the specialization $\mathbf{A}$ as in \eqref{adding_spec}.
\end{theorem}
\begin{proof} 
We will prove \eqref{alpha_main_equation}  by induction on $j$. Case $j=1$ is straightforward because $\bar\la$ is empty (cf. Remark \ref{rmk:one_level_dynamics}).

Assume now that \eqref{alpha_main_equation} holds for signatures 
$\la,\nu$ having length $j-1$, and let us prove 
this identity for $\la,\nu$ or length $j$.
The idea is to expand each term in the sum in the left-hand side of \eqref{alpha_main_equation} with respect to what happens to the 
leftmost particle on the $(j-1)$-st level (and its neighborhood), 
and then use the inductive assumption and the fact that the $\Om$'s sum to $1$. 

For a signature $\mu = (\mu_{1} \geq \ldots \geq \mu_{m})$ we denote by $\mu^{-}$ the signature $(\mu_{1} \geq \ldots \geq \mu_{m-1})$ obtained by deleting the smallest  part of $\mu$, 
and by $\mu + [s]_{\textnormal{lm}}$ the signature $(\mu_{1} \geq \ldots \geq \mu_{m-1} \geq \mu_{m} + s)$ obtained by adding $s$ to the smallest part of $\mu$ (for $s \leq \mu_{m-1} - \mu_{m}$).
To simplify certain notations below, also denote
\begin{align}\label{Vj_notation}
\mathscr{V}_{j}(\la \to \nu \mid \bar \la \to \bar \nu): = \mathscr{U}_{j}(\la \to \nu \mid \bar \la \to \bar \nu)\frac{(\al a_j)^{|\la| - |\bar \la| - |\nu| + |\bar \nu|}}{(\al a_j; q)_{\infty}}.
\end{align}

Temporarily let $t$ stand for $c_{j-1} = \bar \nu_{j-1} - \bar \la_{j-1}$ which is the move of the leftmost particle on the $(j-1)$-st level.  
In order to have at least one nonzero summand in the left-hand side of \eqref{alpha_main_equation},
we need to have (see Fig.~\ref{fig:arp}): 
\begin{itemize}
	\item 
$t \geq \nu_{j} - \la_{j}$, since the jump of the leftmost particle on the $j$-th level happens due to contribution of a part of the jump of the leftmost particle on the $(j-1)$-st level.  

\item 
$t \geq \nu_{j} - \la_{j} + \bar \nu_{j-1} - \la_{j-1}$, 
since  $\Om_{q^{-1},\muq_{j-1}, \eta_{j-1}}(t - \nu_{j} + \la_{j} \mid t) > 0$ implies by \eqref{phinequalities} that $\nu_{j} - \la_{j} \leq \la_{j-1} - \bar \nu_{j-1} + t$.

\item $t \leq \bar \nu_{j-1} - \la_{j}$, since we must have $\bar \la_{j-1} \geq \la_{j}$.

\item $t \leq \nu_{j-1} - \la_{j-1} + \nu_{j} - \la_{j}$, since the contribution from the jump of the leftmost particle on the $(j-1)$-st level is split between particles $\la_{j}$
and $\la_{j-1}$ at the level $j$.
\end{itemize}
Denote the interval of $t$ satisfying the above inequalities by $I$.  We also must have
\begin{itemize}
	\item 
	$t \leq \bar \la_{j-2} - \la_{j-1} + \nu_{j} - \la_{j}$,  since  $\Om_{q^{-1},\muq_{j-1}, \eta_{j-1}}(t - \nu_{j} + \la_{j} \mid t) > 0$ implies by \eqref{phinequalities} 
	that 
	$t -\nu_{j} + \la_{j} \leq \bar \la_{j-2} -  \la_{j-1}$. 
	For $j=2$ this last inequality should be omitted.
\end{itemize}
We will use the notation $\tilde{\la}:= \bar \la^{-}$. 
Denote by $J(t)$ the set of signatures $\tilde{\la}$ of length $j-2$, 
such that  $\tilde{\la} \prech \bar \nu$, $\tilde\la\prech\la^{-}$, 
and $\tilde{\la}_{j-2} \geq t + \la_{j-1}-\nu_{j} + \la_{j}$. 
For $j=2$ this set consists of just the empty signature $\varnothing$. If $\bar \la$ is such that $\bar \la \prech \bar \nu$, $\bar \la \prech \la$ and $\mathscr{V}_{j}(\la \to \nu \mid \bar \la \to \bar \nu) \neq 0$, then $\bar \la^{-} \in \bigsqcup_{t \in I} J(t)$.

\begin{figure}[h] 

\hspace*{0 cm}\includegraphics[width = 0.7\textwidth]{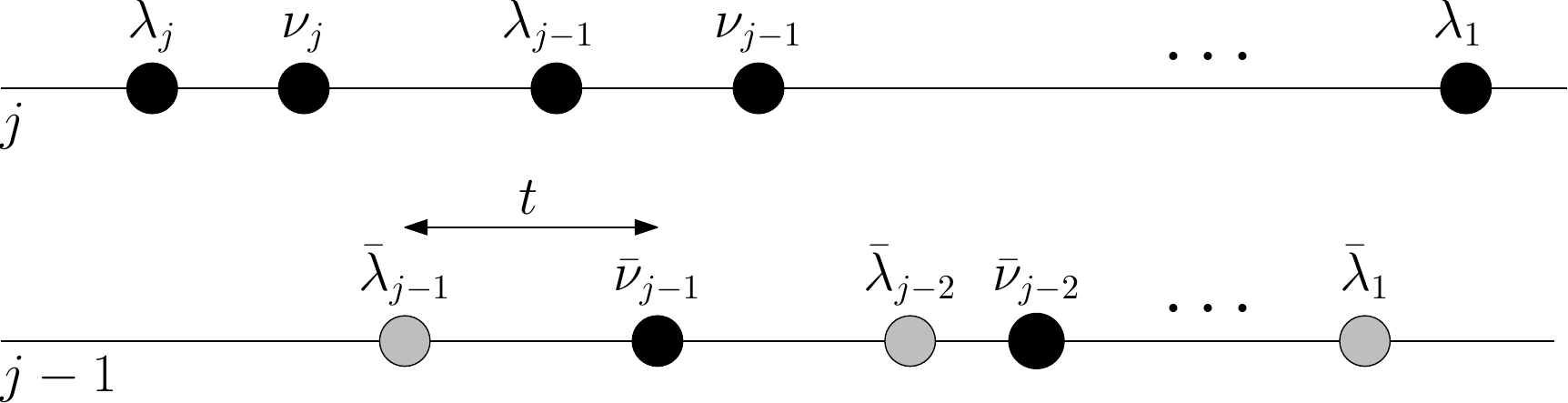}

\caption{We expand sum with respect to the 
jump $t = c_{j-1} =\bar \nu_{j-1} - \bar \la_{j-1}$ 
of the leftmost particle on the $(j-1)$-st level. 
Note that the signatures $\nu, \la, \bar \nu$ are fixed, 
while the positions of particles $\bar \la_{i}$ vary in the sum.} 

\label{fig:arp}
\end{figure}
The left-hand side of \eqref{alpha_main_equation} divided by the 
right-hand side of the same equation is equal to
\begin{align*}
& \sum_{\bar \la} \mathscr{V}_{j}(\la \to \nu \mid \bar \la \to \bar \nu) \frac{\psi_{\la / \bar\la} \varphi_{\bar \nu / \bar \la}}{\psi_{\nu / \bar \nu} \varphi _{\nu / \la}} 
\\&=\sum_{t \in I} \sum_{\tilde{\lambda} \in J(t)} \underbrace{\binom {t}{\nu_{j} - \la_{j}}_{q^{-1}}  \frac{(q^{\la_{j-1} - \bar \nu_{j-1} + t}; q^{-1})_{\nu_{j} - \la_{j}}(q^{\tilde{\la}_{j-2} - \la_{j-1}}; q^{-1})_{t - \nu_{j} + \lambda_{j}}}{(q^{\tilde{\la}_{j-2} - \bar \nu_{j-1}+t}; q^{-1})_{t}}q^{(\la_{j-1} - \bar \nu_{j-1} + t)(t - \nu_{j} + \la_{j})}}_{\Om_{q^{-1},q^{\tilde{\la}_{j-2}- \la_{j-1}}, q^{\tilde{\la}_{j-2}- \bar \nu_{j-1}+t}}(t - \nu_{j} + \la_{j} \mid t)} 
\\&\hspace{30pt} \times \mathscr{V}_{j-1}(\la^{-} + [t - \nu_{j} + \la_{j}]_{\textnormal{lm}} \to \nu^{-} \mid \tilde{\la} \to \bar \nu^{-}) 
\cdot\frac{\psi_{\la^{-} + [t - \nu_{j} + \la_{j}]_{\textnormal{lm}} / \tilde{\la}} \varphi_{\bar \nu^{-} / \tilde{\la}}}{\psi_{\nu^{-} / \bar\nu^{-}} 
\varphi _{\nu^{-} / \la^{-} + [t - \nu_{j} + \la_{j}]_{\textnormal{lm}}}}
\\&\hspace{30pt} \times \underbrace{
\frac{\binom {\la_{j-1} - \lambda_{j}}{\bar \nu_{j-1} - t - \la_{j}}_{q}}{\binom {\nu_{j-1} - \nu_{j}}{\bar \nu_{j-1} - \nu_{j}}_{q}} 
\cdot \frac{\binom {\tilde{\la}_{j-2} - \bar \nu_{j-1} + t}{t}_{q}}{\binom {\la_{j-1} - \la_{j}} {\nu_{j} - \la_{j}}_{q}}\cdot \frac{(q^{\nu_{j-1} - \la_{j-1}}; q^{-1})_{t - \nu_{j} + \la_{j}}}{(q^{\tilde{\la}_{j-2} - \la_{j-1}}; q^{-1})_{t - \nu_{j} + \la_{j}}}}_{
\frac{\psi_{\la / \bar \la}\psi_{\nu^{-} / \bar \nu^{-}}}
{\psi_{\la^{-} + [t - \nu_{j} + \la_{j}]_{\textnormal{lm}}/\tilde{\la}}\psi_{\nu / \bar \nu}}
\cdot\frac{\varphi_{\bar \nu / \bar \la}
\varphi_{\nu^{-}/ \la^{-} + [t-\nu_{j}+\la_{j}]_{\textnormal{lm}}}}
{\varphi_{\nu / \la}\varphi_{\bar \nu^{-}/ \tilde{\la}}}}
\\&= \sum_{t \in I} \bigg(\Om_{q^{-1}, q^{\nu_{j-1} - \bar \nu_{j-1}}, q^{\nu_{j-1} - \nu_{j}}}(t - \nu_{j} + \la_{j} \mid \nu_{j-1} - \la_{j-1}) 
\\&\hspace{30pt}\times \sum_{\tilde{\lambda} \in J(t)} 
\mathscr{V}_{j-1}(\la^{-} + [t - \nu_{j} + \la_{j}]_{\textnormal{lm}} 
\to \nu^{-} \mid \tilde{\la} \to \bar \nu^{-}) \cdot
\frac{\psi_{\la^{-} + [t - \nu_{j} + \la_{j}]_{\textnormal{lm}} / \tilde{\la}} \varphi_{\bar \nu^{-} / \tilde{\la}}}
{\psi_{\nu^{-} / \bar\nu^{-}} \varphi _{\nu^{-} / \la^{-} + [t - \nu_{j} + \la_{j}]_{\textnormal{lm}}}}\bigg) 
\\&= \sum_{t \in I} \Om_{q^{-1}, q^{\nu_{j-1} - \bar \nu_{j-1}}, q^{\nu_{j-1} - \nu_{j}}}(t - \nu_{j} + \la_{j} \mid \nu_{j-1} - \la_{j-1}) = 1.
\end{align*}
Above $\mathscr{V}_{j-1}$ and $\mathscr{V}_{j}$ have the same value of the parameter $a=a_j$. 
We have also used the fact that 
\begin{align*}
 |\nu| - |\la|  - |\bar \nu| + |\bar \la| & = |\nu^{-}| - |\la^{-}| + \nu_{j} - \la_{j}  - |\bar \nu^{-}| + |\bar \la^{-}| - \bar \nu_{j-1} + \bar \la_{j-1} \\ & = 
 |\nu^{-}| - |\la^{-} + [t - \nu_{j} + \la_{j}]_{\textnormal{lm}}| 
 - |\bar \nu^{-}| + |\bar \la^{-}| ,
\end{align*}
hence $\mathscr{V}_{j}(\la \to \nu \mid \bar \la \to \bar \nu)$ involves the same power of $\al a_j$ as $\mathscr{V}_{j-1}(\la^{-} + [t - \nu_{j} + \la_{j}]_{\textnormal{lm}} \to \nu^{-} \mid \bar \la^{-} \to \bar \nu^{-})$. 
Also, \eqref{phinequalities} implies that  $\Om_{q^{-1}, q^{\nu_{j-1} - \bar \nu_{j-1}}, q^{\nu_{j-1} - \nu_{j}}}(t - \nu_{j} + \la_{j} \mid \nu_{j-1} - \la_{j-1})$ is nonzero only for $t \in I$, hence one gets $1$ after summing these quantities over $t \in I$.

This concludes the proof, and also establishes Theorem \ref{thm:sampling_intro}
from Introduction.
\end{proof}

\subsection{Geometric $q$-PushTASEP} 
\label{sub:geometric_q_pushtasep}

Under the dynamics 
$\Qqarow$
we have just constructed, 
the rightmost $N$ particles $\la^{(j)}_{1}$
of the interlacing array
evolve in a \emph{marginally Markovian manner}
(i.e., their evolution does not depend
on the dynamics of the rest of the interlacing
array). 
Namely, at each discrete time step $t\to t+1$
the bottommost particle
is updated as $\la^{(1)}_{1}(t+1)=\la^{(1)}_{1}(t)+V_{1}$, and for any $j=2,\ldots,N$ if we let $\gap_j(t)= \la^{(j)}_{1}(t) - \la^{(j-1)}_{1}(t)$  be the gap between the rightmost particles on the $(j-1)$-st and the $j$-th levels at time $t$, then
\begin{align*}
\la^{(j)}_{1}(t+1)=\la^{(j)}_{1}(t)+ V_{j}+W_{j, t}
\end{align*}
for an independent random variable $W_{j, t}$ 
distributed according to 
\begin{align*}
	\Om_{q^{-1}, q^{\gap_j(t)}, 0}\big(\cdot \mid \la^{(j-1)}_{1}(t+1) - \la^{(j-1)}_{1}(t)\big).
\end{align*}
The random variable $V_j$ (recall that it has the $q$-geometric distribution with parameter
$\al a_j$ which is resampled during each time step) 
represents an independent jump of $\la^{(j)}_{1}$.
The variable $W_{j, t}$ represents the pushing
of $\la^{(j)}_{1}$ by the move of 
$\la^{(j-1)}_{1}$.

This evolution of the 
rightmost particles 
$\la^{(j)}_{1}$,
$1\le j\le N$,
leads to a new interacting particle system
on $\Z$ which we call the (\emph{discrete time}) \emph{geometric \mbox{$q$-PushTASEP}}.


\subsection{Column insertion dynamics $\Qqacol$. Description and discussion} 
\label{sub:col_insertion_dynamics_q__hatal_q-rrsk_}

Let us now describe one time step $\boldsymbol\la \to \boldsymbol\nu$ of the multivariate Markov dynamics $\Qqacol$ on $q$-Whittaker processes of depth $N$. As in the previous case, the bottommost particle of the interlacing array is updated as $\nu_{1}^{(1)} = \la_{1}^{(1)} + X$ for a $q$-geometric random variable $X$ with parameter $\al a_{1}$. Next, sequentially for each $j=2,\ldots,N$, given the movement $\bar\la \to\bar\nu$ at level $j-1$, we will randomly update 
$\la\to\nu$ at level $j$. To describe this update we write, as usual, 
\begin{align*}
	\bar\nu-\bar\la=\sum_{i=1}^{j-1}c_i \bar{\mathrm{e}}_{i},
	\qquad c_i\in\Z_{\ge0}.
\end{align*}

All randomness during this update comes from a collection of $3j$ dependent random variables $X_{1}, \ldots, X_{j}, Y_{1}, \ldots, Y_{j}, Z_{1}, \ldots, Z_{j}$ (they are resampled during each time step), and 
\begin{align*}
	\nu_{j-i+1} - \la_{j-i+1} = \underbrace{X_i}_{\text{voluntary jump}}+
	\underbrace{Y_i}_{\text{push from $\bar\la_{j-i+1}$}}+\underbrace{Z_i}_{\text{push 
		from the ``stabilization fund''}},\qquad i=1,\ldots,j.
\end{align*}
(It will be convenient to let
$i$ represent the position of the particle counted from the left.)
Observe that $Y_1$ must be identically zero. The ``stabilization fund''
means the leftover push from the first $i-2$ particles from the left 
at level $j-1$ (i.e., from $\bar\la_{j-1},\ldots,\bar\la_{j-i+2}$)
(in particular, $Z_1$ and $Z_2$ are identically zero).






Let us first formally define the distribution of all the parts of the jumps:
\begin{enumerate}
\item[(1)] Set $\theta_{1}: =1$. For $i$ from $1$ to $j$ sample $X_{i}$ according to 
\begin{align}\label{Xi_distribution}
X_i\sim\Om_{q, \al a_{j}\theta_{i}, 0}(\cdot \mid \bar \la_{j-i} - \la_{j-i+1})
\end{align}
and set 
\begin{align*}
\theta_{i+1}:= \theta_{i}q^{\bar \la_{j-i} - \la_{j-i+1} - X_{i}}.
\end{align*} 
The jump $X_{i}$  comes from the input $V_{j}$, see Remark \ref{rmk:Xi_from_Vj} below. Here the convention $\bar \la_{0} = +\infty$ applies when $i=j$.

\item[(2)] Set $Y_{1} := 0$. For $i$ from $2$ to $j-1$ take $Y_{i} = y$ with probability 
\begin{align}\label{Yi_distribution}
Y_i\sim\Om_{q^{-1}, q^{c_{j-i+1}}, q^{\bar \la_{j-i} - \bar \la_{j-i+1}}}(\bar \la_{j-i} - \la_{j-i+1} - X_{i} - y \mid  \bar \la_{j-i} - \la_{j-i+1} - X_{i}).
\end{align}
Finally, set $Y_{j}:= c_{1}$.

\item[(3)] Set $r_{1}=r_{2}=1$ and $Z_{1} = Z_{2} := 0$. Set $r_{3}: =r_{2}q^{c_{j-1}-Y_{2}}$. For $i$ from $3$ to $j-1$ take $Z_{i} = z$ with probability
\begin{align}\label{Zi_distribution}
Z_i\sim\Om_{q^{-1}, r_{i}, 0}(\bar \la_{j-i} - \la_{j-i+1} - X_{i} - Y_{i} - z \mid  \bar \la_{j-i} - \la_{j-i+1} - X_{i} - Y_{i})
\end{align}
and set
\begin{align*}
r_{i+1}:= r_{i}q^{c_{j-i+1} - Y_{i} - Z_{i}}.
\end{align*}
Finally, let $Z_{j} := \log_{q} r_{j}$.
\end{enumerate}

\begin{remark}
For fixed $s, u, d \geq 0$ (possibly $u = \infty$) and $D \to \infty$ 
observe that 
$$\Om_{q^{-1}, q^{s}, q^{u+D}}(D-d \mid D) = q^{(s-d)(D-d)}(q^{s}; q^{-1})_{d} \binom {D}{d}_{q} \frac{(q^{u+D-s}; q^{-1})_{D-d}}{(q^{u+D}; q^{-1})_{D}} \to \textbf{1}_{d=s}.$$ 
Therefore, the definitions of
$Z_{j} = \log_{q} r_{j}$ 
and $Y_{j} = c_{1}$ 
are consistent
with the definitions of $Z_{i}$ and $Y_{i}$ ($i<j$), respectively.
In words, the consistency for $Z_j$
means that
the stabilization fund is depleted for 
the push of the rightmost particle 
on the $j$-th level.
The consistency for $Y_j$
means that 
the whole 
value of the jump of the rightmost particle on the $(j-1)$-st level is transferred to the 
rightmost particle on the $j$-th level via immediate pushing.
\end{remark}

\begin{figure}[h] 

\hspace*{0 cm} \includegraphics[width = 0.7\textwidth]{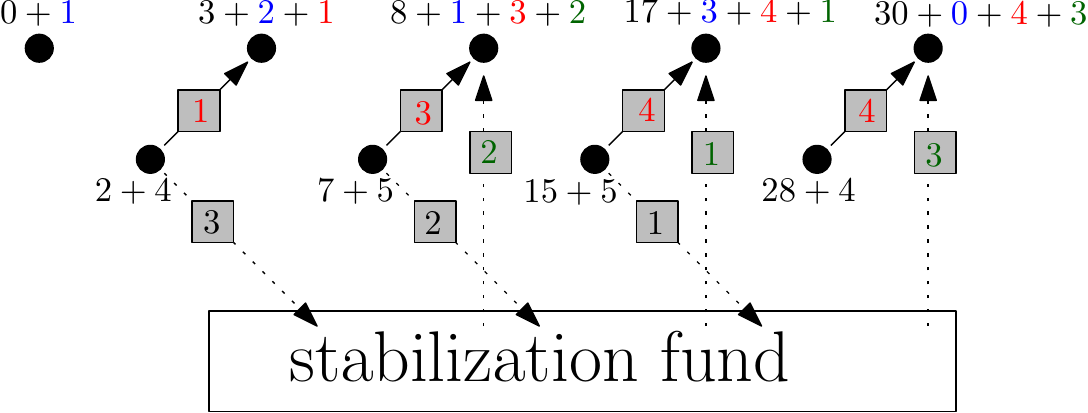}

\caption{An example of a step of $\Qqacol$ at levels 4 and 5.} 

\label{fig:ale2}

\end{figure}

\begin{lemma}
	If $\bar\la \prech \bar \nu$, $\bar\la\prech \la$ 
	and $\mathscr{U}_j(\la \to \nu \mid \bar \la \to \bar \nu) > 0$, 
	then $\bar \nu\prech\nu$ and $\la \prech \nu$.
\end{lemma}
\begin{proof}
	It is straightforward from the definition of the dynamics $\Qqacol$ that $\nu_{j-i+1} \leq \bar \la_{j-i} \le \min(\bar\nu_{j-i}, \la_{j-i})$.  Also for $2 \leq i \leq j$ \eqref{Yi_distribution} together with \eqref{phinequalities} implies that $\bar \la_{j-i} - \la_{j-i+1} - X_{i} - Y_{i} \leq \bar \la_{j-i} - \bar \la_{j-i+1} - c_{j-i+1}$, hence $\nu_{j-i+1} \geq \la_{j-i+1} + X_{i} + Y_{i} \geq \bar \nu_{j-i+1}$. 
It follows that the interlacing properties are preserved.  	
\end{proof}

In the rest of this subsection we will describe the column insertion dynamics
in words, and also discuss its various properties. 
The (rather involved) proof that this dynamics acts on $q$-Whittaker processes in a desired way
is postponed to the next subsection.

There are two stages of the update of particle positions $\la_{j}, \la_{j-1},\ldots,\la_1$, 
performed in order
\emph{from left to right}, which we will describe below.

During the first stage of the update, the particles at level $j$ level make voluntary jumps in order from left to right. 
The value $X_i$ of the voluntary jump of $\la_{j+1-i}$ 
depends on the previous jump $X_{i-1}$, where $2\le i \le j$.
Indeed, this dependence comes from the parameters $\theta_{i}$ (note that they are nonincreasing in $i$), 
see \eqref{Xi_distribution}.
Note that unlike the $\Qqarow$ case, 
in which all random movements not coming from pushing 
are restricted to the right edge, 
in the case of $\Qqacol$ any particle might make a voluntary jump.

\begin{remark}
\label{rmk:Xi_from_Vj}
The random variable $X_1+\ldots+X_j$
has the 
$q$-geometric distribution with parameter $\al a_{j}$,
as it should be by Remark \ref{rmk:one_level_dynamics} 
and the discussion of \S \ref{sub:rsk_type_dynamics}. 
This is seen by applying inductively the next Lemma.
\end{remark}
\begin{lemma} Let $A$ and $B$ be random variables
such that
$A$ is distributed according to $\Om_{q, \al, 0}(\cdot \mid a)$,
and $B$ given $A$ is distributed according to $\Om_{q, \al q^{a-A}, 0}(\cdot \mid b)$ (where $b$ might be $+\infty$).  
Then $A+B$ is distributed according to $\Om_{q, \al, 0}(\cdot \mid a+b)$.
\end{lemma}
\begin{proof}Indeed, we have
\begin{align*}
& Prob(A+B = y) = \sum_{s=0}^{y} Prob(A=s)Prob(B= y-s | A=s)
\\&\hspace{30pt} = \sum_{s=0}^{y} \al^{s} (\al; q)_{a-s} \binom {a}{s}_{q} (\al q^{a-s})^{y-s} (\al q^{a-s}; q)_{b-y+s} \binom {b}{y-s}_{q}
\\&\hspace{30pt} = \sum_{s=0}^{y} \al^{y} (\al; q)_{a+b-y} q^{a(y-s)} \frac{(q^{a};q^{-1})_{s}(q^{b};q^{-1})_{y-s}}{(q; q)_{y}} q^{-s(y-s)} \binom ys_{q}
\\&\hspace{30pt} = \al^{y} (\al; q)_{a+b-y} \binom {a+b}{y}_{q} \cdot \sum_{s=0}^{y} \Om_{q^{-1}, q^{a}, q^{a+b}}(y-s \mid y) 
\\&\hspace{30pt} = \Om_{q, \al, 0}(y \mid a+b),
\end{align*}
which establishes the desired statement.
\end{proof}

The second stage of the update
consists of pushing, 
in order from left to right. 
We start an initially empty stabilization fund, 
which will collect impulses not immediately used for pushing,
and will be a source of the pushes $Z_i$. 
The value of the stabilization fund just before the movement of $\la_{j+1-i}$ 
is $\log_{q}r_{i}$ (by agreement, $r_{1} = r_{2} = 1$ always). 
For each $i$ ranging from $2$ to $j$, the following three steps happen:
\begin{align}\label{col_pushing_3steps}
\parbox{.88\textwidth}{
\begin{enumerate}[(1)]
	\item The particle $\la_{j+1-i}$
	gets a push $Y_i$ from its lower left neighbor $\bar\la_{j+1-i}$.
	The size of this push (distributed according to \eqref{Yi_distribution}) is at most 
	$c_{j-i+1}$.
	\item Then $\la_{j+1-i}$ gets a push from the stabilization fund 
	(if it is not empty) of size not exceeding the current value of the stabilization fund.
	This push is distributed according to $Z_i$ \eqref{Zi_distribution}.
	\item Finally, the 
	amount of pushing not used in (1) above, i.e., $c_{j-i+1}-Y_i$,
	is added to the stabilization fund.
\end{enumerate}}
\end{align}

One can also think that the above two update
stages are performed together for each particle
$\la_{j}, \la_{j-1},\ldots,\la_1$.

\begin{proposition}\label{prop:Kr1}
	One can switch the order of the lower left neighbor pushing 
	and stabilization fund pushing (i.e., steps {\rm{}(1)\/} and {\rm{}(2)\/} 
	in \eqref{col_pushing_3steps}) 
	without changing the dynamics.\footnote{Here for the version with interchanged steps {\rm{}(1)\/} and {\rm{}(2)\/} the distributions of the jump components are changed as $Z_i\sim\Om_{q^{-1}, r_{i}, 0}(\bar \la_{j-i} - \la_{j-i+1} - X_{i} - \cdot \mid  \bar \la_{j-i} - \la_{j-i+1} - X_{i}) $ and $Y_i\sim\Om_{q^{-1}, q^{c_{j-i+1}}, q^{\bar \la_{j-i} - \bar \la_{j-i+1}}}(\bar \la_{j-i} - \la_{j-i+1} - X_{i} - Z_{i}-\cdot \mid  \bar \la_{j-i} - \la_{j-i+1} - X_{i}-Z_{i})$.}
\end{proposition}
\begin{proof}
Fix $k=2,\ldots,j$.
Suppose that after the voluntary displacement stage the distance from 
the $k$-th particle from the left at level $j$
(denote this particle by $P$)
to $\bar\la_{j+1-k}$ is 
$h := \bar \la_{j-k} - \la_{j-k+1} - X_{j-k+1}$. 
Also set
$\ell := \bar \nu_{j-k+1} - \bar \la_{j-k+1}$,
$b := \bar \la_{j-k} - \bar \la_{j-k+1}$,
and let the current size of the stabilization fund be $R$. 

If the steps (1) and (2) in \eqref{col_pushing_3steps} are not interchanged,
then the probability
that $P$ jumps by $s\ge0$ is 
\begin{align*}
	\sum_{y=0}^{s} \Om_{q^{-1}, q^{\ell}, q^{b}}(h-y \mid h)\Om_{q^{-1}, q^{R}, 0}(h-s \mid h-y).
\end{align*}
If the steps (1) and (2) in \eqref{col_pushing_3steps} are interchanged, then the 
same probability is given by 
\begin{align*}
	\sum_{y=0}^{s} \Om_{q^{-1}, q^{R}, 0}(h-s+y \mid h) \Om_{q^{-1}, q^{\ell}, q^{b}}(h-s \mid h-s+y).
\end{align*}
	
After dividing each of these two expressions by $\frac{q^{(R+\ell)(h-s)}(q^{b-\ell}; q^{-1})_{h-s}}{(q^{b}; q^{-1})_{h-s}} \binom hs_{q^{-1}}$ we arrive to the following identity we need to verify:
\begin{multline}\label{qbinom_id_1} 
\sum_{y=0}^{s} \binom {s}{y}_{q^{-1}} q^{\ell(s-y)}(q^{\ell}; q^{-1})_{y}(q^{R}; q^{-1})_{s-y}(q^{b-\ell-h+s}; q^{-1})_{s-y}
\\=\sum_{y=0}^{s} \binom {s}{y}_{q^{-1}} q^{Ry}(q^{\ell}; q^{-1})_{y}(q^{R}; q^{-1})_{s-y}(q^{b-h+1}; q)_{s-y}.
\end{multline}
We are very grateful to Christian Krattenthaler
for providing us with a proof of the $q$-binomial 
identity \eqref{qbinom_id_1}, which we reproduce below.

First, use a transformation formula 
for $_{3}\varphi_{2}$ series \cite[(III.12)]{GasperRahman}:
\begin{align*} 
_{3}\varphi_{2} \left[ \begin{array}{c} q^{-n}, b, c \\ d, e \end{array}; q, q \right] 
= \frac{(e/c; q)_{n}}{(e;q)_{n}}c^{n} {_{3}\varphi_{2}} \left[ \begin{array}{c} q^{-n}, c, d/b \\ d, cq^{1-n}/e \end{array}; q, \frac{bq}{e} \right]
\end{align*}
Sending $b \to 0$ we obtain
\begin{align} \label{transform1}
_{3}\varphi_{2} \left[ \begin{array}{c} q^{-n}, 0, c \\ d, e \end{array}; q, q \right] 
= \frac{(e/c; q)_{n}}{(e;q)_{n}}c^{n} {_{2}\varphi_{2}} \left[ \begin{array}{c} q^{-n}, c \\ d, cq^{1-n}/e \end{array}; q, \frac{dq}{e} \right]
\end{align}
Multiply both sides of \eqref{transform1} by $c^{-n}(d; q)_{n}(e;q)_{n}$ to obtain
\begin{align*}
c^{-n}\sum_{y=0}^{n}\frac{(q^{-n}; q)_{y}}{(q;q)_{y}}(c; q)_{y}(dq^{n-1}; q^{-1})_{n-y}(eq^{n-1}; q^{-1})_{n-y}q^{y} = \\ \sum_{y=0}^{n}\frac{(q^{-n}; q)_{y}(e/c; q)_{n}}{(q;q)_{y}(cq^{1-n}/e; q)_{y}}(c; q)_{y}(dq^{n-1}; q^{-1})_{n-y} (-1)^{y} q^{y(y-1)/2}(dq/e)^{y}.
\end{align*}
This equality can be rewritten as 
\begin{align*}
\sum_{y=0}^{n}\binom ny_{q^{-1}}c^{y-n}(c^{-1}; q^{-1})_{y}(dq^{n-1}; q^{-1})_{n-y}(eq^{n-1}; q^{-1})_{n-y} = \\ \sum_{y=0}^{n}\binom ny_{q^{-1}} (e/c; q)_{n-y}(c^{-1}; q^{-1})_{y}(dq^{n-1}; q^{-1})_{n-y} (dq^{n-1})^{y}.
\end{align*}
Now make the substitution $n:=s$, $d:= q^{1+R-s}$, $c:=q^{-\ell}$, $e:= q^{1+b-h-\ell}$ to arrive to \eqref{qbinom_id_1}.
\end{proof}

\begin{remark} (Schur degeneration)
If one sets $q=0$, then the dynamics $\Qqacol$ reduces to the dynamics $\Qacol$ on Schur processes based on the classical Robinson--Schensted--Knuth column insertion (\S \ref{sub:rsk_type_dynamics_schur}). Indeed, observe that
\begin{align*}
	\lim_{q \to 0}\Om_{q,uq^{t}, 0}(s \mid g)  = \mathbf{1}_{s=0} \quad \text{for} \ t > 0, \ \text{and} \quad \lim_{q \to 0}\Om_{q,u, 0}(s \mid g)=  (1-u+u\mathbf{1}_{s=g})u^{s}.
\end{align*}
Thus, the first update stage (voluntary movements) reduces to the propagation of the impulse the leftmost particle receives
(which has geometric distribution with parameter $\al a_{j}$).
The lower left neighbor pushing due to \eqref{q=0_Omega} and the stabilization fund pushing together degenerate to performing $c_{j-1} + \cdots + c_{1}$ operations $\mathsf{push}$ (Definition \ref{def:push}) 
in order from left to right.
\end{remark}

\subsection{Column insertion dynamics $\Qqacol$. Proof} 
\label{sub:column_insertion_dynamics_qqacol_proof}

\begin{theorem}\label{thm:QqcRSK_alpha}
The dynamics $\Qqacol$ defined above satisfies the main equations \eqref{alpha_main_equation}, and hence preserves the class of $q$-Whittaker processes and adds a new usual parameter $\al$ to the specialization $\mathbf{A}$ as in \eqref{adding_spec}.
\end{theorem}
\begin{proof}
We aim to prove the desired statement by induction on $j$. 
To apply this induction, 
we will need a more general statement. To describe it, introduce
the following notation.
For a nonnegative integer $h$, use $\mathscr{U}_{j}^{h}(\la \to \nu \mid \bar \la \to \bar \nu)$ to denote the 
probability that transition $\bar \la \to \bar \nu$ on the $(j-1)$-st level 
spurs a transition $\lambda \to \nu$ on the $j$-th level according to 
the rules of $\Qqacol$ specified above, but modified so that $Z_{2}=z$ with probability
\begin{align*}
\Om_{q^{-1}, r_{2}, 0}(\bar \la_{j-2} - \la_{j-1} - X_{2} - Y_{2} - z \mid  \bar \la_{j-2} - \la_{j-1} - X_{2} - Y_{2}),
\qquad
r_{2}:=q^{h}.
\end{align*}
Note that the original dynamics $\Qqacol$ has 
$r_{2}=1$.
In other words, the modification $\mathscr{U}_{j}^{h}$
means that we introduce an additional impulse of size $h$
which is distributed among particles at level $j$ (except for $\la_j$), 
as if coming from (nonexistent) particles preceding the leftmost particle on the $(j-1)$-st level. 

Let $\sigma:= |\nu^{-}| - |\la^{-}| - |\bar \nu| + |\bar \la|$
(recall that the notation $\mu^{-}$ means $\mu$ without the last coordinate). 
Under the modified probabilities $\mathscr{U}_{j}^{h}$ as above, 
$\sigma-h$ is a sum of voluntary movements of particles on the 
$j$-th level except for the leftmost one. Note also that 
$\mathscr{U}_{j}^{h}(\la \to \nu \mid \bar \la \to \bar \nu)=0$ for $h > \sigma$.

For the purposes of the proof, let (see Fig.~\ref{fig:alp})
\begin{align*}
	&a:=\la_{j},\qquad k:= \nu_{j} - \la_{j},\qquad b:=\bar \nu_{j-1},\qquad 
	t:=\bar \nu_{j-1} - \bar \la_{j-1},\\&
	c:=\la_{j-1},\qquad d:=\bar \nu_{j-2},\qquad 
	s:=\bar \nu_{j-2} - \bar \la_{j-2},\\& 
	\ell:=\nu_{j-1}-\la_{j-1},
	\qquad x:=X_{2}\qquad y:=Y_{2}. 
\end{align*}
For a nonnegative integer $H$ define
\begin{align}\label{tilde_U_notation}
\tilde{\mathscr{U}}_{j}^{H}(\la \to \nu \mid \bar \la \to \bar \nu):= \sum_{h=0}^{H} \binom Hh_{q^{-1}}q^{(H-h)\sigma + h(b-t-a-k)}\mathscr{U}_{j}^{h}(\la \to \nu \mid \bar \la \to \bar \nu).
\end{align}
In particular, $\tilde{\mathscr{U}}_{j}^{0}(\la \to \nu \mid \bar \la \to \bar \nu) = \mathscr{U}_{j}(\la \to \nu \mid \bar \la \to \bar \nu)$. 
In general, the quantities
$\tilde{\mathscr{U}}_{j}^{H}$ are not probability distributions in $\nu$.
Their only meaning is that they come up in the inductive proof below.

With all the above notation we are now able to describe 
and prove the generalized statement which we will prove by induction:
\begin{align}
\label{alpha_main_equation_mod}
		\sum_{\bar\la\in\GT_{j-1}}
		\tilde{\mathscr{V}}_j^{H}
		(\la\to\nu\mid \bar\la\to\bar\nu)
		\frac{\psi_{\la/\bar\la}\varphi_{\bar\nu/\bar\la}}{\psi_{\nu/\bar\nu}\varphi_{\nu/\la}}
        = 1 \qquad \text{for any $H\ge0$.}
\end{align}
Here and below $\tilde{\mathscr{V}}_j^{H}$
is related to $\tilde{\mathscr{U}}_j^{H}$
as in \eqref{Vj_notation}. For $H=0$ this statement gives us \eqref{alpha_main_equation}.

For $j=1$ we have $\sigma=0$, so 
only the term $h=0$ contributes to \eqref{tilde_U_notation}.
Therefore, 
checking this induction base is the same as in the proof for $\Qqarow$ dynamics.  
\begin{figure}[h] 
\hspace*{0 cm} \includegraphics[width = 0.7\textwidth]{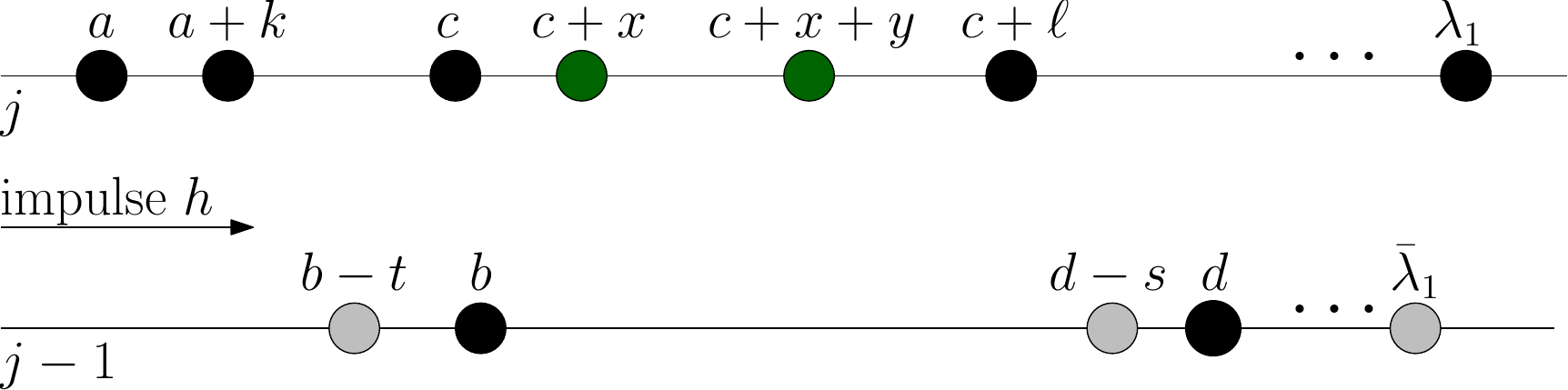}
\caption{We expand sum with respect to jump $t = c_{j-1} =\bar \nu_{j-1} - \bar \la_{j-1}$ of the leftmost particle on the $(j-1)$-st level, voluntary movement $x$ of the second leftmost particle on the $j$-th level and push $y$ from the leftmost particle on the $(j-1)$-st level.} 
\label{fig:alp}
\end{figure}


Let us now perform the inductive step.
Denote by $I$ the range of $(t, x, y, h)$, such that 
\begin{align*}
	t, x, y, h \geq 0, \qquad
	x + y \leq \ell,\qquad h+t+x -\ell \geq 0,\qquad t \leq b-a-k.
\end{align*}
Then we may write (see Fig.~\ref{fig:alp})
\begin{align*}
& \sum_{\bar\la\in\GT_{j-1}} \tilde{\mathscr{V}}_{j}^{H}(\la\to\nu\mid \bar\la\to\bar\nu) 
\frac{\psi_{\la/\bar\la}\varphi_{\bar\nu/\bar\la}}{\psi_{\nu/\bar\nu}\varphi_{\nu/\la}}
{}
\\&\hspace{20pt} =\sum_{(t, x, y, h) \in I} 
\frac{(q;q)_{H}}{(q;q)_{h}(q;q)_{H-h}} 
\frac{(q;q)_{c-a}}{(q;q)_{b-t-a}(q;q)_{c-b+t}} 
\frac{(q;q)_{b-a-k}(q;q)_{c+\ell-b}}{(q;q)_{c+\ell-a-k}} 
\frac{(q;q)_{d-s-b+t}}{(q;q)_{t}(q;q)_{d-s-b}} 
\\& \hspace{40pt}\times 
\frac{(q;q)_{k}(q;q)_{c-a-k}}{(q;q)_{c-a}}
\frac{(q;q)_{b-t-a}}{(q;q)_{k}(q;q)_{b-t-a-k}} 
\frac{(q;q)_{d-s-c}}{(q;q)_{x}(q;q)_{d-s-c-x}} 
\frac{(q;q)_{d-s-c-x}}{(q;q)_{y}(q;q)_{d-s-c-x-y}} 
\\& \hspace{40pt}\times 
\frac{(q^{t};q^{-1})_{y}(q^{d-s-b}; q^{-1})_{d-s-c-x-y}}{(q^{d-s-b+t};q^{-1})_{d-s-c-x}}
\frac{(q;q)_{d-s-c-x-y}}{(q;q)_{\ell -x -y}(q;q)_{d-s-c-\ell}} (q^{h}; q^{-1})_{\ell-x-y} 
\\& \hspace{40pt}\times q^{-h(H-h)-y(d-s-c-x-y) -(d-s-c-\ell)(\ell-x-y) + t(d-s-c-x-y)} 
\\& \hspace{40pt}\times  q^{h(d-s-c-\ell) + (H-h)\sigma + h(b-t-a-k) + (b-t-a-k)x + (b-t-a-k + \ell-x)(\sigma-x-h)} 
\\& \hspace{60pt} \times 
\sum_{\tilde{\la}\in\GT_{j-2}} \mathscr{V}_{j}^{h+t+x-\ell}(\la^{-}\to\nu^{-}\mid \tilde{\la}\to\bar\nu^{-}) \frac{\psi_{\la^{-}/\tilde{\la}}\varphi_{\bar\nu^{-}/\tilde{\la}}}{\psi_{\nu^{-}/\bar\nu^{-}}\varphi_{\nu^{-}/\la^{-}}} 
\\&\hspace{20pt} 
= \sum_{r=0}^{H+B}\binom {H+B}{r}_{q^{-1}} q^{(H+B-r)(\sigma - \ell+t) + r(d-s-c-\ell)} 
\\& \hspace{60pt} \times 
\sum_{\tilde{\la}\in\GT_{j-2}} \mathscr{V}_{j}^{r}(\la^{-}\to\nu^{-}\mid \tilde{\la}\to\bar\nu^{-}) 
\frac{\psi_{\la^{-}/\tilde{\la}}\varphi_{\bar\nu^{-}/\tilde{\la}}}{\psi_{\nu^{-}/\bar\nu^{-}}\varphi_{\nu^{-}/\la^{-}}} 
\\&\hspace{20pt} 
= \sum_{\tilde\la\in\GT_{j-2}} \tilde{\mathscr{V}}_j^{H+B} 
(\la^{-}\to\nu^{-}\mid \tilde{\la}\to\bar\nu^{-}) 
\frac{\psi_{\la^{-}/\tilde{\la}}\varphi_{\bar\nu^{-}/\tilde{\la}}}{\psi_{\nu^{-}/\bar\nu^{-}}\varphi_{\nu^{-}/\la^{-}}} 
\\&\hspace{20pt} = 1.
\end{align*}
Here we have applied Proposition \ref{prop:Kr2}
(see below) 
with $A = H$, $B = b-a-k$, $C = c-b+ \ell$, 
where
$r: = h+t - \ell + x$ is 
the value of the stabilization fund just before the push of the third 
leftmost particle on the $j$-th level plus the value of the additional impulse in the inductive assumption.
This completes the inductive step in proving \eqref{alpha_main_equation_mod},
and thus implies the
theorem.
\end{proof}

\begin{proposition}\label{prop:Kr2}  For $A, B, C, \ell, r \geq 0$, such that $A+B \geq r$ and $B+C \geq \ell$, one has
\begin{multline*} 
	\sum_{t = 0}^{B} \sum_{x=0}^{\ell} \sum_{y=0}^{\ell-x} 
	\Bigg[
	\binom {\ell}{x, \ y}_{q^{-1}}  \binom {B}{t}_{q^{-1}} (q^{t}; q^{-1})_{y}
	(q^{r+\ell-x}; q^{-1})_{t} (q^{r+\ell-t-x}; q^{-1})_{\ell-x-y}  \frac{(q;q)_{r}}{(q;q)_{r+\ell-x}} 
	\\\times
	\frac{(q;q)_{A}}{(q;q)_{A+B}}
	\frac{(q^{A+B-r}; q^{-1})_{B-t+\ell - x}(q^{C+t}; q^{-1})_{\ell}(q^{C}; q^{-1})_{\ell-x-y} }{(q^{B+C}; q^{-1})_{\ell} (q^{C+t}; q^{-1})_{\ell-x}}
	\\\times
	q^{t(\ell-x-y)+ (r+\ell-x)(B-t)+(A+B-r)x} 
	\Bigg]
	= 1.
\end{multline*}
Here and thereafter we use $q$-multinomial 
notation $\displaystyle\binom {n}{m,\ k}_{q}:= \frac{(q;q)_{n}}{(q;q)_{k}(q;q)_{m}(q;q)_{n-m-k}}$.
\end{proposition}
We are extremely grateful to Christian Krattenthaler
for providing us with a proof of this proposition. 
We reproduce the proof below.
\begin{proof}
The left hand side of the equality can be expressed as a power series in $q^{A}, q^{r}, q^{C}$, hence we can set $\alpha = q^{A}$, $\beta = q^{r}$, $\gamma = q^{C}$ and prove a more general equality:
\begin{multline*} 
	\sum_{t = 0}^{B} \sum_{x=0}^{\ell} \sum_{y=0}^{\ell-x} 
	\Bigg[
	\binom {\ell}{x, \ y}_{q^{-1}}  \binom {B}{t}_{q^{-1}} (q^{t}; q^{-1})_{y}
	\frac{(\beta q^{\ell-x}; q^{-1})_{t + \ell-x-y}}{(\beta q^{\ell-x}; q^{-1})_{\ell-x}} 
	\\\times
	\frac{(\alpha q^{B}/ \beta; q^{-1})_{B-t+\ell - x}(\gamma q^{t}; q^{-1})_{\ell}(\gamma; q^{-1})_{\ell-x-y} }{(\alpha q^{B}; q^{-1})_{B}(\gamma q^{B}; q^{-1})_{\ell} (\gamma q^{t}; q^{-1})_{\ell-x}}
	\times
	\alpha^{x} \beta^{B-t-x} q^{-ty + B\ell} 
	\Bigg]
	= 1.
\end{multline*}
By first summing over $y$, the left hand side can be written as 
\begin{multline*}
\sum_{t = 0}^{B} \sum_{x=0}^{\ell} \Bigg[ {_{3}\varphi_{2}} \left[ \begin{array}{c} q^{-t}, 0, q^{-\ell + x} \\ \beta q^{1-t}, \gamma q^{1-\ell+x}  \end{array}; q, q \right] \binom {\ell}{x}_{q^{-1}}  \binom {B}{t}_{q^{-1}}
	\frac{(\beta q^{\ell-x}; q^{-1})_{t + \ell-x}}{(\beta q^{\ell-x}; q^{-1})_{\ell-x}} 
	\\\times
	\frac{(\alpha q^{B}/ \beta; q^{-1})_{B-t+\ell - x}(\gamma q^{t}; q^{-1})_{\ell}(\gamma; q^{-1})_{\ell-x} }{(\alpha q^{B}; q^{-1})_{B}(\gamma q^{B}; q^{-1})_{\ell} (\gamma q^{t}; q^{-1})_{\ell-x}}
	\times
	\alpha^{x} \beta^{B-t-x} q^{B\ell} 
	\Bigg].
\end{multline*}
We now apply transformation formula 
\eqref{transform1} to rewrite this as
\begin{align*}
&=\sum_{t = 0}^{B} \sum_{x=0}^{\ell} \Bigg[ {_{2}\varphi_{2}} \left[ \begin{array}{c} q^{-t}, q^{-\ell + x}  \\ \beta q^{1-t}, q^{-t}/\gamma  \end{array}; q, \beta q^{1 + \ell-t-x}/\gamma \right] \binom {\ell}{x}_{q^{-1}}  \binom {B}{t}_{q^{-1}}
	\frac{(\beta q^{\ell-x}; q^{-1})_{t + \ell-x}}{(\beta q^{\ell-x}; q^{-1})_{\ell-x}} 
           \frac{(\gamma q; q)_{t}}{(\gamma q^{1-\ell+x}; q)_{t}}
	\\&\hspace{40pt}\times
	\frac{(\alpha q^{B}/ \beta; q^{-1})_{B-t+\ell - x}(\gamma q^{t}; q^{-1})_{\ell}(\gamma; q^{-1})_{\ell-x} }{(\alpha q^{B}; q^{-1})_{B}(\gamma q^{B}; q^{-1})_{\ell} (\gamma q^{t}; q^{-1})_{\ell-x}}
	\times
	\alpha^{x} \beta^{B-t-x} q^{B\ell - \ell t + tx} 
	\Bigg]
\\&\hspace{20pt}=
\sum_{t = 0}^{B} \sum_{x=0}^{\ell} \sum_{y=0}^{\min\{t, \ell-x\}} \Bigg[(-1)^{y} \frac{(q^{-t}; q)_{y}(q^{-\ell+x}; q)_{y}}{(q^{-t}/\gamma; q)_{y}(\beta q^{1-t}; q)_{y}(q; q)_{y}}  \binom {\ell}{x}_{q^{-1}}  \binom {B}{t}_{q^{-1}}
	\frac{(\gamma q; q)_{t}}{(\gamma q^{1- \ell+x}; q)_{t}}
	\\&\hspace{40pt}\times
           \frac{(\beta; q^{-1})_{t}(\alpha q^{B}/ \beta; q^{-1})_{B-t+\ell - x}(\gamma q^{t}; q^{-1})_{\ell}(\gamma; q^{-1})_{\ell-x} }{(\alpha q^{B}; q^{-1})_{B}(\gamma q^{B}; q^{-1})_{\ell} (\gamma q^{t}; q^{-1})_{\ell-x}}
	\\&\hspace{40pt}\times
	\alpha^{x} \beta^{B-t-x+y} \gamma^{-y} q^{B\ell + \ell y - \ell t + tx + y^{2}/2 + y/2 -ty - xy} 
	\Bigg] 
\\&\hspace{20pt}=
\sum_{t = 0}^{B} \sum_{y=0}^{\ell} \Bigg[{_{1}\varphi_{1}} \left[ \begin{array}{c} q^{-\ell + y} \\ \alpha q^{1-\ell+t}/ \beta \end{array}; q, \alpha q^{1+t-y}/\beta \right] (-1)^{y} \frac{ (q^{-\ell}; q)_{y} (q^{-t}; q)_{y}}{(q^{-t}/\gamma; q)_{y}(\beta q^{1-t}; q)_{y} (q; q)_{y}} 
	\\&\hspace{40pt}\times
	\binom {B}{t}_{q^{-1}}  \frac{(\beta; q^{-1})_{t}(\alpha q^{B}/ \beta; q^{-1})_{B-t+\ell }(\gamma q^{t}; q^{-1})_{\ell}}{(\alpha q^{B}; q^{-1})_{B}(\gamma q^{B}; q^{-1})_{\ell}}
	\times
	\beta^{B-t+y} \gamma^{-y} q^{B\ell - \ell t + \ell y + y^{2}/2 + y/2 -ty} 
	\Bigg].
\end{align*}
The last equality is obtained by summing over $x$. 
We now use the summation formula 
\cite[(II.5)]{GasperRahman}:
\begin{align*}
{_{1}\varphi_{1}} \left[ \begin{array}{c} a \\ c \end{array}; q, c/a \right] = \frac{(c/a; q)_{\infty}}{(c; q)_{\infty}},
\end{align*}
and by summing over $y$ rewrite our expression as
\begin{multline*}
=\sum_{t = 0}^{B} \Bigg[ {_{3}\varphi_{2}} \left[ \begin{array}{c} q^{-\ell} , \beta q^{-t}/\alpha, q^{-t}  \\ \beta q^{1-t}, q^{-t}/\gamma  \end{array}; q, \alpha q^{1+\ell}/\gamma \right] 
\\\times
	\binom {B}{t}_{q^{-1}}  \frac{(\beta; q^{-1})_{t}(\alpha q^{B}/ \beta; q^{-1})_{B-t+\ell }(\gamma q^{t}; q^{-1})_{\ell}}{(\alpha q^{t}/\beta; q^{-1})_{\ell}(\alpha q^{B}; q^{-1})_{B}(\gamma q^{B}; q^{-1})_{\ell}}
	\times
	\beta^{B-t} q^{B\ell - \ell t} 
	\Bigg].
\end{multline*}
We now aim to use
the transformation formula \cite[(III.13)]{GasperRahman}:
\begin{align}
{_{3}\varphi_{2}} \left[ \begin{array}{c} q^{-n}, b, c \\ d, e  \end{array}; q, deq^{n}/bc \right] 
= \frac{(e/c; q)_{n}}{(e; q)_{n}}{_{3}\varphi_{2}} \left[ \begin{array}{c} q^{-n}, c, d/b \\ d, cq^{1-n}/e  \end{array}; q, q \right].
\end{align} 
Applying it, we can 
rewrite our expression as 
\begin{align*}
&=\sum_{t = 0}^{B} \Bigg[ {_{3}\varphi_{2}} \left[ \begin{array}{c} q^{-\ell}, q^{-t},  \alpha q \\ \beta q^{1-t}, \gamma q^{1-\ell}  \end{array}; q, q \right] 
\\&\hspace{40pt}\times
	\binom {B}{t}_{q^{-1}}  \frac{(1/\gamma; q)_{\ell}(\beta; q^{-1})_{t}(\alpha q^{B}/ \beta; q^{-1})_{B-t+\ell }(\gamma q^{t}; q^{-1})_{\ell}}{(\alpha q^{t}/\beta; q^{-1})_{\ell}(\alpha q^{B}; q^{-1})_{B}(\gamma q^{B}; q^{-1})_{\ell}(q^{-t}/\gamma; q)_{\ell}(q; q)_{y}}
	\times
	\beta^{B-t} q^{B\ell - \ell t} 
	\Bigg]
\\&\hspace{5pt}=
\sum_{t = 0}^{B} \sum_{y=0}^{\ell} \Bigg[\binom {B}{t}_{q^{-1}}  \frac{(q^{-\ell}; q)_{y}(q^{-t}; q)_{y}(\alpha q; q)_{y}(\gamma; q^{-1})_{\ell-y}(\beta q^{y}; q^{-1})_{t}(\alpha q^{B}/ \beta; q^{-1})_{B-t+\ell }}{(\alpha q^{t}/\beta; q^{-1})_{\ell}(\alpha q^{B}; q^{-1})_{B}(\gamma q^{B}; q^{-1})_{\ell}(\beta q^{y}; q^{-1})_{y}(q; q)_{y}}
	\times
	\beta^{B-t} q^{B\ell + y} 
	\Bigg] 
\\&\hspace{5pt}=
\sum_{y=0}^{\ell} \sum_{t = y}^{B}  \Bigg[\binom {B-y}{t-y}_{q^{-1}}\frac{\beta^{B-t} (\beta; q^{-1})_{t-y} (\alpha q^{B}/ \beta; 
q^{-1})_{B-t} }{(\alpha q^{B}; q^{-1})_{B-y}} \times \binom {\ell}{y}_{q^{-1}} \frac{q^{B(\ell-y)}(\gamma; q^{-1})_{\ell-y}(q^{B}; q^{-1})_{y}}{(\gamma q^{B}; q^{-1})_{\ell}}   
	\Bigg].
\end{align*}
The fact that this expression is equal to $1$ now
follows by applying \eqref{Om_qmunu_sum} twice.
\end{proof}


\subsection{Geometric $q$-TASEP} 
\label{sub:geometric_q_tasep}

Under the dynamics 
$\Qqacol$, 
the leftmost $N$ particles
$\la^{(j)}_{j}$
of the interlacing array
evolve in a \emph{marginally Markovian manner}.

Namely, let $\gap_j(t) := \la^{(j-1)}_{j-1}(t) - \la^{(j)}_{j}(t)$ 
be the gap between the consecutive 
leftmost particles at time $t$. We assume $\gap_1(t) = +\infty$. 
Then at each discrete time step $t\to t+1$
the leftmost particle on the $j$-th level is updated as
\begin{align*}
\la^{(j)}_{1}(t+1)=\la^{(j)}_{1}(t)+ W_{j, t}
\end{align*}
for an independent random variable $W_{j, t}$ distributed according to $\Om_{q, \al a_{j} , 0}(\cdot \mid \gap_j(t))$.

This evolution of
$\la^{(j)}_{j}$,
$1\le j\le N$,
is the 
(\emph{discrete time}) 
\emph{geometric \mbox{$q$-TASEP}}
which was introduced and studied in 
\cite{BorodinCorwin2013discrete}.


\subsection{Small $\al$ continuous time limit} 
\label{sub:small_al_continuous_time_limit}

Let us send the parameter $\al$ to zero and simultaneously rescale time from discrete to continuous.
Namely, set
$\al: = (1-q)\Delta$, and let each discrete time step
correspond to continuous time 
$\Delta$.
In the limit 
$\Delta \to 0$, 
both dynamics
$\Qqarow$ and $\Qqacol$ turn into the same continuous time Markov dynamics on $q$-Whittaker processes 
as in \S\ref{sub:small_be_continuous_time_limit} above. 
That is, the limit of $\Qqarow$ is the dynamics introduced 
in \cite{BorodinPetrov2013NN}, same as for $\Qqbrow$. 
The limit of $\Qqacol$ is the dynamics 
introduced in \cite{OConnellPei2012}, same as for $\Qqbcol$.
To see this, note that repeated $q$-geometric trials
can be approximated by (continuous time) Poisson processes in this scaling.

\section{Moments and Fredholm determinants} 
\label{sec:moments_and_fredholm_determinants_for_bernoulli_}

In this section we briefly discuss 
moment and Fredholm determinantal
formulas for the Bernoulli $q$-PushTASEP started from the step initial configuration
(corresponding to $\la^{(j)}_{1}(0)=0$, $j=1,\ldots,N$).
The Fredholm determinantal formula
which we extract from moment formulas
in a way similar to \cite{BorodinCorwinSasamoto2012} 
allow us to prove (in a small $\be$ continuous time limit, cf. \S \ref{sub:small_be_continuous_time_limit}) 
the conjectural 
Fredholm determinantal formula \cite[Conjecture 1.4]{CorwinPetrov2013}
for the continuous time $q$-PushASEP (which is a two-sided dynamics
unifying continuous time $q$-TASEP and $q$-PushTASEP).

\subsection{Bernoulli $q$-PushTASEP on the line} 
\label{sub:bernoulli_qpush_def}

In this section it will be convenient to work in 
the shifted coordinates
\begin{align*}
	x_i:=-\la_1^{(i)}-i,\qquad i=1,\ldots,N,
\end{align*}
so that
$x_1>\ldots >x_N$. We will think that the $x_j$'s
encode positions of particles on 
the line $\Z$ which jump to the \emph{left}.
Let us reformulate the definition of the Bernoulli $q$-PushTASEP (\S \ref{sub:bernoulli_q_pushtasep})
in these terms.

\begin{definition}\label{def:Bernoulli_qpush}
	Each discrete time step $t\to t+1$
	of the Bernoulli $q$-PushTASEP  
	consists of the following sequential updates (see Fig.~\ref{fig:Bernoulli}):
	\begin{enumerate}
		\item 
		The first particle $x_{1}$
		jumps to the left by one with probability 
		$\frac{a_1\be}{1+a_1\be}$, and stays put with
		the complementary probability $\frac{1}{1+a_1\be}$.
		\item Sequentially for $j=2,\ldots,N$: 
		\begin{enumerate}
			\item 
		If the particle 
		$x_{j-1}$ has not jumped, then
		$x_{j}$
		jumps to the left by one with probability 
		$\frac{a_j\be}{1+a_j\be}$, and stays put with
		the complementary probability $\frac{1}{1+a_j\be}$.
		\item 
		If the particle 
		$x_{j-1}$ has jumped (to the left by one), then
		$x_{j}$
		jumps to the left by one with probability 
		$\frac{a_j\be+q^{\gap_j(t)}}{1+a_j\be}$, and stays put with
		the complementary probability $\frac{1-q^{\gap_j(t)}}{1+a_j\be}$,
		where ${\gap_j}(t):=x_{j-1}(t)-x_{j}(t)-1$
		is the number of holes between the particles before the jump of 
		$x_{j-1}$.\footnote{Note that if $x_{j-1}$ has jumped and 
		$x_{j}(t)=x_{j-1}(t)-1$,
		then the probability that $x_j$
		jumps is equal to one, as it should be.}
		\end{enumerate}
	\end{enumerate}
\end{definition}
\begin{figure}[htbp]
	\begin{center}
		\begin{tabular}{lll}
		{\small\rm{}(1) Update $x_1(t+1)$ first:}&
		\multicolumn{2}{l}{\small\rm{}{(2a), (2b) Then update $x_2(t+1)$ based on 
		whether $x_1$ has jumped:}}\\
		\begin{adjustbox}{max width=.31\textwidth}
		\begin{tikzpicture}
			[scale=1,very thick]
			\def\pt{.17}
			\def\ee{.1}
			\def\h{.45}
			\draw[->] (1,0) -- (6.5,0);
			\foreach \ii in {(3*\h,0),(7*\h,0),(5*\h,0),(6*\h,0) ,(8*\h,0),(10*\h,0),(9*\h,0),(12*\h,0),(13*\h,0)}
			{
				\draw \ii circle(\pt);
			}
			\foreach \ii in {(4*\h,0),(11*\h,0)}
			{
				\draw[fill] \ii circle(\pt);
			}
			\node at (4*\h,-5*\ee) {$x_{2}(t)$};
			\node at (11*\h,-5*\ee) {$x_{1}(t)$};
		    \draw[->, very thick] (11*\h,.3) to [in=0, out=90] (10.5*\h,.6)
		    to [in=90, out=180] (10*\h,.3)
		    node [xshift=0,yshift=20] {$\mathrm{Prob}=\frac{a_1\be}{1+a_1\be}
		    \phantom{\frac{\be q^{6}}{1+\be}}$};
		\end{tikzpicture}
		\end{adjustbox}
		&
		\begin{adjustbox}{max width=.31\textwidth}
		\begin{tikzpicture}
			[scale=1,very thick]
			\def\pt{.17}
			\def\ee{.1}
			\def\h{.45}
			\draw[->] (1,0) -- (6.5,0);
			\foreach \ii in {(3*\h,0),(7*\h,0),(5*\h,0),(6*\h,0) ,(8*\h,0),(10*\h,0),(9*\h,0),(12*\h,0),(13*\h,0)}
			{
				\draw \ii circle(\pt);
			}
			\foreach \ii in {(4*\h,0),(11*\h,0)}
			{
				\draw[fill] \ii circle(\pt);
			}
			\node at (4*\h,-5*\ee) {$x_{2}(t)$};
			\node at (11*\h,-5*\ee) {$x_{1}(t+1)$};
		    \draw[->, very thick] (4*\h,.3) to [in=0, out=90] (3.5*\h,.6)
		    to [in=90, out=180] (3*\h,.3)
		    node [xshift=35,yshift=20] {$\mathrm{Prob}=
		    \frac{a_2\be}{1+a_2\be}$};
		\end{tikzpicture}
		\end{adjustbox}
		&
		\begin{adjustbox}{max width=.31\textwidth}
		\begin{tikzpicture}
			[scale=1,very thick]
			\def\pt{.17}
			\def\ee{.1}
			\def\h{.45}
			\draw[->] (1,0) -- (6.5,0);
			\foreach \ii in {(3*\h,0),(7*\h,0),(5*\h,0),(6*\h,0) ,(8*\h,0),(11*\h,0),(9*\h,0),(12*\h,0),(13*\h,0)}
			{
				\draw \ii circle(\pt);
			}
			\foreach \ii in {(4*\h,0),(10*\h,0)}
			{
				\draw[fill] \ii circle(\pt);
			}
			\node at (4*\h,-5*\ee) {$x_{2}(t)$};
			\node at (10*\h,-5*\ee) {$x_{1}(t+1)$};
		    \draw[->, very thick, dashed] (11*\h,.3) to [in=0, out=90] (10.5*\h,.6)
		    to [in=90, out=180] (10*\h,.3);
		    \draw[->, very thick] (4*\h,.3) to [in=0, out=90] (3.5*\h,.6)
		    to [in=90, out=180] (3*\h,.3)
		    node [xshift=35,yshift=20] {$\mathrm{Prob}=
		    \frac{a_2\be+q^{6}}{1+a_2\be}$};
		\end{tikzpicture}
		\end{adjustbox}
		\end{tabular}
	\end{center}
  	\caption{Bernoulli $q$-PushTASEP (on this picture, $\gap_{2}(t)=6$).}
  	\label{fig:Bernoulli}
\end{figure}
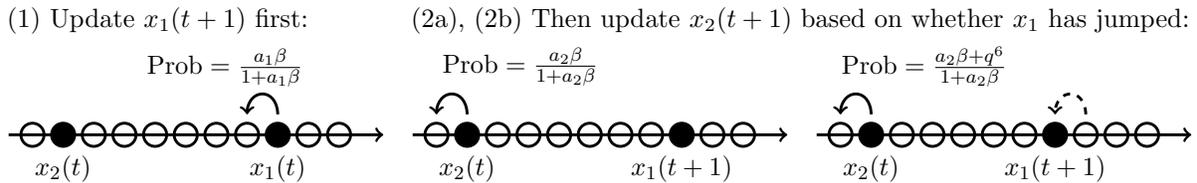
We will assume that the Bernoulli $q$-PushTASEP  
starts from the step initial configuration
$x_i(0)=-i$, $i=1,\ldots,N$. 

\begin{remark}
	Similarly to \cite{BorodinCorwin2013discrete},
	one could readily make the parameter 
	$\be$ of the process depend on time, so that at each 
	discrete time step $t\to t+1$, a new parameter
	$\be_{t+1}$ is used. This will not affect the presence and the general
	structure of 
	moment and Fredholm determinantal formulas.
	Below we will use a fixed parameter $\be$.
\end{remark}


\subsection{Connection to the Bernoulli $q$-TASEP} 
\label{sub:connection_to_the_bernoulli_}

The Bernoulli $q$-PushTASEP looks quite similar 
to the Bernoulli $q$-TASEP introduced in \cite{BorodinCorwin2013discrete}
(see also \S \ref{sub:bernoulli_tasep} above
for an explanation of how the latter process arises
from the dynamics $\Qqbcol$ on $q$-Whittaker processes).

Moreover, there exists a \emph{direct coupling} between the two processes 
which we now explain. 
Recall that under the
Bernoulli $q$-TASEP 
(we will denote its particles with tildes:
$\tilde x_1(t)>\ldots>\tilde x_N(t)$)
particles jump to the \emph{right} by one
according to the rules on Fig.~\ref{fig:Bernoulli_TASEP}.
Let this process also start from the step initial configuration
$\tilde x_i(0)=-i$, $i=1,\ldots,N$. 

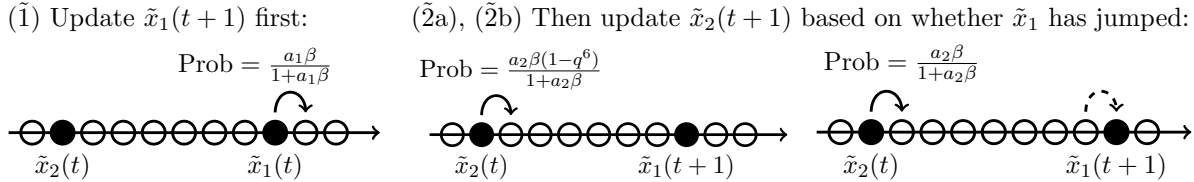
\begin{figure}[htbp]
	\begin{center}
		\begin{tabular}{lll}
		{\small\rm{}(\~{1}) Update $\tilde x_1(t+1)$ first:}&
		\multicolumn{2}{l}{\small\rm{}{(\~2a), (\~2b) Then update $\tilde x_2(t+1)$ based on 
		whether $\tilde x_1$ has jumped:}}\\
		\begin{adjustbox}{max width=.31\textwidth}
		\begin{tikzpicture}
			[scale=1,very thick]
			\def\pt{.17}
			\def\ee{.1}
			\def\h{.45}
			\draw[->] (1,0) -- (6.5,0);
			\foreach \ii in {(3*\h,0),(7*\h,0),(5*\h,0),(6*\h,0) ,(8*\h,0),(10*\h,0),(9*\h,0),(12*\h,0),(13*\h,0)}
			{
				\draw \ii circle(\pt);
			}
			\foreach \ii in {(4*\h,0),(11*\h,0)}
			{
				\draw[fill] \ii circle(\pt);
			}
			\node at (4*\h,-5*\ee) {$\tilde x_{2}(t)$};
			\node at (11*\h,-5*\ee) {$\tilde x_{1}(t)$};
		    \draw[->, very thick] (11*\h,.3) to [in=180, out=90] (11.5*\h,.6)
		    to [in=90, out=0] (12*\h,.3)
		    node [xshift=-12,yshift=20] {$\mathrm{Prob}=\frac{a_1\be}{1+a_1\be}
		    \phantom{\frac{\be q^{6}}{1+\be}}$};
		\end{tikzpicture}
		\end{adjustbox}
		&
		\begin{adjustbox}{max width=.31\textwidth}
		\begin{tikzpicture}
			[scale=1,very thick]
			\def\pt{.17}
			\def\ee{.1}
			\def\h{.45}
			\draw[->] (1,0) -- (6.5,0);
			\foreach \ii in {(3*\h,0),(7*\h,0),(5*\h,0),(6*\h,0) ,(8*\h,0),(10*\h,0),(9*\h,0),(12*\h,0),(13*\h,0)}
			{
				\draw \ii circle(\pt);
			}
			\foreach \ii in {(4*\h,0),(11*\h,0)}
			{
				\draw[fill] \ii circle(\pt);
			}
			\node at (4*\h,-5*\ee) {$\tilde x_{2}(t)$};
			\node at (11*\h,-5*\ee) {$\tilde x_{1}(t+1)$};
		    \draw[->, very thick] (4*\h,.3) to [in=180, out=90] (4.5*\h,.6)
		    to [in=90, out=0] (5*\h,.3)
		    node [xshift=0,yshift=20] {$\mathrm{Prob}=
		    \frac{a_2\be(1-q^{6})}{1+a_2\be}$};
		\end{tikzpicture}
		\end{adjustbox}
		&
		\begin{adjustbox}{max width=.31\textwidth}
		\begin{tikzpicture}
			[scale=1,very thick]
			\def\pt{.17}
			\def\ee{.1}
			\def\h{.45}
			\draw[->] (1,0) -- (6.5,0);
			\foreach \ii in {(3*\h,0),(7*\h,0),(5*\h,0),(6*\h,0) ,(8*\h,0),(10*\h,0),(9*\h,0),(11*\h,0),(13*\h,0)}
			{
				\draw \ii circle(\pt);
			}
			\foreach \ii in {(4*\h,0),(12*\h,0)}
			{
				\draw[fill] \ii circle(\pt);
			}
			\node at (4*\h,-5*\ee) {$\tilde x_{2}(t)$};
			\node at (12*\h,-5*\ee) {$\tilde x_{1}(t+1)$};
		    \draw[->, very thick, dashed] (11*\h,.3) to [in=180, out=90] (11.5*\h,.6)
		    to [in=90, out=0] (12*\h,.3);
		    \draw[->, very thick] (4*\h,.3) to [in=180, out=90] (4.5*\h,.6)
		    to [in=90, out=0] (5*\h,.3)
		    node [xshift=0,yshift=20] {$\mathrm{Prob}=
		    \frac{a_2\be}{1+a_2\be}$};
		\end{tikzpicture}
		\end{adjustbox}
		\end{tabular}
	\end{center}
  	\caption{Bernoulli $q$-TASEP (on this picture, $\gap_{2}(t)=6$).}
  	\label{fig:Bernoulli_TASEP}
\end{figure}

\begin{proposition}\label{prop:TASEP_coupling}
	Let $\{x_i(t)\}_{t=0,1,\ldots}$
	be the Bernoulli $q$-PushTASEP
	started from the step initial configuration
	and depending on parameters $\{a_i\}$ and $\be$.
	
	Then the evolution of the process $\{t+x_i(t)\}_{t=0,1,\ldots}$
	coincides with the Bernoulli $q$-TASEP 
	$\{\tilde x_i(t)\}_{t=0,1,\ldots}$
	started from the step initial configuration and
	depending on the parameters $\{a_i^{-1}\}$ and $\be^{-1}$.
\end{proposition}
\begin{proof}
	The process $\{t+x_i(t)\}$
	jumps to the right, and, moreover, 
	each of its particles makes a jump precisely when 
	the corresponding $q$-PushTASEP particle $x_i(t)$
	stays put. 
	In particular, the first particle $x_1(t)$
	stays put with probability $1/(1+a_1\be)=(a_1^{-1}\be^{-1})/(1+a_1^{-1}\be^{-1})$.
	Next, if 
	$x_1(t)$ stayed put, then $x_2(t)$
	stays put with probability 
	$1/(1+a_2\be)=(a_2^{-1}\be^{-1})/(1+a_2^{-1}\be^{-1})$.
	Otherwise, if $x_1(t)$ jumped to the left, then 
	$x_2(t)$
	stays put with probability
	\begin{align*}
		1- \frac{a_2\be+q^{\gap_2(t)}}{1+a_2\be}
		=
		\frac{1-q^{\gap_2(t)}}{1+a_2\be}
		=
		\frac{a_2^{-1}\be^{-1}(1-q^{\gap_2(t)})}{1+a_2^{-1}\be^{-1}}.
	\end{align*}
	We see that the particles 
	$\{t+x_i(t)\}$ indeed perform the Bernoulli $q$-TASEP evolution
	with the desired parameters.
\end{proof}
One can think that this coupling 
between the two particle systems on $\Z$
comes 
from the 
complementation procedure 
(\S \ref{sub:Complementation}) relating the corresponding 
two-dimensional dynamics.


\subsection{Nested contour integral formulas for $q$-moments} 
\label{sub:nested_contour_integral_formulas_for_q_moments}

The above coupling between the Bernoulli $q$-PushTASEP and the Bernoulli $q$-TASEP allows
to readily write down moment formulas for the former process:
\begin{theorem}\label{thm:Bernoulli_qpush_nested_contours}
	Let $\{x_i(t)\}_{t=0,1,\ldots}$
	be the Bernoulli $q$-PushTASEP jumping to the left,
	started from the step initial configuration.
	Fix $k\ge 1$. For all $t=0,1,2,\ldots$
	and all integers $N\ge n_1\ge n_2\ge \ldots\ge n_k\ge0$,
	\begin{multline}
		\E^{\textnormal{step}}\bigg(\prod_{i=1}^{k}q^{x_{n_i}(t)+n_i}\bigg)
		\\=
		\frac{(-1)^{k}q^{\frac{k(k-1)}2}}{(2\pi\i)^{k}}
		\oint 
		\ldots
		\oint
		\prod_{1\le A<B\le k}
		\frac{z_A-z_B}{z_A-qz_B}
		\prod_{j=1}^{k}
		\left(\prod_{i=1}^{n_j}
		\frac{1}{1-a_iz_j}\right)
		\left(\frac{1+q^{-1}\be z_j^{-1}}{1+\be z_j^{-1}}\right)^{t}
		\frac{dz_j}{z_j},
		\label{Bernoulli_qpush_nested_contours}
	\end{multline}
	where the contour of integration for each $z_A$
	contains $a_1^{-1},\ldots,a_N^{-1}$,
	and the contours $\{qz_B\}_{B>A}$,
	but not poles $0$ or $(-\be)$.
\end{theorem}
\begin{proof}
	Immediately follows from 
	Proposition \ref{prop:TASEP_coupling} and
	\cite[Theorem 2.1.(3)]{BorodinCorwin2013discrete}.
\end{proof}

\begin{remark}\label{rmk:determine_distribution}
	Since $x_i(t)+i\ge -t$ for any $i=1,\ldots,N$ and any $t\ge0$, 
	the $q$-moments in \eqref{Bernoulli_qpush_nested_contours} admit an
	a priori bound. Therefore, 
	for a fixed $t\ge0$
	they determine the distribution
	of the random variables $(x_1(t),\ldots,x_N(t))$.
\end{remark}
\begin{remark}\label{rmk:Fredholm_Bernoulli}
	The moment formula \eqref{Bernoulli_qpush_nested_contours}
	readily leads (similarly to \cite{BorodinCorwinSasamoto2012}) 
	to a Fredholm determinantal formula for the
	$q$-Laplace transform of $x_n(t)$ for any $n$. 
	A similar Fredholm determinantal formula already appeared in 
	\cite{BorodinCorwinFerrariVeto2013}
	(based on the technique of Macdonald difference operators,
	cf. \cite{BorodinCorwin2011Macdonald}). 
	We discuss this and more general two-sided
	formulas in \S \ref{sub:fredholm_determinants}
	below (in particular, see Proposition
	\ref{prop:two-sided_Fredholm}. See
	also \S \ref{sub:remark_geometric_q_pushtasep_formulas}
	for a related discussion of the case of the geometric $q$-PushTASEP.
\end{remark}
\begin{remark}\label{rmk:direct_proof_of_nested}
	One can also establish the nested contour
	integral formula \eqref{Bernoulli_qpush_nested_contours}
	directly, similarly to \cite{BorodinCorwin2013discrete}
	(see also \cite{CorwinPetrov2015}).
	Indeed, denote
	\begin{align*}
		I_t(\vec y):=q^{t(y_1+y_2+\ldots+y_N)}\E^{\textnormal{step}}
		\bigg(\prod_{i=0}^{N}q^{y_i(x_i(t)+i)}\bigg),
	\end{align*}
	where $(y_0,\ldots,y_N)\in\Z^{N}_{\ge0}$
	and, by agreement, the product is zero
	if $y_0>0$.\footnote{One should think that the
	variables $y_j$ encode the $n_i$'s 
	in \eqref{Bernoulli_qpush_nested_contours}:
	each $y_j$ denotes the number
	of $n_i$'s which are equal to $j$.}
	One can directly show that
	these quantities satisfy 
	certain linear equations in the $y_j$'s.
	For each $i=1,\ldots,N$, 
	consider the following difference operators
	acting on functions in $\vec y$:
	\begin{align}\label{Hi_oper}
		[\mathcal{H}^{q,\muq}]_{i}f(\vec y)
		:=\sum_{s_i=0}^{y_i}\Om_{q,a_i^{-1}\muq,0}(s_i\mid y_i)
		f(y_0,y_1,\ldots,y_{i-2},y_{i-1}+s_i,y_i-s_i,y_{i+1},\ldots,y_N).
	\end{align}
	Here the quantities $\Om$ 
	are defined in \eqref{Om_qmunu_definition}.

	Also, denote by $\mathcal{H}^{q,\muq}$
	the operator which acts as 
	$[\mathcal{H}^{q,\muq}]_{i}$
	in each variable $y_i$:
	\begin{align}\label{H_oper}
		\mathcal{H}^{q,\muq}:=
		[\mathcal{H}^{q,\muq}]_{N}
		[\mathcal{H}^{q,\muq}]_{N-1}
		\ldots
		[\mathcal{H}^{q,\muq}]_{1}.
	\end{align}
	Applying operators 
	$[\mathcal{H}^{q,\muq}]_{i}$
	in this order corresponds to 
	first changing $y_1$ (by decreasing it by $s_1$), 
	then $y_{2}$ (by sending $s_2$ to $y_1-s_1$), etc.,
	up to $y_N$.
	In other words, these changes 
	(encoded by $s_1,\ldots,s_N$)
	happen in parallel, simultaneously 
	with each of $y_1,y_2,\ldots,y_N$.
		
	One can then show that 
	for any $t=0,1,2,\ldots$
	and any $\vec{y}=(y_0,y_1,\ldots,y_N)\in\Z_{\ge0}^{N+1}$, 
	the quantities $I_t(\vec y)$ satisfy 
	\begin{align*}
		\mathcal{H}^{q,-\be^{-1}}I_{t+1}(\vec y)=
		\mathcal{H}^{q,-q\be^{-1}}I_{t}(\vec y).
	\end{align*}

	These linear equations 
	can then be solved by the coordinate
	Bethe ansatz technique, because the action of each of 
	the operators $\mathcal{H}^{q,-\be^{-1}}$
	and $\mathcal{H}^{q,-q\be^{-1}}$
	reduces to the action of a free operator (i.e., which acts
	on each of the variables $n_i$ separately; note the identification 
	of the $y_j$'s and $n_i$'s in the previous footnote)
	plus two-body boundary conditions. This immediately leads to 
	the desired nested contour integral formula.
\end{remark}


\subsection{Remark. Geometric $q$-PushTASEP formulas} 
\label{sub:remark_geometric_q_pushtasep_formulas}

There are also nested contour integral formulas for $q$-moments of the geometric
$q$-PushTASEP (\S \ref{sub:geometric_q_pushtasep}). 
They can be obtained directly using the definition of the 
dynamics, similarly to the approach outlined in Remark \ref{rmk:direct_proof_of_nested}.
The moment formulas (for the geometric
$q$-PushTASEP jumping to the left)
will have the same form as in \eqref{Bernoulli_qpush_nested_contours},
with the following replacement of factors:
\begin{align*}
	\prod_{j=1}^{k}\left(\frac{1+q^{-1}\be z_j^{-1}}{1+\be z_j^{-1}}\right)^{t}
	\longrightarrow
	\prod_{j=1}^{k}\frac{1}{(1-\al q^{-1} z_j^{-1})^{t}}.
\end{align*}
However, because particles in the geometric
$q$-PushTASEP can jump arbitrarily far to the left (at least as far as by 
independent $q$-geometric jumps), only a finite number of $q$-moments
of the form
$\E^{\textnormal{step}}\left(\prod_{i=1}^{k}q^{x_{n_i}(t)+n_i}\right)$
exists. Therefore, these $q$-moments do not determine the 
distribution of the geometric
$q$-PushTASEP. 

One can overcome this issue and write down a Fredholm determinantal
formula for the distribution of the geometric
$q$-PushTASEP via a certain analytic continuation 
from the Bernoulli case. This analytic continuation
is performed in \cite[Section 3]{BorodinCorwinFerrariVeto2013}
and heavily relies on 
properties of $q$-Whittaker symmetric functions.
Namely, \cite[Theorem 3.3]{BorodinCorwinFerrariVeto2013} contains
a Fredholm determinantal
expression for the distribution of the 
particle $\la^{(N)}_{1}$ under a $q$-Whittaker process.
Our results on RSK-type dynamics (\S \ref{sec:geometric_q_rsks})
show that the same distribution arises under the geometric
$q$-PushTASEP started from the step initial configuration,\footnote{In other 
words, the geometric
$q$-PushTASEP provides a \emph{coupling} of
the quantities $\la^{(N)}_{1}$ arising
from $q$-Whittaker processes $\MP_{\mathbf{A}}^{\vec a}$
(\S \ref{sub:qwhit_proc})
differing 
by adding usual parameters to the specialization 
$\mathbf{A}$.}
thus \cite[Theorem 3.3]{BorodinCorwinFerrariVeto2013}
provides a 
Fredholm determinantal formula for the 
geometric
$q$-PushTASEP.


\subsection{Fredholm determinantal formula for the continuous time $q$-PushASEP} 
\label{sub:fredholm_determinants}
 
The moment formulas of Theorem \ref{thm:Bernoulli_qpush_nested_contours}
combined with certain spectral ideas allow us to 
establish Conjecture 1.4 from \cite{CorwinPetrov2013}
concerning a Fredholm determinantal formula for the 
continuous time $q$-PushASEP.

For simplicity, from now on we assume that all parameters $a_j\equiv1$. 
Let us recall the definition of the $q$-PushASEP.

\begin{definition}\label{def:cont_qpush}
	The \emph{continuous time $q$-PushASEP} 
	$\xx(\tau)$
	depending on parameters $L,R\ge0$ 
	lives on 
	particle configurations 
	$(\xx_N<\xx_{N-1}< \ldots<\xx_1)$
	on $\Z$ and evolves in continuous time $\tau$ 
	as follows:
	\begin{itemize}
		\item Each particle $\xx_i$ jumps to the right by one at rate $R(1-q^{\xx_{i-1}(\tau)-\xx_{i}(\tau)-1})$
		(by agreement, $\xx_0\equiv +\infty$).

		\item Each particle $\xx_i$ jumps to the left by one at rate $L$.
		If a particle $\xx_i$ has jumped to the left, then it has the possibility
		to push its left neighbor $\xx_{i+1}$ with probability $q^{\xx_i(\tau)-\xx_{i+1}(\tau)-1}$.
		(Pushing means that $\xx_{i+1}$ instantaneously moves to the left by one.)
		If the pushing of $\xx_{i+1}$ has occured, then $\xx_{i+1}$
		may push $\xx_{i+2}$ with the corresponding probability
		$q^{\xx_{i+1}(\tau)-\xx_{i+2}(\tau)-1}$, and so on.
	\end{itemize}
	We are assuming that this process starts from the step initial configuration
	$\xx_i(0)=-i$.
\end{definition}
\begin{remark}\label{rmk:qpush_limit}
	When $\be\searrow0$ and one rescales the discrete 
	time to the continuous one as
	$\sim\tau/\be$,
	the Bernoulli $q$-TASEP (Fig.~\ref{fig:Bernoulli_TASEP})
	clearly converges to $L=0$, $R=1$ version of the above process 
	(this is the pure, one-sided $q$-TASEP).
	Under the same limit, the
	Bernoulli $q$-PushTASEP (Fig.~\ref{fig:Bernoulli})
	becomes the $L=1$, $R=0$ version (pure $q$-PushTASEP).
\end{remark}

\begin{theorem}[{\cite[Conjecture 1.4]{CorwinPetrov2013}}]\label{claim:Fredholm}
	Let $\xx_n(\tau)$ be the $n$-th particle of the $q$-PushASEP
	started from the step initial configuration.
	For all $\zeta\in\C\setminus\R_{>0}$,
	\begin{align}
		\E\left(\frac{1}{(\zeta q^{\xx_n(\tau)+n};q)_{\infty}}\right)
		=\det(I+K_{\zeta}).
		\label{Fredholm_continuous}
	\end{align}
	Here $\det(I+K_{\zeta})$ is the Fredholm determinant
	of $K_{\zeta}\colon L^{2}(C_1)\to L^{2}(C_1)$,
	where $C_1$ is a small positively 
	oriented circle containing 1, and $K_{\zeta}$ 
	is an integral operator with kernel
	\begin{align*}
		K_{\zeta}(w,w')=\frac{1}{2\pi\i}
		\int_{-\i\infty+1/2}^{{\i\infty+1/2}}
		\frac{\pi}{\sin(-\pi s)}(-\zeta)^{s}
		\frac{(wq^{s};q)_{\infty}^{n}}{(w;q)_{\infty}^{n}}
		\frac{e^{\tau Rw(q^{s}-1)}e^{\tau Lw^{-1}(q^{-s}-1)}}{q^{s}w-w'}ds.
	\end{align*}
\end{theorem}
The $L=0$ case of the above theorem 
corresponds to the $q$-TASEP and was established
in
\cite{BorodinCorwin2011Macdonald}, \cite{BorodinCorwinSasamoto2012}. 
The $R=0$ case (one-sided $q$-PushTASEP) 
with different contours
is contained in 
\cite[Theorem 3.3]{BorodinCorwinFerrariVeto2013},
and it is obtained by an analytic continuation 
from the Bernoulli case.\footnote{The contours used
in \cite{BorodinCorwinFerrariVeto2013}
seem to be more suitable to 
the particular type of asymptotic analysis performed
in that paper. Here we do not pursue
further discussion of integration contours.} This analytic continuation relies 
on algebraic properties of $q$-Whittaker symmetric functions.
The corresponding algebraic picture for the two-sided dynamics
is not developed (cf. the discussion in \cite[Appendix A]{CorwinPetrov2013}).
To establish the above theorem we will utilize
a certain two-sided process
at the Bernoulli level 
whose $\be\searrow0$ limit leads to the distribution of particles under the
two-sided $q$-PushASEP at any given time $\tau>0$.

\medskip

The rest of this subsection is devoted to the proof of the above theorem.

\begin{definition}\label{def:two-part}
	Fix $t\in\Z_{\ge0}$, and
	let $\xxi^{(\be)}(t)=(\xxi_N^{(\be)}(t)<\ldots<\xxi_1^{(\be)}(t))\in\Z^{N}$ be a random vector
	which encodes positions of particles  
	started from $\xxi_i^{(\be)}(0)=-i$ ($i=1,\ldots,N$),
	which have evolved according to the Bernoulli $q$-TASEP (with right jumps)
	for $\lfloor tR\rfloor$ steps,\footnote{Here and below, 
	$\lfloor\cdots\rfloor$ means the floor function.} and then 
	according to the Bernoulli $q$-PushTASEP (with left jumps)
	for $\lfloor tL\rfloor$ steps. Both discrete time processes are now assumed to
	depend on the same parameter $\be$ (also recall
	that $a_j\equiv1$).
\end{definition}
It is possible to write down $q$-moments
of this random vector:
\begin{proposition}
	Fix $k\ge 1$. For all $t=0,1,2,\ldots$
	and all $N\ge n_1\ge n_2\ge \ldots\ge n_k\ge0$,
	\begin{multline}
		\E^{\textnormal{step}}\bigg(\prod_{i=1}^{k}q^{\xxi^{(\be)}_{n_i}(t)+n_i}\bigg)
		=
		\frac{(-1)^{k}q^{\frac{k(k-1)}2}}{(2\pi\i)^{k}}
		\oint 
		\ldots
		\oint
		\prod_{1\le A<B\le k}
		\frac{z_A-z_B}{z_A-qz_B}
		\\\times\prod_{j=1}^{k}
		\frac{1}{(1-z_j)^{n_j}}
		\left(\frac{1+q\be z_j}{1+\be z_j}\right)^{\lfloor tR\rfloor}
		\left(\frac{1+q^{-1}\be z_j^{-1}}{1+\be z_j^{-1}}\right)^{\lfloor tL\rfloor}
		\frac{dz_j}{z_j},
		\label{2Bernoulli_qpush_nested_contours}
	\end{multline}
	where the contour of integration for each $z_A$
	contains $1$,
	and the contours $\{qz_B\}_{B>A}$,
	but not poles $0$ or $(-\be)^{\pm1}$.
\end{proposition}
\begin{proof}
	Let $\mathcal{T}_{right}$ and $\mathcal{T}_{left}$
	denote the one-step transition operators
	of the Bernoulli $q$-TASEP and the 
	Bernoulli $q$-PushTASEP acting in $\vec x$ variables, respectively. 
	Also denote
	\begin{align*}
		\mathsf{H}(\vec x,\vec y):=\prod_{i=0}^{N}q^{y_i(x_i(t)+i)}
	\end{align*}
	(by agreement, the product is zero
	if $y_0>0$).

	Then it follows from the results of \cite{BorodinCorwin2013discrete}
	and the discussion of Remark \ref{rmk:direct_proof_of_nested}
	that
	\footnote{We are writing ``transpose''
	simply to indicate that the operators 
	in the right-hand side act in variables $\vec y$.}
	\begin{align*}
		\mathcal{T}_{right}
		\mathsf{H}&=
		\mathsf{H}
		\left(
		(\mathcal{H}^{q,-\be})^{-1}
		\mathcal{H}^{q,-q\be}
		\right)^{transpose},
		\\
		\qquad
		\mathcal{T}_{left}
		\mathsf{H}&=
		\mathsf{H}
		\left(
		q^{-(y_1+\ldots+y_N)}(\mathcal{H}^{q,-\be^{-1}})^{-1}
		\mathcal{H}^{q,-q\be^{-1}}
		\right)^{transpose},
	\end{align*}
	where the operators $\mathcal{H}^{q,\xi}$
	are defined in
	\eqref{Hi_oper}--\eqref{H_oper},
	and $q^{-(y_1+\ldots+y_N)}$ is the multiplication operator.
	Observe that the operators $\mathcal{H}^{q,\xi}$
	applied to $\mathsf{H}$ do not change
	$y_1+\ldots+y_N$ and hence commute with this multiplication operator.
	The inverse operators above exist on a certain space 
	$\mathcal{W}_{\max}$ of functions in $\vec{y}$, see 
	\cite{CorwinPetrov2015}.

	Applying the Plancherel theory \cite{BorodinCorwinPetrovSasamoto2013}
	(see also \cite{CorwinPetrov2015}),
	we now see that the application 
	of the product
	$(\mathcal{T}_{left})^{\lfloor tL\rfloor}
	(\mathcal{T}_{right})^{\lfloor tR\rfloor}$
	which is needed to compute the expectation
	in the left-hand side of \eqref{2Bernoulli_qpush_nested_contours},
	reduces to multiplication by the 
	corresponding eigenvalue
	$\prod_{j=1}^{k}
	\big(\frac{1+q\be z_j}{1+\be z_j}\big)^{\lfloor tR\rfloor}
	\big(\frac{1+q^{-1}\be z_j^{-1}}{1+\be z_j^{-1}}\big)^{\lfloor tL\rfloor}$
	under the nested contour integral.
	This completes the proof.
\end{proof}
\begin{remark}\label{rmk:T_T_commute}
	Similarly to Remark \ref{rmk:determine_distribution}, 
	one sees that the $q$-moments
	\eqref{2Bernoulli_qpush_nested_contours}
	determine the distribution of the vector $\xxi^{(\be)}(t)$.
	This implies that any powers of operators 
	$\mathcal{T}_{right}$ and $\mathcal{T}_{left}$
	(defined above)
	\emph{commute} when applied to the step initial configuration, 
	i.e., they yield the same probability distribution regardless 
	of the order of application.
\end{remark}
\begin{proposition}\label{prop:two_part_to_two_sided}
	As $\be\searrow0$, 
	the random variables 
	$\xxi_{n}^{(\be)}(\lfloor \tau/\be\rfloor)$
	defined above
	converge in distribution to
	$\xx_n(\tau)$
	(the latter 
	is the particle position
	coming from the
	continuous time $q$-PushASEP started from the step initial configuration).
\end{proposition}
\begin{proof}
	Let $\mathcal{G}_{right}$ denote the infinitesimal generator
	of the pure $q$-TASEP process (with only right jumps)
	corresponding to Definition \ref{def:cont_qpush}
	with $L=0$, $R=1$.
	Similarly, let 
	$\mathcal{G}_{left}$ denote the 
	infinitesimal generator
	of the pure $q$-PushTASEP (with only left jumps)
	corresponding to $L=1$, $R=0$.

	It follows from Remark \ref{rmk:T_T_commute}
	that the semigroups $e^{\tau \mathcal{G}_{right}}$
	and $e^{\tau \mathcal{G}_{left}}$
	commute when applied to the step initial configuration
	(in the same sense as in Remark \ref{rmk:T_T_commute}).
	This means that the semigroup
	of the two-sided $q$-PushASEP
	has the form
	\begin{align*}
		e^{\tau (R\mathcal{G}_{right}+L\mathcal{G}_{left})}
		=
		e^{\tau R\mathcal{G}_{right}}
		e^{\tau L\mathcal{G}_{left}}
	\end{align*}
	(when applied to the step initial configuration).
	Therefore, 
	$\xx_n(\tau)$ has the same distribution
	as if it arises by first running the 
	pure $\mathcal{G}_{right}$ process for time $\tau R$,
	and then the 
	pure $\mathcal{G}_{left}$ process for time $\tau L$.
	The latter two-part evolution clearly is the $\be\searrow0$ limit 
	of $\xxi_{n}^{(\be)}(\lfloor \tau/\be\rfloor)$
	(cf. Remark~\ref{rmk:qpush_limit}),
	as desired.
\end{proof}

\begin{proposition}\label{prop:two-sided_Fredholm}
	Recall the random vector $\xxi^{(\be)}(t)$ of Definition
	\ref{def:two-part}. For all $\zeta\in\C\setminus\R_{>0}$,
	\begin{align}
		\E\left(\frac{1}{(\zeta q^{\xxi^{(\be)}_n(t)+n};q)_{\infty}}\right)
		=\det(I+K_{\zeta}^{(\be)}).\label{Fredholm_beta}
	\end{align}
	Here $\det(I+K_{\zeta}^{(\be)})$ is the Fredholm determinant
	of $K_{\zeta}^{(\be)}\colon L^{2}(C_1)\to L^{2}(C_1)$,
	where $C_1$ is a small positively 
	oriented circle containing 1, and $K_{\zeta}^{(\be)}$ 
	is an integral operator with kernel
	\begin{multline}
		K_{\zeta}^{(\be)}(w,w'):=\frac{1}{2\pi\i}
		\int_{-\i\infty+1/2}^{{\i\infty+1/2}}
		\frac{\pi}{\sin(-\pi s)}(-\zeta)^{s}
		\\\times\frac{(wq^{s};q)_{\infty}^{n}}{(w;q)_{\infty}^{n}}
		\bigg(\frac{1+\be q^{s}w}{1+\be w}\bigg)^{\lfloor tR\rfloor}
		\bigg(\frac{1+\be (q^{s}w)^{-1}}{1+\be w^{-1}}\bigg)^{\lfloor tL\rfloor}
		\frac{1}{q^{s}w-w'}
		ds.
		\label{K_beta_kernel}
	\end{multline}
\end{proposition}
\begin{proof}
	The passage from the moment formulas \eqref{2Bernoulli_qpush_nested_contours}
	to the desired Fredholm determinantal formula
	is done similarly to \cite{BorodinCorwinSasamoto2012}
	(based on \cite{BorodinCorwin2011Macdonald}). 
	Namely, for $|\zeta|$ sufficiently small, the
	Fredholm determinant can be obtained by purely algebraic
	manipulations. Then one must show that the
	resulting right-hand side of \eqref{Fredholm_beta}
	is analytic in $\zeta\in\C\setminus\R_{>0}$,
	which follows from bounds like in
	\cite[Proposition 3.6]{BorodinCorwinSasamoto2012}
	and can be readily checked in our situation. The 
	left-hand side of \eqref{Fredholm_beta}
	is also analytic in $\zeta$
	because the function $\zeta\mapsto (\zeta;q)_{\infty}$
	is uniformly bounded away from zero and analytic
	once $\zeta$ is bounded away from $\R_{>0}$.
	Therefore, the desired claim holds.
\end{proof}

Theorem \ref{claim:Fredholm} now follows by observing that
the kernels $K^{(\be)}_{\zeta}$ converge,
as $\be\searrow0$ and $t=\lfloor \tau/\be \rfloor$,
to the kernel $K_{\zeta}$ from Theorem \ref{claim:Fredholm},
because
$q^{s}$ and $w$ are uniformly bounded on our contours. This also implies the 
convergence of the corresponding Fredholm determinants. 
On the other hand, by Proposition \ref{prop:two_part_to_two_sided},
we know that the left-hand sides of 
\eqref{Fredholm_beta} converge to the 
left-hand side of \eqref{Fredholm_continuous}.
This establishes Theorem \ref{claim:Fredholm}.



\section{Polymer limits of $(\al)$ dynamics on $q$-Whittaker processes} 
\label{sec:polymer_limit}

In this section we explain how the two $(\al)$
dynamics on $q$-Whittaker processes
behave in the limit as $q\nearrow 1$.
This leads to discrete time
stochastic processes related to geometric RSK correspondences
and directed random polymers.

\subsection{Polymer partition functions} 
\label{sub:polymer_partition_functions}

Let us first describe the polymer models we will be dealing with.
They are based on inverse-Gamma random variables:

\begin{definition}
	A positive random variable $X$ has \emph{Gamma distribution} with shape parameter $\theta > 0$ if it has probability density 
	\begin{align*}
	P(X \in dx) = \frac{1}{\Gamma(\theta)}x^{\theta-1}e^{-x}dx.
	\end{align*}
	We abbreviate this by $X \sim$ $\GRV{\theta}$.  Then $X^{-1}$ has probability density  
	\begin{align*}
	P(X^{-1} \in dx) = \frac{1}{\Gamma(\theta)}x^{-\theta-1}e^{-1/x}dx,
	\end{align*}
	which is called \emph{inverse-Gamma distribution} and denoted by $\GRVI{\theta}$. 
\end{definition}

We recall partition functions of two models of log-Gamma polymers in $1+1$ dimensions studied previously in \cite{Seppalainen2012}, \cite{BorodinCorwinRemenik}, \cite{COSZ2011}, \cite{OSZ2012}, \cite{OConnellOrtmann2014}, \cite{CorwinSeppalainenShen2014}
(see also \cite{Oconnell2009_Toda} for a continuous time version). 
Both models are defined on the lattice strip $\{(t, j) \ | \ t \in \{0, 1, 2, \ldots\}, j \in \{ 1, 2, \ldots, n \}\}$. One should think of $t$ as time. Suppose we have two collections of real numbers $\theta_{j}$ for  $j \in \{ 1, 2, \ldots, n \}$ and $\hat \theta_{t}$ for $t \in \{0, 1, 2, \ldots\}$, such that $\theta_{j} + \hat \theta_{t} > 0$ for all $j$ and $t$.

\begin{definition}[Log-Gamma polymer \cite{Seppalainen2012}; Fig.~\ref{fig:ppf}, left]\label{def:R_polymer}
Each vertex $(t, j)$ in the strip 
is equipped with a random weight $d_{t, j}$. 
These weights are independent, and $d_{t, j}$  is 
distributed according to $\GRVI{\theta_{j} + \hat \theta_{t}}$.  
The \emph{log-Gamma polymer partition function} with parameters $\theta_{j}, \hat \theta_{t}$ is given by
\begin{align} \label{R1j}
R^{j}_{1}(t) := \sum_{\pi: (1, 1) \to (t, j)} \prod_{(s, i) \in \pi}d_{s, i},
\end{align}
where the sum is over directed up/right lattice paths $\pi$ from $(1, 1)$ to $(t, j)$, which are made of horizontal edges $(s, i) \to (s+1, i)$ and vertical edges $(s, i) \to (s, i+1)$. Extend this definition to denote by $R^{j}_{k}(t)$ for $t \geq k$ the weighted sum over all $k$-tuples of \emph{nonintersecting} up/right lattice paths starting from $(1, 1), (1, 2), \ldots, (1, k)$ and going respectively to ${(t, j-k+1)}, {(t, j-k+2)}, \ldots,$ $(t, j)$. The weight of a tuple of paths is defined by taking a product of weights of vertices of these paths. The inequality $t \geq k$ ensures that $R^{j}_{k}(t)$ is positive.
\end{definition}

\begin{definition}[Strict-weak polymer \cite{CorwinSeppalainenShen2014}, \cite{OConnellOrtmann2014};\footnote{These two
papers independently introduce essentially the same model. 
We will be using the notation of~\cite{CorwinSeppalainenShen2014}.}
Fig.~\ref{fig:ppf}, right]
\label{def:L_polymer}
	Each horizontal edge $e$ in the strip is equipped with a random weight $d_{e}$. These weights are independent, and $d_{(t-1, j) \to (t, j)}$ is distributed according to $\GRV{\theta_{j} + \hat \theta_{t}}$.  The \emph{strict-weak polymer partition function} with parameters $\theta_{j}, \hat \theta_{t}$ is given by 
	\begin{align} \label{L1j}
	L^{j}_{1}(t) := \sum_{\pi: (0, 1) \to (t, j)} \prod_{e \in \pi}d_{e},
	\end{align} 
	where the sum is over directed lattice paths from $(0, 1)$ to $(t, j)$ which are made of horizontal edges $(s, i) \to (s+1, i)$ and diagonal moves $(s, i) \to (s+1, i+1)$. The product is taken only over horizontal edges of the path. Extend this definition to denote by $L^{j}_{k}(t)$ for $t \geq j-k$ the weighted sum over all $k$-tuples of the corresponding
	\emph{nonintersecting} lattice paths starting from $(0, 1), (0, 2), \ldots, (0, k)$ and going respectively to ${(t, j-k+1)}, {(t, j-k+2)}, \ldots,$ $(t, j)$. The weight of a tuple of paths is defined by taking a product of weights of horizontal edges of these paths. The inequality $t \geq j-k$ ensures that $L^{j}_{k}(t)$ is positive.
\end{definition}

\begin{figure}[h] 
\hspace*{0 cm} \includegraphics[width = 0.8\textwidth]{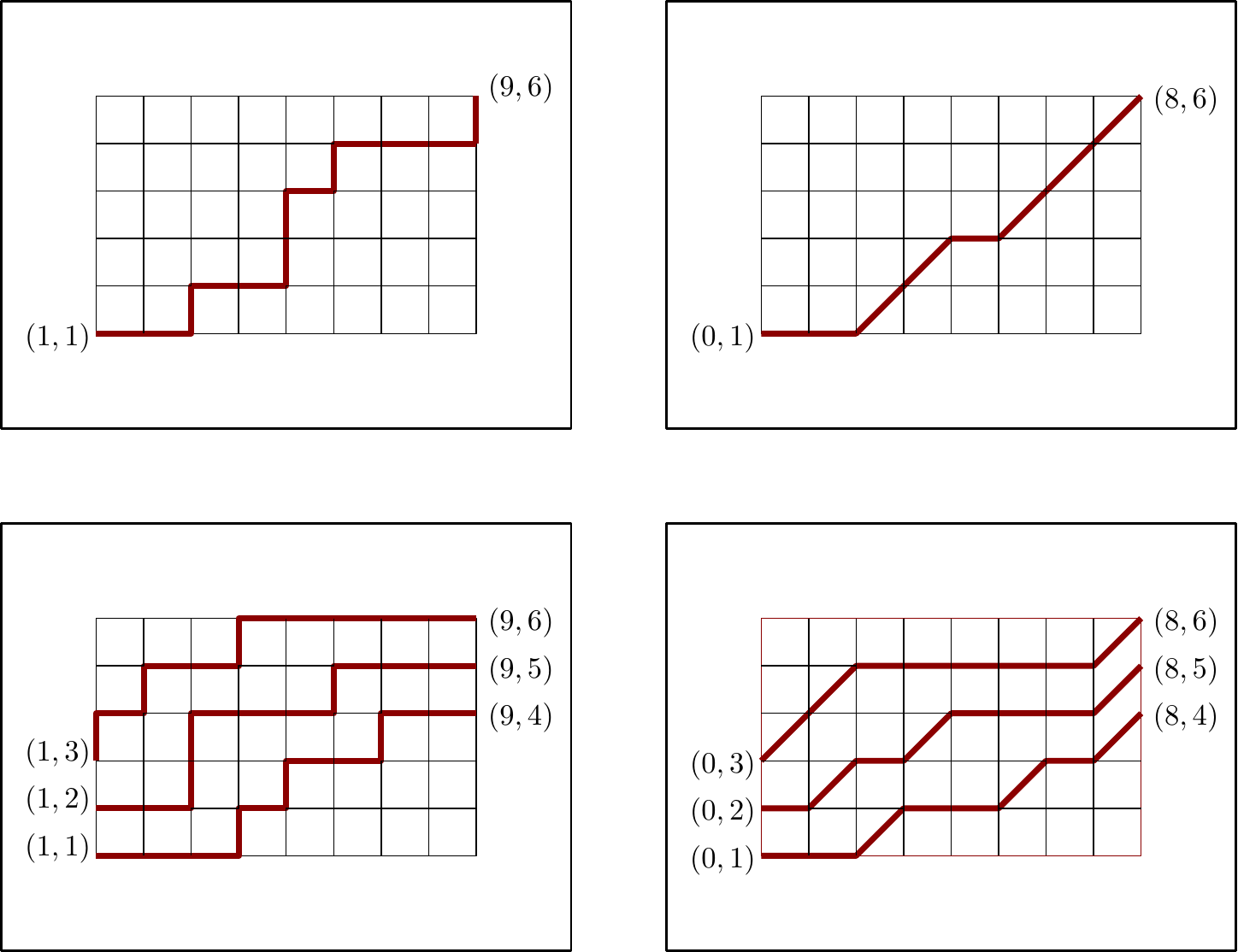}
\caption{Paths and tuples of paths that contribute to the polymer partition functions: 
$R^{5}_{1}(9)$ (top left),
$R^{5}_{3}(9)$ (bottom left),
$L^{6}_{1}(8)$ (top right),
$L^{6}_{3}(8)$ (bottom right).}
\label{fig:ppf}
\end{figure}

Distributions of 
ratios of the polymer partition functions
defined above
are sometimes called
\emph{Whittaker processes} (or, to be more precise, \emph{$\al$-Whittaker processes}), cf. \cite{BorodinCorwin2011Macdonald}.
They arise
as 
limits (as $q$, the $a_j$'s and the $\al_t$'s simultaneously go to $1$) 
of suitably rescaled particle positions in an interlacing integer array 
distributed according to the $q$-Whittaker process $\MP_{\mathbf{A}}^{\vec a}$
(\S \ref{sub:qwhit_proc}), where 
\begin{align*}
	\mathbf{A} = (\al_{1}, \ldots, \al_{t}),
	\qquad
	{\vec a} = (a_{1}, \ldots, a_{n}).
\end{align*}
The convergence
of $q$-Whittaker processes to Whittaker processes
is known in the literature, 
see 
\cite[Thm. 4.2.4]{BorodinCorwin2011Macdonald}. 
Since both dynamics
$\Qqarow$ and $\Qqacol$ 
constructed in \S \ref{sec:geometric_q_rsks}
sample the $q$-Whittaker processes, we can employ them 
to give another proof of this limit
transition. Moreover, we also establish 
the convergence of the corresponding
\emph{stochastic dynamics}. 

Let us first define the appropriately scaled pre-limit dynamics.
In what follows, for $\epsilon > 0$ and $\theta_{j}, \hat \theta_{t}$ as above, we set $q := e^{-\epsilon}$, $a_{j} = e^{- \theta_{j}\epsilon}$ and $\alpha_{t} := e^{-\hat \theta_{t}\epsilon}$.

\begin{definition}[Scaled $\Qqarow$ dynamics]\label{def:row_scaled}
	Start the dynamics 
	$\Qqarow$
	from the zero initial condition (that is, $\la^{(j)}_{i}(0)\equiv 0$).
	Denote by $r_{j, k}(t, \epsilon)$ the position of the $k$-th particle \emph{from the right} on the $j$-th level of the array after $t$ steps of the dynamics (at each time step $t\to t+1$, 
	apply the dynamics $\Qqarow$ with parameter $\al=\al_{t+1}$).
	For $t \geq k$, define the random variables $\hat R^{j}_{k}(t, \epsilon)$ via
	\begin{align*}
		r_{j,k}(t, \epsilon) = (t+j-2k+1) \epsilon^{-1} \log \epsilon^{-1} +  \epsilon^{-1} \log(\hat R^{j}_{k}(t, \epsilon)).
	\end{align*} 
	The reason for the restriction $t \geq k$ comes from the fact that $r_{j, k}(t, \epsilon)=0$ 
	for $t < k$. 

	We will view the collection of random variables $\{\hat R^{j}_{k}(t, \epsilon)\}$
	as a stochastic process $\hat R(t, \epsilon)$
	which at a fixed time $t$ becomes 
	an array $\hat R^{j}_{k}(t, \epsilon)$, $1 \leq k \leq j \leq n$ for $t \geq n$, 
	or a truncated array $\hat R^{j}_{k}(t, \epsilon)$, $1 \leq k \leq \min\{t, j\} \leq n$ for $0 < t < n$. 
\end{definition}

\begin{definition}[Scaled $\Qqacol$ dynamics]\label{def:col_scaled}
	Start the dynamics 
	$\Qqacol$
	from the zero initial condition,
	and
	denote by $\ell_{j, k}(t, \epsilon)$ the position of the $k$-th particle \emph{from the left} on the $j$-th level 
	of the array after $t$ steps of the dynamics
	(again, at each time step $t\to t+1$, 
	apply the dynamics $\Qqacol$ with parameter $\al=\al_{t+1}$). For $t \geq j-k+1$, 
	define the random variable $\hat L^{j}_{k}(t, \epsilon)$ via 
	\begin{align*}
		\ell_{j, k}(t, \epsilon) = (t-j+2k-1) \epsilon^{-1} \log \epsilon^{-1} -  \epsilon^{-1} \log(\hat L^{j}_{k}(t, \epsilon)).
	\end{align*}
	The reason for the restriction $t \geq j-k+1$ 
	comes from the fact that 
	$\ell_{j, k}(t, \epsilon)=0$
	for $t < j-k+1$. 

	We will view the collection of random variables $\{\hat L^{j}_{k}(t, \epsilon)\}$ as a stochastic 
	process $\hat L(t, \epsilon)$, which at a fixed time $t$ becomes 
	an array $\hat L^{j}_{k}(t, \epsilon)$, $1 \leq k \leq j \leq n$ for $t \geq n$, 
	or a truncated array $\hat L^{j}_{k}(t, \epsilon)$, $1 \leq k \leq j \leq \min\{n, k+t-1\}$ for $0 < t < n$. 
\end{definition}

\begin{remark}
Observe that
for a fixed time $t$, the array $r_{j, k}(t, \epsilon)$ has the same distribution as the array $\ell_{j, j-k+1}(t, \epsilon)$
(by Theorems \ref{thm:QqrRSK_alpha} and \ref{thm:QqcRSK_alpha},
they are distributed as $q$-Whittaker processes).
Hence the
(possibly truncated) arrays $\hat R^{j}_{k}(t, \epsilon)$ and $1/\hat L^{j}_{j-k+1}(t, \epsilon)$ for $1 \leq k \leq \min\{t, j\} \leq n$ have the same distribution. 
However, these arrays will not be identically distributed 
as \emph{stochastic processes} in $t$ since they come from different multivariate dynamics.
\end{remark}

In the setting of polymer partition functions, 
define random processes $\hat R(t)$ and $\hat L(t)$ on (possibly truncated) arrays via
\begin{align*}
\hat R^{j}_{k}(t) :=  R^{j}_{k}(t) / R^{j}_{k-1}(t), \quad \text{for $1 \leq k \leq \min\{t, j\} \leq n$}
\end{align*}
and
\begin{align*}
\hat L^{j}_{k}(t) := L^{j}_{k}(t) / L^{j}_{k-1}(t), \quad \text{for $1 \leq k \leq j \leq \min\{n, k+t-1\}$}.
\end{align*}
They are well defined, because $R^{j}_{k}(t), R^{j}_{k-1}(t) > 0$ for $t \geq k$ 
and $L^{j}_{k}(t), L^{j}_{k-1}(t) > 0$ for $t \geq j-k+1$.  

\medskip

We are now in a position to formulate results
on the limiting behavior of dynamics $\Qqarow$ and $\Qqacol$.
In this section we prove the following:

\begin{theorem}\label{thm:TR}
As $\epsilon \to 0$, the process $\hat R(t, \epsilon)$ of Definition \ref{def:row_scaled} 
converges in distribution to the process $\hat R(t)$.
\end{theorem}

\begin{theorem}\label{thm:TL}
As $\epsilon \to 0$, the process $\hat L(t, \epsilon)$ 
of Definition \ref{def:col_scaled}
converges in distribution to the process $\hat L(t)$.
\end{theorem}

\begin{corollary}\label{cor:arrays_R_L}
The (possibly truncated)
arrays $\hat R^{j}_{k}(t)$ and $1/\hat L^{j}_{j-k+1}(t)$ for $1 \leq k \leq \min\{t, j\} \leq n$ have the same distribution.
\end{corollary}
In particular, $1/\hat R^{j}_{j}(t)$ and $\hat L^{j}_{1}(t)$ have the same distribution. 
The latter fact was proven in \cite{OConnellOrtmann2014}, and was used 
to analyze the strict-weak polymer partition function 
via the geometric RSK row insertion (see \S \ref{ssub:geometric_rsk_row_insertion} below), 
and to establish the Tracy-Widom asymptotics for the strict-weak polymer. 

See also \cite{OConnellOrtmann2014} for the close relation between 
the log-gamma and strict-weak polymers, where it is explained 
that one is the complement of the other at the level of lattice paths.
To the best of our knowledge,
the full statement of Corollary \ref{cor:arrays_R_L}
has not previously appeared in the literature.

\medskip

Let us provide a brief outline of our proofs of Theorems 
\ref{thm:TR} and \ref{thm:TL} which are 
presented in the rest of this section. 
First, 
in \S \ref{sub:geometric_rsks} we describe the
constructions of the \emph{geometric RSK correspondences},
which will serve as $\epsilon\to0$
limits of elementary steps 
used in dynamics $\Qqarow$ and $\Qqacol$.
Then in \S \ref{sub:some_asymptotics_of_q_deformed_binomial_distributions}
we prove a number of lemmas concerning 
$\epsilon\to0$ behavior of the 
$q$-distributions from \S \ref{sub:the_q_deformed_binomial_distribution}.
Finally, in \S \ref{sub:proofs_of_theorems_thm:tr_and_thm:tl} we 
use these ingredients to establish the desired statements.


\subsection{Geometric RSKs} 
\label{sub:geometric_rsks}

As we already know, the dynamics
on $q$-Whittaker processes constructed in \S \ref{sec:geometric_q_rsks}
degenerate for $q=0$ into the
dynamics $\Qarow$ and $\Qacol$ based on the classical 
RSK row or column insertion, respectively.
In this subsection we describe the corresponding
\emph{geometric Robinson--Schensted--Knuth insertions},
which will serve as building blocks 
for understanding $q\nearrow 1$
limits of the dynamics
on $q$-Whittaker processes.

The $q=0$ and $q\nearrow1$
pictures (i.e., the classical and the geometric RSK correspondences)
are related via a certain procedure called \emph{detropicalization}.
Namely, the 
geometric RSK \emph{row insertion}
introduced in \cite{Kirillov2000_Tropical}
is obtained by detropicalizing
the classical RSK row insertion
by replacing the $(\max, +)$ operations in its definition by $(+, \times)$.
About the geometric RSK row insertion see also, e.g.,  
\cite{NoumiYamada2004}, 
\cite{COSZ2011}, 
\cite{OSZ2012},
and \cite{Chhaibi2013}.

By analogy with the geometric RSK row insertion, 
one can define the 
\emph{geometric RSK column insertion}, 
by detropicalizing the classical RSK column insertion, 
this time replacing the $(\min, +)$ operations by $(+, \times)$.

\begin{remark}[Names and notation]\label{rmk:names_notation}
	The geometric RSK correspondences 
	are also sometimes called 
	\emph{tropical RSK correspondences}
	\cite{Kirillov2000_Tropical},
	\cite{NoumiYamada2004}, 
	\cite{COSZ2011},
	despite the fact that they come from the process
	of detropicalization. We adopt a convention of calling them
	the geometric RSK correspondences 
	(following, e.g., \cite{OSZ2012}, 
	\cite{Chhaibi2013}, \cite{OConnell2013geomToda}). The latter name arises in connection
	with geometric crystals (see \cite{Chhaibi2013} for more background).

	Note that the 
	word ``geometric'' in the 
	name of the geometric RSK correspondences
	should be distinguished from the same word
	in the names of the
	geometric $q$-PushTASEP and the geometric $q$-TASEP
	(described in \S  \ref{sub:geometric_q_pushtasep} and \S \ref{sub:geometric_q_tasep}, respectively).
	The former refers to detropicalization
	of the classical RSK correspondences, while the latter
	is attached to the $q$-geometric jump distribution.

	Below in this section, by $\la,\nu,\ldots$ we will 
	denote vectors (words) with continuous components,
	and not signatures as before. To indicate the 
	difference, we will use superscripts to denote their components.
\end{remark}

\subsubsection{Geometric RSK row insertion} 
\label{ssub:geometric_rsk_row_insertion}

Consider a triangular array $z^{j}_{k}$ ($1 \leq k \leq j \leq n$) of nonnegative real numbers, such that a \emph{word} $z_{k} = (z^{k}_{k}, \ldots, z^{n}_{k})$ either has all positive entries or is equal to $(1, 0, \ldots, 0)$ (in which case we call it an {\it empty word}).  

First, define the \emph{geometric row insertion} 
of a nonempty word $a = (a^{k}, \ldots, a^{n})$ 
into a nonempty word $\lambda = (\lambda^{k}, \ldots, \lambda^{n})$ 
as an operation that takes the pair 
$\{\lambda, a\}$ as input, 
and produces a pair of words $\{ \nu =(\nu^{k}, \ldots, \nu^{n}), b = (b^{k+1}, \ldots, b^{n})\}$ 
as output 
via the following rule:
\begin{align*}
	\scalebox{.95}{
	\begin{tikzpicture}
		[scale=1, very thick, >=latex, ->]
		\draw[->] (0,1) -- (0,-1) node [pos=-.2] {$a$} node [pos=1.2] {$b$};
		\draw[->] (-1,0) -- (1,0) node [pos=-.2] {$\la$} node [pos=1.2] {$\nu$};
		\node[anchor=west] at (3,.6) {$\displaystyle\nu^{j}=\sum_{i=k}^{j}\la^{i}a^{i}\ldots a^{j}$};
		\node[anchor=west] at (3,-.6) {$\displaystyle b^{j}=a^{j}\frac{\la^{j}\nu^{j-1}}{\la^{j-1}\nu^{j}}$};
	\end{tikzpicture}
	}
\end{align*}
If $\lambda$ is an empty word, then by definition $b$ is not produced, while 
\begin{align*}
	\nu := (a^{k}, a^{k}a^{k+1},\ldots, a^{k}a^{k+1} \cdots a^{n})
\end{align*}
is produced according to the same rule. 
The word $b$
is also not produced for $k=n$. 
Observe that always $\nu^{j} = (\lambda^{j} + \nu^{j-1})a^{j}$ for $k < j \le n$ and $\nu^{k} = \la^{k} a^{k}$.

\begin{figure}[htb] 
	\center \includegraphics[width = 0.12\textwidth]{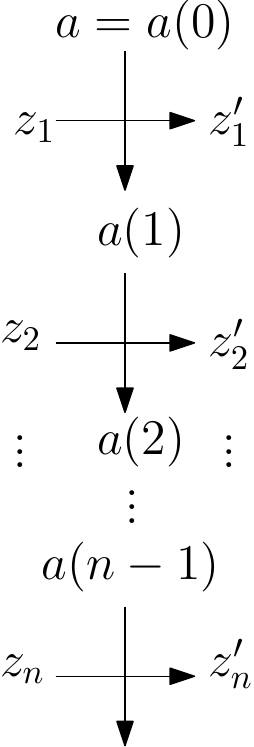}
	\caption{Geometric RSK row insertion.}
	\label{fig:grr}
\end{figure}
\begin{definition}\label{def:geom_RSK_row_insertion}
	The \emph{geometric RSK row insertion} of a word $a = (a^{1}, \ldots, a^{n})$ into an array $z^{j}_{k}$ 
	is defined by consecutively modifying the words 
	$z_{1}, \ldots, z_{n}$ via the insertion according to the diagram on Fig.~\ref{fig:grr}. 
	The bottom output word $a(1),a(2),\ldots$ 
	of each insertion is then used as a top input word for the next insertion. 
	If after some insertion no bottom output word is produced, then no further insertions are performed.  
\end{definition}

\begin{figure}[htbp]
\includegraphics[width = 0.8\textwidth]{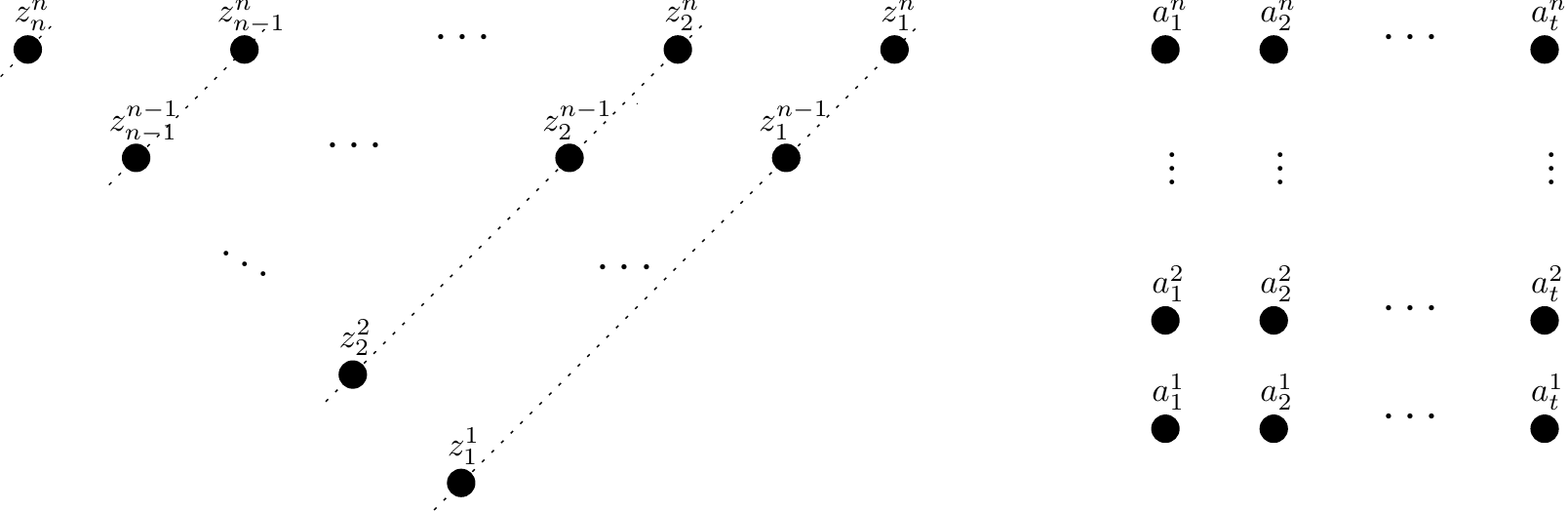}
\caption{Array and strip for the geometric RSK row insertion.}
\label{fig:grra}
\end{figure}
The geometric RSK row insertion is related to the
polymer partition functions of 
\S \ref{sub:polymer_partition_functions}
in the following way:
\begin{proposition}[\cite{NoumiYamada2004}]\label{prop:geom_RSK_row_insertion}
	If we start with an array $z$ 
	of empty words, and consecutively 
	insert into it 
	nonempty fixed words $a_{1}, \ldots, a_{t}$, $a_{i} = (a_{i}^{1}, \ldots, a_{i}^{n})$, 
	via the geometric RSK row insertion, then in the obtained array we have
	\begin{align*}
	z^{j}_{k}(t) = \frac{R^{j}_{k}(t)(a_{1}, \ldots, a_{t})}{R^{j}_{k-1}(t)(a_{1}, \ldots, a_{t})} 
	\qquad \text{for all $t \geq k$.}
	\end{align*}
	Here with a slight abuse of notation we denote by 
	$R^{j}_{k}(t)(a_{1}, \ldots, a_{t})$ the same weighted sum over $k$-tuples of 
	nonintersecting paths as in Definition \ref{def:R_polymer}, 
	but in a strip in which each node $(s, i)$ has a deterministic weight $a^{i}_{s}$
	(see Fig.~\ref{fig:grra}).
\end{proposition}


\subsubsection{Geometric RSK column insertion} 
\label{ssub:geometric_rsk_column_insertion}

Consider a triangular array $y^{j}_{k}$ ($1 \leq k \leq j \leq n$) of nonnegative real numbers, such that in each 
word $y_{k} = (y^{k}_{k}, \ldots, y^{n}_{k})$ either all entries are positive, 
or there is $k \leq j \leq n$, 
such that $y^{j}_{k} = 1$, $y^{i}_{k} = 0$ for $ j < i \leq n$ and $y^{i}_{k} > 0$ for $k \leq i \leq j$.  
We again call $(1, 0, \ldots, 0)$ an  empty word. 

To define the \emph{geometric RSK column insertion} first define the
insertion of a word $a = (a^{k}, \ldots, a^{n})$ with positive entries into 
a word $\lambda = (\lambda^{k}, \ldots, \lambda^{n})$ as an operation 
that takes the pair $\{\lambda, a\}$ as input, 
and produces a pair of words $\{\nu =(\nu^{k}, \ldots, \nu^{n}), b = (b^{k+1}, \ldots, b^{n})\}$ 
as output
via the following rule:
\begin{align*}
	\scalebox{.95}{
	\begin{tikzpicture}
		[scale=1, very thick, >=latex, ->]
		\draw[->] (0,1.2) -- (0,-.8) node [pos=-.2] {$a$} node [pos=1.2] {$b$};
		\draw[->] (-1,0.2) -- (1,0.2) node [pos=-.2] {$\la$} node [pos=1.2] {$\nu$};
		\node[anchor=west] at (3,1.7) {$\nu^{k}=a^{k}\la^{k}$};
		\node[anchor=west] at (3,1) {$\nu^{j}=\la^{j}a^{j}+\la^{j-1}$ for $k<j\le n$};
		\node[anchor=west] at (3,-.4) {\parbox{.5\textwidth}{$\displaystyle b^{j}=
		\begin{cases}
			\displaystyle a^{j}\frac{\la^{j}\nu^{j-1}}{\la^{j-1}\nu^{j}},& \text{if $\la^{j}>0$},\\
			\rule{0pt}{14pt}a^{j}\nu^{j-1},&\text{if $\la^{j}=0$ and $\la^{j-1}>0$},\\
			a^{j},&\text {if $\la^{j-1}=0$}.
		\end{cases}$}};
	\end{tikzpicture}
	}
\end{align*}


\begin{definition}\label{def:geom_RSK_col_insertion}
	The \emph{geometric RSK column insertion} of a word 
	into an array 
	is defined similarly to
	the row insertion (Definition \ref{def:geom_RSK_row_insertion}),
	by consecutively performing the 
	column insertion operations
	defined above, in order as on
	Fig.~\ref{fig:grr}. 
\end{definition}

\begin{figure}[htbp] 
\includegraphics[width = 0.8\textwidth]{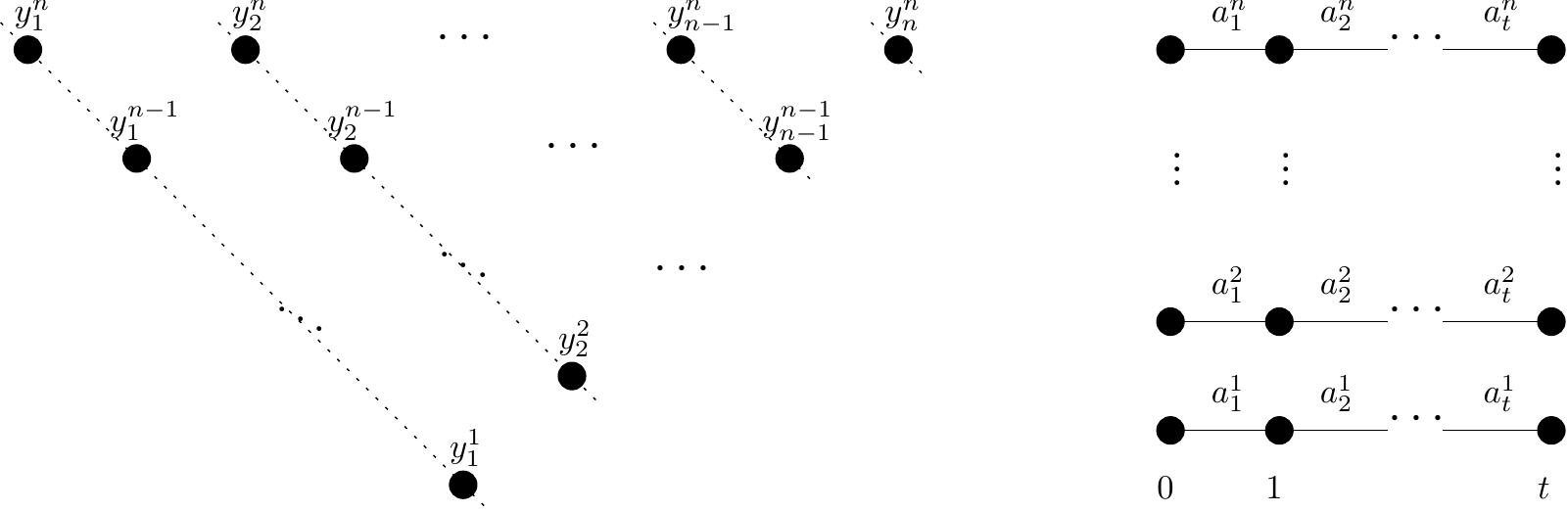}
\caption{Array and strip for the geometric RSK column insertion.}
\label{fig:grca}
\end{figure}
Note that $y^{j}_{k}$ in this definition corresponds to $\la^{(j)}_{j-k+1}$ in the classical RSK column insertion.
We will need the following fact which is analogous to
Proposition \ref{prop:geom_RSK_row_insertion}:
\begin{proposition}\label{prop:geom_RSK_col_insertion}
If we start with an array $y$ of empty words, and consecutively insert 
into it words $a_{1}, \ldots, a_{t}$ with positive entries 
via the geometric RSK column insertion, then in the obtained array 
we have
\begin{align*}
y^{j}_{k}(t) = \frac{L^{j}_{k}(t)(a_{1}, \ldots, a_{t})}{L^{j}_{k-1}(t)(a_{1}, \ldots, a_{t})} \qquad \text{for all $t \geq j-k+1$}. 
\end{align*}
Here again we denote by $L^{j}_{k}(t)(a_{1}, \ldots, a_{t})$ the same weighted sum over $k$-tuples of
nonintersecting paths as in Definition \ref{def:L_polymer}, 
but in a strip in which each edge $(s-1, i) \to (s, i)$ has a deterministic weight $a^{i}_{s}$ (see Fig.~\ref{fig:grca}).
\end{proposition}

\begin{proof}
Our proof is similar to that of Proposition \ref{prop:geom_RSK_row_insertion}
(the latter is given in \cite{NoumiYamada2004}).

For $a = (a^{1}, \ldots, a^{n})$, denote by $H(a)$ the $n \times n$ matrix such that $H(a)_{i,i}:= a^{i}$, $H(a)_{i,i+1} = 1$, and other entries are $0$. 
For  $a = (a^{k}, \ldots, a^{n})$, denote by $H_{k}(a)$ 
the $n \times n$ matrix of the form 
$\left( \begin{array}{cc} Id_{k-1} & 0 \\ 0 & H(a) \end{array} \right).$   
For $\lambda = (\lambda^{k}, \ldots, \lambda^{n})$ such that $\lambda^{i} > 0 $ for $k \leq i \leq j$ 
and $\lambda^{i} = 0$ for $j < i \leq n$, 
denote by $G(\lambda)$ the $n \times n$ matrix of the form $$\left( \begin{array}{ccc} Id_{k-1} & 0 & 0 \\ 0 & G & 0 \\ 0 & 0 & Id_{n-j} \end{array} \right),$$ where $G$ is the upper-triangular $(j-k+1) \times (j-k+1)$ matrix with 
\begin{align*}
G_{p, r} = \frac{\lambda^{r+k-1}}{\lambda^{p+k-2}} \quad \text{for $1 \leq p \leq r \leq j-k+1$.} 
\end{align*}
Assume $\lambda^{k-1} = 1$. 

The key to the proof is the commutation relation 
\begin{align}\label{comm_rel_HG}
	G(\lambda)H_{k}(a) = H_{k+1}(b)G(\nu), 
\end{align}
whenever a pair of words $\nu =(\nu^{k}, \ldots, \nu^{n}), b = (b^{k+1}, \ldots, b^{n})$ 
is obtained by inserting $a = (a^{k}, \ldots, a^{n})$ into $\lambda = (\lambda^{k}, \ldots, \lambda^{n})$. 

To check \eqref{comm_rel_HG}, denote its left-hand side by $\LL$ and right-hand side by $\RR$. 
Clearly, $\LL_{i,i} = \RR_{i,i} = 1$ for $1 \leq i \leq k-1$ and $\LL_{i,i} = \RR_{i,i} = a^{i}$ 
for  $j+2 \leq i \leq n$, and $\LL_{j+1,j+2} = \RR_{j+1,j+2} = 1$. Otherwise $\LL_{i, m} = \RR_{i, m} = 0$ unless $k \leq i \leq j+1$ and $k \leq m \leq j+1$. 
Let us thus assume that the two latter inequalities hold.
On the diagonal, for $k < i < j+1$, we have
\begin{align*}
	\LL_{i,i} = a^{i}\frac{\lambda^{i}}{\lambda^{i-1}} = b^{i}\frac{\nu^{i}}{\nu^{i-1}} = \RR_{i,i},
	\qquad
	\LL_{k, k} = a^{k}\lambda^{k} = \nu^{k} = \RR_{k, k}, 
\end{align*}
and 
\begin{align*}
	\LL_{j+1, j+1} = a^{j+1}   = \frac{b^{j+1}}{\nu^{j}} = \RR_{j+1, j+1}.
\end{align*}
Above the diagonal, for $k < i < m \leq j+1$, we have 
\begin{align*}
	\LL_{i,m} = \frac{\lambda^{m-1}}{\lambda^{i-1}} + \frac{\lambda^{m}}{\lambda^{i-1}}a^{m} = \frac{\nu^{m}}{\lambda^{i-1}} 
	= b^{i} \frac{\nu^{m}}{\nu^{i-1}} + \frac{\nu^{m}}{\nu^{i}} = \RR_{i, m}, 
\end{align*}
since $\frac{b^{i}}{\nu^{i-1}} + \frac{1}{\nu^{i}} = \frac{1}{\nu^{i}}(a^{i}\frac{\lambda^{i}}{\lambda^{i-1}} + 1) = \frac{1}{\lambda^{i-1}}$, and finally 
\begin{align*}
	\LL_{k, m} = \lambda^{m-1} + \lambda^{m}a^{m} = \nu^{m} =  \RR_{k, m}.
\end{align*}
This completes the proof of the commutation relation \eqref{comm_rel_HG}.

By applying the commutation relation multiple times according
to the geometric column RSK insertion (Definition \ref{def:geom_RSK_col_insertion}), 
we get 
\begin{align}\label{comm_rel_HG_many}
	G(y_{n}(t)) \cdot \cdots \cdot G(y_{1}(t)) = H(a_{1}) \cdot \cdots \cdot H(a_{t}).
\end{align}
Observe that the $(i, j)$-entry of the right-hand side above is equal to the sum of weights 
of all directed strict-weak (as on Fig.~\ref{fig:ppf}, right) paths from $(0, i)$ to $(t, j)$,
where the weight of a path is given by the product of weights of horizontal edges, as before. Indeed, this entry is equal to 
\begin{align*}
\sum_{1 \le i_{1}, \ldots, i_{t+1} \le n: \ i_{1}=i, i_{t+1}=j} \prod_{\ell=1}^{t}H(a_{\ell})_{i_{\ell}, i_{\ell+1}} = \sum_{1 \le i_{1}, \ldots, i_{t+1} \le n: i_{1}=i, i_{t+1}=j} \prod_{\ell=1}^{t} (\mathbf{1}_{i_{\ell}=i_{\ell+1}-1}+ a_{\ell}\mathbf{1}_{i_{\ell}=i_{\ell+1}}).
\end{align*}
By the Lindstr\"{o}m-Gessel-Viennot principle \cite{lindstrom1973vector},
\cite{gessel1985binomial}, the determinant of the minor of the right-hand side at the intersection of the first $k$ rows, and columns from $(j-k+1)$-st to $j$-th, is $L^{j}_{k}(t)(a_{1}, \ldots, a_{t})$.  

Next, observe that for $1 \leq s \leq k$, $j-k+1 \leq p \leq j$, the $(s, p)$-entry of the left-hand side 
of \eqref{comm_rel_HG_many}
is equal to the sum of weights of directed up/right (as on Fig.~\ref{fig:ppf}, left) lattice paths 
from $(k+1-s, s)$ to $(\min \{k+t+1-p, k \}, p)$ 
in the array as on 
Fig.~\ref{fig:RSKproof} (the left picture if $t \geq j$, and the right one if $j-k+1 \leq t < j$).
\begin{figure}[h] 
\begin{center}
	\includegraphics[width = 0.75\textwidth]{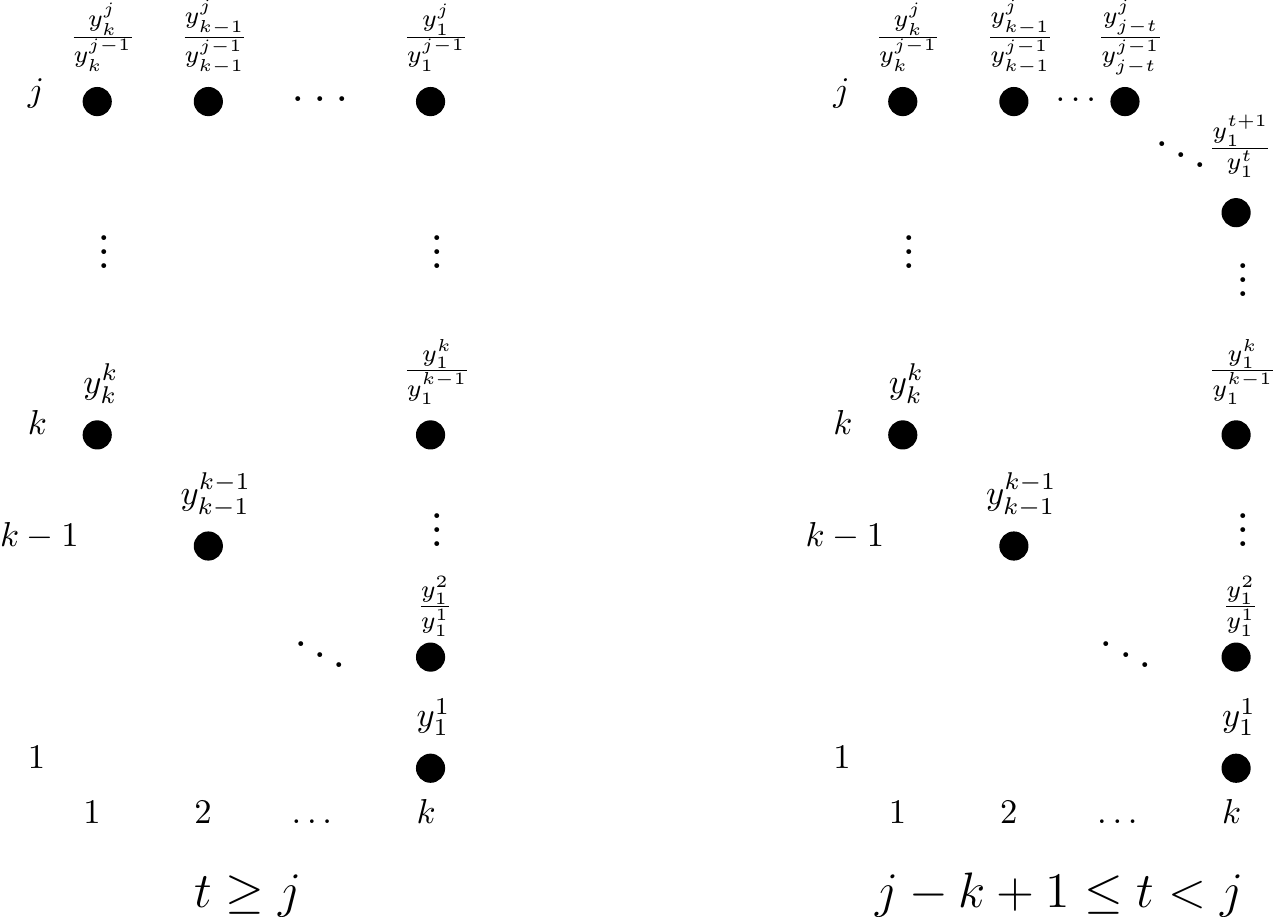}	
\end{center}
\caption{Arrays used in the proof of Proposition 
\ref{prop:geom_RSK_col_insertion}.}
\label{fig:RSKproof}
\end{figure}The weight of each path is defined to be the product of weights of all nodes along the path. 
By Lindstr\"{o}m-Gessel-Viennot principle, 
the determinant of the minor of the left-hand side of \eqref{comm_rel_HG_many} 
at the intersection of the first $k$ rows, and columns from $(j-k+1)$-st to $j$-th, 
is equal to the sum of weights of all $k$-tuples of 
nonintersecting paths from $(k, 1), \ldots, (2, k-1), (1, k)$ to ${(\min\{2k+t - j, k\}, j-k+1)}, \ldots,$ ${(\min \{k+t+2-j, k \}, j-1)},$  ${(\min \{k + t+1 - j, k\}, j)}$.  
There is only one such tuple, which covers all points on Fig.~\ref{fig:RSKproof} and has weight $\displaystyle \prod_{i=1}^{k} y^{\min \{ i+t, j \}}_{i}$. 

Therefore,
\begin{align*}
 	\displaystyle L^{j}_{k}(t)(a_{1}, \ldots, a_{t}) = \prod_{i=1}^{k} y^{\min \{ i+t, j \}}_{i}
 	,
 	\qquad t \geq j - k + 1,
\end{align*}
which establishes the desired statement.
\end{proof}



\subsection{Asymptotics of $q$-deformed Beta-binomial distributions} 
\label{sub:some_asymptotics_of_q_deformed_binomial_distributions}

We will need several lemmas about the limiting properties of the distributions 
$\Om_{q,\muq,\eta}(s\mid y)$ \eqref{Om_qmunu_definition}.

\begin{lemma} \label{L1} Let $X^{\epsilon}$ be a $\Z_{\ge0}$-valued random variable with 
\begin{align*}
Prob(X^{\epsilon} = j) = (\alpha; q)_{\infty} \frac{\alpha^{j}}{(q; q)_{j}} \quad \text{for $\alpha = e^{-\theta\epsilon}$ and $q = e^{-\epsilon}$.}
\end{align*}  
Then as $\epsilon \to 0$, $\epsilon \exp \{ \epsilon X^{\epsilon}\}$ converges in distribution to $\GRVI{\theta}$.

\end{lemma}

\begin{lemma} \label{L2}  Let $n(\epsilon)$ be a function $\R_{\ge0} \to \Z_{\ge0}$, such that $$\displaystyle \lim_{\epsilon \to 0} \epsilon^{-1} \exp \{-\epsilon n(\epsilon)\} = \phi.$$ Let  $X^{\epsilon}$ be a $\Z_{\ge0}$-valued random variable with 
\begin{align*}
Prob(X^{\epsilon} = j) = \Om_{q, \alpha, 0}(j \mid n(\epsilon)) \quad \text{for $\alpha = e^{-\theta\epsilon}$ and $q = e^{-\epsilon}$.}
\end{align*} Then as $\epsilon \to 0$, $\epsilon^{-1} \exp \{ -\epsilon X^{\epsilon}\}$ converges in distribution to $\phi  \ +$ $\GRV{\theta}$. 

\end{lemma}

These two lemmas were both proven in \cite{CorwinSeppalainenShen2014}
(Lemma 2.1 and a part of proof of Theorem 1.4, respectively). 
In the next three lemmas, parameters of distributions which are not explicitly fixed 
are assumed to
depend on $\epsilon$, and sometimes might also be random themselves.

\begin{lemma} \label{L3} Fix $C$ and $0<\sigma<1$. Let $Y^{\epsilon}$ be a $\Z_{\ge0}$-valued random variable distributed according to $\Om_{q^{-1}, \muq, \eta}(\cdot \mid n)$ with $q = e^{-\epsilon}$, $\muq q^{-n} \to \sigma$, $n \geq \epsilon^{-1} \log \epsilon^{-1} - \epsilon^{-1}C$, and $\log \eta  + 2n\epsilon \leq \log \sigma$. Then as $\epsilon \to 0$, $\epsilon Y^{\epsilon} \to \log (1 + \sigma)$.

\end{lemma}

\begin{proof}

The fact that $\Om_{q^{-1}, \muq, \eta}$ with such parameters is indeed a probability distribution for $\epsilon$ small enough follows from inequalities in the statement of the lemma. Let $A(\epsilon): = (e^{-\epsilon}; e^{-\epsilon})_{\infty}$.  
By \cite[Corollary 4.1.10]{BorodinCorwin2011Macdonald},  
\begin{align*}
	(e^{-\epsilon}; e^{-\epsilon})_{\lfloor \epsilon^{-1} \log \epsilon^{-1} - C \epsilon^{-1} \rfloor}  \leq A(\epsilon) C'
\end{align*}
for all $\epsilon$ small enough and some constant $C'$ that depends only on $C$.
As $\epsilon \to 0$, $$\displaystyle \epsilon\log (e^{-\epsilon}; e^{-\epsilon})_{\lceil r/ \epsilon \rceil} \to \int_{0}^{r} \log(1-e^{-x}) dx,$$ since the left-hand side is a Riemann sum for the right-hand side integral, hence we have $$\displaystyle \epsilon\log \frac{(e^{-\epsilon}; e^{-\epsilon})_{\infty}}{(e^{-r}; e^{-\epsilon})_{\infty}} \to \int_{0}^{r} \log(1-e^{-x}) dx.$$ 
(Note that
although this integral blows up at $0$, it is still finite and convergence holds.) 

Fix $\delta > 0$. For $\epsilon$ small enough,
\begin{align*}
 Prob(Y^{\epsilon} = k) & = (\muq q^{-n+k})^{k}\frac{({\eta}/{\muq}; q^{-1})_{k} (\muq; q^{-1})_{n-k}}{(\eta; q^{-1})_{n}} \binom nk_{q}  \leq \frac{(\muq q^{-n})^{k} e^{-\epsilon k^{2}}C'}{(e^{-\epsilon}; e^{-\epsilon})_{k}}
\\&  \leq \frac{(2\sigma)^{k} e^{-\epsilon k^{2}}C'}{(e^{-\epsilon}; e^{-\epsilon})_{k}} \leq C' \exp \left( k \log 2 \sigma  - \epsilon k^{2} -  \frac{1}{\epsilon}\int_{0}^{k \epsilon} \log(1 - e^{-x})dx \right)
\\& \leq e^{-T^{2}/2\epsilon} \quad \text{for $T$ large enough and $k \geq T/\epsilon$.} 
\end{align*}
Hence $$\displaystyle Prob(Y^{\epsilon} \geq T/\epsilon) \leq \sum_{k =\lceil T/\epsilon \rceil}^{\infty} e^{-\epsilon k^{2}/2} \leq e^{-T^{2}/2\epsilon}\sum_{i=0}^{\infty}  
e^{-Ti},$$ which can be made less than $\delta/2$ for all $\epsilon$ small enough by choosing sufficiently large $T$. 

Observe that for $k \leq T/\epsilon$ and $\epsilon$ small enough there is 
some constant $C_{0}$ that depends only on $C$ and $T$, such that
\begin{align*}
	C_{0}^{-1} \leq \frac{({\eta}/{\muq}; q^{-1})_{k}}{(\muq q; q)_{\infty}(\eta; q^{-1})_{n}} 
	\frac{(q; q)_{n}}{(q; q)_{n-k}} \leq C_{0}.
\end{align*}
Let $$f(\psi) := -\psi^{2} + (\log \sigma) \psi - \int_{0}^{\psi} \log(1-e^{-x}) dx - \int_{0}^{\psi - \log \sigma} \log(1-e^{-x}) dx$$ for $\psi \geq 0$. Then $$f'(\psi) = -2\psi + \log \sigma - \log(1 - e^{-\psi}) - \log(1 - e^{-\psi + \log \sigma}),$$ which is strictly decreasing, and $f'(\log (1 + \sigma)) = 0$.  Hence $f$ attains a unique maximum at $\log(1 + \sigma)$, so one can choose $M_{1} > M_{2} > M_{3} > M_{4}$ such that 
\begin{align*}
f(\psi) > M_{1} \quad \text{for $\psi \in (\log (1 + \sigma) - \delta /2,  \log (1 + \sigma) + \delta /2)$} 
\end{align*}
and
\begin{align*}
 f(\psi) < M_{4} \quad \text{for $\psi \notin  (\log (1 + \sigma) - \delta,  \log (1 + \sigma) + \delta)$,}
\end{align*}
and for $\epsilon$ small enough $$Prob(\epsilon Y^{\epsilon} \in (\log (1 + \sigma) - \delta /2,  \log (1 + \sigma) + \delta /2)) \geq C_{0}^{-1}\left(\frac{\delta}{\epsilon} - 1\right) A(\epsilon) e^{M_{2}/\epsilon}$$ and $$Prob(\epsilon Y^{\epsilon} \notin (\log (1 + \sigma) - \delta,  \log (1 + \sigma) + \delta) \cup [T, \infty)) \leq C_{0}\left(\frac{T - 2\delta}{\epsilon}+2\right) A(\epsilon) e^{M_{3}/\epsilon}.$$
Therefore, for $\epsilon$ small enough $$Prob(\epsilon Y^{\epsilon} \in (\log (1 + \sigma) - \delta,  \log (1 + \sigma) + \delta)) \geq 1 - \delta,$$ and this completes the proof.
\end{proof}

\begin{lemma} \label{L4} Fix $C$ and $0 \leq \sigma < 1$.  Let  $Y^{\epsilon}$ be a $\Z_{\ge0}$-valued random variable with 
\begin{align*}
Prob(Y^{\epsilon} = j) = \Om_{q, \alpha, 0}(j \mid n) 
\end{align*}
for $\alpha \to \sigma$ as $\epsilon \to 0$, $q = e^{-\epsilon}$ and $n \geq \epsilon^{-1} \log \epsilon^{-1} - \epsilon^{-1}C$. Then as $\epsilon \to 0$, $\epsilon Y^{\epsilon} \to -\log (1 - \sigma)$.
\end{lemma}

\begin{proof} Suppose $\sigma > 0$ and fix $\delta > 0$. We can write
\begin{align*} 
& Prob(Y^{\epsilon} = k) = \frac{\alpha^{k}}{(e^{-\epsilon}; e^{-\epsilon})_{k}} \frac{(\alpha; e^{-\epsilon})_{n-k}(e^{-\epsilon}, e^{-\epsilon})_{n}}{(e^{-\epsilon}; e^{-\epsilon})_{n-k}} \leq\frac{\alpha^{k}}{(e^{-\epsilon}; e^{-\epsilon})_{k}}
\\&\hspace{100pt} \leq \exp \left( \frac{1}{\epsilon}((\frac{1}{2} \log \sigma) T - \int_{0}^{T} \log (1-e^{-x})dx) \right) \leq e^{(\log \sigma) T/ 4 \epsilon}
\end{align*}
for $T$ large enough and $\epsilon$ small enough, such that 
$k \geq T/\epsilon$. Hence $$Prob(Y^{\epsilon} \geq T/\epsilon) \leq e^{(\log \sigma) T/ 4 \epsilon} \sum_{i=0}^{\infty} e^{\frac{i \log \sigma}{4}}$$ which is less than $\delta/2$ for $T$ large enough. 

For $\epsilon$ small enough, $k \leq T/\epsilon$, and some constant $C_{0}$ that depends only on $C$, $T$, and $\sigma$,
one can write $$C_{0}^{-1} (\alpha; e^{-\epsilon})_{\infty} \leq \frac{(\alpha; e^{-\epsilon})_{n-k}(e^{-\epsilon}; e^{-\epsilon})_{n}}{(e^{-\epsilon}; e^{-\epsilon})_{n-k}} \leq C_{0} (\alpha; e^{-\epsilon})_{\infty}.$$
Let $$f(\psi) := (\log \sigma) \psi - \int_{0}^{\psi} \log(1-e^{-x}) dx$$ for $\psi \geq 0$. Then $$f'(\psi) =  \log \sigma - \log(1 - e^{-\psi}),$$ which is strictly decreasing, and $f'(-\log (1 - \sigma)) = 0$.  Hence $f$ attains a unique maximum at $-\log(1 - \sigma)$, so one can choose $M_{1} > M_{2} > M_{3} > M_{4}$ such that
\begin{align*}
 f(\psi) > M_{1} \quad \text{for $\psi \in (-\log (1 - \sigma) - \delta /2,  -\log (1 - \sigma) + \delta /2)$}
\end{align*}
and
\begin{align*} 
f(\psi) < M_{4} \quad \text{for $\psi \notin  (-\log (1 - \sigma) - \delta,  -\log (1 - \sigma) + \delta)$,} 
\end{align*}
and for $\epsilon$ small enough $$Prob(\epsilon Y^{\epsilon} \in (-\log (1 - \sigma) - \delta /2,  -\log (1 - \sigma) + \delta /2)) \geq C_{0}^{-1}\left(\frac{\delta}{\epsilon} - 1\right)  (\alpha; e^{-\epsilon})_{\infty}e^{M_{2}/\epsilon}$$ and $$Prob(\epsilon Y^{\epsilon} \notin (-\log (1 - \sigma) - \delta,  -\log (1 - \sigma) + \delta) \cup [T, \infty)) \leq C_{0}\left(\frac{T - 2\delta}{\epsilon}+2\right) (\alpha; e^{-\epsilon})_{\infty} e^{M_{3}/\epsilon}.$$
Therefore, for $\epsilon$ small enough we can write 
$$Prob(\epsilon Y^{\epsilon} \in (-\log (1 - \sigma) - \delta,  -\log (1 - \sigma) + \delta)) \geq 1 - \delta,$$ and this completes the proof for the case $\sigma > 0$. 

If $\sigma = 0$, fix arbitrary $u > 0$. For all large enough $U$, such that $Uu >  \frac{1}{2}Uu  - \int_{0}^{\infty} \log(1-e^{-x}) dx$, $\epsilon$ small enough, and $k \geq \frac{u}{\epsilon}$,  $$Prob(Y^{\epsilon} = k)  \leq\frac{\alpha^{k}}{(e^{-\epsilon}; e^{-\epsilon})_{k}} \leq \exp \left(\frac{1}{\epsilon}(-Uk\epsilon - \int_{0}^{\infty} \log(1-e^{-x}) dx )  \right) \leq e^{-\frac{Uu}{2\epsilon}}.$$ So, $$Prob(\epsilon Y^{\epsilon} \geq u) \leq  e^{-\frac{Uu}{2\epsilon}} \sum_{i=0}^{\infty} e^{-Ui/2},$$ which can be made less than any given  $\delta > 0$ for $U$ large enough. Thus, we have convergence 
$\epsilon Y^{\epsilon} \to 0$ in distribution.
\end{proof}

\begin{lemma} \label{L5} Fix $\alpha, \sigma\in(0,1)$ such that $\alpha(1 + \sigma) < 1$.  Let  $Y^{\epsilon}$ 
be a $\Z_{\ge0}$-valued random variable with $Prob(Y^{\epsilon} = j) = \Om_{q^{-1}, \alpha, 0}(j \mid n) $ for ${n\epsilon \to \log(1 +\sigma)}$ as $\epsilon \to 0$, $q = e^{-\epsilon}$. Then as $\epsilon \to 0$, $\epsilon Y^{\epsilon} \to \log (1 + \alpha \sigma)$.
\end{lemma}
\begin{proof} $\alpha(1+\sigma) < 1$ ensures that this is indeed a probability distribution for $\epsilon$ small enough.
This distribution looks as
$$Prob(Y^{\epsilon} = k) = \frac{(q; q)_{n}}{(q; q)_{k}(q; q)_{n-k}} (\alpha q^{k-n})^{k} (\alpha; q^{-1})_{n-k}.$$

Let 
\begin{align*}
f(\psi) &:= - \int_{0}^{\psi} \log(1-e^{-x})dx - \int_{0}^{\log(1+\sigma) -\psi} \log(1-e^{-x})dx + \psi \log \alpha - \psi^{2}+ \psi \log (1 + \sigma)
\\&  +  \int_{-\log \alpha - \log (1 + \sigma) + \psi}^{-\log \alpha} \log(1-e^{-x})dx \quad \text{for $\log (1+\sigma) \geq \psi \geq 0$.}
\end{align*} 
Then $$f'(\psi) = - \log(1 - e^{-\psi}) +  \log(1 - \frac{e^{\psi}}{1 + \sigma}) + \log \alpha - 2 \psi + \log(1 + \sigma) - \log(1 - \alpha (1+\sigma) e^{- \psi}),$$  which is strictly decreasing, and $f'(\log (1 + \alpha \sigma)) = 0$.  Hence $f$ attains a unique maximum at $\log(1 + \alpha \sigma)$, so one can choose $M_{1} > M_{2}$ such that
\begin{align*} 
f(\psi) > M_{1} \quad \text{for $\psi \in (\log (1 + \alpha \sigma) - \delta /2,  \log (1 + \alpha \sigma) + \delta /2)$}
\end{align*}
and
\begin{align*}
f(\psi) < M_{2} \quad \text{for $\psi \notin  (\log (1 + \alpha \sigma) - \delta,  \log (1 + \alpha \sigma) + \delta)$.}
\end{align*}
 Hence for $\epsilon$ small enough, $$Prob(\epsilon Y^{\epsilon} \in (\log (1 + \alpha \sigma) - \delta /2,  \log (1  + \alpha \sigma) + \delta /2)) \geq \left(\frac{\delta}{\epsilon} - 1\right)(e^{-\epsilon}; e^{-\epsilon})_{n}e^{M_{1}/\epsilon},$$ and $$Prob(\epsilon Y^{\epsilon} \notin (\log (1 + \alpha \sigma) - \delta,  \log (1 + \alpha \sigma) + \delta)) \leq \left(\frac{\log(1+ \sigma) - 2\delta}{\epsilon}+2\right) (e^{-\epsilon}; e^{-\epsilon})_{n} e^{M_{2}/\epsilon}.$$
Thus, for $\epsilon$ small enough, we have $$Prob(\epsilon Y^{\epsilon} \in (\log (1 + \alpha \sigma) - \delta,  \log (1 + \alpha \sigma) + \delta)) \geq 1 - \delta,$$ and this completes the proof.
\end{proof}


\subsection{Proofs of Theorems \ref{thm:TR} and \ref{thm:TL}} 
\label{sub:proofs_of_theorems_thm:tr_and_thm:tl}

In our proofs, 
we denote by $$A(\epsilon), B(\epsilon), C(\epsilon), D(\epsilon), E(\epsilon), F(\epsilon)$$ 
possibly random positive valued functions tending to 
deterministic constants $A, B, C, D, E, F$, respectively, as $\epsilon \to 0$. 
This notation will be repeatedly used in the arguments below for defining conditional probabilities.

\subsubsection{Proof of Theorem \ref{thm:TR}} 
\label{ssub:proof_of_theorem_thm:tr}

We must show that 
for fixed $(t, k, j)$, such that $k \leq \min\{t, j\}$ and $j \leq n$,
the random variable
$\hat R^{j}_{k}(t, \epsilon)$
conditioned on 
$\hat R^{J}_{K}(T, \epsilon) \to X ^{J}_{K}(T)$ for all ${(T, K, J) < (t, k, j)}$ in the lexicographic 
order,\footnote{That is, $T < t$, or $T=t$ and $K < k$, or $T=t, K = k$ and $J < j$.} converges as $\epsilon \to 0$ 
to $\hat R^{j}_{k}(t)$ 
conditioned on $\hat R^{J}_{K}(T) = X^{J}_{K}(T)$ for all ${(T, K, J) < (t, k, j)}$ in the lexicographic order.
Here $X^{J}_{K}(T)$ are some fixed constants.
In the rest of the proof we will always assume this conditioning. 

\proofstep{Right edge ($k=1$)} The Markovian projection of the $\Qqarow$ dynamics on the right edge is the geometric 
$q$-PushTASEP (\S \ref{sub:geometric_q_pushtasep}), 
hence the proof of the theorem for the right edge is the same as showing that suitably rescaled positions of particles in the geometric $q$-PushTASEP converge to the partition functions of the log-Gamma polymer (Definition \ref{def:R_polymer}). 

\proofa{a} If $t=1$, then 
$r_{j, 1}(1, \epsilon) = r_{j-1, 1}(1, \epsilon) \  +$ an  independent random movement $d$ distributed according to $\Om_{q, a_{j}\alpha_{1}, 0}(d \mid \infty)$ (assume $r_{0, 1}(1, \epsilon) =0$ and $X^{0}_{1}(1)=1$). By Lemma \ref{L1},
\begin{align*}
 \log(\hat R^{j}_{1}(1, \epsilon))  &=  {r_{j, 1}(1, \epsilon)\epsilon - j \log \epsilon^{-1}} \\&=  {r_{j-1,1}(1, \epsilon)\epsilon - (j-1) \log \epsilon^{-1}} + d \epsilon - \log \epsilon^{-1}
\\& \to \log(X^{j-1}_{1}(1)) + \log(\Gamma)
\end{align*} 
for an independent random variable $\Gamma = a^{j}_{1}$ distributed according to $\GRVI{\theta_{j} + \hat \theta_{1}}$, which is consistent with $\hat R^{j}_{1}(1) = X^{j-1}_{1}(1)a^{j}_{1}$. 

\proofa{b} If $j=1$ and $t \geq 2$, then $r_{1, 1}(t, \epsilon) = {r_{1, 1}(t-1, \epsilon)} \  +$ an  independent random movement $d$ distributed according to $\Om_{q, a_{1}\alpha_{t}, 0}(d \mid \infty)$. By Lemma \ref{L1},
\begin{align*}
\log(\hat R^{1}_{1}(t, \epsilon))  &= r_{1,1}(t, \epsilon)\epsilon - t\log \epsilon^{-1} \\&= r_{1, 1}(t-1, \epsilon)\epsilon - (t-1)\log \epsilon^{-1} +  d \epsilon - \log \epsilon^{-1}  
\\& \to \log(X^{1}_{1}(t-1)) + \log(\Gamma)
\end{align*}  
for an independent random variable $\Gamma = a^{1}_{t}$ distributed according to $\GRVI{\theta_{1} + \hat \theta_{t}}$, which is consistent with $\hat R^{1}_{1}(t) = X^{1}_{1}(t-1)a^{1}_{t}$. 

\proofa{c} Assume $t \geq 2$ and $j \geq 2$. Condition on 
\begin{align*}
& r_{j, 1}(t-1, \epsilon) = (t+ j-2)\epsilon^{-1} \log \epsilon^{-1} + \epsilon^{-1} \log A(\epsilon), \quad \text {so $X^{j}_{1}(t-1)=A$};
\\& r_{j-1, 1}(t-1, \epsilon) = (t+ j-3)\epsilon^{-1} \log \epsilon^{-1} + \epsilon^{-1} \log B(\epsilon), \quad \text {so $X^{j-1}_{1}(t-1)=B$};
\\& r_{j-1, 1}(t, \epsilon) = (t+ j-2)\epsilon^{-1} \log \epsilon^{-1} + \epsilon^{-1} \log C(\epsilon), \quad \text{so $X^{j-1}_{1}(t)=C$}.
\end{align*}
The movement of the rightmost particle on the $j$-th level during the time step ${t-1 \to t}$ 
which happens due to the pushing
by the 
rightmost particle at level $j-1$
behaves as $\epsilon^{-1}\log (1 + \frac{C}{A})$ (by Lemma \ref{L3}). 
The independent movement of the rightmost particle on the $j$-th level 
behaves as $\epsilon^{-1}\log \Gamma+\epsilon^{-1}\log\epsilon^{-1}$
(by Lemma \ref{L1}),
for an independent random variable $\Gamma = a^{j}_{t}$ distributed according to $\GRVI{\theta_{j} + \hat \theta_{t}}$. 
Therefore, 
\begin{align*}
 	\log(\hat R^{j}_{1}(t, \epsilon)) \to \log A + \log (1 + \frac{C}{A}) + \log \Gamma = \log((A+C)\Gamma),
\end{align*}
which is consistent with $\hat R^{j}_{1}(t) = (X^{j}_{1}(t-1) + X^{j-1}_{1}(t))a^{j}_{t}$.   

\proofstep{$k$-th edge from the right for $k \geq 2$}

\proofa{a} Assume $t=k$. We have $r_{j, k}(k, \epsilon) = r_{j-1, k}(k, \epsilon) \  +$ a movement due to pulling from the 
\mbox{$(k-1)$-st} particle from the right on the $(j-1)$-st level (assume $r_{k-1, k}(k, \epsilon)  = 0$ and $X^{k-1}_{k}(k) = 1$).  Condition on
\begin{align*} 
& r_{j-1, k-1}(k-1, \epsilon) = (j- k+1)\epsilon^{-1} \log \epsilon^{-1} + \epsilon^{-1} \log D(\epsilon), \quad \text{so $X^{j-1}_{k-1}(k-1)=D$};
\\& r_{j-1, k-1}(k, \epsilon) = (j- k + 2)\epsilon^{-1} \log \epsilon^{-1} + \epsilon^{-1} \log E(\epsilon), \quad \text{so $X^{j-1}_{k-1}(k)=E$};
\\& r_{j, k-1}(k-1, \epsilon) = (j- k+2)\epsilon^{-1} \log \epsilon^{-1} + \epsilon^{-1} \log F(\epsilon), \quad \text{so $X^{j}_{k-1}(k-1)=F$}. 
\end{align*}  
 By Lemma \ref{L3}, this movement times $\epsilon$ and minus $\log \epsilon^{-1}$ converges to 
 ${\log ({E}/{D}) - \log (1 + {E}/{F})}$. Therefore, 
\begin{align*}
 \log(\hat R^{j}_{k}(k, \epsilon)) &= r_{j, k}(k, \epsilon)\epsilon - (j-k+1)\log \epsilon^{-1}  
 \\&\to \log(X^{j-1}_{k}(k)) + \log ({E}/{D}) - \log (1 + {E}/{F})
\\& = \log \left( \frac{X^{j-1}_{k}(k)EF}{(F+E)D} \right),
\end{align*}
which is consistent with  $$\hat R^{j}_{k}(k) = \frac{X^{j-1}_{k}(k)X^{j-1}_{k-1}(k)X^{j}_{k-1}(k-1)}{\left(X^{j}_{k-1}(k-1)\ + \ X^{j-1}_{k-1}(k)\right)X^{j-1}_{k-1}(k-1)}.$$
 
Indeed, if we insert (via the geometric row insertion) a nonempty word ${b = (b^{k}, \ldots, b^{n})}$ into the empty word $\lambda_{k} = (1, 0, \ldots, 0)$, where $b$ is itself the bottom output of the insertion of $a = (a^{k-1}, \ldots, a^{n})$ into $\lambda_{k-1} = (\lambda^{k-1}_{k-1}, \ldots, \lambda^{n}_{k-1}) = (X^{k-1}_{k-1}(k-1), \ldots, X^{n}_{k-1}(k-1))$, 
then  we get $$\nu^{j}_{k} = \nu^{j-1}_{k}b^{j} = \nu^{j-1}_{k} \frac{a^{j} \lambda_{k-1}^{j} \nu_{k-1}^{j-1}}{\lambda_{k-1}^{j-1} \nu_{k-1}^{j}} = {\nu^{j-1}_{k} \frac{\lambda_{k-1}^{j}}{\lambda_{k-1}^{j-1}} \cdot \frac{\nu^{j-1}_{k-1}}{\nu^{j-1}_{k-1} + \lambda^{j}_{k-1}}},$$ which is the same as the expression for $\hat R^{j}_{k}(k)$ above.

\proofa{b} Assume $t \geq k+1$ and $k < j \leq n$. Condition on  
\begin{align*}
& r_{j,k}(t-1, \epsilon) = (t+ j-2k)\epsilon^{-1} \log \epsilon^{-1} + \epsilon^{-1} \log A(\epsilon), \quad
\text{so $X^{j}_{k}(t-1)=A$};
\\& r_{j-1,k}(t-1, \epsilon) = (t+ j-2k-1)\epsilon^{-1} \log \epsilon^{-1} + \epsilon^{-1} \log B(\epsilon), \quad
\text{so $X^{j-1}_{k}(t-1)=B$};
\\& r_{j-1,k}(t, \epsilon) = (t+ j-2k)\epsilon^{-1} \log \epsilon^{-1} + \epsilon^{-1} \log C(\epsilon), \quad 
\text{so $X^{j-1}_{k}(t)=C$};
\\& r_{j-1,k-1}(t-1, \epsilon) = (t+ j-2k + 1)\epsilon^{-1} \log \epsilon^{-1} + \epsilon^{-1} \log D(\epsilon), \quad
\text{so $X^{j-1}_{k-1}(t-1)=D$};
\\& r_{j-1,k-1}(t, \epsilon) = (t+ j-2k + 2)\epsilon^{-1} \log \epsilon^{-1} + \epsilon^{-1} \log E(\epsilon), \quad
\text{so $X^{j-1}_{k-1}(t)=E$};
\\& r_{j,k-1}(t-1, \epsilon) = (t+ j-2k + 2)\epsilon^{-1} \log \epsilon^{-1} + \epsilon^{-1} \log F(\epsilon), \quad 
\text{so $X^{j}_{k-1}(t-1)=F$}.
\end{align*}
Movement of $r_{j,k}$ during time step ${t-1 \to t}$ due to pushing by the $k$-th particle from the right on the $(j-1)$-st level  behaves as $\epsilon^{-1}\log (1 + \frac{C}{A})$ (by Lemma \ref{L3}). 
The movement due to the pulling (by the $(k-1)$-st particle from the right on the $(j-1)$-st level) times $\epsilon^{-1}$ minus $\log \epsilon^{-1}$ by the same lemma tends to ${\log (\frac{E}{D}) - \log (1 + \frac{E}{F})}$. Hence we may write
\begin{align*}
\log(\hat R^{j}_{k}(t, \epsilon)) \to  \log A + \log \left(1 + \frac{C}{A}\right) + \log \left(\frac{E}{D}\right) - \log \left(1 + \frac{E}{F}\right) 
= \log \left(\frac{(A+C)EF}{(F+E)D}\right),
\end{align*} 
which is consistent with 
$$\hat R^{j}_{k}(t) = \frac{\left(X^{j}_{k}(t-1) \ + \ X^{j-1}_{k}(t)\right)X^{j-1}_{k-1}(t)X^{j}_{k-1}(t-1)}{\left(X^{j}_{k-1}(t-1)\ + \ X^{j-1}_{k-1}(t)\right)X^{j-1}_{k-1}(t-1)}.$$

Indeed, if we insert (via the geometric row insertion) a nonempty word ${b = (b^{k}, \ldots, b^{n})}$ into a nonempty word $\lambda_{k} = (\lambda^{k}_{k}, \ldots, \lambda^{n}_{k}) = (X^{k}_{k}(t-1), \ldots,  {X^{n}_{k}(t-1))}$, where $b$ is itself the bottom output of the insertion of $a = (a^{k-1}, \ldots, a^{n})$ into $\lambda_{k-1} = (\lambda^{k-1}_{k-1}, \ldots, \lambda^{n}_{k-1}) = (X^{k-1}_{k-1}(t-1), \ldots, X^{n}_{k-1}(t-1))$, then  we get $$\nu^{j}_{k} = (\nu^{j-1}_{k} + \lambda^{j}_{k})b^{j} = (\nu^{j-1}_{k} + \lambda^{j}_{k}) \frac{a^{j} \lambda_{k-1}^{j} \nu_{k-1}^{j-1}}{\lambda_{k-1}^{j-1} \nu_{k-1}^{j}} = {(\nu^{j-1}_{k} + \lambda^{j}_{k}) \frac{\lambda_{k-1}^{j}}{\lambda_{k-1}^{j-1}} \cdot \frac{\nu^{j-1}_{k-1}}{\nu^{j-1}_{k-1} + \lambda^{j}_{k-1}}},$$ which is the same as the expression for $\hat R^{j}_{k}(t)$ above.

\proofa{c} Finally, if $j=k$ and $t \geq k+1$, then the previous argument carries out with the exception that in this case the leftmost particle on the $j$-th level experiences only pulling of the leftmost particle on the $(j-1)$-st level, hence we should take $C=0$ in the formulas from part b).

This completes the proof of Theorem \ref{thm:TR}.


\subsubsection{Proof of Theorem \ref{thm:TL}} 
\label{ssub:proof_of_theorem_thm:tl}

We must show that 
for fixed $(t, k, j)$, such that $k \leq j \leq t+k-1$ and $j \leq n$, 
the random variable
$\hat L^{j}_{k}(t, \epsilon)$, conditioned on $\hat L^{J}_{K}(T, \epsilon) \to Y^{J}_{K}(T)$ 
for all ${(T, K, J) < (t, k, j)}$ in the lexicographic order, 
converges as $\epsilon \to 0$ to $\hat L^{j}_{k}(t)$, conditioned on $\hat L^{J}_{K}(T) = Y^{J}_{K}(T)$ for ${(T, K, J) < (t, k, j)}$ in the lexicographic order
(here $Y^{J}_{K}(T)$ are some fixed constants). 
In the rest of the proof we will always assume this conditioning. 

\proofstep{Left edge $(k=1)$.} The Markovian projection of the $\Qqacol$ dynamics on the left edge is the geometric $q$-TASEP
(\S \ref{sub:geometric_q_tasep}), 
hence the proof of the theorem for the left edge is the same as showing 
that suitably rescaled positions of particles in the geometric $q$-TASEP 
converge to the partition functions of the strict-weak polymer (Definition \ref{def:L_polymer}). 
This was already done in \cite{CorwinSeppalainenShen2014}, but we still include this part in 
the proof for the reader's convenience.

\proofa{a} If $t=j=1$, then by Lemma \ref{L1}, 
$ \log(\hat L^{1}_{1}(1, \epsilon)) = \log \epsilon^{-1} - \ell_{1,1}(1, \epsilon)\epsilon$ converges in distribution to $\log \Gamma$ for a random variable $\Gamma = a^{1}_{1}$ distributed according to  $\GRV{\theta_{1} + \hat \theta_{1}}$. 

\proofa{b} Assume $t=j > 1$. Then $m= \ell_{j,1}(j, \epsilon)$ 
is distributed according to $\Om_{q, a_{j}\alpha_{j}, 0}(m \mid \ell_{j-1, 1}(j-1, \epsilon))$. If we condition on 
\begin{align*}
 \ell_{j-1,1}(j-1, \epsilon) = \epsilon^{-1} \log \epsilon^{-1} - \epsilon^{-1} \log F(\epsilon), \quad \text{so $Y^{j-1}_{1}(j-1) = F$},\end{align*}
then by Lemma \ref{L2},  $$\log (\hat L^{j}_{1}(j, \epsilon)) = \log \epsilon^{-1} - \ell_{j,1}(j, \epsilon)\epsilon \to \log (F + \Gamma)$$ for an independent random variable $\Gamma = a_{j}^{j}$ distributed according to $\GRV{\theta_{j} + \hat \theta_{j}}$, which is consistent with $\hat L^{j}_{1}(j) = a^{j}_{j} + Y^{j-1}_{1}(j-1)$. 

\proofa{c} Let $t > j = 1$. By Lemma \ref{L1}, the quantities
$$ \log(\hat L^{1}_{1}(t, \epsilon)) - \log(\hat L^{1}_{1}(t-1, \epsilon)) = \log \epsilon^{-1} - \left(\ell_{1,1}(t, \epsilon) - \ell_{1,1}(t-1, \epsilon)\right)\epsilon$$ converge in distribution to $\log \Gamma$ for a random variable $\Gamma = a^{1}_{t}$ distributed according to  $\GRV{\theta_{1} + \hat \theta_{t}}$, which is consistent with $\hat L^{1}_{1}(t) = a^{1}_{t} Y^{1}_{1}(t-1)$.

\proofa{d} Assume $t > j >1$. Condition on 
\begin{align*}
& \ell_{j,1}(t-1, \epsilon)  = (t-j)\epsilon^{-1}\log \epsilon^{-1} - \epsilon^{-1}\log A(\epsilon), \quad \text{so $Y^{j}_{1}(t-1) = A$};
\\& \ell_{j-1, 1}(t-1, \epsilon)  = (t-j+1)\epsilon^{-1}\log \epsilon^{-1} - \epsilon^{-1} \log F(\epsilon), \quad \text{so $Y^{j-1}_{1}(t-1) = F$}.  
\end{align*}
By Lemma \ref{L2}, the movement of the leftmost particle on the $j$-th level during the time step $t-1 \to t$ times $\epsilon$ and minus $\log \epsilon^{-1}$ converges to $-\log(\frac{F}{A} + \Gamma)$ for an independent random variable  $\Gamma = a_{t}^{j}$ distributed according to $\GRV{\theta_{j} + \hat \theta_{t}}$. Therefore,  
$$\log(\hat L^{j}_{1}(t, \epsilon)) \to \log A + \log({F}/{A} + \Gamma) = \log(F + A\Gamma),$$ 
which is consistent with $\hat L^{j}_{1}(t) = Y^{j-1}_{1}(t-1) + a^{j}_{t} Y^{j}_{1}(t-1)$.

\proofstep{Second edge from the left $(k=2)$}
\proofa{a} If $j = 2$, $t=1$, then we have to look at $\ell_{2,2}(1, \epsilon)  = \ell_{1,1}(1, \epsilon) \ +$ a jump $m$ distributed according to $\Om_{q, a_{2}\alpha_{1}, 0}(m \mid \infty)$. By Lemma \ref{L1}, 
$$\log(\hat L^{2}_{2}(1)) - \log(\hat L^{1}_{1}(1)) = \log \epsilon^{-1} - m\epsilon \to  \log \Gamma,$$ 
where $\Gamma = a_{1}^{2}$ is an independent random variable distributed according to $\GRV{\theta_{2} + \hat \theta_{1}}$, which is consistent with $\hat L^{2}_{2}(1) = a_{1}^{2}Y^{1}_{1}(1)$.

\proofa{b} Assume $j > 2$, $t=j-1$. We have  $\ell_{j,2}(j-1, \epsilon)  = m_{1} + m_{2}$, where $m_1$ is an independent move distributed according to $\Om_{q, a_{j}\alpha_{j-1}, 0}(m_{1} \mid \ell_{j-1, 2}(j-2, \epsilon))$, and $m_2$ is the push from the leftmost particle on the $(j-1)$-st level distributed according to $$\Om_{q^{-1}, q^{\ell_{j-1,1}(j-1, \epsilon)}, q^{\ell_{j-1,2}(j-2, \epsilon)}}(\ell_{j-1,2}(j-2, \epsilon) - m_{1} -m_{2} \mid \ell_{j-1,2}(j-2, \epsilon) - m_{1}).$$
Condition on 
\begin{align*}
& \ell_{j-1,1}(j-1, \epsilon) =  \epsilon^{-1}\log \epsilon^{-1} - \epsilon^{-1} \log E(\epsilon), \quad \text{so $Y^{j-1}_{1}(j-1) = E$};
\\& \ell_{j-1,2}(j-2, \epsilon) =  2\epsilon^{-1}\log \epsilon^{-1} - \epsilon^{-1} \log F(\epsilon), \quad \text{so $Y^{j-1}_{2}(j-2) = F$}.
\end{align*}
By Lemma \ref{L2}, $\log \epsilon^{-1} - m_{1}\epsilon \to \log \Gamma$, where $\Gamma = a_{j-1}^{j}$ is an independent random variable distributed according to $\GRV{\theta_{j} + \hat \theta_{j-1}}$. By Lemma \ref{L3}, $$\log(\hat L^{j}_{2}(j-1, \epsilon)) = 2 \log \epsilon^{-1} -{\ell_{j,2}(j-1, \epsilon)}\epsilon \to \log F + \log (1 + \frac{E\Gamma}{F}) = \log (F + E\Gamma),$$ which is consistent with  $Y^{j}_{2}(j-1) = Y^{j-1}_{2}(j-2) + a^{j}_{j-1}Y^{j-1}_{1}(j-1)$.

\proofa{c} Let $j = 2$, $t > 1$. Then 
\begin{align*}
 	\ell_{2, 2}(t, \epsilon)  = \ell_{2,2}(t-1, \epsilon) + \ell_{1,1}(t, \epsilon) - \ell_{1,1}(t-1, \epsilon) +m,
\end{align*}
where the jump $m$ is distributed according to 
$$\Om_{q, a_{2}\al_{t} q^{\ell_{1,1}(t-1, \epsilon) - \ell_{2,1}(t, \epsilon)}, 0}(m \mid \infty).$$ 
Condition on 
\begin{align*}
&\ell_{2, 2}(t-1, \epsilon) = t\epsilon^{-1}\log \epsilon^{-1} - \epsilon^{-1} \log A(\epsilon), \quad \text{so $Y^{2}_{2}(t-1) = A$};
\\& \ell_{1,1}(t-1, \epsilon) = (t-1)\epsilon^{-1}\log \epsilon^{-1} - \epsilon^{-1} \log B(\epsilon), \quad \text {so $Y^{1}_{1}(t-1) = B$};
\\& \ell_{1, 1}(t, \epsilon) = t\epsilon^{-1}\log \epsilon^{-1} - \epsilon^{-1} \log C(\epsilon), \quad \text{so $Y^{1}_{1}(t) = C$};
\\& \ell_{2,1}(t, \epsilon) = (t-1)\epsilon^{-1}\log \epsilon^{-1} - \epsilon^{-1} \log E(\epsilon), \quad \text{so $Y^{2}_{1}(t) = E$}.
\end{align*}
By Lemma \ref{L4}, $m\epsilon \to -\log \left( 1 - \frac{B}{E}\right)$, hence 
$$\log(\hat L^{2}_{2}(t, \epsilon)) =\ell_{2,2}(t, \epsilon)\epsilon - (t+1)\log \epsilon^{-1} \to \log\left( A \left( 1- \frac{B}{E} \right) \frac{C}{B}\right),$$ which is consistent with 
 $\hat L^{2}_{2}(t) = Y^{2}_{2}(t-1)\left(1- \frac{Y^{1}_{1}(t-1)}{Y^{2}_{1}(t)}\right)\frac{Y^{1}_{1}(t)}{Y^{1}_{1}(t-1)}$. 
Indeed, we have
\begin{align*}
\hat L^{2}_{2}(t) &= \hat L^{2}_{2}(t-1)a^{2}_{t}\frac{\hat L^{2}_{1}(t-1) \hat L^{1}_{1}(t)}{\hat L^{1}_{1}(t-1) \hat L^{2}_{1}(t)}  = \hat L^{2}_{2}(t-1)\frac{\hat L^{2}_{1}(t) - \hat L^{1}_{1}(t-1)}{\hat L^{2}_{1}(t-1)} \frac{\hat L^{2}_{1}(t-1) \hat L^{1}_{1}(t)}{\hat L^{1}_{1}(t-1) \hat L^{2}_{1}(t)} 
\\& = \hat L^{2}_{2}(t-1) \left(1- \frac{\hat L^{1}_{1}(t-1)}{\hat L^{2}_{1}(t)}\right)\frac{\hat L^{1}_{1}(t)}{\hat L^{1}_{1}(t-1)}.
\end{align*}

\proofa{d} Assume $j > 2$, $t > j-1$.
Condition on 
\begin{align*}
& \ell_{j,2}(t-1, \epsilon)  = (t-j+2)\epsilon^{-1}\log \epsilon^{-1} - \epsilon^{-1}\log A(\epsilon), \quad \text{so $Y^{j}_{2}(t-1) = A$};
\\& \ell_{j-1,1}(t-1, \epsilon)  = (t-j+1)\epsilon^{-1}\log \epsilon^{-1} - \epsilon^{-1} \log B(\epsilon), \quad \text{so $Y^{j-1}_{1}(t-1) = B$};
\\& \ell_{j-1,1}(t, \epsilon)  = (t-j+2)\epsilon^{-1}\log \epsilon^{-1} - \epsilon^{-1} \log C(\epsilon), \quad \text{so $Y^{j-1}_{1}(t) = C$};  
\\& \ell_{j,1}(t, \epsilon)  = (t-j+1)\epsilon^{-1}\log \epsilon^{-1} - \epsilon^{-1} \log E(\epsilon), \quad \text{so $Y^{j}_{1}(t) = E$}; 
\\& \ell_{j-1,2}(t-1, \epsilon)  = (t-j+3)\epsilon^{-1}\log \epsilon^{-1} - \epsilon^{-1} \log F(\epsilon), \quad \text{so $Y^{j-1}_{2}(t-1) = F$}.
\end{align*}
Denote by $m$ the independent move of the particle which is the second from the left on the $j$-th level.
This move is distributed according to 
$$\Om_{q, a_{j}\alpha_{t} q^{\ell_{j-1,1}(t-1, \epsilon) - \ell_{j,1}(t, \epsilon)}, 0}(m \mid \ell_{j-1,2}(t-1, \epsilon) - \ell_{j,2}(t-1, \epsilon)).$$ 
As in the previous case, by Lemma \ref{L4}, $m\epsilon \to - \log  \left( 1 - \frac{B}{E}\right)$.

Thus, we see that $\ell_{j,2}(t, \epsilon) = \ell_{j-1,2}(t-1, \epsilon) - M$, where $M$ is distributed according to $$\Om_{q^{-1},  q^{\ell_{j-1,1}(t, \epsilon) - \ell_{j-1,1}(t-1, \epsilon)}, q^{\ell_{j-1, 2}(t-1, \epsilon) - \ell_{j-1,1}(t-1, \epsilon)}}(M \mid \ell_{j-1,2}(t-1, \epsilon) - \ell_{j,2}(t-1, \epsilon) - m).$$ Hence by Lemma \ref{L3}, $M\epsilon \to \left( 1 - \frac{B}{E}\right) \frac{1}{F} \frac{C}{B}A + 1$. Therefore,
$$\log(\hat L^{j}_{2}(t, \epsilon)) = \ell_{j,2}(t, \epsilon) - (t-j+3) \epsilon^{-1}\log \epsilon^{-1} \to \log \left( \left( 1 - \frac{B}{E}\right)\frac{C}{B}A + F\right),$$ which is consistent with  $\hat L^{j}_{2}(t) =  \left( 1 - \frac{Y^{j-1}_{1}(t-1)}{Y^{j}_{1}(t)}\right)\frac{Y^{j-1}_{1}(t)}{Y^{j-1}_{1}(t-1)}Y^{j}_{2}(t-1) + Y^{j-1}_{2}(t-1)$.
Indeed, one checks that
\begin{align*}
\hat L^{j}_{2}(t) &= \hat L^{j}_{2}(t-1)a^{j}_{t}\frac{\hat L^{j}_{1}(t-1) \hat L^{j-1}_{1}(t)}{\hat L^{j-1}_{1}(t-1) \hat L^{j}_{1}(t)} + \hat L^{j-1}_{2}(t-1) \\&= \hat L^{j}_{2}(t-1) \frac{\hat L^{j}_{1}(t) - \hat L^{j-1}_{1}(t-1)}{\hat L^{j}_{1}(t-1)} 
 \frac{\hat L^{j}_{1}(t-1)\hat L^{j-1}_{1}(t)}{\hat L^{j-1}_{1}(t-1) \hat L^{j}_{1}(t)} +  \hat L^{j-1}_{2}(t-1) 
 \\&= \hat L^{j}_{2}(t-1)\left(1- \frac{\hat L^{j-1}_{1}(t-1)}{\hat L^{j}_{1}(t)}\right) \frac{\hat L^{j-1}_{1}(t)}{\hat L^{j-1}_{1}(t-1)} +  \hat L^{j-1}_{2}(t-1).
\end{align*}

\medskip

\proofstep{k-th edge from the left for $k>2$}

\proofa{a} Start with assuming that $j =k$, $t=1$. 
We have 
$\ell_{k,k}(1, \epsilon)  = \ell_{k-1, k-1}(1, \epsilon) \ +$ a jump $m$ distributed according to $\Om_{q, a_{k}\alpha_{1}, 0}(m \mid \infty)$. By Lemma \ref{L1}, $$\log(\hat L^{k}_{k}(1)) - \log(\hat L^{k-1}_{k-1}(1)) = \log \epsilon^{-1} - m\epsilon \to  \log \Gamma,$$ where $\Gamma = a_{1}^{k}$ is an independent random variable distributed according to $\GRV{\theta_{k} + \hat \theta_{1}}$, which is consistent with $\hat L^{k}_{k}(1) = a_{1}^{k}Y^{k-1}_{k-1}(1)$.

\proofa{b} Let $j > k$, $t= j-k+1$. We have $\ell_{j,k}(j-k+1, \epsilon)  = m_{1} + m_{2}$, where $m_1$ is an independent move distributed according to $\Om_{q, a_{j}\alpha_{j-k+1}, 0}(m_{1} \mid \ell_{j-1, k}(j-k, \epsilon))$, and $m_2$ is the push from the $(k-1)$-st particle from the left on the $(j-1)$-st level distributed according to 
$$\Om_{q^{-1}, q^{\ell_{j-1, k-1}(j-k+1, \epsilon)}, q^{\ell_{j-1,k}(j-k)}} 
(\ell_{j-1, k}(j-k, \epsilon) - m_{1} - m_{2} \mid \ell_{j-1, k}(j-k, \epsilon) - m_{1}).$$
Condition on 
\begin{align*}
& \ell_{j-1,k-1}(j-k+1, \epsilon) =  (k-1)\epsilon^{-1}\log \epsilon^{-1} - \epsilon^{-1} \log E(\epsilon), \quad \text{so $Y^{j-1}_{k-1}(j-k+1) = E$};
\\& \ell_{j-1,k}(j-k, \epsilon) =  k\epsilon^{-1}\log \epsilon^{-1} - \epsilon^{-1} \log F(\epsilon), \quad \text{so $Y^{j-1}_{k}(j-k) = F$}.
\end{align*}
By Lemma \ref{L1}, $\log \epsilon^{-1} - m_{1}\epsilon \to \log \Gamma$, where $\Gamma = a_{j-k+1}^{j}$ is an independent random variable distributed according to $\GRV{\theta_{j} + \hat \theta_{j-k+1}}$. By Lemma \ref{L3}, 
$$\log(\hat L^{j}_{k}(j-k+1, \epsilon)) = k \log \epsilon^{-1} - \ell_{j, k}(j-k+1, \epsilon)\epsilon \to \log F + 
\log \left(1 + \frac{E\Gamma}{F}\right) = \log (F + E\Gamma),$$ 
which is consistent with  $Y^{j}_{k}(j-k+1) = Y^{j-1}_{k}(j-k) + a^{j}_{j-k+1}Y^{j-1}_{k-1}(j-k+1)$.

\proofa{c} Assume $j =k$, $t \geq  k$. Condition on (for $1 \leq i \leq k-1$)
\begin{align*}
& \ell_{k-1, i}(t-1, \epsilon)  = (t-k+2i-1)\epsilon^{-1}\log \epsilon^{-1} - \epsilon^{-1} \log B_{i}(\epsilon), \quad \text{so $Y^{k-1}_{i}(t-1) = B_{i}$};
\\& \ell_{k-1, i}(t, \epsilon)  = (t-k+2i)\epsilon^{-1}\log \epsilon^{-1} - \epsilon^{-1} \log C_{i}(\epsilon), \quad \text{so $Y^{k-1}_{i}(t) = C_{i}$};  
\\& \ell_{k,i}(t-1, \epsilon)  = (t-k+2i-2)\epsilon^{-1}\log \epsilon^{-1} - \epsilon^{-1} \log D_{i}(\epsilon), \quad \text{so $Y^{k}_{i}(t-1) = D_{i}$};
\\& \ell_{k, i}(t, \epsilon)  = (t-k+2i-1)\epsilon^{-1}\log \epsilon^{-1} - \epsilon^{-1} \log E_{i}(\epsilon), \quad \text{so $Y^{k}_{i}(t) = E_{i}$}; 
\end{align*}
By Lemma \ref{L4}, the independent move of the rightmost particle on the $k$-th level times $\epsilon$ converges to $0$, while the push from the previous particles times $\epsilon$ and minus $\log \epsilon^{-1}$ converges to 
\begin{align*}
	\displaystyle -\log \left(\frac{\prod_{i=2}^{k-1}D_{i}\prod_{i=1}^{k-1}C_{i}}{\prod_{i=2}^{k-1}E_{i}\prod_{i=1}^{k-1}B_{i}} \left( 1-\frac{B_{1}}{E_{1}} \right) \right),
\end{align*}
which is consistent with $$Y^{k}_{k}(t) = \hat L^{k}_{k}(t-1)\frac{\prod_{i=1}^{k-1}D_{i}\prod_{i=1}^{k-1}C_{i}}{\prod_{i=1}^{k-1}E_{i}\prod_{i=1}^{k-1}B_{i}} \cdot \frac{E_{1} - B_{1}}{D_{1}}.$$ For $t=k$ we take $D_{1} = 1$.

\proofa{d} Let $j =k$, $k > t > 1$. Make the same conditioning as in the previous part, 
but with different ranges of indices:
$k-t+1 \leq i \leq k-1$ for $D_i$, $ E_i$, and $k-t \leq i \leq k-1$ for $B_i$, $C_i$. 
Take $B_{k-t}=D_{k-t+1} = 1$. 
The independent move of the rightmost particle on the $k$-th level times $\epsilon$ converges to $0$, while the push from the previous particles times $\epsilon$ and minus $\log \epsilon^{-1}$ converges to 
\begin{align*}
	\displaystyle -\log \left(\frac{\prod_{i=k-t+1}^{k-1}D_{i}\prod_{i=k-t}^{k-1}C_{i}}{\prod_{i=k-t+1}^{k-1}E_{i}\prod_{i=k-t}^{k-1}B_{i}}a^{k}_{t} \right),
\end{align*}
which is consistent with $$Y^{k}_{k}(t) = \hat L^{k}_{k}(t-1)\frac{\prod_{i=k-t+1}^{k-1}D_{i}\prod_{i=k-t+1}^{k-1}C_{i}}{\prod_{i=k-t+1}^{k-1}E_{i}\prod_{i=k-t+1}^{k-1}B_{i}}(E_{k-t+1}-B_{k-t+1}).$$

\proofa{e} Let $j >k$, $t \geq j$. Condition on
\begin{align*}
& \ell_{j-1,i}(t-1, \epsilon)  = (t-j+2i-1)\epsilon^{-1}\log \epsilon^{-1} - \epsilon^{-1} \log B_{i}(\epsilon), \quad \text{so $Y^{j-1}_{i}(t-1) = B_{i}$ for $1 \leq i \leq k$};
\\& \ell_{j-1,i}(t, \epsilon)  = (t-j+2i)\epsilon^{-1}\log \epsilon^{-1} - \epsilon^{-1} \log C_{i}(\epsilon),\quad \text{so $Y^{j-1}_{i}(t) = C_{i}$ for $1 \leq i \leq k-1$};  
\\& \ell_{j, i}(t-1, \epsilon)  = (t-j+2i-2)\epsilon^{-1}\log \epsilon^{-1} - \epsilon^{-1} \log D_{i}(\epsilon), \quad \text{so $Y^{j}_{i}(t-1) = D_{i}$ for $1 \leq i \leq k$};
\\& \ell_{j, i}(t, \epsilon)  = (t-j+2i-1)\epsilon^{-1}\log \epsilon^{-1} - \epsilon^{-1} \log E_{i}(\epsilon), \quad \text{so $Y^{j}_{i}(t) = E_{i}$ for $1 \leq i \leq k-1$}.
\end{align*} 
The independent move of the $k$-th particle from the left on the $j$-th level times $\epsilon$ converges to $0$. Denote by $m_{1}$ the distance from this particle to $\ell_{j-1,k}(t-1, \epsilon)$ after the push from the $(k-1)$-st particle from the left on the $(j-1)$-st level. Let $m_2$ be this distance after pushes from other particles. 
By Lemma \ref{L3}, $ m_{1} \epsilon \to \log \left(1 + \frac{C_{k-1}D_{k}}{B_{k-1}B_{k}} \right)$.
By Lemma \ref{L5}, $$m_{2} \epsilon \to \log \left(1 + \frac{C_{k-1}D_{k}}{B_{k-1}B_{k}} \frac{\prod_{i=2}^{k-1}D_{i}\prod_{i=1}^{k-2}C_{i}}{\prod_{i=2}^{k-1}E_{i}\prod_{i=1}^{k-2}B_{i}} \left( 1-\frac{B_{1}}{E_{1}} \right)   \right).$$ 
This is consistent with $$\hat L^{j}_{k}(t) = B_{k} + Y^{j}_{k}(t-1) \frac{\prod_{i=1}^{k-1}D_{i}\prod_{i=1}^{k-1}C_{i}}{\prod_{i=1}^{k-1}E_{i}\prod_{i=1}^{k-1}B_{i}} \frac{E_{1} - B_{1}}{D_{1}}.$$

\proofa{f} Finally, assume that $j >k$, $j > t>j-k+1$. 
Make the same conditioning as in the previous part, but 
with different ranges of indices:
$j-t+1 \leq i \leq k$ for $D_i$, $D_{j-t+1} =1$, $j-t+1 \leq i \leq k-1$ for $E_i$, $j-t \leq i \leq k$ for $B_i$, $B_{j-t} =1$,
and $j-t \leq i \leq k-1$ for $C_i$.  The independent move of the $k$-th particle from the left on the $j$-th level times $\epsilon$ converges to $0$. Denote by $m_{1}$ the distance from this particle to $\ell_{j-1,k}(t-1, \epsilon)$ after the push from the $(k-1)$-st particle from the left on the $(j-1)$-st level. Let $m_2$ be 
this distance after pushes from other particles. By Lemma \ref{L3}, $ m_{1} \epsilon \to \log \left(1 + \frac{C_{k-1}D_{k}}{B_{k-1}B_{k}} \right)$. By lemma \ref{L5}, 
$$m_{2} \epsilon \to \log \left(1 + \frac{C_{k-1}D_{k}}{B_{k-1}B_{k}} \frac{\prod_{i=j-t+1}^{k-1}D_{i}\prod_{i=j-t}^{k-2}C_{i}}{\prod_{i=j-t+1}^{k-1}E_{i}\prod_{i=j-t}^{k-2}B_{i}} a^{j}_{t}   \right).$$ This is consistent with $$\hat L^{j}_{k}(t) = B_{k} + Y^{j}_{k}(t-1) \frac{\prod_{i=j-t+1}^{k-1}D_{i}\prod_{i=j-t+1}^{k-1}C_{i}}{\prod_{i=j-t+1}^{k-1}E_{i}\prod_{i=j-t+1}^{k-1}B_{i}} (E_{j-t+1} - B_{j-t+1}).$$

This completes the proof of Theorem \ref{thm:TL}.




\bibliography{bib}
\bibliographystyle{plain}

\end{document}